\documentclass[11pt,letterpaper]{amsart}

\usepackage{amssymb,amscd,amsthm, verbatim,amsmath,color,fancyhdr, mathrsfs}
\usepackage{graphicx}
\usepackage{tikz}
\usepackage{tikz-cd}
\usepackage{bm}
\usetikzlibrary{cd}
\usetikzlibrary{shapes.geometric,positioning} 
\usepackage{appendix}
\usepackage[foot]{amsaddr}
\usepackage[colorlinks=true,allcolors=black]{hyperref}
\makeatletter \def\l@subsection{\@tocline{2}{0pt}{1pc}{5pc}{}} \def\l@subsection{\@tocline{2}{0pt}{2pc}{6pc}{}} \makeatother

\usepackage[letterpaper, left=2.5cm, right=2.5cm, top=2.5cm,
bottom=2.5cm,dvips]{geometry}

\newtheorem{thm}{Theorem}[subsection]
\newtheorem{prop}[thm]{Proposition}
\newtheorem{lem}[thm]{Lemma}

\newtheorem{cor}[thm]{Corollary}

\newtheorem{exmp}[thm]{Example}
\newtheorem{sublem}[thm]{Sublemma}

\theoremstyle{definition}
\newtheorem{defn}[thm]{Definition}
\newtheorem{assumption}[thm]{Assumption}
\newtheorem{notat}[thm]{Notation}
\theoremstyle{remark}
\newtheorem*{remark}{Remark}

\newcommand\textcat[1]{\textnormal{\textbf{#1}}}
\newcommand\texthom[1]{\underline{\textnormal{#1}}}
\title{A Representability Theorem for Stacks in Derived Geometry Contexts}

\author{Rhiannon Savage}
\address{Rhiannon Savage, University College London}
\email{rhiannon.savage@ucl.ac.uk}

\date{}
\begin{document}
\begin{abstract}The representability theorem for stacks, due to Artin in the underived setting and Lurie in the derived setting, gives conditions under which a stack is representable by an $n$-geometric stack. In recent work of Ben-Bassat, Kelly, and Kremnizer, a new theory of derived analytic geometry has been proposed as geometry relative to the $(\infty,1)$-category of simplicial commutative Ind-Banach $R$-modules, for $R$ a Banach ring. In this paper, we prove a representability theorem which holds in a very general context, which we call a representability context, encompassing both the derived algebraic geometry context of To\"en and Vezzosi and these new derived analytic geometry contexts. The representability theorem gives natural and easily verifiable conditions for checking that derived stacks in these contexts are $n$-geometric, such as having an $n$-geometric truncation, being nilcomplete, and having an obstruction theory. Future work will explore representability of certain moduli stacks arising in derived analytic geometry, for example moduli stacks of Galois representations. 
\end{abstract}
\date{}

\maketitle
\vskip.5in

\tableofcontents

\section{Introduction}
\addtocontents{toc}{\setcounter{tocdepth}{-1}}

Representability is one of the most fundamental concepts in algebraic geometry and algebraic topology. Classically, a representability theorem aims to give conditions under which certain functors are representable, i.e. in the image of the Yoneda embedding. If a functor is representable by an object $X$ then we can study (i.e. do geometry or topology on) the functor by just studying $X$. We can interpret this idea in several different settings, and also provide higher and derived analogues.

In algebraic geometry, we want to give conditions under which functors are represented by geometric objects such as schemes and stacks. For example, the functor of points approach to algebraic geometry associates schemes with the Zariski sheaves they represent. Thanks to the work of Artin, Lurie, and To\"en and Vezzosi we now have very strong representability theorems for stacks in derived algebraic geometry. A representability theorem for a certain model of derived analytic geometry in the style of Lurie was proven by Porta and Yue-Yu \cite{porta_representability_2020}.

The aim of this paper is to unite and extend these representability theorems using a very general framework. In particular, we would like to have a representability theorem which holds in the model of derived analytic geometry recently proposed by Ben-Bassat, Kremnizer, and Kelly \cite{ben-bassat_perspective_2024}, building on previous work of Ben-Bassat and Kremnizer \cite{ben-bassat_non-archimedean_2017}\cite{ben-bassat_frechet_2023}, along with work with Bambozzi \cite{bambozzi_dagger_2016}\cite{bambozzi_stein_2018}\cite{bambozzi_sheafyness_2024} and Mukherjee \cite{kelly_analytic_2022}. This theory of derived analytic geometry is done relative, in the sense of To\"en and Vezzosi \cite{toen_homotopical_2008}, to the $(\infty,1)$-category of simplicial Ind-Banach $R$-modules for $R$ a Banach ring.

In algebraic and derived algebraic geometry, the main uses of these representability theorems are in proving that certain moduli problems are representable.  They are instrumental in, for example, constructing the derived moduli stack of elliptic curves \cite{lurie_survey_2009}. Other examples include the proof of representability of derived moduli stacks of polarised projective varieties, vector bundles, and abelian varieties \cite{pridham_derived_2012}. The representability theorem has also been used to prove representability of certain derived mapping stacks \cite{toen_homotopical_2008}, as well as representability of derived Picard stacks and derived Hilbert schemes \cite{lurie_derived_2012}.

\subsection*{The Artin-Lurie Representability Theorem}

In 1975, Artin \cite{artin_versal_1974} introduced his notion of an \textit{algebraic stack}, a generalisation of previous notions of Deligne and Mumford \cite{deligne_irreducibility_1969}. The main difference being that Deligne-Mumford algebraic stacks correspond to the case where you consider only \'etale morphisms, whereas Artin algebraic stacks consider smooth morphisms. Artin also provided a representability theorem; giving conditions for a functor valued in groupoids to be representable by an algebraic stack. These included certain conditions, such as possessing an obstruction theory, notably harder to define in an underived setting, as well as certain conditions on the diagonal morphism. 

The remaining challenge was to generalise the representability theorem to the derived setting. There are two main opposing theories for derived algebraic geometry, but they start off from a similar point of view. We start with the notion of a \textit{generalised ring} from which we can build up a theory of \textit{derived schemes}. One approach, due to Lurie \cite{lurie_higher_2017}, uses $\mathbb{E}_\infty$-rings. The second approach, of To\"en and Vezzosi \cite{toen_homotopical_2008}, uses simplicial commutative rings. These two approaches both have their own advantages depending on whether one is coming from an algebraic or topological viewpoint.

In \cite{lurie_derived_2012}, Lurie proved an analogue of Artin's theorem in the derived setting which he called the \textit{Spectral Artin Representability Theorem}. This theorem provides conditions for an $(\infty,1)$-functor $\mathcal{F}:\textcat{CAlg}^{cn}\rightarrow{\mathcal{S}}$ from the category of connective $\mathbb{E}_\infty$-algebras  to the category of spaces to be representable by a spectral Deligne-Mumford $n$-stack. These conditions include similar ones to before, as well as additional conditions endowing $\mathcal{F}$ with a `well-behaved deformation theory' \cite[p. 22]{lurie_derived_2012}. The conditions of Lurie are in general hard to verify. However, in \cite{pridham_representability_2012}, Pridham simplifies these conditions, providing several variants of the theory.

\subsection*{A Relative Approach to Derived Algebraic Geometry}

We note that, simplistically, geometry can be thought of as the combination of algebra and topology. For example, schemes can be glued together from affine schemes using the Zariski topology. In \cite{toen_au-dessous_2009}, To\"en and Vaqui\'e use this approach to define the theory of \textit{algebraic geometry relative to a category}. There is the following equivalence of categories
\begin{equation*}
    \textnormal{Aff}\simeq \textnormal{CRing}^{op}\simeq \textnormal{Comm}(\textnormal{Ab})^{op}
\end{equation*}where $\textnormal{Comm}(\textnormal{Ab})$ denotes the category of commutative monoids in $\textnormal{Ab}$. Therefore, to obtain geometry relative to some nice symmetric monoidal category $\textnormal{C}$ we can just replace $\textnormal{Ab}$ with $\textnormal{C}$, define the following category of affines
\begin{equation*}
    \textnormal{Aff}_\textnormal{C}:=\textnormal{Comm}(\textnormal{C})^{op}
\end{equation*}and then endow it with some suitable Grothendieck topology. 

In \cite{toen_homotopical_2005} and \cite{toen_homotopical_2008}, To\"en and Vezzosi generalise this idea to their theory of \textit{homotopical algebraic geometry}, geometry relative to a nice symmetric monoidal model category equipped with additional structure making it a \textit{homotopical algebraic geometry (HAG) context} \cite[Definition 1.3.2.13]{toen_homotopical_2008}. Suppose we have a fixed commutative ring $k$. Then, classical algebraic geometry can be recovered as geometry relative to the category of $k$-modules. To\"en and Vezzosi propose that homotopical algebraic geometry relative to the category of simplicial $k$-modules provides a good theory of derived algebraic geometry. Now the category of affine objects, which we will denote by $\textnormal{DAff}^{cn}_k$, is the opposite model category of the category of commutative simplicial $k$-algebras. They also provide general notions of flat, \'etale and smooth morphisms, as well as Zariski open immersions, which, in the case of derived algebraic geometry, all have explicit descriptions as `strong' versions of the standard notions.

The theory of algebraic $n$-stacks, or \textit{$n$-geometric stacks}, was developed by Simpson \cite{simpson_algebraic_1996} in 1996. He noticed that Artin's notion of an algebraic stack could be renamed an algebraic $1$-stack and that a theory of algebraic $n$-stacks could be inductively built up from there. He decided to call these stacks $n$-geometric stacks rather than algebraic $n$-stacks because he noticed they were more geometrically versatile as they looked locally like schemes.

In \cite{toen_homotopical_2008}, To\"en and Vezzosi generalise the notion of an $n$-geometric stack to the homotopy category of stacks on a model category endowed with a HAG structure. In particular, when one works relative to the category $\textnormal{Mod}_k$ of modules over a fixed commutative base ring $k$, one recovers Simpson's algebraic $n$-stacks. To\"en and Vezzosi specialised the representability theorem of Lurie to \textit{$D^-$-stacks}, objects in the homotopy category of stacks over the model category $\textnormal{DAff}_k$ endowed with the model \'etale topology. Their criteria gives conditions under which a $D^-$-stack is an $n$-geometric stack. We state it in full here. 

\begin{thm}\label{reptheoremtoen}
\cite[Theorem C.0.9]{toen_homotopical_2008} Let $\textnormal{F}$ be a $D^-$-stack. The following conditions are equivalent. \begin{enumerate}
    \item $\textnormal{F}$ is an $n$-geometric $D^-$-stack, 
    \item $\textnormal{F}$ satisfies the following three conditions:
    \begin{enumerate}
        \item The truncation $t_0(\textnormal{F})$ is an Artin $(n+1)$-stack, 
        \item $\textnormal{F}$ has an obstruction theory relative to $\textnormal{sMod}_{k,1}$, 
        \item For any $A\in \textnormal{sAlg}_k$, the natural morphism
        \begin{equation*}
            \mathbb{R}\textnormal{F}(A)\rightarrow{Holim}_n \mathbb{R}\textnormal{F}(A_{\leq n})
        \end{equation*}is an isomorphism in $\textnormal{Ho}(\textnormal{sSet})$. 
    \end{enumerate}
\end{enumerate}
\end{thm}

\subsection*{Derived Analytic Geometry}

It would be useful to provide a similar categorical framework for derived analytic geometry. Indeed, Lurie has an approach to derived complex analytic geometry, detailed in \cite[Section 4.4]{lurie_derived_2009-2}, built up using the pregeometry associated to the $(\infty,1)$-category of finite dimensional Stein manifolds. 

Porta and Yue Yu's formulation of derived analytic geometry is based on this theory of Lurie's but also includes an approach to derived non-archimedean analytic geometry \cite{porta_derived_2018}. Their representability theorem \cite[Theorem 7.1]{porta_representability_2020} shares many similarities with To\"en and Vezzosi's. They instead consider a notion of an $n$-geometric stack $\mathcal{F}$ with respect to the $(\infty,1)$-category $\textcat{dAfd}_k$ of derived $k$-affinoid spaces when $k$ is non-archimedean, or the $(\infty,1)$-category $\textcat{dAfd}_\mathbb{C}$ of derived Stein spaces in the case when $k=\mathbb{C}$, see \cite[Definition 2.4]{porta_representability_2020}.  

There are several limitations to the Porta and Yue Yu theory, as detailed in \cite{ben-bassat_perspective_2024}. In particular, it is hard to formulate a definition and work with quasi-coherent sheaves on their derived analytic spaces. Recent work of Ben-Bassat, Kelly, and Kremnizer circumnavigates these problems by working with a more relative approach, motivated by the work of To\"en and Vezzosi \cite{toen_homotopical_2008} and Raksit \cite{raksit_hochschild_2020}. 

In \cite{ben-bassat_perspective_2024}, they present a formulation of geometry relative to an \textit{$(\infty,1)$-algebra context}. This is essentially a locally presentable symmetric monoidal $(\infty,1)$-category $\mathcal{C}$ equipped with a graded monad $\textcat{D}$. Our algebras will then be algebras over this monad. Most of the examples we will be interested in will be \textit{derived algebraic contexts} in the sense of Raksit \cite[Definition 4.2.1]{raksit_hochschild_2020} where we use the $\textcat{LSym}$-monad. When equipped with a suitable Grothendieck topology $\bm\tau$ and a class of maps $\textcat{P}$, these contexts generalise the HAG contexts of To\"en and Vezzosi. In particular, we can define the notion of an $n$-geometric stack in these contexts. 

In choosing an appropriate category to work with for derived analytic geometry, some natural choices would include, for $R$ a Banach ring, the category $\textnormal{Ban}_R$ of Banach modules over $R$, or $\textnormal{Fr}_R$, the category of Fr\'echet modules over $R$. Indeed, we note that any affinoid algebra and any Stein algebra can be naturally considered as a Fr\'echet algebra. However, the category $\textnormal{Fr}_R$ has several limitations, as it is neither complete nor cocomplete and does not have enough projectives. The category $\textnormal{Fr}_R$ sits fully faithfully within the category $\textnormal{CBorn}_R$, the category of complete bornological spaces over $R$, see Appendix \ref{appendixbornological}. This category is an elementary closed symmetric monoidal exact, in fact quasi-abelian, category with symmetric projectives. Moreover, it sits fully faithfully within the category $\textnormal{Ind}(\textnormal{Ban}_R)$, the category of Ind-Banach spaces over $R$. Therefore, analogously to how we pass from rings to simplicial commutative rings when passing from algebraic geometry to derived algebraic geometry, we could consider derived analytic geometry as geometry relative to the categories of simplicial objects in $\textnormal{Ind}(\textnormal{Ban}_R)$ or $\textnormal{CBorn}_R$, for $R$ a Banach ring. This is the perspective taken in \cite{ben-bassat_perspective_2024}.

\subsection*{A Representability Theorem for Derived Stacks}

In Section \ref{chapter1}, we begin by fixing our notation for what we mean by an $\infty$-stack, which should be an $(\infty,1)$-presheaf valued in $\infty$-groupoids, or spaces, satisfying descent for certain hypercovers. We then prove certain useful results concerning functoriality of stacks which will be used in Section \ref{chapter5} when we are considering truncations of stacks. 

In Section \ref{hagsection}, we introduce the notion of an $(\infty,1)$-relative geometry tuple $(\mathcal{M},\bm\tau,\textcat{P},\mathcal{A})$. This bears many similarities to the homotopical algebraic geometry contexts of To\"en and Vezzosi \cite{toen_homotopical_2008}. The category $\mathcal{M}$ will be a large ambient category in which our affines of interest, $\mathcal{A}$, lie as a full subcategory. The topology $\bm\tau$ and the class of maps $\textcat{P}$ allow us to define what we mean by an $n$-geometric stack. These are inductively defined starting with the definition that a $(-1)$-geometric stack on $\mathcal{A}$ is one of the form $\textnormal{Map}_\mathcal{M}(-,X)$ for some $X\in\mathcal{M}$. We also explore conditions under which $n$-geometric stacks are preserved under adjunctions. In particular, several conditions can be relaxed if $\mathcal{A}$ is closed under $\bm\tau$-descent, which in particular implies that any stack with an $n$-atlas is $n$-geometric. 

In Section \ref{chapter3}, we state the notion of a \textit{derived algebraic context}, due to Raksit \cite{raksit_hochschild_2020}. This consists of a stable $(\infty,1)$-category $\mathcal{C}$ equipped with a $t$-structure $(\mathcal{C}_{\geq 0},\mathcal{C}_{\leq 0})$ such that $\mathcal{C}_{\geq 0}$ is generated by a certain set of compact projective generators $\mathcal{C}^0$. In particular, we note that if $\textnormal{E}$ is a nice enough exact category, see Theorem \ref{modelderivedtheorem}, endowed with a collection $\textnormal{P}^0$ of compact projective generators, then the category of simplicial objects in $\textnormal{E}$ forms a homotopical algebraic context, \cite[Definition 1.1.0.11]{toen_homotopical_2008}. Furthermore,
\begin{equation*}
    (\textcat{L}^H(\textnormal{Ch}(\textnormal{E})),\textcat{L}^H(\textnormal{Ch}_{\geq 0}(\textnormal{E})),\textcat{L}^H(\textnormal{Ch}_{\leq 0}(\textnormal{E})),\textcat{L}^H(\textnormal{P}^0))
\end{equation*}is a derived algebraic context, where $\textcat{L}^H(-)$ denotes the associated $(\infty,1)$-category. In particular, when we take $\textnormal{E}=\textnormal{Mod}_k$ for $k$ a commutative ring, then we obtain the context where the connective objects are simplicial commutative $k$-modules. The exact category $\textnormal{E}=\textnormal{Ind}(\textnormal{Ban}_R)$ also satisfies the required properties on $\textnormal{E}$. 

In these derived algebraic contexts, we can construct a monad $\textcat{LSym}_\mathcal{C}$ and hence obtain an $(\infty,1)$-category of \textit{derived commutative algebra objects}, which we denote by $\textcat{DAlg}(\mathcal{C})$. Motivated by the constructions in \cite{lurie_higher_2017} and \cite{toen_homotopical_2008}, we can define the notions of square-zero extensions, cotangent complexes, and infinitesimal extensions for the subcategory $\textcat{DAlg}^{cn}(\mathcal{C})$ of connective algebra objects. These satisfy the expected properties, and in particular allow us to define notions of formally smooth, \'etale, and perfect morphisms as in \cite{toen_homotopical_2008}, generalising the classical notions. 

In Section \ref{obstructionsection}, we combine our definitions of derived algebraic contexts and $(\infty,1)$-relative geometry tuples to obtain what we call a \textit{derived geometry context} $(\mathcal{C},\mathcal{C}_{\geq 0},\mathcal{C}_{\leq 0},\mathcal{C}^0,\bm\tau,\textcat{P},\mathcal{A},\textcat{M})$, where $\textcat{M}$ is a \textit{good system of categories of modules on $\mathcal{A}$}, in the sense of Definition \ref{goodsystemmodules}. In this setting, we can define cotangent complexes of stacks. If a stack $\mathcal{F}$ has a \textit{global cotangent complex} and preserves infinitesimal extensions, then we say that $\mathcal{F}$ has an \textit{obstruction theory}, resulting in it having nice lifting properties with respect to infinitesimal extensions. In Section \ref{obstructionsection}, we also prove one of the main results of the paper
\begin{thm}(Theorem \ref{obstructiongeometric})
    Suppose that $\bm\tau$ and $\textcat{P}$ satisfy the obstruction conditions of Definition \ref{artinsconditions} relative to $\mathcal{A}$ for a class of morphisms $\textcat{S}$. If $f:\mathcal{F}\rightarrow\mathcal{G}$ is an $n$-representable morphism of stacks in $\textcat{Stk}(\mathcal{A},\bm\tau|_\mathcal{A})$, then it has an obstruction theory. 
\end{thm}In particular, in this setting any $n$-geometric stack has an obstruction theory. As a result of this theorem, we can prove conditions under which $n$-geometric stacks are \textit{nilcomplete}, i.e. compatible with Postnikov towers. 

In Section \ref{chapter5}, we introduce our notion of a \textit{representability context} and prove our representability theorem in these contexts. We note that there is an adjunction 
\begin{equation*}
    \iota:\textcat{DAff}^\heartsuit(\mathcal{C}):=\textcat{DAlg}^\heartsuit(\mathcal{C})^{op}\leftrightarrows\textcat{DAlg}^{cn}(\mathcal{C})^{op}=:\textcat{DAff}^{cn}(\mathcal{C}):t_0
\end{equation*}where $t_0$ is the truncation functor arising from the $t$-structure on $\mathcal{C}$. Our representability context, see Definition \ref{repcontextdefn}, consists of a derived geometry context $(\mathcal{C},\mathcal{C}_{\geq 0},\mathcal{C}_{\leq 0},\mathcal{C}^0,\bm\tau,\textcat{P},\mathcal{A},\textcat{M})$ along with a class of morphisms $\textcat{S}$ satisfying certain conditions. In particular, we have a strong relative $(\infty,1)$-geometry tuple $(\mathcal{A},\bm\tau|_\mathcal{A},\textcat{P}|_\mathcal{A},\mathcal{A}^\heartsuit)$ where $\mathcal{A}^\heartsuit$ is the subcategory of $\textcat{DAff}^\heartsuit(\mathcal{C})$ consisting of truncations of objects in $\mathcal{A}$. We also have other natural expected conditions such as a `strongness' assumption on morphisms in $\textcat{P}$, and the assumption that $\bm\tau$ and $\textcat{P}$ satisfy the obstruction conditions relative to $\mathcal{A}$ for the class $\textcat{S}$ of morphisms. 

We prove the following representability theorem by lifting an atlas of the truncated stack $t_0(\mathcal{F})$ to an atlas for $\mathcal{F}$. 
\begin{thm}(Theorem \ref{representabilitytheorem}) Suppose that $(\mathcal{C},\mathcal{C}_{\geq 0},\mathcal{C}_{\leq 0},\mathcal{C}^0,\bm\tau,\textcat{P},\mathcal{A},\textcat{M},\textcat{S})$ is a representability context and that $\mathcal{F}$ is a stack in $\textcat{Stk}(\mathcal{A},\bm{\tau}|_{\mathcal{A}})$. The following conditions are equivalent. 
\begin{enumerate}
    \item $\mathcal{F}$ is an $n$-geometric stack relative to $\mathcal{A}$,
    \item $\mathcal{F}$ satisfies the following three conditions:
    \begin{enumerate}
\item The truncation $t_0(\mathcal{F})$ is an $n$-geometric stack relative to $\mathcal{A}^\heartsuit$,
        \item $\mathcal{F}$ has an obstruction theory relative to $\mathcal{A}$,
        \item For every $X=\textnormal{Spec}(A)\in\mathcal{A}$, $\mathcal{F}(A)\simeq\varprojlim_k\mathcal{F}(A_{\leq k})$.
        \end{enumerate}
\end{enumerate}
\end{thm}

Finally, in Section \ref{topologiessection}, we give several examples of these representability contexts. In particular, when we choose our topology to be a finite homotopy monomorphism topology $\bm{hm}^{fin}$ and when we choose our class of maps $\textcat{fP}^{\bm{hm}^{fin}}$ to be maps which, locally for this topology, are formally perfect, then we can simplify the conditions for a tuple to be a representability context. In particular, in the case where $\mathcal{C}=\textcat{Ch}(\textnormal{Ind}(\textnormal{Ban}_\mathbb{C}))$, we obtain a suitable representability context to discuss complex analytic geometry. 

\subsection*{Applications}

In future work, we will be motivated primarily by studying derived moduli stacks. Indeed, we can make the following simple definitions. Suppose that we have a category of stacks $\textcat{Stk}(\mathcal{C},\bm\tau)$ where $\mathcal{C}$ is an $(\infty,1)$-category and $\bm\tau$ is a Grothendieck pre-topology on $\textnormal{Ho}(\mathcal{C})$. Suppose that $G=\textnormal{Spec}(H)$ is a group object, i.e. $H$ is a Hopf algebra.

\begin{defn}Suppose that $\mathcal{F}$ is a stack in $\textcat{Stk}(\mathcal{C},\bm\tau)$ and $G$ is a group object. Let
\begin{equation*}
    {G}\times\mathcal{F}\rightarrow{\mathcal{F}}
\end{equation*}be an action of $G$ on $\mathcal{F}$. Then, define the \textit{quotient stack}
\begin{equation*}
    [\mathcal{F}/\mathcal{G}]:=\varinjlim_{n\geq 0}G^n\times\mathcal{F}
\end{equation*}

\end{defn}
This allows us to define the stack $\textcat{BG}:=[*/G]$ and the \textit{moduli stack of $G$-bundles on a stack}
\begin{equation*}
    \textcat{Bun}_G(\mathcal{F}):=\textnormal{Map}_{\textcat{Stk}(\mathcal{C},\bm\tau)}(\mathcal{F},\textcat{BG})
\end{equation*}
Moreover, we can define the \textit{moduli stack of $n$-dimensional representations of $G$} by 
\begin{equation*}
    \textcat{Rep}_n(G):=\textnormal{Map}_{\textcat{Stk}(\mathcal{C},\bm\tau)}(\textcat{BG},\textcat{BGL}_n)
\end{equation*}
In future research, we will explore conditions under which, given such a representability context $(\mathcal{C},\mathcal{C}_{\geq 0},\mathcal{C}_{\leq 0},\mathcal{C}^0,\bm\tau,\textcat{P},\mathcal{A},\textcat{M},\textcat{S})$ and a stack $\mathcal{F}\in\textcat{Stk}(\mathcal{A},\bm\tau|_{\mathcal{A}})$, we get representable moduli stacks of $G$-bundles or moduli stacks of $G$-representations. In particular, we will look at moduli stacks of Galois representations. Other examples where we could apply our representability theorem could be in studying the moduli stack of instantons, the moduli stack of complex structures on Riemannian manifolds, or the moduli stack of solutions to non-linear elliptic Fredholm partial differential equations.

\subsection*{Connections with Condensed Mathematics}

Some differences and similarities between the theories of bornological and condensed mathematics are highlighted in \cite{ben-bassat_perspective_2024}. We note that the category of analytic rings, in the sense of Clausen and Scholze \cite{scholze_lectures_nodate} is not quite a full subcategory of commutative monoids in some symmetric monoidal category and therefore cannot be dealt with in a similar way. However, we expect that the theory of geometric stacks lines up in the analytic setting. There is yet to be a representability theorem in this setting.

\subsection*{Acknowledgements}

I am very grateful for the support and guidance of Jack Kelly and my supervisor Kobi Kremnizer throughout this whole project. Thanks also to Sof\'ia Marlasca Aparicio for several helpful discussions.

\subsection*{Funding}

This research was conducted during the author's PhD at the University of Oxford, funded by an EPSRC studentship [EP/W523781/1 - no. 2580843].

\addtocontents{toc}{\setcounter{tocdepth}{2}}

\section*{Notation}
In this paper we will use the following notation.
\begin{itemize}
    \item We will use normal font $\textnormal{C}$ and $\textnormal{Mod}_A$ to denote ordinary categories, and normal font $\textnormal{F}$ for ordinary stacks, 
    \item We will use mathcal font $\mathcal{C}$ and bold font $\textcat{Mod}_A$ to denote $(\infty,1)$-categories, and mathcal font $\mathcal{F}$ for higher stacks. 
    \item We will use the letters $A,B$ for elements of a category $\mathcal{C}$, and $X,Y$ for elements of a category $\mathcal{A}$. We tend to think of our category $\mathcal{A}$ as our relevant category of affines. 
\end{itemize}

\section{$(\infty,1)$-Categories and $\infty$-Stacks}\label{chapter1}

In this section, we will fix what we mean by an `$\infty$-stack', a notion which can mean many different things in the literature. We will then prove certain functoriality results. 

\subsection{$(\infty,1)$-Categories}

There are various models for an $(\infty,1)$-category, such as quasi-categories \cite{lurie_higher_2009}, Segal spaces \cite{segal_categories_1974}, complete Segal spaces \cite{rezk_model_2000} or simplicial categories \cite{toen_homotopical_2005}. These specific models all have different advantages depending on what context you are working in, but are equivalent to one another in the sense that they can be connected by chains of Quillen equivalent model categories \cite{bergner_survey_2006}.

In \cite{gaitsgory_study_2019}, Gaitsgory and Rozenblyum work with the philosophy that `one believes that the notion of an $(\infty,1)$-category exists, and all one needs to know is how to use the words correctly'. We will use this philosophy in this thesis and the reader should feel free to use whichever model for an $(\infty,1)$-category they feel most comfortable with. However, it will at times be clearer and more useful to fix a model for an $(\infty,1)$-category, and in this case we will take our $(\infty,1)$-categories to be simplicial categories along with the model structure described by Bergner in \cite[Theorem 1.1]{bergner_model_2006}.

We note that any \textit{relative category} $\textnormal{M}$, i.e. a category with weak equivalences, presents an $(\infty,1)$-category $\textcat{L}^H(\textnormal{M})$ known as its \textit{hammock localisation} \cite[Section 2]{dwyer_calculating_1980}. If $\textnormal{M}$ is a simplicial model category, then the hammock localisation $\textcat{L}^H(\textnormal{M})$ is connected by an equivalence of $(\infty,1)$-categories to the simplicial nerve $\textcat{N}(\textnormal{M}^\circ)$ (see \cite[Definition 1.1.5.5]{lurie_higher_2009}) of the subcategory $\textnormal{M}^\circ$ of $\textnormal{M}$ consisting of fibrant-cofibrant objects. 

There is in fact an intimate connection between combinatorial simplicial model categories and certain $(\infty,1)$-categories known as locally presentable $(\infty,1)$-categories \cite[Definition 5.5.0.1]{lurie_higher_2009}.
\begin{prop}\phantomsection\label{locallypresentablemodel}\cite[c.f. Proposition A.3.7.6]{lurie_higher_2009}\cite[c.f. Proposition 4.8] {dwyer_function_1980} Suppose that $\mathcal{M}$ is an $(\infty,1)$-category. Then, the following two conditions are equivalent, 
\begin{enumerate}
    \item The $(\infty,1)$-category $\mathcal{M}$ is locally presentable, 
    \item There exists a combinatorial simplicial model category $\textnormal{M}$ and an equivalence $\mathcal{M}\simeq \textcat{L}^H(\textnormal{M})$,
\end{enumerate}We will call $\mathcal{M}$ the underlying $(\infty,1)$-category of $\textnormal{M}$.

\end{prop}

We define the category of $\infty$-groupoids, $\infty\textcat{Grpd}$, to be $\textcat{L}^H(\textnormal{sSet}_Q)$ where $\textnormal{sSet}_Q$ denotes the category of simplicial sets endowed with the standard Quillen model structure.  There is a Quillen adjunction \begin{equation*}
     i:\textnormal{Set}_Q\leftrightarrows{\textnormal{sSet}_Q}:\pi_0
\end{equation*}
which induces an adjunction of $(\infty,1)$ categories
\begin{equation*}
     i:\textnormal{Set}\leftrightarrows{\infty\textcat{Grpd}}:\pi_0
\end{equation*}

We can use this definition of $\pi_0$ to define an associated ordinary category, the \textit{homotopy category}, to any $(\infty,1)$-category.

\begin{defn}\cite[c.f. Definition 1.1.3.2]{lurie_higher_2009}
Suppose that $\mathcal{M}$ is an $(\infty,1)$-category. Then, the \textit{homotopy category} $\textnormal{Ho}(\mathcal{M})$ has the same objects as $\mathcal{M}$ but, for $X,Y\in\mathcal{M}$, 
\begin{equation*}
    \textnormal{Hom}_{\textnormal{Ho}(\mathcal{M})}(X,Y)=\pi_0\textnormal{Map}_\mathcal{M}(X,Y)
\end{equation*}Composition of morphisms in $\textnormal{Ho}(\mathcal{M})$ is induced from composition of morphisms in $\mathcal{M}$ by applying $\pi_0$. 
\end{defn}

\subsection{Groupoid Objects}

Suppose that we have an $(\infty,1)$-category $\mathcal{M}$ with all $(\infty,1)$-pullbacks. Compare the following definition with that of a \textit{Segal groupoid object} in \cite[Definition 1.3.1.6]{toen_homotopical_2008}. 
\begin{defn}\cite[c.f. Proposition 6.1.2.6 (4'')]{lurie_higher_2009}
    A \textit{groupoid object in $\mathcal{M}$} is a simplicial object $X_*:\Delta^{op}\rightarrow{\mathcal{M}}$ such that, for every $n\geq 0$ and every partition $[n]=S\cup S'$ such that $S\cap S'$ consists of a single element $s$, the diagram 
    \begin{equation*}
        \begin{tikzcd}
            X_*([n])\arrow{r} \arrow{d} & X_*(S)\arrow{d}\\
            X_*(S') \arrow{r} & X_*(\{s\})
        \end{tikzcd}
    \end{equation*}is an $(\infty,1)$-pullback square in $\mathcal{M}$. 
\end{defn}

\begin{remark}
    For a simplicial object $X_*$, we denote $X_*([n])$ by $X_n$. We note that, for any groupoid object $X_*$ in $\mathcal{M}$ and any $n\geq 0$, there are equivalences $X_n\rightarrow{X_1\times_{X_0} \dots_{X_0} X_1}$ where we are considering an $n$-fold fibre product of $X_1$ over $X_0$. 
\end{remark}

\begin{defn}\cite[c.f. Proposition 6.1.2.11]{lurie_higher_2009} Suppose that we have a morphism $f:X\rightarrow{Y}$ in $\mathcal{M}$. An augmented simplicial object $\check{C}(f)_*\rightarrow{Y}$ is the \textit{\v{C}ech nerve of $f$} if it is a groupoid object of $\mathcal{M}$ and there is an $(\infty,1)$-pullback square in $\mathcal{M}$
\begin{equation*}
    \begin{tikzcd}
        \check{C}(f)_1\arrow{r} \arrow{d} & X \arrow{d}\\
        X\arrow{r} &  Y
    \end{tikzcd}
\end{equation*}
    
\end{defn}

Suppose that we have a simplicial object $X_*$. If the colimit exists, then we denote it by $|X_*|$ and call it the \textit{geometric realisation of $X_*$}. 
\begin{defn}\phantomsection\label{cecheffective}
    \cite[c.f. Corollary 6.2.3.5]{lurie_higher_2009} In any $(\infty,1)$-category $\mathcal{M}$ with $(\infty,1)$-pullbacks and all geometric realisations of simplicial objects, a morphism $f:X\rightarrow{Y}$ is an \textit{effective epimorphism} if and only if we have an equivalence $|\check{C}(f)_*|\rightarrow{Y}$. 
\end{defn}

\subsection{$\infty$-Stacks}\phantomsection\label{descentsubsection}

Suppose that $\mathcal{M}$ is an $(\infty,1)$-category. The \textit{category of $(\infty,1)$-presheaves on $\mathcal{M}$}, denoted by $\textcat{PSh}(\mathcal{M})$, is defined to be the $(\infty,1)$-functor category $\textcat{Fun}(\mathcal{M}^{op},\infty\textcat{Grpd})$. We note that there is a Yoneda embedding $h:\mathcal{M}\rightarrow{\textcat{PSh}(\mathcal{M})}$ defined on objects $X\in\mathcal{M}$ by $\textnormal{Map}_\mathcal{M}(-,X)$. 

In \cite[Definition 6.2.2.1]{lurie_higher_2009}, Lurie defines the notion of a \textit{Grothendieck topology} $\bm\tau$ on an $(\infty,1)$-category $\mathcal{M}$ using the notion of a \textit{sieve}. We note that, by \cite[Proposition 6.2.2.5]{lurie_higher_2009}, for each $X\in \mathcal{M}$, each sieve on $X$ corresponds to an equivalence class of monomorphisms $\{\mathcal{U}_i\rightarrow{h(X)}\}_{i\in I}$ in $\textcat{PSh}(\mathcal{M})$. Let $S$ denote a collection of representative monomorphisms corresponding to covering sieves in $\bm\tau$.

Classically, sheaves are presheaves whose values are determined by evaluating on covers. We can generalise this notion to the $(\infty,1)$-categorical setting. 

\begin{defn}\cite[c.f. Definition 6.2.2.6]{lurie_higher_2009} A presheaf $\mathcal{F}$ in $\textcat{PSh}(\mathcal{M})$ is \textit{$S$-local} if, for every monomorphism $\mathcal{U}\rightarrow{h(X)}$ in $S$, we have an equivalence of $\infty$-groupoids
   \begin{equation*}
       \mathcal{F}(X)\simeq \textnormal{Map}_{\textcat{PSh}(\mathcal{M})}(h(X),\mathcal{F})\rightarrow{\textnormal{Map}_{\textcat{PSh}(\mathcal{M})}(\mathcal{U},\mathcal{F})}
   \end{equation*}
The \textit{category of $(\infty,1)$-sheaves}, denoted $\textcat{Sh}(\mathcal{M},\bm{\tau})$, is the full subcategory of $\textcat{PSh}(\mathcal{M})$ spanned by $S$-local presheaves. 
\end{defn}

We note that the category $\textcat{PSh}(\mathcal{M})$ is complete and cocomplete by \cite[Corollary 5.1.2.4]{lurie_higher_2009}. The category of sheaves $\textcat{Sh}(\mathcal{M},\bm{\tau})$ is an $(\infty,1)$-topos by \cite[Section 6.5.2]{lurie_higher_2009}, so in particular it is a reflective subcategory of $\textcat{PSh}(\mathcal{M})$ and is also complete and cocomplete. Moreover, the sheafification functor $\textcat{PSh}(\mathcal{M})\rightarrow{\textcat{Sh}(\mathcal{M},\bm{\tau})}$ preserves colimits.

The definition of a Grothendieck topology on an $(\infty,1)$-category $\mathcal{M}$ lines up with the usual notion of a Grothendieck topology on its homotopy category $\textnormal{Ho}(\mathcal{M})$ by \cite[Remark 6.2.2.3]{lurie_higher_2009}. We fix the following notion of a \textit{Grothendieck pre-topology} on $\textnormal{Ho}(\mathcal{M})$. 

\begin{defn}
    A \textit{Grothendieck pre-topology}, $\bm\tau$, on $\textnormal{Ho}(\mathcal{M})$ is a collection $\bm\tau$ of families of maps $\{f_i:U_i\rightarrow{X}\}_{i\in I}$ such that 
    \begin{itemize}
        \item For any isomorphism $f_i$, we have that $\{f_i\}$ is in $\bm\tau$, 
        \item If $\{U_i\rightarrow{X}\}_{i\in I}$ is in $\bm\tau$ and $\{V_{i,j}\rightarrow{U_i}\}_{j\in J}$ is in $\bm\tau$ for each $i$, then the composition $\{V_{i,j}\rightarrow{X}\}_{(i,j)\in I\times J}$ is in $\bm\tau$, 
        \item If $\{U_i\rightarrow{X}\}_{i\in I}$ is in $\bm\tau$ and $V\rightarrow{X}$ is a morphism, then $U_i\times_X V$ exists and $\{U_i\times_X V\rightarrow{V}\}_{i\in I}$ is in $\bm\tau$.
    \end{itemize}We will call the covers in our Grothendieck pre-topology \textit{$\bm\tau$-covering families} or \textit{$\bm\tau$-covers}. We will call a category $\mathcal{M}$ equipped with a pre-topology $\bm\tau$ on $\mathcal{M}$ \textit{an $(\infty,1)$-site}.
\end{defn} 

Suppose that we have a Grothendieck pre-topology $\bm\tau$ on $\textnormal{Ho}(\mathcal{M})$. We note that $\bm\tau$ generates a Grothendieck topology on $\textnormal{Ho}(\mathcal{M})$ by defining the covering sieves to be those which contain $\bm\tau$-covering families. Conversely, when $\textnormal{Ho}(\mathcal{M})$ has pullbacks, any Grothendieck topology defines a pre-topology whose covering families are families of morphisms which generate covering sieves.

\begin{defn}\phantomsection\label{cechdescentdef}
A presheaf $\mathcal{F}$ in $\textcat{PSh}(\mathcal{M})$ satisfies \textit{(\v{C}ech-)descent for $\bm\tau$-covers} if, whenever we have a $\bm\tau$-covering family $\{U_i\rightarrow{X}\}_{i\in I}$, we have an equivalence of $\infty$-groupoids\begin{equation*}
       \mathcal{F}(X)\simeq \textnormal{Map}_{\textcat{PSh}(\mathcal{M})}(h(X),\mathcal{F})\rightarrow{\textnormal{Map}_{\textcat{PSh}(\mathcal{M})}(|\mathcal{U}_*|,\mathcal{F})}
   \end{equation*}where $\mathcal{U}_*\rightarrow{h(X)}$ is the \v{C}ech nerve of the morphism $\mathcal{U}=\coprod_{i\in I} h(U_i)\rightarrow{h(X)}$ in $\textcat{PSh}(\mathcal{M})$.
\end{defn}

\begin{prop}
    Let $(\mathcal{M},\bm\tau)$ be an $(\infty,1)$-site. The category of presheaves satisfying descent for $\bm\tau$-covers is precisely the category of sheaves $\textcat{Sh}(\mathcal{M},\bm\tau)$.
\end{prop}

\begin{proof}It suffices to show that $|\mathcal{U}_*|$ corresponds to the sieve generated by the covering family $\{U_i\rightarrow{X}\}_{i\in I}$. Indeed, this follows because $\mathcal{U}_*$ computes $(-1)$-truncations in $\textcat{PSh}(\mathcal{M})_{/h(X)}$ by \cite[Proposition 6.2.3.4]{lurie_higher_2009} and therefore defines a monomorphism $|\mathcal{U}_*|\rightarrow{h(X)}$, and hence a sieve on $X$ by \cite[Proposition 6.2.2.5]{lurie_higher_2009}. 
\end{proof}

We note however that the category of $(\infty,1)$-sheaves is not hypercomplete in the sense of \cite[Section 6.5.2]{lurie_higher_2009}, and therefore doesn't satisfy Whitehead's theorem. Given an $(\infty,1)$-site $(\mathcal{M},\bm\tau)$, we will define the category $\textcat{Stk}(\mathcal{M},\bm\tau)$ of $\infty$-stacks to be the hypercompletion of $\textcat{Sh}(\mathcal{M},\bm\tau)$. Equivalently, we can describe this category as follows. 

We will say that a morphism $f:\coprod_{i\in I} h(U_i)\rightarrow{\coprod_{j\in J}h(V_j)}$ in $\textcat{PSh}(\mathcal{M})$ is a \textit{generalised $\bm\tau$-cover} if, for each $j\in J$, the family of morphisms $\{U_{k}\rightarrow{V_j}\}_{k\in K}$ corresponds to a $\bm\tau$-cover \cite[c.f. p.13]{dugger_hypercovers_2001}, where $k\in K$ if there is a map $h(U_k)\rightarrow{h(V_j)}$.
\begin{defn}
    An augmented simplicial object $\mathcal{U}_*\rightarrow{h(X)}$ in $\textcat{PSh}(\mathcal{M})$ is a \textit{pseudo-representable $\bm\tau$-hypercover of $h(X)$} if each $\mathcal{U}_n$ is a coproduct of representables and, for each $n\geq 0$, the map $\mathcal{U}_n\rightarrow{(cosk_{n-1}\mathcal{U}_*)_n}$ corresponds to a generalised $\bm\tau$-cover of $(cosk_{n-1}\mathcal{U}_*)_n$.
\end{defn}

\begin{remark}
    We note that the \v{C}ech nerve $\mathcal{U}_*$ of the morphism $\coprod_{i\in I} h(U_i)\rightarrow{h(X)}$ associated to a $\bm\tau$-cover $\{U_i\rightarrow{X}\}_{i\in I}$ is a $\bm\tau$-hypercover of height zero; the morphisms $\mathcal{U}_n\rightarrow{(cosk_{n-1}\mathcal{U}_*)_n}$ are all isomorphisms.
\end{remark}

\begin{defn}
    \begin{enumerate}
        \item A presheaf $\mathcal{F}$ satisfies \textit{descent for $\bm\tau$-hypercovers} if, whenever we have a pseudo-representable $\bm\tau$-hypercover $\mathcal{U}_*\rightarrow{h(X)}$, we have an equivalence of $\infty$-groupoids
        \begin{equation*}
            \mathcal{F}(X)\simeq \textnormal{Map}_{\textcat{PSh}(\mathcal{M})}(|\mathcal{U}_*|,\mathcal{F})
        \end{equation*}
        \item The category of \textit{$\infty$-stacks} (or \textit{hypercomplete $(\infty,1)$-sheaves}) is the full subcategory of $\textcat{PSh}(\mathcal{M})$ consisting of presheaves satisfying descent for $\bm\tau$-hypercovers.
    \end{enumerate} 
\end{defn}

\begin{remark}
    We note that, by \cite[Proposition 6.5.2.14]{lurie_higher_2009}, if $\textnormal{M}$ is a small ordinary category equipped with a Grothendieck pre-topology $\bm\tau$ and $\textnormal{A}$ is the category of simplicial presheaves on $\textnormal{M}$ endowed with the local model structure, then $\mathcal{A}=\textcat{L}^H(\textnormal{A})$ is equivalent to $\textcat{Stk}(\mathcal{M},\bm\tau)$ where $\mathcal{M}=\textcat{L}^H(\textnormal{M})$. 
\end{remark}

The category of $\infty$-stacks is an $(\infty,1)$-topos, and is moreover a reflective subcategory of $\textcat{Sh}(\mathcal{M},\bm\tau)$ \cite[c.f. Section 6.5.2]{lurie_higher_2009}. Therefore, it is complete and cocomplete. Moreover, we note that the stackification functor $\textcat{PSh}(\mathcal{M})\rightarrow{\textcat{Stk}(\mathcal{M},\bm{\tau})}$ preserves colimits and finite limits.

\subsection{Continuous and Cocontinuous Functors}\phantomsection\label{continuousfunctorsection}

Recall that there is an $(\infty,1)$-adjunction
\begin{equation*}i:\textnormal{Set}\leftrightarrows\infty\textcat{Grpd}:\pi_0
\end{equation*}Suppose that $(\mathcal{M},\bm\tau)$ is an $(\infty,1)$-site and that $\textnormal{F}:\textnormal{Ho}(\mathcal{M})^{op}\rightarrow{\textnormal{Set}}$ is an ordinary presheaf. We will denote the ordinary categories of presheaves and sheaves by $\textnormal{PSh}(\textnormal{Ho}(\mathcal{M}))$ and $\textnormal{Sh}(\textnormal{Ho}(\mathcal{M}),\bm{\tau})$ respectively. We let $i^{\bm{\tau}}(\textnormal{F})$ be the $\infty$-stackification of the $(\infty,1)$-presheaf 
\begin{equation*}\begin{aligned}
    i:\mathcal{M}^{op}&\rightarrow{\infty\textcat{Grpd}}\\
    A&\rightarrow{i(\textnormal{F}(A))}
\end{aligned}\end{equation*}on the $(\infty,1)$-site $(\mathcal{M},\bm{\tau})$.

Now, suppose that $\mathcal{F}:\mathcal{M}^{op}\rightarrow{\infty\textcat{Grpd}}$ is an $(\infty,1)$-presheaf on $\mathcal{M}$. We let $\pi_0^{\bm{\tau}}(\mathcal{F})$ be the sheaf associated to the presheaf 
\begin{equation*}
    \begin{aligned}
    \pi_0:\textnormal{Ho}(\mathcal{M})^{op}&\rightarrow{\textnormal{Set}}\\
    A&\rightarrow{\pi_0(\mathcal{F}(A))}
    \end{aligned}
\end{equation*}on the site $(\textnormal{Ho}(\mathcal{M}),\bm{\tau})$. If we denote by $\tau_{\leq 0}$ the $0$-truncation functor defined in \cite[Proposition 5.5.6.18]{lurie_higher_2009} then we can easily see that $\pi_0^{\bm{\tau}}$ defines an equivalence of ordinary categories
\begin{equation*}
    \textnormal{Ho}(\tau_{\leq 0}(\textcat{Stk}(\mathcal{M},\bm{\tau})))\simeq \textnormal{Sh}(\textnormal{Ho}(\mathcal{M}),\bm{\tau})
\end{equation*}
By \cite[Proposition 7.2.1.14]{lurie_higher_2009}, a morphism in an $(\infty,1)$-topos $\mathcal{M}$ is an effective epimorphism precisely if its $0$-truncation is an epimorphism in $\textnormal{Ho}(\tau_{\leq 0}\mathcal{M})$. This motivates the following definition.

\begin{defn}\phantomsection\label{epimorphismstacks}
A morphism of stacks $f:\mathcal{F}\rightarrow{\mathcal{G}}$ in $\textcat{Stk}(\mathcal{M},\bm{\tau})$ is an \textit{epimorphism} if the induced morphism $\pi_0^{\bm{\tau}}(f):\pi_0^{\bm{\tau}}(\mathcal{F})\rightarrow{\pi_0^{\bm{\tau}}(\mathcal{G})}$ is an epimorphism in $\textnormal{Sh}(\textnormal{Ho(C)},\bm{\tau})$.
\end{defn}

Using \cite[Corollary 6.2.3.12]{lurie_higher_2009} and \cite[Proposition 6.2.3.14]{lurie_higher_2009}, we easily see that epimorphisms of stacks are stable by compositions, equivalences, and pullbacks. 

\begin{defn}
Suppose that $(\mathcal{M},\bm{\tau})$ and $(\mathcal{N},\bm\sigma)$ are $(\infty,1)$-sites, and that we have a functor $F:\mathcal{M}\rightarrow{\mathcal{N}}$. Then, $F$ is a \textit{continuous functor of $(\infty,1)$-sites} if 

\begin{enumerate}
    \item Whenever $\{U_i\rightarrow{X}\}_{i\in I}$ is a $\bm{\tau}$-cover, then $\{F(U_i)\rightarrow{F(X)}\}_{i\in I}$ is a $\bm\sigma$-cover, 
    \item For any morphism $X\rightarrow{Y}$ in $\mathcal{M}$ and every $\bm{\tau}$-cover $\{U_i\rightarrow{Y}\}_{i\in I}$, we have an isomorphism
    \begin{equation*}
        F(X\times_Y U_i)\simeq F(X)\times_{F(Y)} F(U_i)
    \end{equation*}
\end{enumerate}
\end{defn}

Suppose that $F:\mathcal{M}\rightarrow{\mathcal{N}}$ is a functor of $(\infty,1)$-categories and consider the precomposition functor $F^*:\textcat{PSh}(\mathcal{N})\rightarrow{\textcat{PSh}(\mathcal{M})}$ . Then, by left and right $(\infty,1)$-Kan extension, we obtain a chain of adjunctions $F_!\dashv F^*\dashv F_*$ on the level of presheaves.  Now, if we have a continuous functor $F:(\mathcal{M},\bm{\tau})\rightarrow{(\mathcal{N},\bm\sigma)}$ of $(\infty,1)$-sites, then $\bm\tau$-covers get sent to $\bm\sigma$-covers, and we obtain a functor on stacks
\begin{equation*}
    F^*:\textcat{Stk}(\mathcal{N},\bm\sigma)\rightarrow{\textcat{Stk}(\mathcal{M},\bm{\tau})}
\end{equation*}with left adjoint
\begin{equation*}
    F_\#:\textcat{Stk}(\mathcal{M},\bm{\tau})\rightarrow{\textcat{Stk}(\mathcal{N},\bm\sigma)}
\end{equation*}given by composing the stackification functor with $F_!$. We remark that, for a representable stack $\mathcal{F}\simeq\textnormal{Map}_\mathcal{M}(-,X)$, $F_\#$ acts by $F_\#(\mathcal{F})= \textnormal{Map}_\mathcal{N}(-,F(X))$. We refer the reader to \cite[Section 2]{porta_higher_2016} for full proofs of these statements. 

\begin{prop}\phantomsection\label{epimorphismfunctor}Suppose that $F:(\mathcal{M},\bm{\tau})\rightarrow{(\mathcal{N},\bm\sigma)}$ is a continuous functor of $(\infty,1)$-sites. Then, $F_\#$ preserves epimorphisms of $\infty$-stacks. 
\end{prop}

\begin{proof}
Suppose that $f:\mathcal{F}\rightarrow{\mathcal{G}}$ is an epimorphism of stacks. It suffices to show that $\pi_0^{\bm\sigma}(F_\#(f)):\pi_0^{\bm\sigma}(F_\#(\mathcal{F}))\rightarrow{\pi_0^{\bm\sigma}(F_\#(\mathcal{G}))}$ is an epimorphism in the category of sheaves. We note that, if $\mathcal{H}$ is any stack over an $(\infty,1)$-site $(\mathcal{N},\bm\sigma)$, then $\pi_0^{\bm\sigma}(\tilde{\mathcal{H}})\simeq \pi_0^{\bm\sigma}(\mathcal{H})$, where $\tilde{\mathcal{H}}$ is $\mathcal{H}$, considered as a presheaf. Hence, it suffices to show that 
\begin{equation*}
    \pi_0^{\bm\sigma}(F_!(f)):\pi_0^{\bm\sigma}(F_!(\mathcal{F}))\rightarrow{\pi_0^{\bm\sigma}(F_!(\mathcal{G}))}
\end{equation*}is an epimorphism of sheaves. Since both $\pi_0:\infty\textcat{Grpd}\rightarrow{\textnormal{Set}}$ and the sheafification functor commutes with left $(\infty,1)$-Kan extensions, we have an equivalence
\begin{equation*}
    \pi_0^{\bm\sigma}(F_!(\mathcal{F}))\simeq \textnormal{Ho}(F)_{\#}(\pi_0^{\bm{\tau}}(\mathcal{F}))
\end{equation*}The result follows since $\textnormal{Ho}(F)_\#$, being a left adjoint, must preserve the epimorphism of sheaves given by $\pi_0^{\bm{\tau}}(f):\pi_0^{\bm{\tau}}(\mathcal{F})\rightarrow{\pi_0^{\bm{\tau}}(\mathcal{G})}$. 
\end{proof}

\begin{defn}
Suppose that $(\mathcal{M},\bm{\tau})$ and $(\mathcal{N},\bm\sigma)$ are $(\infty,1)$-sites, and that we have a functor $F:\mathcal{M}\rightarrow{\mathcal{N}}$. Then, $F$ is a \textit{cocontinuous functor of $(\infty,1)$-sites} if, whenever $\{Y_j\rightarrow{F(X)}\}_{j\in J}$ is a $\bm\sigma$-cover, there exists a $\bm{\tau}$-cover $\{U_i\rightarrow{X}\}_{i\in I}$ such that the family of maps $\{F(U_i)\rightarrow{F(X)}\}_{i\in I}$ refines the cover $\{Y_j\rightarrow{F(X)}\}_{j\in J}$.
\end{defn}

In the same setting as before, we see that if $F:(\mathcal{M},\bm{\tau})\rightarrow{(\mathcal{N},\bm\sigma)}$ is a cocontinuous functor of $(\infty,1)$-sites, then we obtain a functor $F_*:\textcat{Stk}(\mathcal{M},\bm{\tau})\rightarrow{\textcat{Stk}(\mathcal{N},\bm\sigma)}$ whose left adjoint is $F^*$. Similarly to Proposition \ref{epimorphismfunctor}, we have the following result. 
\begin{prop}\phantomsection\label{epimorphismcocontinuous}
    Suppose that $F:(\mathcal{M},\bm{\tau})\rightarrow{(\mathcal{N},\bm\sigma)}$ is a cocontinuous functor of sites. Then, $F^*$ preserves epimorphisms of $\infty$-stacks. 
\end{prop}

Suppose that we have an $(\infty,1)$-site $(\mathcal{M},\bm\tau)$. We will say that an object $X$ of $\mathcal{M}$ is \textit{admissible} if $h(X):=\textnormal{Map}_\mathcal{M}(-,X)$ is a stack. By an application of \cite[c.f. Corollary 5.1.5.8]{lurie_higher_2009} along with the statement that stackification preserves colimits, we see that if every $X\in\mathcal{M}$ is admissible, then the Yoneda embedding generates $\textcat{Stk}(\mathcal{M},\bm\tau)$ under small colimits.

\begin{prop}\phantomsection\label{fullyfaithfulstackfunctor}
    Suppose that $F:(\mathcal{M},\bm\tau)\rightarrow{(\mathcal{N},\bm\sigma)}$ is a fully faithful, continuous, and cocontinuous functor of sites, and that every $X\in\mathcal{M}$ is admissible. Then, the induced functor $F_{\#}:\textcat{Stk}(\mathcal{M},\bm\tau)\rightarrow{\textcat{Stk}(\mathcal{N},\bm\sigma)}$ is fully faithful.
\end{prop}
\begin{proof}
Since $F_{\#}$ has a right adjoint $F^{*}$, it suffices to show that there is an equivalence of $(\infty,1)$-functors $F^{*}\circ F_{\#}\simeq \textnormal{id}_{\textcat{Stk}(\mathcal{M},\bm{\tau})}$ by \cite[Proposition 5.2.7.4]{lurie_higher_2009}. Since $F_{\#}$ and $F^*$ commute with colimits and every stack in $\textcat{Stk}(\mathcal{M},\bm\tau)$ is a colimit of objects of the form $h(X)$ for $X\in\mathcal{M}$, it suffices to check on objects of the form $h(X)$. By fully faithfulness of $F$ we have that
    \begin{equation*}
        F^*\circ F_{\#}(h(X))=F^*(\textnormal{Map}_{\mathcal{N}}(-,F(X)))\simeq \textnormal{Map}_{\mathcal{N}}(F(-),F(X))\simeq h(X)
    \end{equation*}
\end{proof}

\section{Homotopical Algebraic Geometry}\label{hagsection}

In this section, we will provide an analogue of To\"en and Vezzosi's \textit{HAG (homotopical algebraic geometry) contexts} \cite{toen_homotopical_2008} suitable for our applications. 

\subsection{Geometries}\phantomsection\label{geometriessection}

The following definition of an $(\infty,1)$-pre-geometry triple was first stated in \cite{kelly_analytic_2022}. As a guiding example, it may help to think of the classical algebraic geometry triple with $\mathcal{M}$ as the ordinary category of affines, $\bm\tau$ the \'etale topology, and $\textcat{P}$ the class of smooth maps. 
\begin{defn}\cite[Definition 6.3]{kelly_analytic_2022}\phantomsection\label{pregeometrytriple} An \textit{$(\infty,1)$-pre-geometry triple} is a triple $(\mathcal{M},\bm{\tau},\textcat{P})$ where $\mathcal{M}$ is an $(\infty,1)$-category, $\bm{\tau}$ is a Grothendieck pre-topology on $\textnormal{Ho}(\mathcal{M})$, and $\textcat{P}$ is a class of maps in $\mathcal{M}$ such that
\begin{enumerate}
    \item\label{geometry1} If $\{U_i\rightarrow{X}\}_{i\in I}$ is a $\bm{\tau}$-covering family, then each $U_i\rightarrow{X}$ is in $\textcat{P}$, 
    \item\label{geometry2} $\textcat{P}$ is \textit{local for the topology} $\bm{\tau}$, in the sense that, whenever we have a morphism $f:Y\rightarrow{X}$ in $\mathcal{M}$ along with a $\bm{\tau}$-covering family $\{U_i\rightarrow{Y}\}_{i\in I}$ such that each induced morphism $U_i\rightarrow{X}$ is in $\textcat{P}$, then $f\in\textcat{P}$,
    \item\label{geometry3} The class $\textcat{P}$ is stable under equivalences, compositions, and pullbacks.
\end{enumerate}An $(\infty,1)$-pre-geometry triple is said to be an \textit{$(\infty,1)$-geometry triple} if every object of $\mathcal{M}$ is admissible.

\end{defn}

Often it is not clear in certain settings when a class of maps $\textcat{P}$ is local for the topology. We can let $\textcat{P}^{\bm\tau}$ be the class of maps $f:Y\rightarrow{X}$ in $\mathcal{M}$ such that there is a $\bm\tau$-cover $\{g_i:U_i\rightarrow{Y}\}_{i\in I}$ with $f\circ g_i\in \textcat{P}$. Then, we see that $\textcat{P}^{\bm\tau}$ is local and $\textcat{P}\subseteq \textcat{P}^{\bm\tau}$ with equality if and only if $\textcat{P}$ is local for the $\bm\tau$-topology. By \cite[Proposition 6.1.4]{ben-bassat_perspective_2024}, if morphisms in $\textcat{P}$ are stable by equivalences, compositions, and pushouts, then so are morphisms in $\textcat{P}^{\bm\tau}$.

A philosophy we use throughout this thesis is that our categories of interest, generally some subcategories of `affines', should be embedded in larger categories with more useful properties. This leads naturally to talking about the notion of a relative geometry tuple as described in \cite[Definition 6.4]{kelly_analytic_2022}.

\begin{defn}\phantomsection\label{relativepregeometry} 
\begin{enumerate}
    \item A \textit{relative $(\infty,1)$-pre-geometry tuple} consists of a tuple $(\mathcal{M},\bm{\tau},\textcat{P},\mathcal{A})$ with $(\mathcal{M},\bm{\tau},\textcat{P})$ an $(\infty,1)$-pre-geometry triple and $\mathcal{A}$ a full subcategory of $\mathcal{M}$ such that, if $f:Y\rightarrow{X}$ is a map in $\textcat{P}\cap\mathcal{A}$ and $Z\rightarrow{X}$ is any map with $Z$ in $\mathcal{A}$, then $Y\times_XZ$ exists and is in $\mathcal{A}$, 
    \item A relative $(\infty,1)$-pre-geometry tuple $(\mathcal{M},\bm{\tau},\textcat{P},\mathcal{A})$ is \textit{strong} if, whenever we have a cover $\{U_i\rightarrow{X}\}_{i\in I}$ in $\mathcal{M}$ and $Y\rightarrow{X}$ is a map with $Y\in\mathcal{A}$, then $\{U_i\times_XY\rightarrow{Y}\}_{i\in I}$ is a cover in $\bm{\tau}|_\mathcal{A}$, where $\bm{\tau}|_{\mathcal{A}}$ denotes the restriction of $\bm{\tau}$ to $\mathcal{A}$.
\end{enumerate}
\end{defn}

\begin{remark}We note that if we have a Grothendieck pre-topology defined on $\mathcal{A}$, then we can extend it to a Grothendieck pre-topology on $\mathcal{M}$ such that the resulting topology is strong relative to $\mathcal{A}$.
\end{remark}

\begin{defn}\cite[Definition 6.6]{kelly_analytic_2022} Let $(\mathcal{M},\bm{\tau},\textcat{P},\mathcal{A})$ be a (strong) relative $(\infty,1)$-pre-geometry tuple. 
\begin{enumerate}
    \item An object $X$ of $\mathcal{M}$ is said to be \textit{$\mathcal{A}$-admissible} if the restriction of $\textnormal{Map}_\mathcal{M}(-,X)$ to $\mathcal{A}$ is a stack in $\textcat{Stk}(\mathcal{A},\bm{\tau}|_{\mathcal{A}})$,
    \item $(\mathcal{M},\bm{\tau},\textcat{P},\mathcal{A})$ is a \textit{(strong) relative $(\infty,1)$-geometry tuple} if each $X\in\mathcal{A}$ is $\mathcal{A}$-admissible.
\end{enumerate}
\end{defn}

We note that, if $(\mathcal{M},\bm{\tau},\textcat{P},\mathcal{A})$ is a relative $(\infty,1)$-pre-geometry tuple, then if we restrict our topology $\bm\tau$ and our class of maps $\textcat{P}$, we see that $(\mathcal{A},\bm{\tau}|_\mathcal{A},\textcat{P}|_{\mathcal{A}})$ is an $(\infty,1)$-pre-geometry triple. We can also construct strong tuples using the following lemma.

\begin{lem}\phantomsection\label{makingthingsstrong}
    To any relative $(\infty,1)$-pre-geometry tuple $(\mathcal{M},\bm\tau,\textcat{P},\mathcal{A})$ we can define an associated strong relative $(\infty,1)$-pre-geometry tuple $(\mathcal{M},\bm\tau_\mathcal{A},\textcat{P}_\mathcal{A},\mathcal{A})$. 
\end{lem}

\begin{proof}
    This follows by taking $\textcat{P}_\mathcal{A}\subseteq\textcat{P}$ to be the class of maps $f:Y\rightarrow{X}$ in $\textcat{P}$ such that, whenever $Z\rightarrow{X}$ is a map with $Z\in\mathcal{A}$, then $Y\times_XZ$ is in $\mathcal{A}$. Our topology $\bm\tau_\mathcal{A}$ will be the class of covers $\{U_i\rightarrow{X}\}_{i\in I}$ in $\bm\tau$ such that, whenever there exists some map $Y\rightarrow{X}$ in $\mathcal{M}$, then $U_i\times_XY$ is in $\mathcal{A}$ if and only if $Y$ is in $\mathcal{A}$. 
\end{proof} 

\begin{exmp}
    In the model category theoretic version of To\"en and Vezzosi \cite{toen_homotopical_2008}, a homotopical algebraic geometry (HAG) context is a tuple $(\textnormal{C},\textnormal{C}_0,\textnormal{A},\bm{\tau},\textcat{P})$. Here, $\textnormal{A}$ is taken to be a full subcategory of $\textnormal{Comm}(\textnormal{C})$, the category of commutative monoids in $\textnormal{C}$. When we take $\textnormal{C}_0=\textnormal{C}$ and consider the associated $(\infty,1)$-categories, then $(\textcat{Aff}_{\textcat{L}^H(\textnormal{C})},\bm{\tau},\textcat{P},\textcat{L}^H(\textnormal{A})^{op})$ is a relative $(\infty,1)$-geometry tuple, where \begin{equation*}
\textcat{Aff}_{\textcat{L}^H(\textnormal{C})}:=\textcat{L}^H(\textnormal{Comm}(\textnormal{C}))^{op}
    \end{equation*} 
\end{exmp}

Consider the inclusion functor $i:\textcat{PSh}(\mathcal{A})\rightarrow{\textcat{PSh}(\mathcal{M}})$, which induces a functor on the level of stacks $i_{\#}:\textcat{Stk}(\mathcal{A},\bm\tau|_\mathcal{A})\rightarrow{\textcat{Stk}(\mathcal{M},\bm\tau)}$. The following result shows that, in the case where the geometry tuple is strong, we can consider stacks on $\mathcal{A}$ by realising them as stacks on $\mathcal{M}$. 

\begin{prop}\phantomsection\label{stronggeometryfullyfaithful}\cite[Proposition 6.7]{kelly_analytic_2022} If $(\mathcal{M},\bm\tau,\textcat{P},\mathcal{A})$ is a strong relative $(\infty,1)$-geometry tuple, then $i_{\#}$ is fully faithful. 
     
\end{prop}

\subsection{Geometric Stacks}\phantomsection\label{geometricstacksection}

Fix a relative $(\infty,1)$-pre-geometry tuple $(\mathcal{M},\bm\tau,\textcat{P},\mathcal{A})$. We will say that $X\in\mathcal{A}$ is a \textit{representable stack} if $X$ is $\mathcal{A}$-admissible. 
\begin{defn}
\begin{enumerate}
    \item A stack $\mathcal{F}$ in $\textcat{Stk}(\mathcal{A},\bm{\tau}|_{\mathcal{A}})$ is \textit{$(-1)$-geometric} if it is of the form $\mathcal{F}\simeq \textnormal{Map}_\mathcal{M}(-,X)$ for some $\mathcal{A}$-admissible $X\in\mathcal{M}$,
    \item A morphism of stacks $f:\mathcal{F}\rightarrow{\mathcal{G}}$ in $\textcat{Stk}(\mathcal{A},\bm{\tau}|_{\mathcal{A}})$ is \textit{$(-1)$-representable} if, for any map $X\rightarrow{\mathcal{G}}$, with $X\in\mathcal{A}$ a representable stack, the pullback $\mathcal{F}\times_\mathcal{G} X$ is $(-1)$-geometric,
    \item  A morphism of stacks $f:\mathcal{F}\rightarrow{\mathcal{G}}$ in $\textcat{Stk}(\mathcal{A},\bm{\tau}|_{\mathcal{A}})$ is \textit{in $(-1)$-\textcat{P}} if it is $(-1)$-representable and, for any map $X\rightarrow{\mathcal{G}}$, with $X\in\mathcal{A}$ a representable stack, the induced map of $(-1)$-geometric stacks $\mathcal{F}\times_\mathcal{G}X\rightarrow{X}$ is represented by a morphism in $\textcat{P}$.
\end{enumerate}
\end{defn}

\begin{remark}
    We will often, when it is clear from context, denote a $(-1)$-geometric stack by the object $X\in\mathcal{M}$ that represents it. 
\end{remark}

Now, for $n\geq 0$, we can inductively build up notions of higher geometric stacks by glueing together representables as follows. 

\begin{defn}\phantomsection\label{geometricstackdefinition}

\begin{enumerate}
    \item Let $\mathcal{F}$ be a stack in $\textcat{Stk}(\mathcal{A},\bm{\tau}|_{\mathcal{A}})$. An \textit{$n$-atlas} for $\mathcal{F}$ is a set of morphisms $\{U_i\rightarrow{\mathcal{F}}\}_{i\in I}$ such that each $U_i$ is $(-1)$-geometric, each map $U_i\rightarrow{\mathcal{F}}$ is in $(n-1)$-\textcat{P}, and there is an epimorphism of stacks 
        \begin{equation*}
            \coprod_{i\in I} U_i\rightarrow{\mathcal{F}}
        \end{equation*}in $\textcat{Stk}(\mathcal{A},\bm{\tau}|_{\mathcal{A}})$, 
    \item A stack $\mathcal{F}$ in $\textcat{Stk}(\mathcal{A},\bm{\tau}|_{\mathcal{A}})$ is \textit{$n$-geometric} if the diagonal morphism $\mathcal{F}\rightarrow{\mathcal{F}\times\mathcal{F}}$ is $(n-1)$-representable and $\mathcal{F}$ admits an $n$-atlas,
    \item A morphism of stacks $f:\mathcal{F}\rightarrow{\mathcal{G}}$ in $\textcat{Stk}(\mathcal{A},\bm{\tau}|_{\mathcal{A}})$ is \textit{$n$-representable} if, for any map $X\rightarrow{\mathcal{G}}$ with $X\in\mathcal{A}$ a representable stack, the pullback $\mathcal{F}\times_\mathcal{G} X$ is $n$-geometric, 
    \item A morphism of stacks $f:\mathcal{F}\rightarrow{\mathcal{G}}$ in $\textcat{Stk}(\mathcal{A},\bm{\tau}|_{\mathcal{A}})$ is in $n$-\textcat{P} if it is $n$-representable and, for any map $X\rightarrow{\mathcal{G}}$ with $X\in\mathcal{A}$ a representable stack, there exists an $n$-atlas of the form $\{U_i\rightarrow \mathcal{F}\times_\mathcal{G} X\}_{i\in I}$ such that each map $U_i\rightarrow{X}$ is in $\textcat{P}$.
\end{enumerate}
\end{defn}

We easily see that the collection, $\textcat{Stk}_n(\mathcal{M},\bm\tau,\textcat{P},\mathcal{A})$, of $n$-geometric stacks in $\textcat{Stk}(\mathcal{A},\bm\tau|_\mathcal{A})$ is a subset of the collection of $n$-geometric stacks $\textcat{Stk}_n(\mathcal{M},\bm\tau,\textcat{P},\mathcal{M})$ in $\textcat{Stk}(\mathcal{M},\bm\tau)$ when $(\mathcal{M},\bm\tau,\textcat{P},\mathcal{M})$ is considered as a relative $(\infty,1)$-pre-geometry tuple. Moreover, if $(\mathcal{M},\bm\tau,\textcat{P})$ is an $(\infty,1)$-geometry tuple and $\mathcal{A}$ has all finite limits which are preserved by the inclusion, then there is an induced functor 
\begin{equation*}
    i_\#:\textcat{Stk}_n(\mathcal{A},\bm\tau|_\mathcal{A},\textcat{P}|_\mathcal{A},\mathcal{A})\rightarrow{\textcat{Stk}_n(\mathcal{M},\bm\tau,\textcat{P},\mathcal{A})}
\end{equation*}which is fully faithful if $(\mathcal{M},\bm\tau,\textcat{P},\mathcal{A})$ is a strong tuple \cite[Corollary 6.13]{kelly_analytic_2022}.

\begin{notat}\phantomsection\label{ngeometricnotation}
    We will want to specify when we are working with $n$-geometric stacks within the $(\infty,1)$-geometry tuple $(\mathcal{A},\bm\tau|_\mathcal{A},\textcat{P}|_\mathcal{A})$. Hence, we will abbreviate the category of $n$-geometric stacks $\textcat{Stk}_n(\mathcal{A},\bm\tau|_\mathcal{A},\textcat{P}|_\mathcal{A},\mathcal{A})$ to $\textcat{Stk}_n(\mathcal{A},\bm\tau|_\mathcal{A},\textcat{P}|_\mathcal{A})$. We will refer to $n$-representable morphisms with respect to this geometry tuple by the notation $n$-representable$|_\mathcal{A}$. The notation $n\textcat{-P}|_\mathcal{A}$ in this context should be clear. 
\end{notat}

\begin{exmp}Suppose that $k$ is a commutative ring. 
    \begin{enumerate}
        \item When $\mathcal{M}=\mathcal{A}=\textnormal{Aff}_k$, the category of affine schemes, $\bm\tau$ is the \'etale topology, and $\textcat{P}$ is the class of smooth maps, $n$-geometric stacks correspond to algebraic $n$-stacks \cite{toen_homotopical_2008}, 
        \item When $\mathcal{M}=\mathcal{A}=\textcat{DAff}^{cn}_k:=\textcat{L}^H(\textnormal{Comm}(\textnormal{sMod}_k))^{op}$, $\bm\tau$ is the \'etale  topology, and $\textcat{P}$ is the class of smooth morphisms (see Section \ref{derivedalgebraicgeosection}) we obtain the derived algebraic geometry context of To\"en and Vezzosi. Several examples of $n$-geometric stacks appear in this context such as the $1$-geometric stack of rank $n$-vector bundles $\textcat{Vect}_n$ \cite[Section 2.2.6.1]{toen_homotopical_2008}, 
        \item We will explore examples where $\mathcal{M}$ is not necessarily equal to $\mathcal{A}$ in Section \ref{topologiessection}.
    \end{enumerate}
\end{exmp}

In a similar way to \cite[Proposition 1.3.3.3]{toen_homotopical_2008}, we can prove the following statements about $n$-geometric stacks. 

\begin{prop}\phantomsection\label{conditionsgeometric}
\begin{enumerate}
    \item $\mathcal{F}$ is $n$-geometric if and only if the map $\mathcal{F}\rightarrow{*}$ is $n$-representable, 
    \item Any $(n-1)$-geometric stack is $n$-geometric,
    \item Any $(n-1)$-representable morphism is $n$-representable, 
    \item Any $(n-1)$-\textcat{P}-morphism is an $n$-\textcat{P}-morphism, 
    \item $n$-representable morphisms are stable by isomorphisms, pullbacks and compositions, 
    \item $n$-\textcat{P}-morphisms are stable by isomorphisms, pullbacks and compositions.
\end{enumerate}\end{prop}

In the course of proving that $n$-representable morphisms are stable by composition we obtain the following result.
\begin{cor}\phantomsection\label{compositioncorollary}
Suppose that $f:\mathcal{F}\rightarrow{\mathcal{G}}$ is an $n$-representable morphism of stacks. If $\mathcal{G}$ is $n$-geometric, then so is $\mathcal{F}$.
\end{cor}

\begin{prop}\phantomsection\label{coolcorollary} Suppose that we have morphisms $\mathcal{F}\rightarrow{\mathcal{G}}$ and $\mathcal{H}\rightarrow{\mathcal{G}}$ with $\mathcal{F},\mathcal{H}$ $n$-geometric stacks and suppose that the diagonal morphism $\mathcal{G}\rightarrow{\mathcal{G}\times\mathcal{G}}$ is $n$-representable for some $n\geq -1$. Then, $\mathcal{F}\times_\mathcal{G}\mathcal{H}$ is an $n$-geometric stack. 
    
\end{prop}

\begin{proof}
We will show that the morphism $\mathcal{F}\times_\mathcal{G} \mathcal{H}\rightarrow{*}$ is $n$-representable.  We have the following pullback diagram,
\begin{equation*}
    \begin{tikzcd}
    \mathcal{F}\times_\mathcal{G}\mathcal{H} \arrow{d} \arrow{r} & \mathcal{G} \arrow{d}\\
    \mathcal{F}\times\mathcal{H} \arrow{r} & \mathcal{G}\times\mathcal{G}
    \end{tikzcd}
\end{equation*}By stability of $n$-representable maps under pullback, the map $\mathcal{F}\times_\mathcal{G}\mathcal{H}\rightarrow{\mathcal{F}\times\mathcal{H}}$ is $n$-representable. Moreover, if we consider the pullback square, 
\begin{equation*}
    \begin{tikzcd}
        \mathcal{F}\times \mathcal{H} \arrow{d}\arrow{r} & \mathcal{H} \arrow{d}\\
        \mathcal{F} \arrow{r} & *
    \end{tikzcd}
\end{equation*}then we see that, since $\mathcal{F}$ is $n$-geometric, the morphism $\mathcal{F}\times \mathcal{H}\rightarrow{\mathcal{H}}$ is $n$-representable. Hence, since the composition of $n$-representable maps is $n$-representable, the morphism $\mathcal{F}\times_\mathcal{G} \mathcal{H}\rightarrow{*}$ is $n$-representable.
\end{proof}

\begin{cor}\phantomsection\label{stablepullbackgeometric}The full subcategory of $n$-geometric stacks in $\textcat{Stk}(\mathcal{A},\bm{\tau}|_\mathcal{A})$ is stable under $(\infty,1)$-pullbacks for $n\geq 0$.

\end{cor}

\subsection{Closure under $\tau$-descent}

Fix a relative $(\infty,1)$-geometry tuple $(\mathcal{M},\textcat{P},\bm\tau,\mathcal{A})$. Suppose that we have a $\bm\tau$-cover $\{U_i\rightarrow{X}\}_{i\in I}$ in $\mathcal{M}$. Then, this defines an epimorphism of stacks $\coprod_{i\in I}U_i\rightarrow{X}$ in the sense of Definition \ref{epimorphismstacks}. However, it does not follow that every epimorphism of representable stacks defines a $\bm\tau$-cover. 

We will say that a stack $\mathcal{F}$ is \textit{coverable} if there is an epimorphism of stacks $\coprod_{i\in I}U_i\rightarrow{\mathcal{F}}$
with each $U_i$ a $(-1)$-geometric stack on $\mathcal{A}$. In this case, we will say that $\mathcal{F}$ is \textit{covered by $\{U_i\}_{i\in I}$}.

\begin{prop}\phantomsection\label{localepimorphismprop}Suppose that $\mathcal{F}$ is a coverable stack in $\textcat{Stk}(\mathcal{A},\bm{\tau}|_\mathcal{A})$, covered by $\{U_i\}_{i\in I}$, and that $X\in\mathcal{A}$ is a representable stack along with a morphism $X\rightarrow{\mathcal{F}}$. Then, there exists a $\bm{\tau}$-cover $\{V_j\rightarrow{X}\}_{j\in J}$ in $\mathcal{A}$, a morphism $u:J\rightarrow{I}$ and, for all $j$, a commutative diagram in $\textcat{Stk}(\mathcal{A},\bm{\tau}|_\mathcal{A})$,
\begin{equation*}
    \begin{tikzcd}
       V_j\arrow{r} \arrow{d} & U_{u(j)}\arrow{d} \\
    X\arrow{r} & \mathcal{F}
    \end{tikzcd}
\end{equation*}
\end{prop}

\begin{proof}
    Since $\mathcal{F}$ is coverable, there is an epimorphism of stacks $\coprod_{i\in I} U_i\rightarrow{\mathcal{F}}$ corresponding to an epimorphism of sheaves $\coprod_{i\in I} \pi_0^{\bm{\tau}} (U_i)\rightarrow{\pi_0^{\bm{\tau}}(\mathcal{F})}$
in $\textnormal{Sh}(\textnormal{Ho}(\mathcal{A}),\bm{\tau}|_\mathcal{A})$. It is known (see \cite[Corollary III.7.5]{maclane_sheaves_1994}) that a morphism of sheaves is an epimorphism if and only if it satisfies a local surjectivity property. Given the map $\pi_0^{\bm{\tau}}(X)\rightarrow{\pi_0^{\bm{\tau}}(\mathcal{F})}$, this local surjectivity property asserts that there exists a $\bm{\tau}$-cover $\{V_j\rightarrow{X}\}_{j\in J}$ and a morphism $u:J\rightarrow{I}$ such that the following diagram commutes in $\textnormal{Sh}(\textnormal{Ho}(\mathcal{A}),\bm{\tau}|_{\mathcal{A}})$,
\begin{equation*}
     \begin{tikzcd}
       V_j\arrow{r} \arrow{d} & \pi_0^{\bm{\tau}}(U_{u(j)})\arrow{d} \\
    \pi_0^{\bm{\tau}}(X)\arrow{r} & \pi_0^{\bm{\tau}}(\mathcal{F})
    \end{tikzcd}
\end{equation*}

Under the inclusion map $i^{\bm{\tau}}:\textnormal{Sh}(\textnormal{Ho}(\mathcal{A}),\bm{\tau}|_{\mathcal{A}})\rightarrow{\textcat{Stk}(\mathcal{A},\bm{\tau}|_{\mathcal{A}})}$, we note that representable sheaves get sent to representable stacks since every element of $\mathcal{A}$ is $\mathcal{A}$-admissible. We have the following commutative diagram 
\begin{equation*}
     \begin{tikzcd}
       V_j\arrow{r} \arrow{d} & U_{u(j)}\arrow{d} \\
    X\arrow{r} & i^{\bm{\tau}}\circ\pi_0^{\bm{\tau}}(\mathcal{F})
    \end{tikzcd}
\end{equation*}
Using the counit of the adjunction $i\circ\pi_0\rightarrow{\textnormal{id}}$, we obtain a map $i^{\bm{\tau}}\circ\pi_0^{\bm{\tau}}(\mathcal{F})\rightarrow\mathcal{F}$ such that the required diagram commutes. 
\end{proof}

\begin{defn}\cite[c.f. Definition 8.3]{porta_derived_2018}\phantomsection\label{taudescent} We say that $\mathcal{A}$ is \textit{closed under $\bm{\tau}$-descent relative to $\mathcal{M}$} if, for any stack $\mathcal{F}$ in $\textcat{Stk}(\mathcal{A},\bm{\tau}|_\mathcal{A})$, any morphism $\mathcal{F}\rightarrow{X}$ with $X$ a $(-1)$-geometric stack, and any $\bm{\tau}$-cover $\{U_i\rightarrow{X}\}_{i\in I}$ of $\mathcal{A}$-admissible objects, we have that, if $\mathcal{F}\times_XU_i$ is a $(-1)$-geometric stack for every $i\in I$, then so is $\mathcal{F}$.

\end{defn}

When $\mathcal{A}$ is closed under $\bm{\tau}$-descent we get several useful criteria for a stack to be $n$-geometric, including the following local condition for checking geometricity. 

\begin{lem}\phantomsection\label{pullbackgeometric}
Suppose that $\mathcal{A}$ is closed under $\bm{\tau}$-descent. Let $f:\mathcal{F}\rightarrow{\mathcal{G}}$ be any morphism of stacks in $\textcat{Stk}(\mathcal{A},\bm{\tau}|_\mathcal{A})$ with $\mathcal{G}$ an $n$-geometric stack. Suppose that there exists an $n$-atlas $\{U_i\rightarrow{\mathcal{G}}\}_{i\in I}$ of $\mathcal{G}$ such that each stack $\mathcal{F}\times_\mathcal{G}U_i$ is $n$-geometric. Then, $\mathcal{F}$ is $n$-geometric. 

Furthermore, if every map $\mathcal{F}\times_\mathcal{G} U_i\rightarrow{U_i}$ is in $n\textcat{-P}$, then so is $f$.

\end{lem}

\begin{proof}
We proceed by induction on $n$. Indeed, when $n=-1$ the result follows since $\mathcal{A}$ is closed under $\bm{\tau}$-descent. If $n\geq 0$, then it suffices to show that $f$ is $n$-representable by Corollary \ref{compositioncorollary}. Suppose that $X$ is a representable stack along with a morphism $X\rightarrow{\mathcal{G}}$. It remains to show that $\mathcal{F}\times_\mathcal{G}X$ is $n$-geometric for all representable stacks $X$. Suppose that $\{U_i\rightarrow{\mathcal{G}}\}_{i\in I}$ is an $n$-atlas for $\mathcal{G}$. By Proposition \ref{localepimorphismprop}, we can construct a $\bm{\tau}$-covering family $\{V_j\rightarrow{X}\}_{j\in J}$ for $X$ such that the morphism $V_j\rightarrow{\mathcal{G}}$ factors through $U_{u(j)}$ for some morphism $u:J\rightarrow{I}$. Since the induced map 
\begin{equation*}
    \mathcal{F}\times_\mathcal{G} V_j\rightarrow{\mathcal{F}\times_\mathcal{G} U_{u(j)}}
\end{equation*}is $n$-representable for all $n\geq -1$ and $\mathcal{F}\times_\mathcal{G} U_{u(j)}$ is $n$-geometric by assumption then, by Corollary \ref{compositioncorollary}, it follows that each $\mathcal{F}\times_\mathcal{G}V_j\simeq (\mathcal{F}\times_\mathcal{G}U_{u(j)})\times_{U_{u(j)}} V_j$ is $n$-geometric. Furthermore, if each map $\mathcal{F}\times_\mathcal{G} U_i\rightarrow{U_i}$ is in $n\textcat{-P}$, then each map $\mathcal{F}\times_\mathcal{G} V_j\rightarrow{V_j}$ is also in $n\textcat{-P}$ since $n\textcat{-P}$ morphisms are stable under pullback. Therefore, it suffices to reduce to the situation where $\mathcal{G}$ is a $(-1)$-geometric stack. The proof then follows identically to \cite[Proposition 1.3.3.4]{toen_homotopical_2008}.
\end{proof}

\begin{cor}\cite[c.f. Corollary 1.3.3.5]{toen_homotopical_2008}
    Suppose that $\mathcal{A}$ is closed under $\bm\tau$-descent relative to $\mathcal{M}$. The full subcategory of $n$-geometric stacks in $\textcat{Stk}(\mathcal{A},\bm\tau|_\mathcal{A})$ is closed under disjoint unions for $n\geq 0$.
\end{cor}

When $\mathcal{A}$ is closed under $\bm\tau$-descent relative to $\mathcal{M}$, the condition that the diagonal morphism $\mathcal{F}\rightarrow{\mathcal{F}\times\mathcal{F}}$ is $(n-1)$-representable for an $n$-geometric stack $\mathcal{F}$ is superfluous.
\begin{prop}\phantomsection\label{onlynatlas}\cite[c.f. Corollary 8.6]{porta_derived_2018}
Suppose that $\mathcal{A}$ is closed under $\bm{\tau}$-descent relative to $\mathcal{M}$. Suppose that $\mathcal{F}$ is a stack in $\textcat{Stk}(\mathcal{A},\bm{\tau}|_\mathcal{A})$ with an $n$-atlas $\{U_i\rightarrow{\mathcal{F}}\}_{i\in I}$ for some $n\geq 0$. Then, $\mathcal{F}$ is an $n$-geometric stack.
\end{prop}

We now want to explore some conditions under which $\mathcal{A}$ is closed under $\bm\tau$-descent relative to $\mathcal{M}$. Suppose that $\mathcal{F}$ is an $n$-geometric stack in $\textcat{Stk}(\mathcal{A},\bm{\tau}|_\mathcal{A})$ with $n$-atlas $\{U_i\rightarrow{\mathcal{F}}\}_{i\in I}$. Consider the \v{C}ech nerve, which we will denote by $\mathcal{U}_*\rightarrow{\mathcal{F}}$, of the epimorphism $\coprod_{i\in I} U_i\rightarrow{\mathcal{F}}$. We notice that 
\begin{equation*}
    \mathcal{U}_m=\coprod_{\underline{i}\in I^{m+1}} U_{\underline{i}}
\end{equation*}with $\underline{i}=(i_0,\dots,i_m)\in I^{m+1}$ and $U_{\underline{i}}:=U_{i_0}\times_\mathcal{F} U_{i_1}\times_{\mathcal{F}}\times \dots \times_\mathcal{F} U_{i_m}$ in $\textcat{Stk}(\mathcal{A},\bm\tau|_\mathcal{A})$. Each $U_{\underline{i}}$ is an $(n-1)$-geometric stack since each morphism $U_{i_j}\rightarrow{\mathcal{F}}$ is $(n-1)$-representable and $(n-1)$-representable morphisms are stable under pullback and composition. 

By Proposition \ref{cecheffective}, since an epimorphism of stacks is an effective epimorphism in the $(\infty,1)$-topos $\textcat{Stk}(\mathcal{A},\bm\tau|_\mathcal{A})$, we have the following result. 

\begin{prop}\phantomsection\label{segalgroupoidequivalence}The natural morphism $|\mathcal{U}_*|:={\varinjlim}_{[m]\in \Delta}\mathcal{U}_m\rightarrow\mathcal{F}$ is an equivalence of stacks.
\end{prop}

\begin{cor}\phantomsection\label{coproductsdisjoint}
   Suppose that $\mathcal{M}$ is closed under finite coproducts and that coproducts are disjoint. If, for any finite family of $\mathcal{A}$-admissible objects $\{U_i\}_{i\in I}$ in $\mathcal{M}$, the family of morphisms
    \begin{equation*}
        \{U_i\rightarrow{\coprod_{j\in I} U_j}\}_{i\in I}
    \end{equation*}is a $\bm\tau$-covering family of $\coprod_{j\in I} U_j$, then the natural morphism 
    \begin{equation*}
        \coprod_{i\in I} h(U_i)\rightarrow{h(\coprod_{i\in I} U_i)}
    \end{equation*}is an equivalence of stacks in $\textcat{Stk}(\mathcal{A},\bm\tau|_\mathcal{A})$.
\end{cor}

\begin{proof}
    The case when $I$ is empty follows easily as the empty family covers the initial object.  Now, if $I$ is non-empty it suffices, by induction, to consider the case where the finite family consists of two objects $X,Y$, say. We then see, by our assumption, that $\{X\rightarrow{X\coprod Y},Y\rightarrow{X\coprod Y}\}$ is a $\bm\tau$-covering family, and hence there is an epimorphism of stacks
    \begin{equation*}
        h(X)\coprod h(Y)\rightarrow{h(X\coprod Y)}
    \end{equation*}Therefore, $h(X\coprod Y)$ is equivalent to the geometric realisation of the \v{C}ech nerve of the epimorphism by Proposition \ref{segalgroupoidequivalence}. Hence, to show our desired equivalence it suffices to show that
    \begin{equation*}
        h(X\times_{X\coprod Y}Y)\simeq h(X)\times_{h(X\coprod Y)}h(Y)\simeq \emptyset
    \end{equation*}which follows since coproducts are disjoint in $\mathcal{M}$. 
\end{proof}

We consider the category $s\mathcal{M}$ of simplicial objects in $\mathcal{M}$. Suppose that $\mathcal{M}$ is closed under geometric realisations of simplicial objects. There exists an $(\infty,1)$-adjunction
\begin{equation*}
    |-|:s\mathcal{M}\leftrightarrows{\mathcal{M}}:\textnormal{const}
\end{equation*}where $\textnormal{const}$ denotes the constant functor taking an object in $\mathcal{M}$ to the constant simplicial object, and $|-|$ is the geometric realisation functor, taking $Y_*\in s\mathcal{M}$ to its geometric realisation $|Y_*|=\varinjlim_{[m]\in\Delta} Y_m$ in $\mathcal{M}$. Given an augmented simplicial object $U_*\rightarrow{X}\in\mathcal{M}$, this defines an adjunction on the over-categories
\begin{equation*}
    \int:s\mathcal{M}_{/U_*}\leftrightarrows{\mathcal{M}_{/X}}:-\times_X U_*
\end{equation*} by \cite[Lemma 5.2.5.2]{lurie_higher_2009}, where $\int$ takes an object $Y_*\rightarrow{U_*}$ of $s\mathcal{M}_{/U_*}$ to the object $|Y_*|\rightarrow{|U_*|}\rightarrow{X}$ in $\mathcal{M}_{/X}$. Motivated by the `descent conditions' defined in \cite{toen_homotopical_2008} we make the following definitions.

\begin{defn}\phantomsection\label{descentcondition}
\begin{enumerate}
    \item We say that a morphism of simplicial objects $Y_*\rightarrow{U}_*$ in $s\mathcal{M}$ is \textit{co-cartesian} if, for each $[m]\rightarrow{[n]}$ in $\Delta$, $Y_n\times_{U_n} U_m$ exists and the natural morphism $Y_m\rightarrow{Y_n\times_{U_n}U_m}$ is an isomorphism. 
    \item We say that an augmented simplicial object $U_*\rightarrow{X}$ in $s\mathcal{M}$ satisfies \textit{the co-cartesian descent condition} if the unit of the above adjunction is an isomorphism when restricted to the full subcategory of $s\mathcal{M}_{/U_*}$ consisting of co-cartesian morphisms $Y_*\rightarrow{U_*}$.
\end{enumerate}
\end{defn}

\begin{remark}We note that, if $U_*\in s\mathcal{M}$ satisfies the co-cartesian descent condition, then for all objects $Y_*\rightarrow{U_*}$ which are co-cartesian, $Y\times_X U_*\simeq Y_*$ where $Y=|Y_*|$.   
\end{remark}

\begin{prop}\phantomsection\label{taudescentundercondn}
     Suppose that
     \begin{enumerate}
         \item $\mathcal{M}$ is closed under geometric realisations,
         \item\label{cover2} Any $\bm\tau$-cover has a finite subcover, 
         \item For any finite collection $\{U_i\}_{i\in I}$ of $\mathcal{A}$-admissible objects, the map $\coprod_{i\in I} h(U_i)\rightarrow{h(\coprod_{i\in I} U_i)}$ is an equivalence in $\textcat{Stk}(\mathcal{A},\bm\tau|_\mathcal{A})$,
         \item The \v{C}ech nerve of any $\bm\tau$-cover of $\mathcal{A}$-admissible objects in $\mathcal{M}$ satisfies the co-cartesian descent condition.
     \end{enumerate}Then, $\mathcal{A}$ is closed under $\bm{\tau}$-descent relative to $\mathcal{M}$ in the sense of Definition \ref{taudescent}. 
\end{prop}

\begin{proof}Suppose that we have a stack $\mathcal{F}$ in $\textcat{Stk}(\mathcal{A},\bm{\tau}|_\mathcal{A})$ and a morphism $\mathcal{F}\rightarrow{h(X)}$ for some $(-1)$-geometric $X$. Suppose that we have a $\bm\tau$-cover $\{U_i\rightarrow{X}\}_{i\in I}$ such that $\mathcal{F}\times_{h(X)} h(U_i)$ is a $(-1)$-geometric stack for every $i\in I$. By Assumption (\ref{cover2}), we may assume that $I$ is a finite set. We want to show that $\mathcal{F}$ is a $(-1)$-geometric stack.

Consider the \v{C}ech nerve $\mathcal{U}_*$ of the epimorphism of stacks $\coprod_{i\in I}h(U_i)\rightarrow{h(X)}$. We note that, by our assumptions, $\mathcal{U}_m=h((U)_m)$ where $(U)_*$ is the \v{C}ech nerve of the morphism $\coprod_{i\in I} U_i\rightarrow{X}$ in $\mathcal{M}$. Moreover, by Proposition \ref{segalgroupoidequivalence}, there is an equivalence 
    \begin{equation*}
    |\mathcal{U}_*|:=\varinjlim_{[m]\in\Delta} \mathcal{U}_m\rightarrow{h(X)}
    \end{equation*}We let $\mathcal{F}_m$ be the pullback $\mathcal{F}\times_{h(X)} \mathcal{U}_m$ which we know, by our assumptions, is representable for every $m$, say $\mathcal{F}_m=h(Y_m)$ for some $\mathcal{A}$-admissible $Y_m\in\mathcal{M}$. Consider the induced simplicial objects $Y_*$ and $\mathcal{F}_*$ as objects of $\textnormal{s}\mathcal{M}$ and $\textnormal{s}\textcat{Stk}(\mathcal{A},\bm\tau|_\mathcal{A})$ respectively. There is a fully faithful functor $\underline{h}:\textnormal{s}\mathcal{M}\rightarrow{\textnormal{s}\textcat{Stk}(\mathcal{M},\bm\tau|_\mathcal{A})}$ induced by the Yoneda embedding and we easily see that $\mathcal{F}_*=\underline{h}(Y_*)$. By construction, $\mathcal{F}_*$ is the \v{C}ech nerve of the morphism $\mathcal{F}\times_{h(X)}\coprod_{i\in I}h({U_i})\rightarrow{\mathcal{F}}$, and hence, by Proposition \ref{segalgroupoidequivalence}, the natural morphism 
    \begin{equation*}|\mathcal{F}_*|:=\varinjlim_{[m]\in\Delta} \mathcal{F}_m\rightarrow{\mathcal{F}}
    \end{equation*}is an equivalence in $\textcat{Stk}(\mathcal{A},\bm{\tau}|_\mathcal{A})$. 
    
     Let $Y:=\varinjlim_{[m]\in\Delta}Y_m$ which is in $\mathcal{M}$ since $\mathcal{M}$ is closed under geometric realisations. By fully faithfulness of $\underline{h}$, there is a morphism $Y_*\rightarrow{(U)_*}$ corresponding to the morphism $\mathcal{F}_*=\underline{h}(Y_*)\rightarrow{\underline{h}((U)_*)=\mathcal{U}_*}$. We note that, for any morphism $[m]\rightarrow[n]$ in $\Delta$, we have, by construction, the following isomorphism in $\textcat{Stk}(\mathcal{A},\bm{\tau}|_\mathcal{A})$,
    \begin{equation*}
        \mathcal{F}_n\times_{\mathcal{U}_n}\mathcal{U}_m\simeq (\mathcal{F}\times_{h(X)} \mathcal{U}_n)\times_{\mathcal{U}_n}\mathcal{U}_m\simeq  \mathcal{F}\times_{h(X)} \mathcal{U}_m\simeq \mathcal{F}_m
    \end{equation*}and hence we have an isomorphism $Y_n\times_{(U)_n} (U)_m\rightarrow{Y_m}$ in $\mathcal{M}$. Since the \v{C}ech nerve $(U)_*\rightarrow{X}$ satisfies the co-cartesian descent condition, there is an isomorphism $Y\times_X (U)_*\rightarrow{Y_*}$ in $\textcat{s}\mathcal{M}$, and hence an isomorphism $h(Y)\times_{h(X)} \underline{h}((U)_*)\rightarrow{\underline{h}(Y_*)}$ in $s\textcat{PSh}(\mathcal{M})$. It then follows that 
    \begin{equation*}
        |\mathcal{F}_*|\simeq |\underline{h}({Y_*})|\simeq h({Y})\times_{h(X)}|\underline{h}((U)_*)|\simeq h(Y)\times_{h(X)}h(X)\simeq h({Y})
    \end{equation*}Therefore $|\mathcal{F}_*|$, and hence $\mathcal{F}$, is a $(-1)$-geometric stack in $\textcat{Stk}(\mathcal{A},\bm\tau|_\mathcal{A})$. 

\end{proof}

\subsection{Preservation of Geometric Stacks}\phantomsection\label{preservationgeometricstacks}

In this subsection, we will fix two relative $(\infty,1)$-geometry tuples $(\mathcal{M},\bm{\tau},\textcat{P},\mathcal{A})$ and $(\mathcal{N},\bm\sigma,\textcat{Q},\mathcal{B})$ and a colimit-preserving functor $L:\textcat{Stk}(\mathcal{A},\bm{\tau}|_{\mathcal{A}})\rightarrow{\textcat{Stk}(\mathcal{B},\bm\sigma|_{\mathcal{B}})}$. Then, by \cite[Corollary 5.5.2.9]{lurie_higher_2009}, there is an induced adjunction
\begin{equation*}
    L:\textcat{Stk}(\mathcal{A},\bm{\tau}|_{\mathcal{A}})\leftrightarrows{\textcat{Stk}(\mathcal{B},\bm\sigma|_{\mathcal{B}})}:R
\end{equation*}

\begin{prop}\phantomsection\label{pullbacksegal} Suppose $L$ preserves pullbacks of $(-1)$-geometric stacks along $m\textcat{-P}$-morphisms for all $m\geq -1$. Then, for each $n\geq -1$, $L$ preserves pullbacks of $n$-geometric stacks along $m\textcat{-P}$-morphisms for all $m\geq -1$.
\end{prop}

\begin{proof}We prove this claim by induction on $n$, with the case $n=-1$ true by assumption.
Suppose that we have an $m\textcat{-P}$-morphism $\mathcal{F}\rightarrow{\mathcal{G}}$ of $n$-geometric stacks for some $m\geq -1$ and a morphism of stacks $\mathcal{H}\rightarrow{\mathcal{G}}$, with $\mathcal{H}$ $n$-geometric. We note that, by Proposition \ref{segalgroupoidequivalence}, if $\{U_i\rightarrow{\mathcal{G}}\}_{i\in I}$ is an $n$-atlas for $\mathcal{G}$, then we can write $\mathcal{G}=|\mathcal{U}_*|$. Let
\begin{equation*}
    \mathcal{F}_*=\mathcal{F}\times_\mathcal{G}\mathcal{U}_*\quad\text{and}\quad\mathcal{H}_*=\mathcal{H}\times_\mathcal{G}\mathcal{U}_*
\end{equation*}We note that $\mathcal{F}\simeq |\mathcal{F}_*|$ and $\mathcal{H}\simeq |\mathcal{H}_*|$ and $\mathcal{F}\times_\mathcal{G} \mathcal{H}\simeq |\mathcal{F}_*\times_{\mathcal{U}_*}\mathcal{H}_*|$. Since $L$ commutes with colimits,
\begin{equation*}
    \begin{aligned}
L(\mathcal{F}\times_\mathcal{G}\mathcal{H})&=L(\varinjlim_{[l]\in\Delta}(\mathcal{F}_l\times_{\mathcal{U}_l}\mathcal{H}_l))\\
        &\simeq \varinjlim_{[l]\in\Delta}L(\mathcal{F}_l\times_{\mathcal{U}_l}\mathcal{H}_l)\\
        &\simeq \varinjlim_{[l]\in\Delta}\coprod_{\underline{i}\in I^{l+1}}L((\mathcal{F}\times_\mathcal{G} U_{\underline{i}})\times_{U_{\underline{i}}}(\mathcal{H}\times_\mathcal{G} U_{\underline{i}}))\\
    \end{aligned}
\end{equation*}Suppose that $[l]\in\Delta$. For each $\underline{i}$ in $I^{l+1}$, we note that $\mathcal{F}\times_\mathcal{G} U_{\underline{i}}$, $U_{\underline{i}}$ and $\mathcal{H}\times_\mathcal{G} U_{\underline{i}}$ are all $(n-1)$-geometric stacks. Therefore, by the inductive hypothesis, since the morphism $\mathcal{F}\times_\mathcal{G} U_{\underline{i}}\rightarrow U_{\underline{i}}$ is in $m\textcat{-P}$ for some $m$, then we have that 
\begin{equation*}
    \begin{aligned}
        L(\mathcal{F}\times_\mathcal{G}\mathcal{H})&\simeq \varinjlim_{[l]\in \Delta}
\coprod_{i\in I^{l+1}} L(\mathcal{F}\times_\mathcal{G} U_{\underline{i}})\times_{L(U_{\underline{i}})}L(\mathcal{H}\times_\mathcal{G} U_{\underline{i}}) \\
&\simeq \varinjlim_{l\in \Delta} L(\mathcal{F}_l)\times_{L(\mathcal{U}_l)}L(\mathcal{H}_l)\\
&\simeq L(\mathcal{F})\times_{L(\mathcal{G})}L(\mathcal{H})\end{aligned}
\end{equation*}

\end{proof}
Suppose for the rest of the section that $L$ preserves epimorphisms of stacks, and that $L$ sends $(-1)$-geometric stacks on $\mathcal{A}$ to $(-1)$-geometric stacks on $\mathcal{B}$. 

\begin{exmp}\phantomsection\label{adjunctionsFG}Suppose that we have a functor of sites $F:(\mathcal{A},\bm\tau|_\mathcal{A})\rightarrow{(\mathcal{B},\bm\sigma|_\mathcal{B})}$. We recall from Section \ref{continuousfunctorsection} that, if $F$ is continuous, we have an $(\infty,1)$-adjunction
\begin{equation*}
    F_\#:\textcat{Stk}(\mathcal{A},\bm{\tau}|_{\mathcal{A}})\leftrightarrows{\textcat{Stk}(\mathcal{B},\bm\sigma|_{\mathcal{B}})}:F^*
\end{equation*}We showed that $F_\#$ preserves epimorphisms and sends $(-1)$-geometric stacks to $(-1)$-geometric stacks. We note that all the results hold in this section with $L$ replaced with $F_\#$. Now, if $F$ is instead cocontinuous, then we have an $(\infty,1)$-adjunction
\begin{equation*}
    F^*:\textcat{Stk}(\mathcal{B},\bm{\tau}|_{\mathcal{B}})\leftrightarrows{\textcat{Stk}(\mathcal{A},\bm\sigma|_{\mathcal{A}})}:F_*
\end{equation*}We also showed that $F^*$ preserves epimorphisms and sends $(-1)$-geometric stacks to $(-1)$-geometric stacks. Therefore, all the results in this section also hold with $L$ replaced with $F^*$ and with our tuples $(\mathcal{M},\bm{\tau},\textcat{P},\mathcal{A})$ and $(\mathcal{N},\bm\sigma,\textcat{Q},\mathcal{B})$ swapped. 

\end{exmp}
\begin{prop}\phantomsection\label{therealtruncationprop}
    Suppose that
    \begin{enumerate}
        \item $L$ preserves pullbacks of $n$-geometric stacks for all $n\geq 0$,
        \item $L$ sends $\textcat{P}$-morphisms between $(-1)$-geometric stacks to $\textcat{Q}$-morphisms between $(-1)$-geometric stacks,  
        \item\label{surjectivityfunctorcondition} Every representable stack on $\mathcal{B}$ is the image under $L$ of a representable on $\mathcal{A}$. 
    \end{enumerate}
    Then, for all $n\geq -1$, $L$ sends $n$-geometric stacks on $\mathcal{A}$ to $n$-geometric stacks on $\mathcal{B}$. Moreover, $L$ sends $n\textcat{-P}$-morphisms to $n\textcat{-Q}$-morphisms, not necessarily between geometric stacks.
\end{prop}

\begin{proof}We prove this by induction on $n$. By assumption, $L$ preserves $(-1)$-geometric stacks. Now, suppose that we have a $(-1)\textcat{-P}$-morphism $\mathcal{F}\rightarrow{\mathcal{G}}$ between stacks. Suppose further that we have a representable stack $Y$ on $\mathcal{B}$ and a map $Y\rightarrow{L(\mathcal{G})}$. We note that, by Assumption \ref{surjectivityfunctorcondition}, there exists some {$X\in\mathcal{A}$} such that $L(X)\simeq Y$. Hence, since the map $\mathcal{F}\rightarrow{\mathcal{G}}$ is in $(-1)\textcat{-P}$, then
\begin{equation*}
L(\mathcal{F})\times_{L(\mathcal{G})}Y\simeq L(\mathcal{F})\times_{L(\mathcal{G})}L(X)\simeq L(\mathcal{F}\times_\mathcal{G}X)
\end{equation*}is a $(-1)$-geometric stack and the induced map $\mathcal{F}\times_\mathcal{G}X\rightarrow{X}$ is in $\textcat{P}$. Therefore, we see that $L(\mathcal{F})\rightarrow{L(\mathcal{G})}$ is in $(-1)\textcat{-Q}$.

Now, suppose that the statement holds for all $m$ such that $-1\leq m<n$. Suppose that $\mathcal{F}\in\textcat{Stk}(\mathcal{A},\bm{\tau}|_\mathcal{A})$ is an $n$-geometric stack with $n$-atlas $\{U_i\rightarrow{\mathcal{F}}\}_{i\in I}$. Then, using our assumptions, we can show that $L(\mathcal{F})$ has an $n$-atlas. Moreover, we easily note that the diagonal morphism $L(\mathcal{F})\rightarrow{L(\mathcal{F}\times\mathcal{F})\simeq L(\mathcal{F}\times\mathcal{F})}$ is $(n-1)$-representable using Assumption (\ref{surjectivityfunctorcondition}).

To conclude, suppose that $f:\mathcal{F}\rightarrow{\mathcal{G}}$ is an $n\textcat{-P}$ morphism of stacks in $\textcat{Stk}(\mathcal{A},\bm{\tau}|_\mathcal{A})$ and that there is a map $Y\rightarrow{L(\mathcal{G})}$ for some representable stack $Y$ on $\mathcal{B}$. We note that there exists a representable stack $X$ on $\mathcal{A}$ such that $L(X)\simeq Y$, and we can easily see that $L(\mathcal{F})\times_{L(\mathcal{G})}Y$ is $n$-geometric. Moreover, we know that there exists an $n$-atlas $\{U_i\rightarrow{\mathcal{F}\times_\mathcal{G} X}\}_{i\in I}$ such that the induced morphism $U_i\rightarrow{X}$ is in $\textcat{P}$. Hence, the $n$-atlas $\{L(U_i)\rightarrow{L(\mathcal{F})\times_{L(\mathcal{G})} Y}\}_{i\in I}$ is such that the induced morphism $L(U_i)\rightarrow{Y}$ is in $\textcat{Q}$. 
    
\end{proof}

If we now suppose that $\mathcal{B}$ is closed under $\bm\tau$-descent, then we can relax some of our conditions and obtain the following result.
\begin{prop}\phantomsection\label{ngeometricity}
Suppose that 
\begin{enumerate}
    \item $L$ preserves pullbacks of $n$-geometric stacks along $(n-1)\textcat{-P}$-morphisms for all $n\geq 0$, 
    \item $L$ sends $(-1)\textcat{-P}$-morphisms between $(-1)$-geometric stacks on $\mathcal{A}$ to $(-1)\textcat{-Q}$-morphisms between $(-1)$-geometric stacks on $\mathcal{B}$.
\end{enumerate}Then, for all $n\geq -1$, $L$ sends $n$-geometric stacks on $\mathcal{A}$ to $n$-geometric stacks on $\mathcal{B}$. Moreover, it sends $n\textcat{-P}$-morphisms between $n$-geometric stacks to $n\textcat{-Q}$-morphisms between $n$-geometric stacks.
\end{prop}

\begin{proof}
We prove this by induction on $n$. By assumption, the case $n=-1$ is satisfied. Now, suppose that the statement holds for all $m$ such that $-1\leq m< n$. Suppose that $\mathcal{F}\in\textcat{Stk}(\mathcal{A},\bm{\tau}|_\mathcal{A})$ is an $n$-geometric stack with $n$-atlas $\{U_i\rightarrow{\mathcal{F}}\}_{i\in I}$. Then, using our assumptions, we can show that $L(\mathcal{F})$ has an $n$-atlas and therefore $L(\mathcal{F})$ is $n$-geometric by Proposition \ref{onlynatlas}. 

Now, suppose that $f:\mathcal{F}\rightarrow{\mathcal{G}}$ is an $n$\textcat{-P}-morphism of geometric stacks in $\textcat{Stk}(\mathcal{A},\bm{\tau}|_{\mathcal{A}})$. Then, since $\mathcal{G}$ is $n$-geometric, it has an $n$-atlas $\{U_i\rightarrow{\mathcal{G}}\}_{i\in I}$ and, as before, $\{L(U_i)\rightarrow{L(\mathcal{G})\}_{i\in I}}$ is an $n$-atlas for $L(\mathcal{G})$. By Lemma \ref{pullbackgeometric}, to show that $L(f)$ is in $n\textcat{-Q}$, it suffices to show that each map $L(\mathcal{F})\times_{L(\mathcal{G})} L(U_i)\rightarrow{L({U}_i)}$ is in $n$\textcat{-Q}. Since $L$ preserves pullbacks of $n$-geometric stacks along $(n-1)\textcat{-P}$-morphisms, \begin{equation*}
    L(\mathcal{F})\times_{L(\mathcal{G})} L(U_i)\simeq L(\mathcal{F}\times_\mathcal{G}{U_i})
\end{equation*}Hence, since $\mathcal{F}\times_\mathcal{G}U_i\rightarrow{U_i}$ is an $n\textcat{-P}$-morphism of geometric stacks in $\textcat{Stk}(\mathcal{A},\bm{\tau}|_{\mathcal{A}})$ because $n\textcat{-P}$-morphisms are stable by pullbacks, we can reduce to the case when $\mathcal{G}$ is a $(-1)$-geometric stack $X$. 

So, suppose we have an $n\textcat{-P}$ morphism $f:\mathcal{F}\rightarrow{X}$ with $\mathcal{F}$ $n$-geometric. In this case, we consider an $n$-atlas $\{V_i\rightarrow{\mathcal{F}}\}_{i\in I}$ for $\mathcal{F}$, which we know induces an $n$-atlas $\{L(V_i)\rightarrow{L(\mathcal{F})}\}_{i\in I}$ for $L(\mathcal{F})$. The induced morphism $V_i\rightarrow{X}$ is in $\textcat{P}$ since $f$ is in $n\textcat{-P}$. Now, since $L$ sends morphisms in $\textcat{P}$ between $(-1)$-geometric stacks to morphisms in $\textcat{Q}$, we see that the image of the map $V_i\rightarrow{X}$ under $L$ is in $\textcat{Q}$. Therefore, $L(f)$ is in $n\textcat{-Q}$.
\end{proof}

\section{Derived Geometry Contexts}\label{chapter3}

In order to be able to discuss cotangent complexes and obstruction theories for derived stacks, we need to introduce more structure. This is facilitated by what we call a \textit{derived geometry context}. The notion of a \textit{derived algebraic context}, formalised by Raksit \cite{raksit_hochschild_2020}, has provided us with a versatile framework in which to explore both connective and non-connective geometric settings. Applications of this framework can be seen in \cite{ben-bassat_blow-ups_2023} and \cite{kelly_analytic_2022}. In \cite{ben-bassat_perspective_2024}, Ben-Bassat, Kelly, and Kremnizer introduce the notion of a \textit{spectral algebraic context} which both encapsulates these derived algebraic contexts and provides an appropriate context in which to do spectral algebraic geometry in the sense of Lurie \cite{lurie_spectral_2018}.

In this chapter we define the notion of a \textit{derived geometry context} by identifying a particular category $\mathcal{A}$ of affines along with a well-behaved system $\textcat{M}$ of categories of modules on $\mathcal{A}$. Within such a context, we obtain several important results about liftings of maps of stacks along first order infinitesimal deformations. 

\subsection{Derived Algebraic Contexts}

Suppose that $\mathcal{C}$ is a stable locally presentable symmetric monoidal $(\infty,1)$-category with unit $I$ and monoidal functor $\otimes$. Suppose that $(\mathcal{C}_{\geq 0},\mathcal{C}_{\leq 0})$ is a $t$-structure on $\mathcal{C}$, i.e. a $t$-structure on the triangulated homotopy category $\textnormal{Ho}(\mathcal{C})$. We recall that the heart $\mathcal{C}^\heartsuit$ of a $t$-structure on a stable $(\infty,1)$-category is an abelian category. 

\begin{defn}
    \cite[Definition 3.3.1]{raksit_hochschild_2020} The $t$-structure is \textit{compatible} if 
    \begin{enumerate}
        \item $\mathcal{C}_{\leq 0}$ is closed under filtered colimits in $\mathcal{C}$, 
        \item The unit object $I$ of $\mathcal{C}$ lies in $\mathcal{C}_{\geq 0}$, 
        \item If $A,B\in\mathcal{C}_{\geq 0}$, then $A\otimes B\in\mathcal{C}_{\geq 0}$.
    \end{enumerate}
\end{defn}

\begin{defn}\cite[Definition 4.2.1]{raksit_hochschild_2020}\phantomsection\label{derivedalgebraiccontext}
    A \textit{derived algebraic context} is a tuple $(\mathcal{C},\mathcal{C}_{\geq 0},\mathcal{C}_{\leq 0},\mathcal{C}^0)$ where
    \begin{enumerate}
        \item $\mathcal{C}$ is a stable locally presentable symmetric monoidal $(\infty,1)$-category, 
        \item $(\mathcal{C}_{\geq 0},\mathcal{C}_{\leq 0})$ is a complete $t$-structure on $\mathcal{C}$ which is compatible with the monoidal structure, 
        \item $\mathcal{C}^0$ is a small full subcategory of $\mathcal{C}^\heartsuit$ which is 
        \begin{itemize}
            \item a symmetric monoidal subcategory of $\mathcal{C}$,
            \item a generating set of compact projectives for $\mathcal{C}_{\geq 0}$,
            \item closed under the formation of $\mathcal{C}^\heartsuit$-symmetric powers, i.e. for any $A\in\mathcal{C}^0$ and $n\geq 0$, $\textnormal{Sym}^n_{\mathcal{C}^\heartsuit}(A)\in\mathcal{C}^0$,
            \item closed under the formation of finite coproducts in $\mathcal{C}$.
        \end{itemize}
    \end{enumerate}   
\end{defn}
We note that, by \cite[c.f. Proposition 5.5.8.22]{lurie_higher_2009}, there is a symmetric monoidal equivalence of $(\infty,1)$-categories
\begin{equation*}
    \mathcal{C}_{\geq 0}\simeq \mathcal{P}_{\Sigma}(\mathcal{C}^0):=\textcat{Fun}^\times(\mathcal{C}^{0,op},\infty\textcat{Grpd})
\end{equation*}The category $\mathcal{P}_{\Sigma}(\mathcal{C}^0)$ inherits a symmetric monoidal structure from $\mathcal{C}^0$ by \cite[Corollary 4.8.1.12]{lurie_higher_2017}. 

Within such a derived algebraic context we can define the notion of a \textit{derived commutative algebra object}. Indeed, consider the functor $\textnormal{Sym}_{\mathcal{C}^{\heartsuit}}:\mathcal{C}^0\rightarrow{\mathcal{C}_{\geq 0}}$ defined to be the composition of the functor $\mathcal{C}^\heartsuit\rightarrow{\mathcal{C}_{\geq 0}}$ with the usual symmetric algebra functor on $\mathcal{C}^\heartsuit$, restricted to $\mathcal{C}^0$. By \cite[Proposition 5.5.8.15]{lurie_higher_2009}, this functor extends to a functor 
\begin{equation*}
    \textcat{LSym}_{\mathcal{C}_{\geq 0}}:\mathcal{C}_{\geq 0}\rightarrow{\mathcal{C}_{\geq 0}}
\end{equation*}This is the \textit{derived symmetric algebra functor on $\mathcal{C}_{\geq 0}$} and is a monadic functor. By \cite[Construction 4.2.20]{raksit_hochschild_2020}, we can extend this functor to a monad on $\mathcal{C}$ which we will call the \textit{derived symmetric algebra monad on $\mathcal{C}$} and denote by $\textcat{LSym}_\mathcal{C}$. 

\begin{defn}
    \cite[Definition 4.2.22]{raksit_hochschild_2020} A \textit{derived commutative algebra object of $\mathcal{C}$} is a module over the derived symmetric algebra monad on $\mathcal{C}$. 
\end{defn}

We denote by $\textcat{DAlg}(\mathcal{C})$ the $(\infty,1)$-category of derived commutative algebra objects. This category is locally presentable, in particular it admits all small limits and colimits. We note that there is an adjunction
\begin{equation*}
    \textcat{LSym}_\mathcal{C}:\mathcal{C}\leftrightarrows{\textcat{DAlg}(\mathcal{C})}:U_\mathcal{C}
\end{equation*}where $U_\mathcal{C}$ is the forgetful functor. 

The $t$-structure on $\mathcal{C}$ allows us to define the following categories
\begin{equation*}\begin{aligned}
    \textcat{DAlg}^{\geq n}(\mathcal{C})&:=\textcat{DAlg}(\mathcal{C})\times_\mathcal{C}\mathcal{C}_{\geq n}\\
    \textcat{DAlg}^{\heartsuit}(\mathcal{C})&:=\textcat{DAlg}(\mathcal{C})\times_\mathcal{C}\mathcal{C}^{\heartsuit}\\
    \textcat{DAlg}^{\leq n}(\mathcal{C})&:=\textcat{DAlg}(\mathcal{C})\times_\mathcal{C}\mathcal{C}_{\leq n}
\end{aligned}\end{equation*}In particular, we will denote by $\textcat{DAlg}^{cn}(\mathcal{C})$ the \textit{$(\infty,1)$-category of connective derived commutative algebra objects}, i.e. the category $\textcat{DAlg}^{\geq 0}(\mathcal{C})$. We note that, as we only work in the connective setting in this thesis, we don't necessarily need the full machinery of derived algebraic contexts, but we use it in this work for consistency with the existing literature. 

\begin{exmp}\phantomsection\label{simplicialcommringexample}
    Let $k$ be a commutative ring. Let $\textnormal{Mod}_k^{fgf}$ be the ordinary category of finitely generated free $k$-modules and let $\textcat{Mod}_{k,\geq 0}$ be $\mathcal{P}_{\Sigma}(\textnormal{Mod}_k^{fgf})$. Define $\textcat{Mod}_k:=\textcat{Stab}(\textcat{Mod}_{k,\geq 0})$. Then, $(\textcat{Mod}_k,\textcat{Mod}_{k,\geq 0},\textcat{Mod}_{k,\leq 0},\textnormal{Mod}_k^{fgf})$ is a derived algebraic context and $\textcat{DAlg}^{cn}(\textcat{Mod}_k)$ is equivalent to the $(\infty,1)$-category of simplicial commutative $k$-algebras as defined in \cite[Definition 4.1.1]{lurie_derived_2009-1}.

\end{exmp}

We note that there is a functor from the $(\infty,1)$-category of simplicial commutative $k$-algebras to the $(\infty,1)$-category of $\mathbb{E}_\infty$-$k$-algebras but that this functor is not necessarily an equivalence unless $k$ contains the field $\mathbb{Q}$ of rational numbers \cite[Warning 7.1.4.21]{lurie_higher_2017}. More generally, by \cite[Proposition 4.2.27]{raksit_hochschild_2020}, there is a functor 
\begin{equation*}
    \Theta:\textcat{DAlg}(\mathcal{C})\rightarrow{\textcat{CAlg}(\mathcal{C})}
\end{equation*}which preserves small limits and colimits.

\begin{defn}
    Suppose that $A\in\textcat{DAlg}(\mathcal{C})$. Then,
    \begin{enumerate}
        \item We denote by $\textcat{Mod}_A$ the $(\infty,1)$-category $\textcat{Mod}_{\Theta(A)}(\mathcal{C})$, 
        \item We denote by $\textcat{DAlg}_A(\mathcal{C})$ the under category ${}^{A/}\textcat{DAlg}(\mathcal{C})$. 
    \end{enumerate}
\end{defn}

We note that, by \cite[Notation 4.2.28]{raksit_hochschild_2020}, $\Theta$ induces a monadic adjunction 
\begin{equation*}
    \textcat{LSym}_A:\textcat{Mod}_A\leftrightarrows{\textcat{DAlg}_A(\mathcal{C})}:U
\end{equation*}By \cite[Theorem 3.4.2 and Proposition 3.6.6]{lurie_derived_2009-2} and \cite[c.f. Remark 7.3.4.16]{lurie_higher_2017}, $\textcat{Mod}_A$ is a stable locally presentable symmetric monoidal $(\infty,1)$-category. We denote by $\otimes^\mathbb{L}_A$ the coproduct on $\textcat{Mod}_A$, which is given by the monoidal product $\otimes^\mathbb{L}_{\Theta(A)}$. Since $\textcat{DAlg}_A(\mathcal{C})$ is co-cartesian monoidal and the functor $U$ is monoidal we see that, for any two objects $B,B'\in\textcat{DAlg}_A(\mathcal{C})$, there is an equivalence between $B\coprod_AB$ and $B\otimes_A^\mathbb{L}B$ where $B\coprod_AB$ denotes the coproduct in $\textcat{DAlg}_A(\mathcal{C})$, considered as an object of $\textcat{Mod}_A$ under the forgetful functor. 

We will denote by $\textcat{Mod}_A^{\geq n}$ and $\textcat{Mod}_A^{\leq n}$ the $(\infty,1)$-categories defined by the fibre products
    \begin{equation*}\begin{aligned}
        \textcat{Mod}_A^{\geq n}&:=\textcat{Mod}_A\times_\mathcal{C} \mathcal{C}_{\geq n}\\
        \textcat{Mod}_A^{\heartsuit}&:=\textcat{Mod}_A\times_\mathcal{C} \mathcal{C}^{\heartsuit}\\
        \textcat{Mod}_A^{\leq n}&:=\textcat{Mod}_A\times_\mathcal{C} \mathcal{C}_{\leq n}
    \end{aligned}\end{equation*}Denote by $\textcat{Mod}_A^{cn}$ the category $\textcat{Mod}_A^{\geq 0}$ of \textit{connective $A$-modules}. 

For each $n\in\mathbb{Z}$, there are inclusion functors $\iota_{\geq n}:\mathcal{C}_{\geq n}\rightarrow{\mathcal{C}}$ and $\iota_{\leq n}:\mathcal{C}_{\leq n}\rightarrow{\mathcal{C}}$. The functor $\iota_{\geq n}$ has a right adjoint $\tau_{\geq n}:\mathcal{C}\rightarrow{\mathcal{C}_{\geq n}}$ and the functor $\iota_{\leq n}$ has a left adjoint $\tau_{\leq n}:\mathcal{C}\rightarrow{\mathcal{C}_{\leq n}}$. We denote by $\pi_n:\mathcal{C}\rightarrow{\mathcal{C}^{\heartsuit}}$ the functors $\pi_n=\tau_{\leq 0}\circ \tau_{\geq 0}\circ [n]$ where $[n]$ denotes the $n^{th}$ power of the suspension functor $\Sigma^n$ on $\mathcal{C}$.

\begin{remark}
    We note that these functors depend on the $t$-structure and don't necessarily correlate with the truncation functors defined in \cite[Proposition 5.5.6.18]{lurie_higher_2009}. 
\end{remark}
Suppose that $A\in\textcat{DAlg}^{cn}(\mathcal{C})$. Then, $(\textcat{Mod}_{A}^{\geq 0},\textcat{Mod}_{A}^{\leq 0})$ is a $t$-structure on $\textcat{Mod}_A$ by \cite[Lemma 2.3.96]{ben-bassat_perspective_2024}. Therefore, we get a well defined notion of truncation on $\textcat{Mod}_A$ which lines up with the truncation functors on $\mathcal{C}$. We want to similarly be able to define good notions of truncation functors on the category $\textcat{DAlg}^{cn}(\mathcal{C})$.

\begin{defn}\phantomsection\label{postnikovcompatible}
    We say that $\textcat{DAlg}^{cn}(\mathcal{C})$ is \textit{compatible with the $t$-structure on $\mathcal{C}$} if, for all $n\geq 0$, there exist adjunctions
    \begin{equation*}\begin{aligned}
        \tau_{\leq n}:\textcat{DAlg}^{cn}(\mathcal{C})&\leftrightarrows\textcat{DAlg}^{cn,\leq n}(\mathcal{C}):\iota_{\leq n}\\
        \iota_{\geq n}:\textcat{DAlg}^{cn,\geq n}(\mathcal{C})&\leftrightarrows\textcat{DAlg}^{cn}(\mathcal{C}):\tau_{\geq n}
    \end{aligned}\end{equation*}such that
$\tau_{\leq n}\circ U_{\mathcal{C}}\simeq U_\mathcal{C}\circ\tau_{\leq n}\quad\text{and}\quad\tau_{\geq n}\circ U_{\mathcal{C}}\simeq U_{\mathcal{C}}\circ\tau_{\geq n}$. In this situation, we will define the functor $\pi_n:\textcat{DAlg}^{cn}(\mathcal{C})\rightarrow{\textcat{DAlg}^\heartsuit(\mathcal{C})}$ to be $\tau_{\leq 0}\circ \tau_{\geq 0}\circ[n]$.
\end{defn}

\begin{remark}
    In many of the cases we are considering, particularly the categories of simplicial objects in an exact category, this condition will hold. 
\end{remark}

\subsection{Model Derived Algebraic Contexts}

Many of the derived algebraic contexts we will use in this thesis will be presented by \textit{model derived algebraic contexts} $(\textnormal{C},\textnormal{C}_{\geq 0},\textnormal{C}_{\leq 0},\textnormal{C}^0)$. A model derived algebraic context has a similar definition to Definition \ref{derivedalgebraiccontext} so we will not reproduce it here (see \cite[Definition 3.14]{kelly_analytic_2022} for details). In particular, any model derived algebraic context presents a derived algebraic context. 

\begin{prop}\cite[Proposition 3.15]{kelly_analytic_2022}\phantomsection\label{modelderivedcontext}
    Suppose that $(\textnormal{C},\textnormal{C}_{\geq 0},\textnormal{C}_{\leq 0},\textnormal{C}^0)$ is a model derived algebraic context. Then, $(\textcat{L}^H(\textnormal{C}),\textcat{L}^H(\textnormal{C}_{\geq 0}),\textcat{L}^H(\textnormal{C}_{\leq 0}),\textcat{L}^H(\textnormal{C}^0))$ is a derived algebraic context. 
\end{prop}

Recall that, in derived algebraic geometry, our derived affines correspond to simplicial commutative rings. By Example \ref{simplicialcommringexample}, the $\infty$-category of simplicial commutative rings can be realised as the category of connective algebras for some derived algebraic context. We obtain a similar result for simplicial algebra objects in certain exact categories. 

\begin{thm}\cite[c.f. Theorems 3.1.41, 3.1.42]{ben-bassat_perspective_2024}\phantomsection\label{modelderivedtheorem}
    Suppose that $\textnormal{E}$ is a complete closed symmetric monoidal elementary exact category such that 
    \begin{enumerate}
    \item $\textnormal{E}$ has symmetric projectives, i.e. for any projective $P$ and any $n\in\mathbb{N}$, $\textnormal{Sym}^n_\textnormal{E}(P)$ is projective,
    \item The tensor product of two projectives is projective, 
    \item The monoidal unit $I$ is projective,
    \item Projectives are flat.
    \end{enumerate}Let $\textnormal{P}^0$ be a set of compact projective generators closed under finite direct sums, tensor products, and the formation of symmetric powers. Then,
    \begin{enumerate}
        \item The category of simplicial objects in $\textnormal{E}$, $\textnormal{sE}$, forms a homotopical algebraic (HA) context $(\textnormal{sE},\textnormal{sE},\textnormal{Comm}(\textnormal{sE}))$ in the sense of \cite[Definition 1.1.0.11]{toen_homotopical_2008}, 
        \item  When $\textnormal{Ch}(\textnormal{E})$ is equipped with the projective model structure and the left $t$-structure as in \cite[Section 3.1]{henrard_left_2023}, $(\textnormal{Ch}(\textnormal{E}),\textnormal{Ch}_{\geq 0}(\textnormal{E}),\textnormal{Ch}_{\leq 0}(\textnormal{E}),\textnormal{P}^0)$ is a model derived algebraic context. 
    \end{enumerate}

\end{thm}
\begin{exmp}\phantomsection\label{derivedalgebraicgeometryexmp}
   An example of an exact category satisfying the conditions in Theorem \ref{modelderivedtheorem} is the category $\textnormal{Mod}_k$, for $k$ a commutative ring. This is exactly the situation in Example \ref{simplicialcommringexample}. 
\end{exmp}

\begin{cor}
    $(\textcat{Ch}(\textnormal{E}),\textcat{Ch}_{\geq 0}(\textnormal{E}),\textcat{Ch}_{\leq 0}(\textnormal{E}),\textcat{L}^H(P^0))$ is a derived algebraic context with $\textcat{Ch}(\textnormal{E}):=\textcat{L}^H(\textnormal{Ch}(\textnormal{E}))$.
\end{cor}
\begin{thm}\cite[Proposition 3.1.66]{ben-bassat_perspective_2024}
   There is an equivalence of $\infty$-categories 
    \begin{equation*}
        \textcat{DAlg}^{cn}(\textcat{Ch}(\textnormal{E}))\simeq \textcat{L}^H(\textnormal{Comm}(\textnormal{sE}))
    \end{equation*}
\end{thm}

\begin{exmp}\phantomsection\label{indbanderivedalgebraic}Suppose that $R$ is a Banach ring. We note, by \cite[Section 4.6]{kelly_analytic_2022}, that $\textnormal{Ind}(\textnormal{Ban}_R)$ is an example of a category satisfying the conditions of Theorem \ref{modelderivedtheorem} with $P^0$ its compact projective generators described in \cite[Section 4.2.1]{kelly_analytic_2022}. For $R$ a Banach ring, we note that $\textnormal{CBorn}_R$ is not necessarily elementary, but instead $\textcat{AdMon}$-elementary in the sense of \cite[Definition 4.5]{kelly_analytic_2022}, and so we cannot directly apply our theorem. However, by \cite[Remark 4.38]{kelly_analytic_2022}, we have a Quillen equivalence $\textnormal{Ch}(\textnormal{CBorn}_R)\simeq \textnormal{Ch}(\textnormal{Ind}(\textnormal{Ban}_R))$ and we can transfer the $t$-structure on $\textcat{Ch}(\textnormal{Ind}(\textnormal{Ban}_R))$ to one on $\textcat{Ch}(\textnormal{CBorn}_R)$ such that the positive parts agree. However, the transferred non-positive part $\widetilde{\textcat{Ch}}_{\leq 0}(\textnormal{CBorn}_R)$ does not correspond to the category $\textcat{Ch}_{\leq 0}(\textnormal{CBorn}_R)$. We obtain equivalent derived algebraic contexts
\begin{equation*}
    (\textcat{Ch}(\textnormal{Ind}(\textnormal{Ban}_R)),\textcat{Ch}_{\geq 0}(\textnormal{Ind}(\textnormal{Ban}_R)), \textcat{Ch}_{\leq 0}(\textnormal{Ind}(\textnormal{Ban}_R)),\textcat{L}^H(P^0))
\end{equation*}and
\begin{equation*}
    (\textcat{Ch}(\textnormal{CBorn}_R),\textcat{Ch}_{\geq 0}(\textnormal{CBorn}_R), \widetilde{\textcat{Ch}}_{\leq 0}(\textnormal{CBorn}_R),\textcat{L}^H(P^0))
\end{equation*}In particular, the categories of connective derived algebras will be equivalent in both contexts. 
\end{exmp}

\subsection{The Cotangent Complex}

The cotangent complex is a fundamental concept in derived algebraic geometry. It controls the \textit{deformation theory} of geometric objects. In this section, we define cotangent complexes for derived affine objects in our derived algebraic contexts, and then generalise to cotangent complexes of derived stacks. For the purposes of this thesis we will only focus on the connective setting, the non-connective setting can be dealt with using the theory of \cite{ben-bassat_perspective_2024}. The definitions we make in this section are analogous to the ones made in derived algebraic geometry, see \cite{toen_homotopical_2008}.

Suppose that $(\mathcal{C},\mathcal{C}_{\geq 0},\mathcal{C}_{\leq 0},\mathcal{C}^0)$ is a derived algebraic context. We denote by $\Sigma$ and $\Omega$ the suspension and loop functors respectively on $\mathcal{C}$. We recall that the \textit{stabilisation} $\textcat{Stab}(\mathcal{D})$ of any $(\infty,1)$-category $\mathcal{D}$ with finite limits exists by \cite[Proposition 1.4.2.17]{lurie_higher_2017} and that the suspension and loop functors on $\mathcal{D}$ induce functors $\Omega^{\infty}:\textcat{Stab}(\mathcal{D})\rightarrow{\mathcal{D}}$ and $\Sigma^\infty:\mathcal{D}\rightarrow{\textcat{Stab}(\mathcal{D})}$.

Classically (see \cite[Section 4]{quillen_co-homology_1970}), for any commutative ring $A$, there is an equivalence between the category of $A$-modules and the category of \textit{Beck modules}, i.e. abelian group objects in $\textnormal{CRing}_{/A}$. By \cite[Remark 2.4.117]{ben-bassat_perspective_2024}, we can obtain a similar result in our derived algebraic contexts. Suppose that $\mathcal{D}$ is an $(\infty,1)$-category with finite products. The \textit{$(\infty,1)$-category of abelian group objects in $\mathcal{D}$}, denoted $\textcat{Ab}(\mathcal{D})$, is defined to be
\begin{equation*}
    \textcat{Ab}(\mathcal{D}):=\textcat{Fun}^\times(\textnormal{Mod}_{\mathbb{Z}}^{fgf,op},\mathcal{D})
\end{equation*}

\begin{thm}
    Suppose that $A\in\textcat{DAlg}^{cn}(\mathcal{C})$. Then, there is an equivalence of $(\infty,1)$-categories
    \begin{equation*}
        \textcat{Mod}_A\simeq \textcat{Stab}(\textcat{Ab}(\textcat{DAlg}^{cn}(\mathcal{C})_{/A}))
    \end{equation*}
\end{thm}

This allows us to define the square-zero extension in a similar way to Lurie \cite[Remark 7.3.4.16]{lurie_higher_2017}. We denote the functor from $\textcat{Stab}(\textcat{Ab}(\textcat{DAlg}^{cn}(\mathcal{C})_{/A}))$ to $\textcat{Ab}(\textcat{DAlg}^{cn}(\mathcal{C})_{/A})$ by $\Omega^{\infty}_{\textcat{Ab}}$ and the natural forgetful functor from $\textcat{Ab}(\textcat{DAlg}^{cn}(\mathcal{C})_{/A})$ to ${\textcat{DAlg}^{cn}(\mathcal{C})_{/A}}$ given by evaluating at the unit, by $F$. 

\begin{defn}
    Suppose that $A\in\textcat{DAlg}^{cn}(\mathcal{C})$. The \textit{square-zero extension functor} is defined to be the functor 
    \begin{equation*}
    \textcat{sqz}_A:\textcat{Mod}_A\xrightarrow{\simeq}{\textcat{Stab}(\textcat{Ab}(\textcat{DAlg}^{cn}(\mathcal{C})_{/A}))}\xrightarrow{\Omega^\infty_{\textcat{Ab}}}{\textcat{Ab}(\textcat{DAlg}^{cn}(\mathcal{C})_{/A})}\xrightarrow{F}{\textcat{DAlg}^{cn}(\mathcal{C})_{/A}}
    \end{equation*}We denote the image of $M\in\textcat{Mod}_A$ under $\textcat{sqz}_A$ by $A\oplus M$ and call it the \textit{square-zero extension of $A$ by $M$.}
\end{defn}

\begin{remark}
    We note that, using a similar reasoning to \cite[Remark 7.3.4.16]{lurie_higher_2017}, this does in fact define something we can consider as the square-zero extension. 
\end{remark}

Suppose that $A\in\textcat{DAlg}^{cn}(\mathcal{C})$. Consider the functor $\Theta:\textcat{DAlg}(\mathcal{C})\rightarrow{\textcat{CAlg}(\mathcal{C})}$ and denote by $\Theta_A$ the induced functor on $\textcat{DAlg}(\mathcal{C})_{/A}$. We note that $\Theta_A(A\oplus M)$ is equivalent to the square-zero extension, in the sense of \cite[Remark 7.3.4.16]{lurie_higher_2017}, of $\Theta(A)\in\textcat{CAlg}^{cn}(\mathcal{C})$ by $M$. 

\begin{exmp}
    Consider the $(\infty,1)$-category of simplicial commutative rings as a derived algebraic context as in Example \ref{simplicialcommringexample}. Consider the functor which takes an object $A\simeq \textcat{LSym}_{\textcat{Mod}_\mathbb{Z}}(P)$ for some $P\in\textnormal{Mod}_\mathbb{Z}^{fgf}$ and a module $M$ of the form $A^{\oplus n}$ for some $n\geq 0$, and then computes the usual $1$-categorical square-zero extension. In \cite[Construction 25.3.1.1]{lurie_spectral_2018}, Lurie defines square-zero extensions for simplicial commutative rings by left Kan extending this functor. 
    
    By \cite[Remark 25.3.1.2]{lurie_spectral_2018}, for $A\in\textcat{DAlg}^{cn}(\textcat{Mod}_\mathbb{Z})$ and $M\in\textcat{Mod}_A^{cn}$, the underlying $\mathbb{E}_\infty$-algebra of this square-zero extension corresponds to $\Theta_A(A\oplus M)$. In particular, by \cite[Remark 4.4.12]{raksit_hochschild_2020}, this corresponds to the underlying $\mathbb{E}_\infty$-algebra of the square-zero extension of Raksit defined in \cite[Construction 4.4.5]{raksit_hochschild_2020}. Hence, in $\textcat{CAlg}^{cn}(\textcat{Mod}_\mathbb{Z})_{/A}$, all the definitions of square-zero extension correspond.
\end{exmp}

Using the definition of the square-zero extension, we can easily write down the following definition.
\begin{defn}
    For a morphism $f:A\rightarrow{B}$ in $\textcat{DAlg}^{cn}(\mathcal{C})$ and $M\in\textcat{Mod}_B$, we define the \textit{$\infty$-groupoid of $A$-derivations from $B$ to $M$} to be
\begin{equation*}
    \textcat{Der}_A(B,M):=\textnormal{Map}_{\textcat{DAlg}^{cn}_A(\mathcal{C})_{/B}}(B,B\oplus M)
\end{equation*}If $A=I$, then we denote the $\infty$-groupoid of derivations from $B$ to $M$ by $\textcat{Der}(B,M)$. 
\end{defn}

\begin{lem}There is an adjunction $L_A:\textcat{DAlg}^{cn}(\mathcal{C})_{/A}\leftrightarrows{\textcat{Mod}_A}:\textcat{sqz}_A$.
    
\end{lem}
\begin{proof}
We note that $\textcat{Mod}_A$ and $\textcat{DAlg}^{cn}(\mathcal{C})_{/A}$ are presentable. The functor $\textcat{sqz}_A$ preserves limits and is accessible since $\Omega_{\textcat{Ab}}^\infty$ is a right adjoint by \cite[Proposition 1.4.4.4]{lurie_higher_2017} and the functor $F$ is an accessible limit-preserving functor. Therefore, $\textcat{sqz}_A$ has a left adjoint by \cite[Corollary 5.5.2.9]{lurie_higher_2009}. 
\end{proof}

For any morphism $f:A\rightarrow{B}$ in $\textcat{DAlg}^{cn}(\mathcal{C})$, there is also an induced adjunction on the slice categories,
\begin{equation*}
L_{B/A}:\textcat{DAlg}^{cn}_A(\mathcal{C})_{/B}\leftrightarrows{\textcat{Mod}_B}:\textcat{sqz}_{B/A}
\end{equation*}By setting $\mathbb{L}_{B/A}:=L_{B/A}(B)$, we obtain the following result. 

\begin{cor}\phantomsection\label{derivationequivalence}Suppose that we have a morphism $f:A\rightarrow{B}$ in $\textcat{DAlg}^{cn}(\mathcal{C})$. Then, there exists a $B$-module $\mathbb{L}_{B/A}$ such that there is an equivalence of $\infty$-groupoids
\begin{equation*}
\textcat{Der}_A(B,M)\simeq \textnormal{Map}_{\textcat{Mod}_B}(\mathbb{L}_{B/A},M)
\end{equation*}for all $M\in\textcat{Mod}_B$.
\end{cor}

\begin{defn}Suppose that we have a morphism $f:A\rightarrow{B}$ in $\textcat{DAlg}^{cn}(\mathcal{C})$. The \textit{relative cotangent complex of $f$} is the unique $B$-module $\mathbb{L}_{B/A}$ from  Corollary \ref{derivationequivalence}. If $A=I$, then we denote the cotangent complex of the morphism $I\rightarrow{B}$ by $\mathbb{L}_B$.
\end{defn}

Suppose that $f:A\rightarrow{B}$ is a morphism in $\textcat{DAlg}(\mathcal{C})$. Then, there is a forgetful functor $f^*:\textcat{Mod}_B\rightarrow{\textcat{Mod}_A}$ which, by \cite[c.f. Theorem 3.6.7]{lurie_derived_2009-2}, has a left adjoint given by $-\otimes_A^\mathbb{L}B=f_!:\textcat{Mod}_A\rightarrow{\textcat{Mod}_B}$. We have the following results about the cotangent complex. 

\begin{lem}\phantomsection\label{cofibrerelcotangent}
    Suppose that we have a morphism $f:A\rightarrow{B}$ in $\textcat{DAlg}^{cn}(\mathcal{C})$. Then, there is a cofibre sequence in $\textcat{Mod}_B$
    \begin{equation*}
        \mathbb{L}_A\otimes_A^\mathbb{L}B\rightarrow{\mathbb{L}_B}\rightarrow{\mathbb{L}_{B/A}}
    \end{equation*}
\end{lem}

\begin{proof}
It suffices to check, since the Yoneda embedding reflects limits, that for all $M\in\textcat{Mod}_B$, $\textnormal{Map}_{\textcat{Mod}_B}({\mathbb{L}}_{B/A},M)$ is the fibre of the map 
    \begin{equation*}
        \textnormal{Map}_{\textcat{Mod}_B}(\mathbb{L}_B,M)\rightarrow{\textnormal{Map}_{\textcat{Mod}_B}(\mathbb{L}_A\otimes^\mathbb{L}_AB,M)}
    \end{equation*}
   For any $M\in\textcat{Mod}_B$, we have a chain of equivalences 
   \begin{equation*}
       \textnormal{Map}_{\textcat{Mod}_B}(\mathbb{L}_A\otimes^\mathbb{L}_A B,M)\simeq \textnormal{Map}_{\textcat{Mod}_A}(\mathbb{L}_A,M)
    \simeq \textnormal{Map}_{\textcat{DAlg}^{cn}(\mathcal{C})_{/A}}(A,A\oplus M)\end{equation*}and therefore we are reduced to evaluating the fibre of the map 
\begin{equation*}
    \textnormal{Map}_{\textcat{DAlg}^{cn}(\mathcal{C})_{/B}}(B,B\oplus M)\rightarrow{\textnormal{Map}_{\textcat{DAlg}^{cn}(\mathcal{C})_{/A}}(A,A\oplus M)}
\end{equation*}which is equivalent to $\textnormal{Map}_{\textcat{DAlg}^{cn}_A(\mathcal{C})_{/B}}(B,B\oplus M)\simeq \textnormal{Map}_{\textcat{Mod}_B}({\mathbb{L}}_{B/A},M)$, i.e. the collection of derivations $B\rightarrow{M}$ which act as the identity on $A$.
\end{proof}

The following results follow using Lemma \ref{cofibrerelcotangent} and using similar reasoning to \cite[Corollary 7.3.3.6, Proposition 7.3.3.7]{lurie_higher_2017}. 

\begin{cor}\phantomsection\label{cofibcotresults}\begin{enumerate}
    \item Suppose that $A\rightarrow{B}$ and $B\rightarrow{C}$ are maps in $\textcat{DAlg}^{cn}(\mathcal{C})$. Then, we have a cofibre sequence
    \begin{equation*}
        \mathbb{L}_{B/A}\otimes^\mathbb{L}_B C\rightarrow{\mathbb{L}_{C/A}\rightarrow{\mathbb{L}_{C/B}}}
    \end{equation*}in $\textcat{Mod}_C$, 
    \item Suppose that we have a pushout square in $\textcat{DAlg}^{cn}(\mathcal{C})$
    \begin{equation*}\begin{tikzcd}
        A\arrow{r} \arrow{d}& B \arrow{d}\\
        A'\arrow{r} & B'
    \end{tikzcd}
    \end{equation*}Then, there is an equivalence $\mathbb{L}_{B/A}\otimes^\mathbb{L}_B B'\simeq \mathbb{L}_{B'/A'}$ in $\textcat{Mod}_{B'}$. Moreover, there is a pushout square in $\textcat{Mod}_{B'}$
    \begin{equation*}
        \begin{tikzcd}
            \mathbb{L}_A\otimes^\mathbb{L}_A B'\arrow{r} \arrow{d}& \mathbb{L}_{B}\otimes^\mathbb{L}_B B' \arrow{d}\\
            \mathbb{L}_{A'}\otimes^\mathbb{L}_{A'}B'\arrow{r} & \mathbb{L}_{B'}
        \end{tikzcd}
    \end{equation*}
\end{enumerate}
    
\end{cor}

\subsection{Infinitesimal Extensions}

Suppose that $f:A\rightarrow{B}$ is a morphism in $\textcat{DAlg}^{cn}(\mathcal{C})$ and suppose that $M\in\textcat{Mod}_B$. If we consider the image of $M$ under $\textcat{sqz}_{B/A}$, then this defines a canonical morphism $i:B\oplus M\rightarrow{B}$. There is an inclusion map $s:B\rightarrow{B\oplus M}$ which comes from considering the image of the zero morphism $\mathbb{L}_{B/A}\rightarrow{M}$ under the equivalence 
\begin{equation*}
    \pi_0(\textnormal{Map}_{\textcat{Mod}_B}(\mathbb{L}_{B/A},M))\simeq\pi_0(\textcat{Der}_A(B,M))
\end{equation*}induced by Corollary \ref{derivationequivalence}. We can consider the map $i:B\rightarrow{B\oplus M}$ as a section of the map $s:B\oplus M\rightarrow{B}$ in $\textcat{DAlg}^{cn}_A(\mathcal{C})_{/B}$.

We denote by $\Omega_A:\textcat{Mod}_A\rightarrow{\textcat{Mod}_A}$ the loop functor and the suspension functor by $\Sigma_A:\textcat{Mod}_A\rightarrow{\textcat{Mod}_A}$ on $\textcat{Mod}_A$. For $M\in\textcat{Mod}_A$, we will occasionally use the notation $M[n]$ to denote the object $\Sigma^n(M)$. If $A\in\textcat{DAlg}^{cn}(\mathcal{C})$ and $M\in\textcat{Mod}_A$ then, since $\textcat{sqz}_A$ commutes with limits, there is a pullback square
    \begin{equation*}
        \begin{tikzcd}
            A\oplus \Omega_A(M) \arrow{r} \arrow{d}& A\arrow{d}{s}\\
            A\arrow{r}{s} & A\oplus M
        \end{tikzcd}
    \end{equation*}in $\textcat{DAlg}^{cn}(\mathcal{C})$. This motivates the following definition. 
     \begin{defn}Suppose that we have a morphism $f:A\rightarrow{B}$ in $\textcat{DAlg}^{cn}(\mathcal{C})$, $M\in\textcat{Mod}_B$, and a derivation $d\in\pi_0(\textcat{Der}_A(B,M))$. Then, the \textit{square zero extension associated to $d$}, denoted $B\oplus_d\Omega M$, is the pullback in $\textcat{DAlg}^{cn}_A(\mathcal{C})_{/B}$ of the diagram 
\begin{equation*}
    \begin{tikzcd}
        B\oplus_d\Omega M\arrow{r}\arrow{d} & B \arrow{d}{d}\\
        B\arrow{r}{s} & B\oplus M
    \end{tikzcd}
\end{equation*}
\end{defn}

\begin{lem}\phantomsection\label{pushoutAsqzM}Suppose that $A\in\textcat{DAlg}^{cn}(\mathcal{C})$, $M\in\textcat{Mod}_A$, and $d\in\pi_0(\textcat{Der}(A,M))$. Then, the following commutative diagram is both a pullback and pushout square in $\textcat{Mod}_{A}$. 
\begin{equation*}
    \begin{tikzcd}
        A\oplus_d\Omega M\arrow{r}\arrow{d} & A \arrow{d}{d}\\
        A\arrow{r}{s} & A\oplus M
    \end{tikzcd}
\end{equation*}Moreover, there is a fibre-cofibre sequence in $\textcat{Mod}_A$, $A\oplus_d\Omega M\rightarrow{A}\rightarrow{M}$.
\end{lem}

\begin{proof}Consider the square as a pullback square in $\textcat{Mod}_{A}$ under the forgetful functor $\textcat{DAlg}^{cn}_A(\mathcal{C})\rightarrow{\textcat{Mod}_A}$. Since $\textcat{Mod}_A$ is stable, the square is also a pushout square in $\textcat{Mod}_A$. The cofibres of the two horizontal morphisms in $\textcat{Mod}_A$ are equivalent to $M\in\textcat{Mod}_A$, and hence we can conclude. 
\end{proof}

\begin{remark}
    We note that, as we are working with connective objects, we will often have to restrict to modules $M$ in $\textcat{Mod}_A^{\geq 1}$ so that $\Omega_A(M)\in\textcat{Mod}_A^{cn}$. We note that, if $M\in\textcat{Mod}_A^{\geq 1}$, then the above sequence is also a fibre-cofibre sequence in $\textcat{Mod}_{A}^{cn}$. 
\end{remark}

\begin{cor}\phantomsection\label{Amodulepullbackpushout}Suppose $A\in\textcat{DAlg}^{cn}(\mathcal{C})$, $M\in\textcat{Mod}_A^{\geq 1}$, and $N\in\textcat{Mod}_{A\oplus_d\Omega M}^{cn}$. Then, there is an equivalence in $\textcat{DAlg}^{cn}(\mathcal{C})$. 
    \begin{equation*}
        N\simeq (N\otimes^\mathbb{L}_{A\oplus_d\Omega M}A)\times_{N\otimes^\mathbb{L}_{A\oplus_d\Omega M}(A\oplus M)}(N\otimes^\mathbb{L}_{A\oplus_d\Omega M}A)
    \end{equation*}
\end{cor}

\begin{proof}
    By the previous lemma, we have the following commutative diagram in $\textcat{Mod}_A$ where both squares are pushout diagrams. 
\begin{equation*}\begin{tikzcd}
        N\otimes^\mathbb{L}_{A\oplus_d\Omega M}A \arrow{r} & A\oplus M\arrow{r} & N\otimes^\mathbb{L}_{A\oplus_d\Omega M}(A\oplus M)\\
        N\arrow{r} \arrow{u}& A\arrow{r} \arrow{u}& N\otimes^\mathbb{L}_{A\oplus_d\Omega M}A \arrow{u}
\end{tikzcd}
\end{equation*}Hence, the outer rectangle is a pushout square in $\textcat{Mod}_A$, equivalently a pullback square in $\textcat{Mod}_A$. Since the forgetful functor $\textcat{DAlg}_A(\mathcal{C})\rightarrow{\textcat{Mod}_A}$ is conservative, we obtain the desired equivalence.  
    
\end{proof}
Suppose that we have a morphism $f:A\rightarrow{B}$ in $\textcat{DAlg}^{cn}(\mathcal{C})$. Consider the induced map of $B$-modules $\eta:\mathbb{L}_B\rightarrow{\mathbb{L}_{B/A}}$. We note that this induces a derivation $d_\eta:B\rightarrow{B\oplus\mathbb{L}_{B/A}}$. Consider the associated square-zero extension $B\oplus_{d_\eta}\Omega \mathbb{L}_{B/A}$. Since the restriction of $\eta$ to $\mathbb{L}_A$ is nullhomotopic, the map $f$ factors as a composition
\begin{equation*}
    A\xrightarrow{f'}{B\oplus_{d_\eta}\Omega \mathbb{L}_{B/A}}\xrightarrow{f''}B
\end{equation*}In particular, there is a map of $A$-modules $\textnormal{cofib}(f)\rightarrow{\textnormal{cofib}(f'')}$ which, using Lemma \ref{pushoutAsqzM}, induces a map of $B$-modules 
\begin{equation*}
    \epsilon_f:B\otimes_A\textnormal{cofib}(f)\rightarrow{\textnormal{cofib}(f'')\simeq \mathbb{L}_{B/A}}
\end{equation*} 

\begin{lem}\phantomsection\label{truncationcotangentequiv}
    Suppose that $f:A\rightarrow{B}$ is a morphism in $\textcat{DAlg}^{cn}(\mathcal{C})$ and that $f$ induces an equivalence $\tau_{\leq n}(A)\rightarrow{\tau_{\leq n}(B)}$ for some $n\in\mathbb{N}$. Then $\tau_{\leq n}(\mathbb{L}_{B/A})\simeq 0$. 
\end{lem}
\begin{proof}
    This follows in a similar way to \cite[Lemma 7.4.3.17]{lurie_higher_2017} using Lemma \ref{cofibrerelcotangent}. 
\end{proof}

By applying this lemma to the morphism $f:I\rightarrow{A}$, we see that if $A\in\textcat{DAlg}^{cn}(\mathcal{C})$, then $\mathbb{L}_A$ is connective. 

Now, suppose that for all $m\geq 2$, if $A\in\mathcal{C}$ is $n$-connective, then $\textcat{LSym}^m(A)$ is $(n+2)$-connective. This condition holds, for example, for $\mathbb{E}_\infty$-rings and simplicial commutative rings \cite[Proposition 25.2.4.1]{lurie_spectral_2018}. In this situation, we have the following. 
\begin{thm}\phantomsection\label{2nconnective}\cite[c.f. Theorem 7.4.3.12, Corollary 7.4.3.2]{lurie_higher_2017} Suppose that $f:A\rightarrow{B}$ is a morphism in $\textcat{DAlg}^{cn}(\mathcal{C})$ and that $\textnormal{cofib}(f)$ is $n$-connective for some $n\geq 1$. Then, 
\begin{enumerate}
    \item $\textnormal{fib}(\epsilon_f)$ is $(n+2)$-connective,
    \item The induced morphism $\mathbb{L}_{A}\rightarrow{\mathbb{L}_B}$ has $n$-connective cofibre. The converse is true if the morphism $\pi_0(A)\rightarrow{\pi_0(B)}$ is an isomorphism. 
\end{enumerate}
\end{thm}

\begin{cor}\phantomsection\label{cotangentcomplexpi0}
    Suppose that $f:A\rightarrow{B}$ is a morphism in $\textcat{DAlg}^{cn}(\mathcal{C})$. Then, as a $\pi_0(B)$-module, $\pi_0(\mathbb{L}_{B/A})$ is isomorphic to $\pi_0(\mathbb{L}_{\pi_0(B)/\pi_0(A)})$. 
\end{cor}
\begin{proof}By the previous result applied to the morphism $A\rightarrow{\pi_0(A)}$, we can deduce that the canonical map $\pi_0(\mathbb{L}_A)\rightarrow{\pi_0(\mathbb{L}_{\pi_0(A)})}$ is an isomorphism. The result then follows using the following morphism of long exact sequences. 
\begin{equation*}
    \begin{tikzcd}
        \dots\arrow{r} & {\pi_0(\mathbb{L}_A\otimes_A B)}\arrow{r}\arrow{d} & {\pi_0(\mathbb{L}_B)}\arrow{r} \arrow{d} & {\pi_0(\mathbb{L}_{B/A})}\arrow{d}\\
        \dots\arrow{r} & {\mathbb{L}_{\pi_0(A)}\otimes_{\pi_0(A)}\pi_0(B)}\arrow{r} & {\pi_0(\mathbb{L}_{\pi_0(B)})}\arrow{r}  & {\pi_0(\mathbb{L}_{\pi_0(B)/\pi_0(A)})}
    \end{tikzcd}
\end{equation*}
\end{proof}

\subsection{Smooth, \'Etale, and Perfect}\phantomsection\label{formalsection}

We now use our definition of the cotangent complex to make the following definitions, motivated by the classical definitions. The definitions of projective and perfect in this context can be found in Appendix \ref{perfectappendix}. 
\begin{defn}\phantomsection\label{formallyetaledefn}\cite[c.f. Definitions 1.2.6.1 and 1.2.7.1]{toen_homotopical_2008} Suppose that $f:A\rightarrow{B}$ is a morphism in $\textcat{DAlg}^{cn}(\mathcal{C})$. Then, 
\begin{itemize}
    \item $f$ is \textit{formally \'etale} if the natural morphism $\mathbb{L}_A\otimes^\mathbb{L}_AB\rightarrow{\mathbb{L}_B}$ is an equivalence in $\textcat{Mod}_B$, 
    \item $f$ is \textit{formally unramified} if $\mathbb{L}_{B/A}\simeq 0$ in $\textcat{Mod}_B$,
\item $f$ is \textit{formally perfect} if the $B$-module $\mathbb{L}_{B/A}$ is perfect.
\end{itemize}
    
\end{defn}

We can easily show that formally \'etale and formally unramified morphisms are stable by equivalences, compositions, and $(\infty,1)$-pushouts using Corollary \ref{cofibcotresults}.

\begin{lem}\phantomsection\label{perfectstable}
    Formally perfect morphisms in $\textcat{DAlg}^{cn}(\mathcal{C})$ are stable under equivalences, compositions, and pushouts. 
\end{lem}

\begin{proof}It is easy to see that equivalences are formally perfect. If we have formally perfect maps $A\rightarrow{B}$ and $B\rightarrow{C}$, then we have a cofibre sequence
    \begin{equation*}
        \mathbb{L}_{B/A}\otimes^\mathbb{L}_BC\rightarrow\mathbb{L}_{C/A}\rightarrow{\mathbb{L}_{C/B}}
    \end{equation*}by Corollary \ref{cofibcotresults}. This sequence splits by \cite[Proposition 7.2.2.6]{lurie_higher_2017} since perfect $C$-modules are projective. Hence, by \cite[Lemma 1.6.11]{calmes_hermitian_2025},
    \begin{equation*}
        \mathbb{L}_{C/A}\simeq (\mathbb{L}_{B/A}\otimes^\mathbb{L}_BC)\coprod\mathbb{L}_{C/B}
    \end{equation*}We note that perfect objects are stable under base-change by \cite[Section 2.1.5]{ben-bassat_perspective_2024}. Hence, since $\mathbb{L}_{C/A}$ is the coproduct of two perfect modules, it is perfect. Now, suppose that we have a formally perfect map $A\rightarrow{B}$ along with a map $A\rightarrow{C}$. Then, we see that, by Corollary \ref{cofibcotresults},
    \begin{equation*}
        \mathbb{L}_{B\otimes^\mathbb{L}_AC/C}\simeq \mathbb{L}_{B/A}\otimes^\mathbb{L}_AC
    \end{equation*}which is the base change of a perfect module and is therefore perfect. 
\end{proof}

\begin{defn}
    Suppose that $A\in\textcat{DAlg}^{cn}(\mathcal{C})$ and that $P\in\mathcal{C}_{\geq 0}$ is projective. Then, an $A$-module $M$ is \textit{$P$-projective} if $M$ is a retract in $\textcat{Mod}_A$ of $A\otimes P$. 
\end{defn}

\begin{defn}\cite[Definition 2.5.125]{ben-bassat_perspective_2024} Let $P\in\mathcal{C}_{\geq 0}$ be projective. Then, a map $f:A\rightarrow{B}$ in $\textcat{DAlg}^{cn}(\mathcal{C})$ is \textit{formally $P$-smooth} if 
\begin{enumerate}
    \item $\mathbb{L}_{B/A}$ is $P$-projective, 
    \item The morphism $\mathbb{L}_{A}\otimes^\mathbb{L}_AB\rightarrow{\mathbb{L}_B}$ has a retraction in $\textcat{Mod}_B^{cn}$. 
\end{enumerate}We say that $f$ is \textit{formally smooth} if it is formally $P$-smooth for some $P$. 
\end{defn}

We note that formally smooth morphisms in $\textcat{DAlg}^{cn}(\mathcal{C})$ are stable under equivalences, compositions, and pushouts by \cite[Lemma 2.5.129]{ben-bassat_perspective_2024}.  

\begin{prop}\phantomsection\label{etaleissmooth}Suppose that $f:A\rightarrow{B}$ is a morphism in $\textcat{DAlg}^{cn}(\mathcal{C})$. Then,
\begin{enumerate}
    \item $f$ is formally \'etale if and only if it is formally unramified,
    \item If $f$ is formally \'etale, then it is formally perfect and formally smooth.
\end{enumerate}
\end{prop}

\begin{proof}For the first result, we consider the cofibre sequence in $\textcat{Mod}_B$
\begin{equation*}
    \mathbb{L}_A\otimes_A^\mathbb{L}B\rightarrow{\mathbb{L}_B}\rightarrow{\mathbb{L}_{B/A}}
\end{equation*}Since $\textcat{Mod}_B$ is a stable $(\infty,1)$-category, the sequence is also a fibre sequence, from which the first result easily follows. If $f:A\rightarrow{B}$ is formally \'etale, then the natural morphism $\mathbb{L}_A\otimes^\mathbb{L}_A B\rightarrow{\mathbb{L}_B}$ is an equivalence and $\mathbb{L}_{B/A}\simeq 0$, which is perfect and projective. Hence, the second result follows.
\end{proof}

\begin{lem}\phantomsection\label{projectivevanishinghomotopy}
    Suppose that $A\in\textcat{DAlg}^{cn}(\mathcal{C})$ and that $P\in\textcat{Mod}_A$ is projective. Then, 
    \begin{enumerate}
        \item If $M$ is an object of $\textcat{Mod}_A^{cn}$, then 
        \begin{equation*}
            \textnormal{Hom}_{\textcat{Mod}_{\pi_0(A)}}(\pi_0(P),\pi_0(M))\simeq \pi_0(\textnormal{Map}_{\textcat{Mod}_A}(P,M))
        \end{equation*}
        \item For any $M\in\textcat{Mod}^{\geq 1}_{A}$, $\pi_0(\textnormal{Map}_{\textcat{Mod}_A}(P,M))=0$. 
    \end{enumerate}
\end{lem}

\begin{proof}
This follows immediately using \cite[Lemma 2.3.75]{ben-bassat_perspective_2024}. 
\end{proof}

\begin{prop}\phantomsection\label{formallyperfectformallysmooth}Formally perfect morphisms in $\textcat{DAlg}^{cn}(\mathcal{C})$ are formally smooth. 
\end{prop}

\begin{proof}
    Indeed, we note that if $f:A\rightarrow{B}$ is formally perfect, then the cotangent complex $\mathbb{L}_{B/A}$ is perfect, and hence projective. Now, if we consider the shifted fibre sequence
    \begin{equation*}
\mathbb{L}_B\rightarrow{\mathbb{L}_{B/A}}\rightarrow(\mathbb{L}_{A}\otimes^\mathbb{L}_AB)[1]
    \end{equation*}and, if we denote by $[-,-]$ the set $\pi_0(\textnormal{Map}_{\textcat{Mod}_B}(-,-))$, then we obtain a long exact sequence in homotopy 
    \begin{equation*}
        \dots\rightarrow[\mathbb{L}_{B/A},\mathbb{L}_B]\rightarrow{[\mathbb{L}_{B/A},\mathbb{L}_{B/A}]}\rightarrow{[\mathbb{L}_{B/A},(\mathbb{L}_{A}\otimes^\mathbb{L}_AB)[1]]}\rightarrow{0}
    \end{equation*}then, since $[\mathbb{L}_{B/A},(\mathbb{L}_{A}\otimes^\mathbb{L}_AB)[1]]\simeq 0$ by Lemma \ref{projectivevanishinghomotopy}, there is a surjection \begin{equation*}
        [\mathbb{L}_{B/A},\mathbb{L}_B]\rightarrow{[\mathbb{L}_{B/A},\mathbb{L}_{B/A}]} 
    \end{equation*}and hence the morphism $\mathbb{L}_B\rightarrow\mathbb{L}_{B/A}$ has a section. Therefore, the fibre sequence $\mathbb{L}_{A}\otimes^\mathbb{L}_AB\rightarrow\mathbb{L}_B\rightarrow{\mathbb{L}_{B/A}}$ splits, and so $\mathbb{L}_A\otimes_A^\mathbb{L}B\rightarrow{\mathbb{L}_B}$ has a retraction. 
\end{proof}

\subsection{Derived Geometry Contexts}

Suppose we have a derived algebraic context $(\mathcal{C},\mathcal{C}_{\geq 0},\mathcal{C}_{\leq 0},\mathcal{C}^0)$. We define the categories of affines and connective affines by 
\begin{equation*}
    \textcat{DAff}(\mathcal{C}):=\textcat{DAlg}(\mathcal{C})^{op}\quad\text{and}\quad \textcat{DAff}^{cn}(\mathcal{C}):=\textcat{DAlg}^{cn}(\mathcal{C})^{op}
\end{equation*}For $A\in\textcat{DAlg}(\mathcal{C})$, we will denote by $\textnormal{Spec}(A)$ its image in $\textcat{DAff}(\mathcal{C})$. Suppose that $\mathcal{A}$ is some full subcategory of $\textcat{DAff}^{cn}(\mathcal{C})$. Since we are not considering all affines, we need to restrict our categories of modules to the following \textit{good systems} in order to be able to define the relevant notions.

For a collection $\textcat{M}_A$ of $\textcat{A}$-modules, we will denote by $\textcat{M}_{A,n}$ the fibre product $\textcat{M}_{A}\times_{\textcat{Mod}_A}\textcat{Mod}_A^{\geq n}$. In particular, $\textcat{M}_A^{cn}$ is the fibre product $\textcat{M}_{A}\times_{\textcat{Mod}_A}\textcat{Mod}_A^{cn}$ and $\textcat{M}_A^\heartsuit$ is the fibre product $\textcat{M}_{A}\times_{\textcat{Mod}_A}\textcat{Mod}_A^{\heartsuit}$.

\begin{defn}\phantomsection\label{goodsystemmodules}
    A \textit{good system of categories of modules} on $\mathcal{A}$, denoted $\textcat{M}$, is an assignment, to each $A\in\mathcal{A}^{op}\subseteq \textcat{DAlg}^{cn}(\mathcal{C})$, of a full subcategory $\textcat{M}_A$ of $\textcat{Mod}_A$ satisfying the following properties
    \begin{enumerate}
        \item\label{good1} $A\in\textcat{M}_A$,
        \item\label{good2} $\mathbb{L}_A\in\textcat{M}_A$,
        \item\label{good3} $\pi_n(A)\in\textcat{M}_A$ for all $n\geq 0$,
        \item\label{good4} Whenever $M\in\textcat{M}_A$ and $f:A\rightarrow{B}$ is a morphism in $\mathcal{A}^{op}$, $M\otimes^\mathbb{L}_AB\in\textcat{M}_B$, 
        \item\label{good5} For any $M\in\textcat{M}_{A,1}$ and any derivation $d\in\pi_0(\textcat{Der}(A,M))$, $\textnormal{Spec}(A\oplus_d\Omega M)$ is $\mathcal{A}$-admissible,
        \item \label{good5half} If $X=\textnormal{Spec}(A)\in \mathcal{A}^{\leq k}:=\mathcal{A}\times_{\textcat{DAff}^{cn} (\mathcal{C})}\textcat{DAff}^{\leq k}(\mathcal{C})$ for some $k>0$, then for any $M\in\textcat{M}_A^\heartsuit$ and any derivation $d\in\pi_0(\textcat{Der}(A,M))$, $\textnormal{Spec}(A\oplus_d\Omega M[k+1])\in\mathcal{A}$,
        \item\label{good6} If $M\in\textcat{M}_A$, then $\Omega_A(M)\in\textcat{M}_A$ and $\Sigma_A(M)\in\textcat{M}_A$,
        \item\label{good7} $\textcat{M}_A$ is closed under equivalences, retracts, and finite colimits. 
    \end{enumerate}
\end{defn}

Combining the definitions of derived algebraic contexts with our notion of a geometry tuple from Section \ref{geometriessection}, we can define a \textit{derived geometry context}.

\begin{defn}\phantomsection\label{derivedgeometrycontext}
    We say that $(\mathcal{C},\mathcal{C}_{\geq 0},\mathcal{C}_{\leq 0},\mathcal{C}^0,\bm\tau,\textcat{P},\mathcal{A},\textcat{M})$ is a \textit{derived geometry context} if
    \begin{enumerate}
        \item $(\mathcal{C},\mathcal{C}_{\geq 0},\mathcal{C}_{\leq 0},\mathcal{C}^0)$ is a derived algebraic context, 
        \item $\mathcal{A}$ is a full subcategory of $\textcat{DAff}^{cn}(\mathcal{C})$,
        \item $(\textcat{DAff}^{cn}(\mathcal{C}),\bm\tau,\textcat{P},\mathcal{A})$ is a strong relative $(\infty,1)$-geometry tuple, 
        \item $\textcat{M}$ is a good system of categories of modules on $\mathcal{A}$.
    \end{enumerate}
\end{defn}

\subsection{Cotangent Complexes of Presheaves}
Suppose that $(\mathcal{C},\mathcal{C}_{\geq 0},\mathcal{C}_{\leq 0},\mathcal{C}^0,\bm\tau,\textcat{P},\mathcal{A},\textcat{M})$ is a derived geometry context. 
We recall that there is a fully faithful functor $i:\textcat{PSh}(\mathcal{A})\rightarrow{\textcat{PSh}(\textcat{DAff}^{cn}(\mathcal{C}))}$ which, since we are assuming our tuple is strong, induces a fully faithful functor on stacks.

Suppose that $X=\textnormal{Spec}(A)\in\mathcal{A}$ and $M\in\textcat{M}_A^{cn}$. Define $X[M]:=\textnormal{Spec}(A\oplus M)$, which is a stack on $\mathcal{A}$ by our assumptions. If we consider the map $A\oplus M\rightarrow{A}$ specified by the square-zero extension construction, then we obtain an induced morphism of affines $X\rightarrow{X[M]}$. Now, suppose that we have some morphism of presheaves $f:\mathcal{F}\rightarrow{\mathcal{G}}$ in $\textcat{PSh}(\textcat{DAff}^{cn}(\mathcal{C}))$ and a morphism $x:X\rightarrow{\mathcal{F}}$. 

\begin{defn}\phantomsection\label{derivationrelative}
    The \textit{$\infty$-groupoid of derived $\mathcal{F}/\mathcal{G}$ derivations from $\mathcal{F}$ to $M$ at $x$} is defined to be
    \begin{equation*}
        \textcat{Der}_{\mathcal{F}/\mathcal{G}}(X,M):=\textnormal{Map}_{^{X/}\textcat{PSh}(\textcat{DAff}^{cn}(\mathcal{C}))_{/\mathcal{G}}}(X[M],\mathcal{F})
    \end{equation*}When $\mathcal{G}=*$, we denote the $\infty$-groupoid of derivations by $\textcat{Der}_\mathcal{F}(X,M)$. 
\end{defn}

\begin{remark}
    By \cite[c.f. Lemma 5.5.5.12]{lurie_higher_2009}, $\textcat{Der}_{\mathcal{F}/\mathcal{G}}(X,M)$ is the fibre of the map $\textcat{Der}_{\mathcal{F}}(X,M)\rightarrow{\textcat{Der}_{\mathcal{G}}(X,M)}$ at $x$. 
\end{remark}

\begin{lem}\phantomsection\label{alternativederivationdefn}Let $X=\textnormal{Spec}(A)\in\mathcal{A}$ and $M\in\textcat{M}_A^{cn}$. Suppose that we have a derivation $d\in\pi_0(\textcat{Der}(A,M))$ and let $p$ be the map $\mathcal{F}(A)\times_{\mathcal{F}(A\oplus M)}\mathcal{F}(A)\rightarrow{\mathcal{F}(A)}$ induced by the map $A\oplus M\rightarrow{A}$. Then, there is an equivalence of $(\infty,1)$-groupoids
\begin{equation*}
    \Omega_{d(x),0}\textcat{Der}_\mathcal{F}(X,M)\simeq \textnormal{fib}(\mathcal{F}(A)\times_{\mathcal{F}(A\oplus M)}\mathcal{F}(A)\xrightarrow{p}{\mathcal{F}(A)})
\end{equation*}where $\Omega_{d(x),0}(-)$ is the path space and the fibre on the right is taken at $x:X\rightarrow{\mathcal{F}}$.
\end{lem}

\begin{proof}We construct the following commutative diagram 
    \begin{equation*}
        \begin{tikzcd}[row sep = 12]
            \textnormal{fib}(p)\arrow{r} \arrow{d} & * \arrow{d}{x}\\
            \mathcal{F}(A)\times_{\mathcal{F}(A\oplus M)}\mathcal{F}(A)\arrow{r}{p}\arrow{d} & \mathcal{F}(A)\arrow{d}{d}\\
            *\arrow{r}{0}\arrow{d} & \textcat{Der}_\mathcal{F}(X,M)\arrow{r}\arrow{d} & *\arrow{d}\\
            \mathcal{F}(A)\arrow{r}& \mathcal{F}(A\oplus M)\arrow{r} &  \mathcal{F}(A)
        \end{tikzcd}
    \end{equation*}We note that the bottom right square is a pullback square by the above remarks and we can easily deduce that all other squares are pullback squares. Hence, we obtain the desired result.
\end{proof}

\begin{defn} We say that $f$ has a \textit{relative cotangent complex at $x$} if there is an integer $n$, a module $\mathbb{L}_{\mathcal{F}/\mathcal{G},x}$ in $\textcat{M}_{A,-n}$, and an equivalence
\begin{equation*}
    \textcat{Der}_{\mathcal{F}/\mathcal{G}}(X,-)\simeq \textnormal{Map}_{\textcat{Mod}_A}(\mathbb{L}_{\mathcal{F}/\mathcal{G},x},-)
\end{equation*}in $\textcat{PSh}((\textcat{M}_A^{cn})^{op})$. If the cotangent complex exists, then we say that $\mathbb{L}_{\mathcal{F}/\mathcal{G},x}$ is the \textit{relative cotangent complex of $f$ at $x$}. If $\mathcal{G}=*$, then we denote the cotangent complex of the morphism $\mathcal{F}\rightarrow{*}$ at $x:X\rightarrow{\mathcal{F}}$ by $\mathbb{L}_{\mathcal{F},x}$.  
    
\end{defn}

\begin{prop}\phantomsection\label{representablecotangent}
    Any representable presheaf $X$ in $\textcat{PSh}(\textcat{DAff}^{cn}(\mathcal{C}))$ has a cotangent complex for any map $y:Y\rightarrow{X}$ where $Y\in\mathcal{A}$. Furthermore, this cotangent complex is connective.
\end{prop}

\begin{proof}
    Suppose that $X=\textnormal{Spec}(A)$ and $Y=\textnormal{Spec}(B)$ are representable presheaves and that there is a point $y:Y\rightarrow{X}$ corresponding to a morphism $f:A\rightarrow{B}$. Then, we can easily show that, for any $M\in\textcat{M}_{B}^{cn}$, there is an equivalence
    \begin{equation*}
    \textcat{Der}_X(Y,M):=\textnormal{Map}_{{}^{Y/}{\textcat{PSh}(\textcat{DAff}^{cn}(\mathcal{C}))}}(Y[M],X)
        \simeq\textnormal{Map}_{\textcat{Mod}_B}(\mathbb{L}_A\otimes_A^\mathbb{L}B,M)
    \end{equation*}Hence, let $\mathbb{L}_{X,y}:=\mathbb{L}_A\otimes_A^\mathbb{L}B$ which we know is connective and, moreover, lies in $\textcat{M}_B$ by our assumptions.
\end{proof}

\begin{lem}\phantomsection\label{Dfullyfaithfulconnective} Suppose that $A\in\mathcal{A}^{op}$. The functor
    \begin{equation*}\begin{aligned}
        D:\textcat{Mod}_A&\rightarrow{\textcat{PSh}((\textcat{M}_A^{cn})^{op})}\\
        M&\rightarrow{\textnormal{Map}_{\textcat{Mod}_A}(M,-)}
    \end{aligned}\end{equation*}is fully faithful when restricted to the subcategory $\textcat{M}_{A,-n}$. 
\end{lem}

\begin{proof}
    Indeed, suppose that $M\in\textcat{M}_{A,-n}$ and $N\in\textcat{M}_A^{cn}$. Then, we note that $\Sigma_A^n(M)\in\textcat{M}_A^{cn}$ and hence, since $\Sigma_A$ is an equivalence, 
    \begin{equation*}
\textnormal{Map}_{\textcat{Mod}_A}(M,N)\simeq \textnormal{Map}_{\textcat{Mod}_A^{cn}}(\Sigma_A^n(M),\Sigma_A^n(N))
    \end{equation*}Therefore, if we denote by $h$ the Yoneda embedding $h:(\textcat{M}^{cn}_A)^{op}\rightarrow{\textcat{PSh}((\textcat{M}_A^{cn})^{op})}$, then, by \cite[Section 1.2.11]{lurie_higher_2009}, we are reduced to showing that for any $M_1,M_2\in\textcat{M}_{A,-n}$, the map
    \begin{equation*}
        \pi_0(\textnormal{Map}_{\textcat{Mod}_A}(M_1,M_2))\rightarrow{\pi_0(\textnormal{Map}_{\textcat{PSh}((\textcat{M}_A^{cn})^{op})}(h(\Sigma_A^n(M_1))\circ \Sigma_A^n,h(\Sigma_A^n(M_2))\circ\Sigma_A^n)}
    \end{equation*}is an isomorphism. By fully faithfulness of $h$, there exists a function
    \begin{equation*}
        \begin{aligned}
            &\textnormal{Map}_{\textcat{PSh}((\textcat{M}_A^{cn})^{op})}(h(\Sigma_A^n(M_1))\circ\Sigma_A^n,h(\Sigma_A^n(M_1))\circ\Sigma_A^n)\\
            &\rightarrow \textnormal{Map}_{\textcat{PSh}((\textcat{M}_A^{cn})^{op})}(h(\Sigma_A^n(M_1))\circ \Sigma_A^n\circ\Omega_A^n,h(\Sigma_A^n(M_1))\circ\Sigma_A^n\circ\Omega_A^n)\\
            &\simeq \textnormal{Map}_{\textcat{PSh}((\textcat{M}_A^{cn})^{op})}(h(\Sigma_A^n(M_1)),h(\Sigma_A^n(M_1)))\\
            &\simeq\textnormal{Map}_{\textcat{M}_A^{cn}}(\Sigma_A^n(M_1),\Sigma_A^n(M_2))\\
            &\simeq \textnormal{Map}_{\textcat{Mod}_A}(M_1,M_2)
        \end{aligned}
    \end{equation*}and taking $\pi_0$ of this map provides the necessary inverse. 
    
\end{proof}

Now, consider the morphism $f:\mathcal{F}\rightarrow{\mathcal{G}}$ of presheaves in $\textcat{PSh}(\textcat{DAff}^{cn}(\mathcal{C}))$ and suppose that we have morphisms $x:X\rightarrow{\mathcal{F}}$ and $y:Y\rightarrow{\mathcal{F}}$. Then, there is an induced morphism 
\begin{equation*}
    \textcat{Der}_{\mathcal{F}/\mathcal{G}}(Y,-)\rightarrow{\textcat{Der}_{\mathcal{F}/\mathcal{G}}(X,-)}
\end{equation*}which, if $f$ has relative cotangent complexes at $x$ and $y$, induces a morphism
\begin{equation*}
    \textnormal{Map}_{\textcat{Mod}_B}(\mathbb{L}_{\mathcal{F}/\mathcal{G},y},-)\rightarrow{\textnormal{Map}_{\textcat{Mod}_A}(\mathbb{L}_{\mathcal{F}/\mathcal{G},x},-)}\simeq {\textnormal{Map}_{\textcat{Mod}_B}(\mathbb{L}_{\mathcal{F}/\mathcal{G},x}\otimes_A^\mathbb{L}B,-)}\end{equation*}and hence, by Lemma \ref{Dfullyfaithfulconnective}, there is a morphism 
\begin{equation*}
\mathbb{L}_{\mathcal{F}/\mathcal{G},x}\otimes_A^\mathbb{L}B\rightarrow{\mathbb{L}_{\mathcal{F}/\mathcal{G},y}}
\end{equation*}

\begin{defn}
    We say that a morphism $f:\mathcal{F}\rightarrow{\mathcal{G}}$ in $\textcat{PSh}(\textcat{DAff}^{cn}(\mathcal{C}))$ has a \textit{relative global cotangent complex (relative to $\mathcal{A}$)} if it satisfies the following two conditions.
    \begin{enumerate}
        \item For any $X\in\mathcal{A}$ and any point $x:X\rightarrow{\mathcal{F}}$, the morphism $f$ has a cotangent complex $\mathbb{L}_{\mathcal{F}/\mathcal{G},x}$, 
        \item For any morphism $g:A\rightarrow{B}$ in $\mathcal{A}^{op}$ and any morphism  \begin{equation*}
    \begin{tikzcd}
        Y\arrow{rr}{g} \arrow{rd}{y}& & X\arrow{ld}{x}\\
        & \mathcal{F}
    \end{tikzcd}
\end{equation*}in $\textcat{PSh}(\textcat{DAff}^{cn}(\mathcal{C}))$, the induced morphism $\mathbb{L}_{\mathcal{F}/\mathcal{G},x}\otimes^\mathbb{L}_AB\rightarrow{\mathbb{L}_{\mathcal{F}/\mathcal{G},y}}$ is an equivalence of modules in $\textcat{M}_{B}$.
    \end{enumerate}
\end{defn}

It is clear, using Proposition \ref{representablecotangent}, that any representable in $\textcat{PSh}(\textcat{DAff}^{cn}(\mathcal{C}))$ has a global cotangent complex. We collect together the following results about the global cotangent complex. The proofs follow in the same way as in \cite[Lemma 1.4.1.16]{toen_homotopical_2008}.
\begin{lem}\phantomsection\label{globalcotangentresults} Suppose that $f:\mathcal{F}\rightarrow{\mathcal{G}}$ is a morphism in $\textcat{PSh}(\textcat{DAff}^{cn}(\mathcal{C}))$. 
\begin{enumerate}
    \item\label{item1} If both $\mathcal{F}$ and $\mathcal{G}$ have global cotangent complexes, then $f$ has a relative global cotangent complex. Furthermore, for any $X=\textnormal{Spec}(A)\in\mathcal{A}$ and any morphism of presheaves $x:X\rightarrow{\mathcal{F}}$, there is a natural fibre-cofibre sequence of modules in $\textcat{M}_A$,
    \begin{equation*}
        \mathbb{L}_{\mathcal{G},x}\rightarrow{\mathbb{L}_{\mathcal{F},x}}\rightarrow{\mathbb{L}_{\mathcal{F}/\mathcal{G},x}}
    \end{equation*}
    \item If the morphism $f$ has a relative global cotangent complex, then for any presheaf $\mathcal{H}\in\textcat{PSh}(\textcat{DAff}^{cn}(\mathcal{C}))$ and any morphism $\mathcal{H}\rightarrow{\mathcal{G}}$, the morphism $\mathcal{F}\times_{\mathcal{G}}\mathcal{H}\rightarrow{\mathcal{H}}$ has a relative global cotangent complex, and furthermore we have 
    \begin{equation*}
        \mathbb{L}_{\mathcal{F}/\mathcal{G},x}\simeq \mathbb{L}_{\mathcal{F}\times_{\mathcal{G}}\mathcal{H}/\mathcal{H},x}
    \end{equation*}for any $X\in\mathcal{A}$ and any morphism of presheaves $x:X\rightarrow{\mathcal{F}\times_\mathcal{G}\mathcal{H}}$, 
    \item \label{globalitem3} If, for any $X=\textnormal{Spec}(A)\in\mathcal{A}$ and any morphism of presheaves $x:X\rightarrow{\mathcal{F}}$, the morphism $\mathcal{F}\times_{\mathcal{G}}X\rightarrow{X}$ has a relative global cotangent complex, then the morphism $f$ has a relative global cotangent complex. Furthermore, we have 
    \begin{equation*}
        \mathbb{L}_{\mathcal{F}/\mathcal{G},x}\simeq \mathbb{L}_{\mathcal{F}\times_{\mathcal{G}}X/X,x^\cdot}\end{equation*}where $x^\cdot:X\rightarrow{\mathcal{F}\times_\mathcal{G}X}$ denotes the lift of the point $x$. 
\end{enumerate}
    
\end{lem}

\begin{cor}\phantomsection\label{cofibrecotangentstack}
    Suppose that we have a commutative diagram in $\textcat{PSh}(\textcat{DAff}^{cn}(\mathcal{C}))$ 
    \begin{equation*}
        \begin{tikzcd}
            X\arrow{r}{x} & \mathcal{F}\arrow{r} &  \mathcal{G}\\
             Y\arrow{r}{y} \arrow{u}& \mathcal{H} \arrow{ur} \arrow{u}
        \end{tikzcd}
    \end{equation*}with $X=\textnormal{Spec}(A)$ and $Y=\textnormal{Spec}(B)$ in $\mathcal{A}$. If $\mathcal{F}, \mathcal{G}$, and $\mathcal{H}$ have global cotangent complexes, then there exists a fibre-cofibre sequence of modules in $\textcat{M}_{B}$,
    \begin{equation*}
\mathbb{L}_{\mathcal{F}/\mathcal{G},x}\otimes^\mathbb{L}_AB\rightarrow\mathbb{L}_{\mathcal{H}/\mathcal{G},y}\rightarrow\mathbb{L}_{\mathcal{H}/\mathcal{F},y}
    \end{equation*}
\end{cor}

\begin{proof}Using the previous lemma, we obtain the following commutative diagram of modules in $\textcat{M}_{B}$,
\begin{equation*}
    \begin{tikzcd}
        \mathbb{L}_{\mathcal{G},y}\arrow{r} \arrow{d} & \mathbb{L}_{\mathcal{F},y}\arrow{r} \arrow{d} & \mathbb{L}_{\mathcal{H},y} \arrow{d}\\
        * \arrow{r}  & \mathbb{L}_{\mathcal{F}/\mathcal{G},y}\arrow{r} \arrow{d}& \mathbb{L}_{\mathcal{H}/\mathcal{G},y} \arrow{d}\\
        & * \arrow{r} & \mathbb{L}_{\mathcal{H}/\mathcal{F},y}
    \end{tikzcd}
\end{equation*}and hence we obtain a fibre-cofibre sequence of modules in $\textcat{M}_{B}$
 \begin{equation*}
\mathbb{L}_{\mathcal{F}/\mathcal{G},y}\rightarrow\mathbb{L}_{\mathcal{H}/\mathcal{G},y}\rightarrow\mathbb{L}_{\mathcal{H}/\mathcal{F},y}
    \end{equation*}Since the morphism $f:\mathcal{F}\rightarrow{\mathcal{G}}$ has a global cotangent complex, we obtain the following fibre-cofibre sequence in $\textcat{M}_B$ 
        \begin{equation*}
        \mathbb{L}_{\mathcal{F}/\mathcal{G},x}\otimes_A^{\mathbb{L}}B\rightarrow\mathbb{L}_{\mathcal{H}/\mathcal{G},y}\rightarrow\mathbb{L}_{\mathcal{H}/\mathcal{F},y}
    \end{equation*}
\end{proof}

\section{Obstruction Theories}\label{obstructionsection}

In this section, we consider the problem of when a presheaf has an \textit{obstruction theory}. These obstruction theories control lifting problems along first-order infinitesimal deformations. 

\subsection{Obstruction Theory of Presheaves}
In this section, we consider the problem of when a presheaf has an \textit{obstruction theory}. These obstruction theories control lifting problems along first-order infinitesimal deformations. 
Suppose 
$(\mathcal{C},\mathcal{C}_{\geq 0},\mathcal{C}_{\leq 0},\mathcal{C}^0,\bm\tau,\textcat{P},\mathcal{A},\textcat{M})$ is a derived geometry context. Suppose that $X=\textnormal{Spec}(A)\in\mathcal{A}$, $M\in\textcat{M}_A^{cn}$, and $d\in\pi_0(\textcat{Der}(A,M))$. We define the object
\begin{equation*}X_d[\Omega M]:=\textnormal{Spec}(A\oplus_d\Omega M)
\end{equation*}If $M\in\textcat{M}_{A,1}$, then $X_d[\Omega M]$ is $\mathcal{A}$-admissible by our assumptions on the good system of categories of modules $\textcat{M}$.

\begin{defn}\phantomsection\label{obstructiontheorydefn} Suppose that $f:\mathcal{F}\rightarrow{\mathcal{G}}$ is a morphism of presheaves in $\textcat{PSh}(\textcat{DAff}^{cn}(\mathcal{C}))$. Then, 
\begin{enumerate}
    \item $f$ is \textit{infinitesimally cartesian relative to $\mathcal{A}$} if, for any $X=\textnormal{Spec}(A)\in\mathcal{A}$, any $A$-module $M\in\textcat{M}_{A,1}$, and any derivation $d\in\pi_0(\textcat{Der}(A,M))$, corresponding to a morphism $d:A\rightarrow{A\oplus M}$, the square
    \begin{equation*}\begin{tikzcd}
        \mathcal{F}(A\oplus_d\Omega M)\arrow{r} \arrow{d}& \mathcal{G}(A\oplus_d\Omega M) \arrow{d}\\
        \mathcal{F}(A)\times_{\mathcal{F}(A\oplus M)}\mathcal{F}(A)\arrow{r} &  \mathcal{G}(A)\times_{\mathcal{G}(A\oplus M)}\mathcal{G}(A)
    \end{tikzcd}
    \end{equation*}is a pullback square in $\infty\textcat{Grpd}$, 
    \item $f$ \textit{has an obstruction theory relative to $\mathcal{A}$} if it has a global cotangent complex relative to $\mathcal{A}$ and is infinitesimally cartesian relative to $\mathcal{A}$.
    \end{enumerate}

We say that a presheaf $\mathcal{F}$ in $\textcat{PSh}(\textcat{DAff}^{cn}(\mathcal{C}))$ \textit{has an obstruction theory} if the morphism $\mathcal{F}\rightarrow{*}$ does. 
    
\end{defn}

\begin{prop}\phantomsection\label{representableobstructiontheory}
    Any representable presheaf $X$ in $\textcat{PSh}(\textcat{DAff}^{cn}(\mathcal{C}))$ has an obstruction theory relative to $\mathcal{A}$. 
\end{prop}

\begin{proof}
    We know that any representable presheaf has a global cotangent complex. It remains to show that it is infinitesimally cartesian relative to $\mathcal{A}$. Indeed, suppose that $X=\textnormal{Spec}(A)$ and that $Y=\textnormal{Spec}(B)\in\mathcal{A}$. Suppose that $M$ is in $\textcat{M}_{B,1}$ and that $d\in\pi_0(\textcat{Der}(B,M))$. Then, 
    \begin{equation*}\begin{aligned}
        &\textnormal{Map}_{\textcat{DAlg}^{cn}(\mathcal{C})}(A,B\oplus_d\Omega M)\\&\simeq \textnormal{Map}_{\textcat{DAlg}^{cn}(\mathcal{C})}(A,B\times_{B\oplus M} B)\\&\simeq \textnormal{Map}_{\textcat{DAlg}^{cn}(\mathcal{C})}(A,B)\times_{\textnormal{Map}_{\textcat{DAlg}^{cn}(\mathcal{C})}(A,B\oplus M)}\textnormal{Map}_{\textcat{DAlg}^{cn}(\mathcal{C})}(A,B)
    \end{aligned}\end{equation*}from which the result follows.
\end{proof}

\begin{prop}
    Suppose that $\mathcal{F}$ is a stack in $\textcat{Stk}(\mathcal{A},\bm{\tau}|_\mathcal{A})$ whose diagonal morphism $\mathcal{F}\rightarrow{\mathcal{F}\times\mathcal{F}}$ is $(n-1)$-representable for some $n\geq 0$. Then, $\mathcal{F}$ has an obstruction theory if and only if it is infinitesimally cartesian.
\end{prop}
\begin{proof}
    Follows in a similar way to \cite[c.f. Proposition 1.4.2.7]{toen_homotopical_2008}.  
\end{proof}

\begin{cor}\phantomsection\label{corollarygeometricinf} Any $n$-geometric stack has an obstruction theory if and only if it is infinitesimally cartesian. 
\end{cor}

We collect together the following results. The proofs follow in the same way as in \cite[c.f. Lemma 1.4.2.3]{toen_homotopical_2008} using the results in Lemma \ref{globalcotangentresults}.
\begin{lem}\phantomsection\label{relativeobstructionresult}Suppose that $f:\mathcal{F}\rightarrow{\mathcal{G}}$ is a morphism in $\textcat{PSh}(\textcat{DAff}^{cn}(\mathcal{C}))$. Then, 
\begin{enumerate}
    \item If both presheaves $\mathcal{F}$ and $\mathcal{G}$ have an obstruction theory, then the morphism $f$ has an obstruction theory, 
    \item If the morphism $f$ has an obstruction theory, then for any presheaf $\mathcal{H}$ and any morphism $\mathcal{H}\rightarrow{\mathcal{G}}$, the morphism $\mathcal{F}\times_{\mathcal{G}}\mathcal{H}\rightarrow{\mathcal{H}}$ has a relative obstruction theory, 
    \item\label{obstructiontheoryresult3} If, for any $X\in\mathcal{A}$ and any morphism of presheaves $x:X\rightarrow{\mathcal{G}}$, the presheaf $\mathcal{F}\times_{\mathcal{G}}X$ has an obstruction theory, then the morphism $f$ has a relative obstruction theory.
\end{enumerate}

\end{lem}

\begin{cor}\phantomsection\label{representableobstruction}
    Any $(-1)$-representable morphism of stacks in $\textcat{Stk}(\mathcal{A},\bm\tau|_\mathcal{A})$ has an obstruction theory.
\end{cor}

\begin{proof}
    Follows using Proposition \ref{representableobstructiontheory} and item (\ref{obstructiontheoryresult3}) of the previous lemma. 
\end{proof}

We state the following two propositions motivating the definition of a presheaf `having an obstruction theory'. Their proofs follow in a similar way to \cite[Propositions 1.4.2.5 and 1.4.2.6]{toen_homotopical_2008}.

\begin{prop}\phantomsection\label{liftingproposition}
    Suppose that $\mathcal{F}$ is a presheaf in $\textcat{PSh}(\textcat{DAff}^{cn}(\mathcal{C}))$ which has an obstruction theory relative to $\mathcal{A}$. Let $X=\textnormal{Spec}(A)\in\mathcal{A}$, $M\in\textcat{M}_{A,1}$, and take some $d\in\pi_0(\textcat{Der}(A,M))$. Suppose that $x:X=\textnormal{Spec}(A)\rightarrow{\mathcal{F}}$ is a morphism. Then, there exists a natural obstruction 
    \begin{equation*}
    \alpha(x)\in\pi_0(\textnormal{Map}_{\textcat{Mod}_{A}}(\mathbb{L}_{\mathcal{F},x},M))
    \end{equation*}vanishing if and only if $x$ extends to a morphism $x'$ in ${}^{X/}\textcat{PSh}(\textcat{DAff}^{cn}(\mathcal{C}))$
    \begin{equation*}
        \begin{tikzcd}
            X\arrow{rr} \arrow{dr}{x} & &  X_d[\Omega M] \arrow[dl,dotted,"x'"]\\
            & \mathcal{F}
        \end{tikzcd}
    \end{equation*}
\end{prop}

\begin{prop}\phantomsection\label{liftpropositionpath} Suppose that $f:\mathcal{F}\rightarrow{\mathcal{G}}$ is a map in $\textcat{PSh}(\textcat{DAff}^{cn}(\mathcal{C}))$ which has an obstruction theory relative to $\mathcal{A}$. Let $X=\textnormal{Spec}(A)\in\mathcal{A}$, $M\in\textcat{M}_{A,1}$, and $d\in\pi_0(\textcat{Der}(A,M))$. Suppose that $x$ is a point in $\mathcal{F}(A)\times_{\mathcal{G}(A)}\mathcal{G}(A\oplus_d\Omega M)$ with projection $y\in\mathcal{F}(A)$ and let $L(x)$ be the homotopy fibre, taken at $x$, of the morphism 
\begin{equation*}
    \mathcal{F}(A\oplus_d\Omega M)\rightarrow{\mathcal{F}(A)\times_{\mathcal{G}(A)}\mathcal{G}(A\oplus_d\Omega M)}
\end{equation*}Then, there exists a natural point $\alpha(x)$ in $\pi_0(\textnormal{Map}_{\textcat{Mod}_A}(\mathbb{L}_{\mathcal{F}/\mathcal{G},y},M))$ and a natural equivalence of $\infty$-groupoids
\begin{equation*}
    L(x)\simeq \Omega_{\alpha(x),0}\textnormal{Map}_{\textcat{Mod}_A}(\mathbb{L}_{\mathcal{F}/\mathcal{G},y},M)
\end{equation*}
\end{prop}

\subsection{Infinitesimally Smooth}

We will need the following notion of \textit{infinitesimal smoothness} which is somewhat weaker than the notion of smoothness. For a more detailed description in homotopical algebraic contexts see \cite{toen_homotopical_2008}. 

\begin{defn}\phantomsection\label{ismoothdefn}
    Suppose that $f:A\rightarrow{B}$ is a morphism in $\textcat{DAlg}^{cn}(\mathcal{C})$. Then, $f$ is \textit{formally infinitesimally smooth (or formally $i$-smooth) relative to $\mathcal{A}$} if, for any $C\in\mathcal{A}^{op}$, any morphism $A\rightarrow{C}$, any $M\in\textcat{M}_{C,1}$ and any $d\in\pi_0(\textcat{Der}_A(C,M))$, the natural morphism 
    \begin{equation*}
        \pi_0(\textnormal{Map}_{\textcat{DAlg}^{cn}_A(\mathcal{C})}(B,C\oplus_d\Omega M))\rightarrow{ \pi_0(\textnormal{Map}_{\textcat{DAlg}^{cn}_A(\mathcal{C})}(B,C))}
    \end{equation*}is surjective.
\end{defn}

We note that formally $i$-smooth morphisms are stable by equivalences, compositions, and pushouts. Moreover, 
\begin{prop}\cite[c.f. Proposition 1.2.8.3]{toen_homotopical_2008}\phantomsection\label{ismoothprop}
    A morphism $f:A\rightarrow{B}$ in $\textcat{DAlg}^{cn}(\mathcal{C})$ is formally $i$-smooth relative to $\mathcal{A}$ if and only if, for any morphism $B\rightarrow{C}$ with $C\in\mathcal{A}^{op}$ and any $M\in\textcat{M}_{C,1}$, the natural morphism 
    \begin{equation*}
        \pi_0(\textnormal{Map}_{\textcat{Mod}_C}(\mathbb{L}_{C/A},M))\rightarrow\pi_0(\textnormal{Map}_{\textcat{Mod}_B}(\mathbb{L}_{B/A},M))
    \end{equation*}is zero.
\end{prop} 

\begin{cor}
    Any formally \'etale or formally unramified morphism $f:A\rightarrow{B}$ in $\textcat{DAlg}^{cn}(\mathcal{C})$ is formally $i$-smooth relative to $\mathcal{A}$. 
\end{cor}

\begin{proof}
    This follows from Proposition \ref{etaleissmooth} and the previous result.
\end{proof}

\begin{cor}\phantomsection\label{formallyperfectismooth}
    Any formally smooth or formally perfect morphism in $\textcat{DAlg}^{cn}(\mathcal{C})$ is formally $i$-smooth.
\end{cor}

\begin{proof}
    Indeed, suppose that $f:A\rightarrow{B}$ is a formally smooth (or perfect) morphism in $\textcat{DAlg}^{cn}(\mathcal{C})$ and that $C\in\textcat{DAlg}^{cn}(\mathcal{C})$ along with a morphism $B\rightarrow{C}$. Then, we have a fibre sequence
    \begin{equation*}\phantomsection\label{fibreequation}
        \mathbb{L}_{B/A}\otimes^\mathbb{L}_B C\rightarrow{\mathbb{L}_{C/A}}\rightarrow{\mathbb{L}_{C/B}}
    \end{equation*} in $\textcat{Mod}_C$. Consider some $M\in\textcat{M}_{C,1}$ and a derivation $d\in\pi_0(\textcat{Der}_A(C,M))$. Then, since $\mathbb{L}_{B/A}$ is projective, we see that $\pi_0(\textnormal{Map}_{\textcat{Mod}_B}(\mathbb{L}_{B/A},M))=0$ by Lemma \ref{projectivevanishinghomotopy}, and hence the morphism
        \begin{equation*}
        \pi_0(\textnormal{Map}_{\textcat{Mod}_C}(\mathbb{L}_{C/A},M))\rightarrow\pi_0(\textnormal{Map}_{\textcat{Mod}_B}(\mathbb{L}_{B/A},M))
    \end{equation*}is zero. The result then follows by Proposition \ref{ismoothprop}.
\end{proof}

\subsection{Obstruction Conditions}

Suppose that $(\mathcal{C},\mathcal{C}_{\geq 0},\mathcal{C}_{\leq 0},\mathcal{C}^0,\bm\tau,\textcat{P},\mathcal{A},\textcat{M})$ is a derived geometry context. We will now impose some extra conditions on $\bm{\tau}$ and $\textcat{P}$ and define a suitable class $\textcat{S}$ of formally smooth morphisms in $\mathcal{A}$ such that we obtain obstruction theories for geometric stacks.

We will call a family $\{U_i\rightarrow{X}\}_{i\in I}$ of morphisms of $\mathcal{A}$-admissible objects in $\textcat{DAff}^{cn}(\mathcal{C})$ \textit{an $\textcat{S}$-covering family} if each morphism $U_i\rightarrow{X}$ is in $\textcat{S}$ and there is an epimorphism of stacks $\coprod_{i\in I}U_i\rightarrow{X}$ in $\textcat{Stk}(\mathcal{A},\bm\tau|_\mathcal{A})$. The induced epimorphism $\coprod_{i\in I}U_i\rightarrow{X}$ will be called an \textit{$\textcat{S}$-cover}.

\begin{lem}\phantomsection\label{liftingderivations}
Suppose that we have a {formally smooth} morphism $A\rightarrow{B}$ in $\textcat{DAlg}^{cn}(\mathcal{C})$, $M\in\textcat{Mod}_A$, and $d\in\pi_0(\textcat{Der}(A,M))$. Then, $d$ lifts to a derivation $d'$ which lies in $\pi_0(\textcat{Der}(B,M\otimes^\mathbb{L}_AB))$. 
\end{lem}

\begin{proof}
    Indeed, we note that since $A\rightarrow{B}$ is formally smooth, we have a morphism $\mathbb{L}_{B}\rightarrow{\mathbb{L}_A\otimes_A^\mathbb{L}B}$ in $\textcat{Mod}_B$ . We define $d'$ to be the image of $d$ under the following map. 
    \begin{equation*}\begin{aligned}
        \pi_0(\textcat{Der}(A,M))=\pi_0(\textnormal{Map}_{\textcat{Mod}_A}(\mathbb{L}_A,M))&\rightarrow{}\pi_0(\textnormal{Map}_{\textcat{Mod}_B}(\mathbb{L}_A\otimes^\mathbb{L}_AB,M\otimes^\mathbb{L}_AB)) \\
        &\rightarrow\pi_0(\textnormal{Map}_{\textcat{Mod}_B}(\mathbb{L}_B,M\otimes^\mathbb{L}_AB))\\
        &=\pi_0(\textcat{Der}(B,M\otimes^\mathbb{L}_AB))
    \end{aligned}\end{equation*}
\end{proof}

\begin{defn}\phantomsection\label{artinsconditions}
    We say that $\bm{\tau}$ and $\textcat{P}$ satisfy \textit{the obstruction conditions relative to $\mathcal{A}$} if there exists a class $\textcat{S}$ of morphisms in $\mathcal{A}$ such that
    \begin{enumerate}
    \item \label{artin0} Any morphism in $\textcat{P}$ is formally $i$-smooth,
        \item\label{artin1} The class $\textcat{S}$ consists of formally {smooth morphisms} and is stable under equivalences, compositions, and pullbacks,
        \item\label{artin5} Suppose that $X=\textnormal{Spec}(A)\in\mathcal{A}$, $M\in\textcat{M}_{A,1}$, and that $d\in\pi_0(\textcat{Der}(A,M))$. Consider the natural morphism 
        \begin{equation*}
            X=\textnormal{Spec}(A)\rightarrow{\textnormal{Spec}(A\oplus_d\Omega M)}=X_d[\Omega M]
        \end{equation*}in $\textcat{DAff}^{cn}(\mathcal{C})$. Suppose that $\{V_j\rightarrow{X}\}_{j\in J}$ is a $\bm\tau|_\mathcal{A}$-covering family of $X$. Then, there exists a finite $\textcat{S}$-covering family $\{V_k'=\textnormal{Spec}(A_k)\rightarrow{X}\}_{k\in K}$ in $\mathcal{A}$ and a morphism $v:K\rightarrow{J}$ such that, for each $k\in K$, there is a commutative diagram
        \begin{equation*}
            \begin{tikzcd}
                & V_{v(k)}\arrow{d}\\
                V_k'\arrow{ur} \arrow{r} & X
            \end{tikzcd}
        \end{equation*} and such that $\{W_k\rightarrow{X_d[\Omega M]}\}_{k\in K}$ is an $\textcat{S}$-covering family of $X_d[\Omega M]$ in $\mathcal{A}$, where $W_k=\textnormal{Spec}(A_k\oplus_{d_k'}\Omega M_k')$ with $d_k'$ the derivation induced by Lemma \ref{liftingderivations} and $M_k'=M\otimes^\mathbb{L}_{A}A_k$. 
    \end{enumerate}
\end{defn}

\begin{remark}
We note that our obstruction conditions differ from the Artin conditions of \cite[Definition 1.4.3.1]{toen_homotopical_2008} because we are not assuming quasi-compactness of our topology. In particular, Condition (\ref{artin5}) is a weaker condition.
\end{remark}

\begin{lem}\phantomsection\label{Ecoverequiv}Suppose that $X=\textnormal{Spec}(A)\in\mathcal{A}$, $M\in\textcat{M}_{A,1}$, and that $d\in\pi_0(\textcat{Der}(A,M))$. Suppose that we have a morphism $A\rightarrow{A_j}$ in $\mathcal{A}^{op}$. Then, there is an equivalence 
\begin{equation*}
    B_j\otimes^\mathbb{L}_{A\oplus_d\Omega M} A\rightarrow{A_j}
\end{equation*}in $\textcat{DAlg}^{cn}(\mathcal{C})$ where $B_j=A_j\oplus_{d'_j}\Omega M_j'$, $d_j'$ is the derivation induced by Lemma \ref{liftingderivations}, and $M_j'=M\otimes^\mathbb{L}_{A}A_j$.
\end{lem}

\begin{proof}
We consider the morphism $A\oplus_d\Omega M \rightarrow{B_j}$ induced by the composition of morphisms $A\rightarrow{A_j}\rightarrow{A_j\oplus M'_j}$. By Lemma \ref{pushoutAsqzM}, there are fibre-cofibre sequences
    \begin{equation*}\begin{aligned}
         A\oplus_d\Omega M&\rightarrow{A}\rightarrow{M}\\
         B_j&\rightarrow{A_j}\rightarrow{M_j'}
    \end{aligned}
    \end{equation*}in $\textcat{Mod}_{A}^{cn}$ and $\textcat{Mod}_{A_j}^{cn}$ respectively, which induce the following commutative diagram where the horizontal maps form fibre-cofibre sequences in $\textcat{Mod}_A$

    \begin{equation*}
        \begin{tikzcd}
            B_j\otimes^\mathbb{L}_{A\oplus_d\Omega M}A\arrow{r}\arrow{d} & A_j\otimes^\mathbb{L}_{A\oplus_d\Omega M}A\arrow{r}\arrow{d} & M_j'\otimes^\mathbb{L}_{A\oplus_d\Omega M} A\arrow{d}\\
            A_j\otimes_{A\oplus_d\Omega M} A\oplus_d\Omega M\arrow{r} & A_j\otimes^\mathbb{L}_{A\oplus_d\Omega M}A\arrow{r} & A_j\otimes^\mathbb{L}_{A\oplus_d\Omega M} M
        \end{tikzcd}
    \end{equation*}Since the two right hand side vertical maps are equivalences, it follows that the first vertical map is an equivalence, and hence we have an equivalence $B_j\otimes^\mathbb{L}_{A\oplus_d\Omega M}A\simeq A_j$ in $\textcat{Mod}_A^{cn}$, and hence also in $\textcat{DAlg}^{cn}(\mathcal{C})$.
\end{proof}

\subsection{Obstruction Theory of Geometric Stacks}

Suppose that we have a derived geometry context $(\mathcal{C},\mathcal{C}_{\geq 0},\mathcal{C}_{\leq 0},\mathcal{C}^0,\bm\tau,\textcat{P},\mathcal{A},\textcat{M})$. Suppose also that, for any finite collection $\{U_i\}_{i\in I}$ of $\mathcal{A}$-admissible objects in $\mathcal{M}$, the map $\coprod_{i\in I} h(U_i)\rightarrow{h(\coprod_{i\in I} U_i)}$ is an equivalence in $\textcat{Stk}(\mathcal{A},\bm\tau|_\mathcal{A})$. Compare the following proof with \cite[c.f. Theorem 1.4.3.2]{toen_homotopical_2008}.

\begin{thm}\phantomsection\label{obstructiongeometric} Suppose that $\bm{\tau}$ and $\textcat{P}$ satisfy the obstruction conditions relative to $\mathcal{A}$. 
\begin{enumerate}
    \item\label{obstructionitem} If $f:\mathcal{F}\rightarrow{\mathcal{G}}$ is an $n$-representable morphism of stacks in $\textcat{Stk}(\mathcal{A},\bm{\tau}|_\mathcal{A})$, then it has an obstruction theory, 
\item\label{homitem} If $\mathcal{F}$ is an $n$-geometric stack in $\textcat{Stk}(\mathcal{A},\bm\tau|_\mathcal{A})$ and $g:U\rightarrow{\mathcal{F}}$ is an $n\textcat{-P}$-morphism with $U$ a $(-1)$-geometric stack then, for any $X\in\mathcal{A}$ and $x:X=\textnormal{Spec}(A)\rightarrow{U}$, there exists an $\textcat{S}$-cover\begin{equation*}
        x':X'=\textnormal{Spec}(A')\rightarrow{X}
    \end{equation*}such that, for any $M\in\textcat{M}_{A',1}$, the natural map 
    \begin{equation*}
        \pi_0(\textnormal{Map}_{\textcat{Mod}_{A'}}(\mathbb{L}_{X'/U},M))\rightarrow\pi_0(\textnormal{Map}_{\textcat{Mod}_A}(\mathbb{L}_{U/\mathcal{F},x},M))
    \end{equation*}is zero.
\end{enumerate}
    
\end{thm}

\begin{proof}We prove this by induction on $n$. Indeed, when $n=-1$, Statement (\ref{obstructionitem}) follows from Corollary \ref{representableobstruction}. 

To prove Statement (\ref{homitem}) at level $n=-1$, suppose that $\mathcal{F}$ is a $(-1)$-geometric stack, say $\mathcal{F}=Y=\textnormal{Spec}(B)$ for some $\mathcal{A}$-admissible $B\in\mathcal{M}$. Suppose that there is a $(-1)\textcat{-P}$-morphism $g:U=\textnormal{Spec}(C)\rightarrow{Y}$ in $\textcat{Stk}(\mathcal{A},\bm\tau|_\mathcal{A})$. By Assumption (\ref{artin0}), this morphism is formally $i$-smooth. Therefore, we see that for any morphism $x:X=\textnormal{Spec}(A)\rightarrow{U}$, and taking the trivial $\textcat{S}$-cover of $X$, our result follows immediately using Proposition \ref{ismoothprop}.

Now, suppose that $f:\mathcal{F}\rightarrow{\mathcal{G}}$ is an $n$-representable morphism of stacks and that the theorem holds for $k$-representable morphisms for $k<n$. Suppose that $X\in\mathcal{A}$ and that there is a morphism $x:X\rightarrow{\mathcal{G}}$. We note that, by Lemma \ref{relativeobstructionresult}, it suffices to show that the $n$-geometric stack $\mathcal{F}\times_\mathcal{G}X$ has an obstruction theory. This leads to the following lemma.

    \begin{lem}\phantomsection\label{ngeoobstruction}
        Any $n$-geometric stack has an obstruction theory. 
    \end{lem}
    \begin{proof}
Suppose that $\mathcal{F}$ is an $n$-geometric stack. By Corollary \ref{corollarygeometricinf}, it suffices to show that $\mathcal{F}$ is infinitesimally cartesian. Indeed, suppose that $X=\textnormal{Spec}(A)\in\mathcal{A}$, $M\in\textcat{M}_{A,1}$ and that there is a derivation $d\in\pi_0(\textcat{Der}(A,M))$. We want to prove that the morphism
        \begin{equation*}
            \mathcal{F}(A\oplus_d\Omega M)\rightarrow{\mathcal{F}(A)\times_{\mathcal{F}(A\oplus M)}\mathcal{F}(A)}
        \end{equation*}is an equivalence in $\infty\textcat{Grpd}$. Suppose that $x\in\pi_0(\mathcal{F}(A)\times_{\mathcal{F}(A\oplus M)}\mathcal{F}(A))$ with projection $x_1\in\pi_0(\mathcal{F}(A))$. By adapting \cite[Lemma 2.1.3.4]{lurie_higher_2009} to the case of $\infty$-groupoids, it suffices to prove that the fibre of the above map at $x$ is contractible. 

    We define a functor $\mathcal{S}:(\textcat{DAff}^{cn}(\mathcal{C})_{/X_d[\Omega M]})^{op}\rightarrow{\infty\textcat{Grpd}}$ which takes a map $A\oplus_d\Omega M\rightarrow{B}$ to the fibre of the morphism 
\begin{equation*}
    \mathcal{F}(B)\rightarrow{\mathcal{F}(B\otimes^\mathbb{L}_{A\oplus_d\Omega M}A)\times_{\mathcal{F}(B\otimes^\mathbb{L}_{A\oplus_d\Omega M}A\oplus M)}\mathcal{F}(B\otimes^\mathbb{L}_{A\oplus_d\Omega M}A)}
\end{equation*}taken at the image of $x$. Since $\mathcal{F}$ is a stack on $\mathcal{A}$, we see that $\mathcal{S}$ defines a stack on $\textcat{DAff}^{cn}(\mathcal{C})_{/X_d[\Omega M]}$. Suppose that we have some epimorphism of representable stacks $\coprod_{k\in K} \textnormal{Spec}(B_k)\rightarrow{X_d[\Omega M]}$ in $\textcat{Stk}(\textcat{DAff}^{cn}(\mathcal{C})_{/X_d[\Omega M]},\bm{\tau})$. To show that $\mathcal{S}(A\oplus_d\Omega M)$ is contractible, it suffices to show that $\prod_{k\in K} \mathcal{S}(B_k)$ is contractible.

Suppose that we have an $(n-1)$-atlas $\{U_i\rightarrow{\mathcal{F}}\}_{i\in I}$ of $\mathcal{F}$. Then, by Proposition \ref{localepimorphismprop}, we may find a $\bm{\tau}$-covering family $\{V_j\rightarrow{X}\}_{j\in J}$ of $X$ such that the morphism $V_j\rightarrow{\mathcal{F}}$ factors through some $U_{u(j)}$. We may assume, by Assumption (\ref{artin5}), that this cover refines to a finite $\textcat{S}$-covering family $\{V_k'=\textnormal{Spec}(A_k')\rightarrow{X}\}_{k\in K}$ and defines an $\textcat{S}$-covering family $\{W_k=\textnormal{Spec}(B_k)\rightarrow{X_d[\Omega M]}\}_{k\in K}$ of $X_d[\Omega M]$, where $B_k=A_k\oplus_{d_k'}\Omega M_k'$ for $d_k'$ the derivation induced by Lemma \ref{liftingderivations} and $M_k'=M\otimes^\mathbb{L}_{A}A_k$.

Since $B_k\otimes^\mathbb{L}_{A\oplus_d\Omega M}(A\oplus M)\simeq A_k\oplus M'_k$ and $B_k\otimes^\mathbb{L}_{A\oplus_d\Omega M} A\simeq A_k$ by Lemma \ref{Ecoverequiv} then, to show that $\prod_{k\in K}\mathcal{S}(B_k)$ is contractible, we just need to check that 
\begin{equation*}
    \prod_{k\in K}\mathcal{F}( B_k)\rightarrow\prod_{k\in K} \mathcal{F}(A_k)\times_{\mathcal{F}(A_k\oplus M_k')}\mathcal{F}(A_k)
\end{equation*}is an equivalence of $\infty$-groupoids. Therefore, for each $k\in K$, we can just replace $X=\textnormal{Spec}(A)$ by $V_k'=\textnormal{Spec}(A_k')$, $d$ by $ d_k'$, and $M$ by $M_k'$. Hence, the point $x_1\in\pi_0(\mathcal{F}(A))$, which is the image of $x\in\pi_0(\mathcal{F}(A)\times_{\mathcal{F}(A\oplus M)}\mathcal{F}(A))$, lifts to a point $y_1\in\pi_0(U_{u(k)}(A))$ for some $k\in K$. Let $U:=U_{u(k)}\in\mathcal{A}$.

\begin{sublem}\phantomsection\label{liftsublemma} The point $x\in\pi_0(\mathcal{F}(A)\times_{\mathcal{F}(A\oplus M)}\mathcal{F}(A))$ lifts to a point $y$ in $\pi_0(U(A)\times_{U(A\oplus M)}U(A))$
\end{sublem}

\begin{proof}
        We note that there is a commutative square of $\infty$-groupoids 
    \begin{equation*}
        \begin{tikzcd}
        U(A)\times_{U(A\oplus M)}U(A) \arrow{r}{f} \arrow{d}{p} & \mathcal{F}(A)\times_{\mathcal{F}(A\oplus M)}\mathcal{F}(A) \arrow{d}{q}\\
            U(A)\arrow{r} & \mathcal{F}(A)
        \end{tikzcd}
    \end{equation*}induced by the map $A\oplus M\rightarrow{A}$. We let  $\textnormal{fib}(p)$ and $\textnormal{fib}(q)$ be the fibres of the morphisms $p$ and $q$ taken at $y_1$ and $x_1$ respectively and consider the induced morphism $h:\textnormal{fib}(p)\rightarrow{\textnormal{fib}(q)}$. We note that there is an induced morphism from $\textnormal{fib}(h)$ taken at $x$ to $\textnormal{fib}(f)$ at $x$. It therefore suffices to show that the fibre of $h$ at $x$ is non-empty. By Lemma \ref{alternativederivationdefn}, we see that $h$ is equivalent to the morphism
    \begin{equation*}
        \Omega_{d(y_1),0}\textcat{Der}_{U}(X,M)\rightarrow{\Omega_{d(x_1),0}\textcat{Der}_{\mathcal{F}}(X,M)}
    \end{equation*}Hence, we see that the fibre of $h$ at $x$ is equivalent to 
    \begin{equation*}
        \Omega_{d(y_1),0}\textcat{Der}_{U/\mathcal{F}}(X,M)\simeq\Omega_{d(y_1),0}\textnormal{Map}_{\textcat{Mod}_A}(\mathbb{L}_{U/\mathcal{F},y_1},M)
    \end{equation*}Since $\mathcal{F}$ is an $n$-geometric stack and $g:U\rightarrow{\mathcal{F}}$ is an $(n-1)\textcat{-P}$ morphism, then by Statement (\ref{homitem}) of the theorem at level $n-1$, we note that there is an $\textcat{S}$-cover $x':X'=\textnormal{Spec}(A')\rightarrow{X}$ such that, for any $M\in\textcat{M}_{A',1}$, the natural map 
     \begin{equation}\phantomsection\label{zeroequation}
        \pi_0(\textnormal{Map}_{\textcat{Mod}_{A'}}(\mathbb{L}_{X'/U},M))\rightarrow\pi_0(\textnormal{Map}_{\textcat{Mod}_A}(\mathbb{L}_{U/\mathcal{F},y_1},M))
    \end{equation}is zero. By Corollary \ref{cofibrecotangentstack}, we have a fibre-cofibre sequence in $\textcat{Mod}_A$
   \begin{equation*}
\mathbb{L}_{U/\mathcal{F},y_1}\rightarrow{\mathbb{L}_{X'/\mathcal{F},x'}}\rightarrow{\mathbb{L}_{X'/U}}
   \end{equation*}which induces the long exact sequence
   \begin{equation*}
\dots\rightarrow{\pi_0(\textnormal{Map}(\mathbb{L}_{X'/\mathcal{F},x'},M))}\rightarrow{\pi_{0}(\textnormal{Map}(\mathbb{L}_{U/\mathcal{F},y_1},M))}\rightarrow{\pi_{-1}(\textnormal{Map}(\mathbb{L}_{X'/U},M}))\rightarrow{\dots} 
   \end{equation*}where the first map is zero by Equation (\ref{zeroequation}). We note that the image of $d(y_1)$ under the second map is zero, and hence we see that $d(y_1)$ must lie in the image of the first map. Therefore, $0$ and $d(y_1)$ lie in the same connected component of $\textnormal{Map}_{\textcat{Mod}_A}(\mathbb{L}_{U/\mathcal{F},y_1},M)$, and hence $\textnormal{fib}(h)$ is non-empty. 
\end{proof}
Now, since the morphism $g:U\rightarrow{\mathcal{F}}$ is $(n-1)$-representable, by our inductive hypothesis it is infinitesimally cartesian and we have a pullback square
\begin{equation*}
    \begin{tikzcd}
         U(A\oplus_d\Omega M)\arrow{r} \arrow{d}{p'}& \mathcal{F}(A\oplus_d\Omega M) \arrow{d}{q'}\\
        U(A)\times_{U(A\oplus M)}U(A)\arrow{r} &  \mathcal{F}(A)\times_{\mathcal{F}(A\oplus M)}\mathcal{F}(A)
    \end{tikzcd}
\end{equation*}We note that, since $U$ is $(-1)$-geometric, it is infinitesimally cartesian, and hence $p'$ is an equivalence. Therefore, the fibre of $q'$ at $x$ is either contractible or empty by \cite[c.f. Proposition 1.2.12.9]{lurie_higher_2009}. But, $x$ lifts to a point in $\pi_0(U(A)\times_{U(A\oplus M)}U(A))$ by the sublemma, and hence the fibre must be non empty. Therefore, $\mathcal{F}$ is infinitesimally cartesian.
    \end{proof}

    With the proof of Statement (\ref{obstructionitem}) complete, it remains to show Statement (\ref{homitem}) holds at level $n$. Suppose that we have an $n$-geometric stack $\mathcal{F}$ and an $n\textcat{-P}$ morphism $g:U\rightarrow{\mathcal{F}}$. Take an $n$-atlas $\{U_i\rightarrow{\mathcal{F}\}_{i\in I}}$ for $\mathcal{F}$, and consider the induced $n$-atlas $\{V_j=\textnormal{Spec}(B_j)\rightarrow{U\times_\mathcal{F}U_i}\}_{j\in J}$ for the $n$-geometric stacks $U\times_\mathcal{F}U_i$. We have an induced morphism $v_j:X\times_UV_j\rightarrow{V_j}$ of $(-1)$-geometric stacks. By Lemma \ref{globalcotangentresults}, it suffices to show that there exists an $\textcat{S}$-cover $X'=\textnormal{Spec}(A')\rightarrow{X}$ such that, for any $M\in\textcat{M}_{A',1}$, the natural map 
    \begin{equation*}
        \pi_0(\textnormal{Map}_{\textcat{Mod}_{A'}}(\mathbb{L}_{X'\times_UV_j/V_j},M))\rightarrow{\pi_0(\textnormal{Map}_{\textcat{Mod}_A}(\mathbb{L}_{U\times_\mathcal{F}U_i/U_i,x^\cdot},M))}
    \end{equation*}is zero. This map factors through the morphism \begin{equation}\label{zeroequation2}
        \pi_0(\textnormal{Map}_{\textcat{Mod}_{B_j}}(\mathbb{L}_{V_j/U\times_\mathcal{F}U_i,v_j},M))\rightarrow{\pi_0(\textnormal{Map}_{\textcat{Mod}_A}(\mathbb{L}_{U\times_\mathcal{F}U_i/U_i,x^\cdot},M))}
\end{equation}and hence it suffices to show, up to replacing $V_j$ with an $\textcat{S}$-cover, that this latter morphism is zero. 

We note that, since the morphism $U_i\rightarrow{\mathcal{F}}$ is in $(n-1)\textcat{-P}$, then the induced morphism $V_j\rightarrow{U_i}$ is in $\textcat{P}$. Hence, we see, by Proposition \ref{ismoothprop} and Condition (\ref{artin0}), that $\pi_0(\textnormal{Map}_{\textcat{Mod}_{B_j}}(\mathbb{L}_{V_j/U_i},M))=0$. Hence, since the morphism in Equation (\ref{zeroequation2}) factors through $\pi_0(\textnormal{Map}_{\textcat{Mod}_{B_j}}(\mathbb{L}_{V_j/U_i},M))=0$ by Corollary \ref{cofibrecotangentstack}, then we see that our morphism is zero as desired.

\end{proof}

\subsection{Postnikov Towers}\phantomsection\label{postnikovsection} A Postnikov tower of a topological space gives you a way to `build up' the homotopy groups of the space. We will generalise this idea to our derived geometry contexts as follows.

Suppose that $(\mathcal{C},\mathcal{C}_{\geq 0},\mathcal{C}_{\leq 0},\mathcal{C}^0)$ is a derived algebraic context and that $\textcat{DAlg}^{cn}(\mathcal{C})$ is compatible with the $t$-structure on $\mathcal{C}$ in the sense of Definition \ref{postnikovcompatible}. For $n\geq 0$, we will denote by $A_{\leq n}$ the object $\tau_{\leq n}(A)\in\textcat{DAlg}^{cn}(\mathcal{C})$. There is an induced morphism $A_{\leq n}\rightarrow{A_{\leq n-1}}$ for each $n$.  
\begin{defn}
    Suppose that $A\in\textcat{DAlg}^{cn}(\mathcal{C})$. The \textit{Postnikov tower} of $A$ is 
    \begin{equation*}
        A\rightarrow{\dots}\rightarrow{A_{\leq n}}\rightarrow{A_{\leq n-1}}\rightarrow{\dots}{\rightarrow{A_{\leq 0}}=\pi_0(A)}
    \end{equation*}We will say that Postnikov towers in $\textcat{DAlg}^{cn}(\mathcal{C})$ \textit{converge} if $A$ is the limit of this sequence in $\textcat{DAlg}^{cn}(\mathcal{C})$. 
\end{defn}

\begin{defn}
   A derived geometry context $(\mathcal{C},\mathcal{C}_{\geq 0},\mathcal{C}_{\leq 0},\mathcal{C}^0,\bm\tau,\textcat{P},\mathcal{A},\textcat{M})$ will be \textit{Postnikov compatible} if
    \begin{enumerate} 
        \item The $t$-structure on $\mathcal{C}$ is left complete,
        \item $\textcat{DAlg}^{cn}(\mathcal{C})$ is compatible with the $t$-structure on $\mathcal{C}$,
        \item\label{postnikovcomp4} Given a Postnikov tower in $\textcat{DAlg}^{cn}(\mathcal{C})$, then $A_{\leq n}\in\mathcal{A}^{op}$ for all $n\geq 0$ if and only if $A\in\mathcal{A}^{op}$,
\item\label{postnikovcomp5} 
        For all $m\geq 2$, if $A\in\mathcal{C}$ is $n$-connective, then $\textcat{LSym}^m(A)$ is $(n+2)$-connective.
    \end{enumerate}
\end{defn}

\begin{remark}
    We remark that the last condition is dealt with in more generality in \cite[Definition 2.1.8]{ben-bassat_perspective_2024}, and is included here in its present form so that Theorem \ref{2nconnective} holds. 
\end{remark}

For the rest of the section, suppose that we have a Postnikov compatible derived geometry context $(\mathcal{C},\mathcal{C}_{\geq 0},\mathcal{C}_{\leq 0},\mathcal{C}^0,\bm\tau,\textcat{P},\mathcal{A},\textcat{M})$. We note that, since $\mathcal{C}$ is left complete, Postnikov towers in $\mathcal{C}$ converge. Hence, since $\textcat{DAlg}^{cn}(\mathcal{C})$ is closed under limits and using our assumptions, Postnikov towers in $\textcat{DAlg}^{cn}(\mathcal{C})$ also converge. In particular, our assumptions imply that, if $X=\textnormal{Spec}(A)$ is in the full subcategory $\mathcal{A}$ of $\textcat{DAff}^{cn}(\mathcal{C})$, then $X$ is the colimit in $\mathcal{A}$ of the sequence
\begin{equation*}
    X\leftarrow{\dots}\leftarrow{X_{\leq n}}\leftarrow{X_{\leq n-1}\leftarrow{\dots}\leftarrow{X_{\leq 0}}}
\end{equation*}In particular, we can use this to show that every $(-1)$-geometric stack on $\mathcal{A}$ is nilcomplete with respect to $\mathcal{A}$ in the following sense. 

\begin{defn}Suppose that $f:\mathcal{F}\rightarrow{\mathcal{G}}$ is a morphism of presheaves in $\textcat{PSh}(\textcat{DAff}^{cn}(\mathcal{C}))$. Then, $f$ is said to be \textit{nilcomplete with respect to $\mathcal{A}$} if, for every $X=\textnormal{Spec}(A)\in \mathcal{A}$, the square 
\begin{equation*}
    \begin{tikzcd}
    \mathcal{F}(A)\arrow{r}\arrow{d} & \varprojlim_k \mathcal{F}(A_{\leq k}) \arrow{d}\\
    \mathcal{G}(A) \arrow{r} & \varprojlim_k \mathcal{G}(A_{\leq k})
    \end{tikzcd}
\end{equation*}is a pullback square in $\infty\textcat{Grpd}$. We say that $\mathcal{F}$ is \textit{nilcomplete} with respect to $\mathcal{A}$ if the morphism $\mathcal{F}\rightarrow{*}$ is nilcomplete.
\end{defn}

\begin{defn}\phantomsection\label{postnikovliftproperty}
Suppose that $f:\mathcal{F}\rightarrow{\mathcal{G}}$ is an $n$-representable morphism in $\textcat{PSh}(\textcat{DAff}^{cn}(\mathcal{C}))$. Then, $f$ satisfies the \textit{Postnikov lifting property with respect to $\mathcal{A}$} if, for every $X\in \mathcal{A}$ and every $k>0$, the following lifting problem has at least one solution $X_{\leq k}\rightarrow{\mathcal{F}}$
\begin{equation*}\begin{tikzcd}
    X_{\leq k-1}\arrow{d} \arrow{r} & \mathcal{F}\arrow{d}{f}\\
    X_{\leq k} \arrow[ur,dotted]\arrow{r} & \mathcal{G}
\end{tikzcd}\end{equation*}
\end{defn}

\begin{prop}\phantomsection\label{postnikovnilcomplete}Suppose that
\begin{enumerate}
    \item $\bm\tau$ and $\textcat{P}$ satisfy the obstruction conditions relative to $\mathcal{A}$ for a suitable class of morphisms $\textcat{S}$, 
    \item For any finite collection $\{U_i\}_{i\in I}$ of $\mathcal{A}$-admissible objects, the map $\coprod_{i\in I} h(U_i)\rightarrow{h(\coprod_{i\in I} U_i)}$ is an equivalence in $\textcat{Stk}(\mathcal{A},\bm\tau|_\mathcal{A})$,
    \item Every $m\textcat{-P}$-representable map of stacks in $\textcat{Stk}(\mathcal{A},\bm\tau|_\mathcal{A})$ satisfies the Postnikov lifting property for $m\geq -1$.
\end{enumerate}Let $f:\mathcal{F}\rightarrow{\mathcal{G}}$ be an $n$-representable morphism of stacks in $\textcat{Stk}(\mathcal{A},\bm\tau|_\mathcal{A})$. Then, $f$ is nilcomplete for all $n\geq -1$.
\end{prop}

\begin{proof}We want to prove that, for any $X=\textnormal{Spec}(A)\in\mathcal{A}$, the map 
\begin{equation*}
    \mathcal{F}(A)\rightarrow{\mathcal{G}(A)\times_{\varprojlim_k \mathcal{G}(A_{\leq k})}\varprojlim_k \mathcal{F}(A_{\leq k})}
\end{equation*}is an equivalence in $\infty\textcat{Grpd}$. It suffices to show that the fibre of this morphism taken at any vertex is contractible. Indeed, take $x\in \pi_0(\mathcal{G}(A)\times_{\varprojlim_k \mathcal{G}(A_{\leq k})}\varprojlim_k \mathcal{F}(A_{\leq k}))$
and denote by $x'$ the projection of $x$ into $\pi_0(\mathcal{G}(A))$. The map $\mathcal{F}\times_\mathcal{G}X \rightarrow{\mathcal{F}}$ is $n$-representable since $n$-representable maps are stable under pullback. Since nilcompleteness of the map $\mathcal{F}\times_\mathcal{G}X\rightarrow{X}$ at $A$ implies nilcompleteness of $f$ at $A$, we can just replace $\mathcal{F}$ by $\mathcal{F}\times_\mathcal{G}X$ and $\mathcal{G}$ by $X$ in our statement. Hence, we are reduced to proving the statement in the case when we have an $n$-representable map $f:\mathcal{F}\rightarrow{Y}$, with $\mathcal{F}$ an $n$-geometric stack and $Y=\textnormal{Spec}(B)$ a $(-1)$-geometric stack on $\mathcal{A}$. We prove the statement by induction on $n$.

We note that the $(-1)$-geometric stack $Y$ is nilcomplete with respect to $\mathcal{A}$. It therefore suffices to show that the map 
\begin{equation*}
    \mathcal{F}(A)\rightarrow{\varprojlim_k\mathcal{F}(A_{\leq k})}
\end{equation*}is an equivalence for every $A\in\mathcal{A}^{op}$. The point $x$ is now a point in $\pi_0( {\varprojlim_k }\mathcal{F}(A_{\leq k}))$. We will let $x_k:X_{\leq k} \rightarrow{\mathcal{F}}$ denote the morphism of stacks corresponding to the projection of $x$ into $\pi_0(\mathcal{F}(A_{\leq k}))$. 

When $n=-1$, we note that $\mathcal{F}$ is a $(-1)$-geometric stack and the result easily follows. Now suppose that $n\geq 0$ and consider the morphism $x_0:X_{\leq 0}\rightarrow{\mathcal{F}}$. Since $\mathcal{F}$ is $n$-geometric, it has an $n$-atlas $\{U_i\rightarrow{\mathcal{F}}\}_{i\in I}$. By Proposition \ref{localepimorphismprop}, there exists a $\bm{\tau}$-covering family $\{V_j\rightarrow{X_{\leq 0}}\}_{j\in J}$ of $(-1)$-geometric stacks on $\mathcal{A}$ along with a morphism $u:J\rightarrow{I}$ such that the induced morphism $V_j\rightarrow{\mathcal{F}}$ factors through $U_{u(j)}$. We note that, by Condition (\ref{artin5}) of Definition \ref{artinsconditions}, we may refine this to a finite $\textcat{S}$-covering family $\{V_k'\rightarrow{X_{\leq 0}}\}_{k\in K}$. We let $V=\coprod_{k\in K} V_k'$, and consider this as a representable stack $\textnormal{Spec}(A')$ on $\mathcal{A}$.

We define a functor $\mathcal{S}:(\mathcal{A}_{/X_{\leq 0}})^{op}\rightarrow{\infty\textcat{Grpd}}$ which takes a map $\pi_0(A)\rightarrow{A'}$ to the fibre of the morphism 
\begin{equation*}
    \mathcal{F}(A')\rightarrow{\varprojlim_k\mathcal{F}(A'_{\leq k})}
\end{equation*}at $x_0$. We see that this defines a stack on $\mathcal{A}_{/X_{\leq 0}}$. To show that $\mathcal{S}(A)$ is contractible, it suffices to show that $\mathcal{S}(A')$ is contractible for any epimorphism $\textnormal{Spec}(A')\rightarrow{\textnormal{Spec}(A)}$ of stacks. Therefore, we can replace $X_{\leq 0}$ with its $\textcat{S}$-cover $V$ and there will exist a corresponding $U:=U_i$ for some $i\in I$ such that we have a commutative diagram 
\begin{equation*}
    \begin{tikzcd}
    & U\arrow{d}\\
    X_{\leq 0}\arrow{r}{x_0} \arrow{ur}{y_0} &\mathcal{F}
    \end{tikzcd}
\end{equation*}with the morphism $U\rightarrow{\mathcal{F}}$ in $(n-1)\textcat{-P}$. We claim that there is a point $y\in \pi_0(\varprojlim_k U(A_{\leq k}))$ whose image in $\pi_0(\varprojlim_k\mathcal{F}(A_{\leq k}))$ is $x$. We proceed by induction on $k\geq 0$ to construct a sequence of compatible maps $y_k:X_{\leq k}\rightarrow{U}$. Indeed, when $k=0$, we can consider the induced map $y_0:X_{\leq 0}\rightarrow{U}$ coming from the above commutative diagram. Now suppose that $k>0$ and that we have a map $y_{k-1}:X_{\leq k-1}\rightarrow{U}$. Using our Postnikov lifting property for the $(n-1)\textcat{-P}$ map $U\rightarrow{\mathcal{F}}$, there exists a map $y_{k}:X_{\leq k}\rightarrow{U}$ such that the subdiagrams in the following commutative diagram commute 
\begin{equation*}
    \begin{tikzcd}
    X_{\leq k-1}\arrow{d} \arrow{r}{y_{k-1}} & U\arrow{d}\\
    X_{\leq k} \arrow[ur,dotted,"y_{k}"]\arrow{r}{x_{k}} & \mathcal{F}
    \end{tikzcd}
\end{equation*}Given such a sequence $(y_k$), then $y=\varprojlim_k y_k\in \pi_0(\varprojlim_k U(A_{\leq k}))$ has image $x$ in $\pi_0(\varprojlim_k \mathcal{F}(A_{\leq k}))$.

We examine the following commutative diagram
\begin{equation*}
    \begin{tikzcd}
    U(A)\arrow{d} \arrow{r} & \varprojlim_k U(A_{\leq k}) \arrow{d}\\
    \mathcal{F}(A) \arrow{r} & \varprojlim_k \mathcal{F}(A_{\leq k})
    \end{tikzcd}
\end{equation*}
We note that, since $U\rightarrow{\mathcal{F}}$ is $(n-1)$-representable, then by our inductive hypothesis this is a pullback square. Since each $x\in\pi_0(\varprojlim_k \mathcal{F}(A_{\leq k}))$ corresponds to some point $y\in \pi_0(\varprojlim_k U(A_{\leq k}))$, the fibre of the morphism $\mathcal{F}(A)\rightarrow \varprojlim_k \mathcal{F}(A_{\leq k})$ at $x$ is equivalent to the fibre of the morphism $U(A)\rightarrow\varprojlim_k U(A_{\leq k})$ at $y$. Now, since $U(A)\rightarrow{\varprojlim_k U(A_{\leq k})}$ is an equivalence by the case $n=-1$, its fibre at $y$ is contractible, and hence the fibre of $\mathcal{F}(A)\rightarrow{\varprojlim_k \mathcal{F}(A_{\leq k})}$ at $x$ is also contractible. 
\end{proof}

\begin{remark}
    As an immediate corollary we see that in this situation any $n$-geometric stack is nilcomplete. 
\end{remark}

\subsection{Climbing the Postnikov Tower}

Suppose that $(\mathcal{C},\mathcal{C}_{\geq 0},\mathcal{C}_{\leq 0},\mathcal{C}^{0},\bm\tau,\textcat{P},\mathcal{A},\textcat{M})$ is a Postnikov compatible derived geometry context and that $k\geq 1$.

\begin{lem}\phantomsection\label{postnikovsquarezero}
     There exists a derivation $d_{k}\in\pi_0(\textcat{Der}(A_{\leq k-1},\pi_k(A)[k+1]))$ such that $A_{\leq k-1}\oplus_{d_k}\pi_k(A)[k]$ is isomorphic to $A_{\leq k}$.
\end{lem}

\begin{proof}
This follows from \cite[c.f. Theorem 2.1.35]{ben-bassat_perspective_2024} using Theorem \ref{2nconnective} and Assumption (\ref{postnikovcomp5}) of Postnikov compatibility.
\end{proof}

\begin{lem}\phantomsection\label{postnikovcotangentvanish}
There exist natural equivalences 
\begin{equation*}\begin{aligned}
\pi_{k+2}(\mathbb{L}_{A_{\leq k-1}/A_{\leq k}})&\simeq 0\\
    \pi_{k+1}(\mathbb{L}_{A_{\leq k-1}/A_{\leq k}})&\simeq \pi_k(A)\\
    \pi_i(\mathbb{L}_{A_{\leq k-1}/A_{\leq k}})&=0\quad\text{for $i\leq k$}\end{aligned}
\end{equation*}

\end{lem}

\begin{proof}Denote the morphism $A_{\leq k}\rightarrow{A_{\leq k-1}}$ by $f$. Since $\textnormal{cofib}(f)$ is $k$-connective, by Theorem \ref{2nconnective} and Assumption (\ref{postnikovcomp5}) of Postnikov compatibility, the fibre of the induced map $A_{\leq k-1}\otimes^\mathbb{L}_{A_{\leq k}}\textnormal{cofib}(f)\rightarrow{\mathbb{L}_{A_{\leq k-1}/A_{\leq k}}}$ is $(k+2)$-connective. It follows, by considering the induced long exact sequence, that 
\begin{equation*}
    \pi_i(\mathbb{L}_{A_{\leq k-1}/A_{\leq k}})\simeq \pi_i(A_{\leq k-1}\otimes^\mathbb{L}_{A_{\leq k}}\textnormal{cofib}(f))
\end{equation*}for $i\leq k+1$ and there is a surjection $\pi_{k+2}(A_{\leq k-1}\otimes_{A_{\leq k}}\textnormal{cofib}(f))\rightarrow\pi_{k+2}(\mathbb{L}_{A_{\leq k-1}/A_{\leq k}})$. Hence, we just need to calculate $\pi_i(A_{\leq k-1}\otimes^\mathbb{L}_{A_{\leq k}}\textnormal{cofib}(f))$. 

Let $K=\pi_k(A)[k+1]=\textnormal{cofib}(f)$. We easily note that there is a fibre-cofibre sequence in $\textcat{Mod}_{A_{\leq k-1}}$
\begin{equation*}
    {K\otimes^\mathbb{L}_{A_{\leq k}}K}\xrightarrow{m}{K\simeq A_{\leq k}\otimes^\mathbb{L}_{A_{\leq k}}K}\rightarrow{A_{\leq k-1}\otimes^\mathbb{L}_{A_{\leq k}}\textnormal{cofib}(f)}
\end{equation*}where $m$ is the multiplication map. Therefore, it suffices to calculate $\pi_i(\textnormal{cofib}(m))$. We note that the multiplication morphism $m:K\otimes^\mathbb{L}_{A_{\leq k}} K\rightarrow{K}$ is nullhomotopic, using the previous lemma and similar reasoning to \cite[Proposition 7.4.1.14]{lurie_higher_2017}. Therefore, we have a splitting
\begin{equation*}
     \textnormal{cofib}(m)\simeq K\coprod (K\otimes^\mathbb{L}_{A_{\leq k}} K)[1]
\end{equation*}from which our result easily follows. 

\end{proof}

Suppose we want to show that an $n$-representable map in $\textcat{PSh}(\textcat{DAff}^{cn}(\mathcal{C}))$ satisfies the Postnikov lifting property with respect to $\mathcal{A}$ as in Definition \ref{postnikovliftproperty}. By Lemma \ref{postnikovsquarezero}, the morphism $A_{\leq k}\rightarrow{A_{\leq k-1}}$ is given by a square zero extension. Hence, it suffices to show the following condition. 

\begin{defn}
    Suppose that $f:\mathcal{F}\rightarrow{\mathcal{G}}$ is an $n$-representable morphism of presheaves in $\textcat{PSh}(\textcat{DAff}^{cn}(\mathcal{C}))$. Then, $f$ satisfies the \textit{square-zero lifting property with respect to $\mathcal{A}$} if, for every $X=\textnormal{Spec}(A)\in\mathcal{A}$, $M\in\textcat{M}_{A,1}$, and $d\in\pi_0(\textcat{Der}(A,M))$, the following lifting problem has at least one solution $X_d[\Omega M]\rightarrow{\mathcal{F}}$
\begin{equation*}\begin{tikzcd}
    X\arrow{d} \arrow{r} & \mathcal{F}\arrow{d}{f}\\
    X_d[\Omega M] \arrow[ur,dotted]\arrow{r}  & \mathcal{G}
\end{tikzcd}\end{equation*}
\end{defn}
Recall the definition of a morphism of stacks being in $n\textcat{-P}|_\mathcal{A}$ from Notation \ref{ngeometricnotation}. We will make the following assumption on such morphisms.

\begin{assumption}\phantomsection\label{assumptiononP}
    If a morphism $f:\mathcal{F}\rightarrow{\mathcal{G}}$ of stacks in $\textcat{Stk}(\mathcal{A},\bm\tau|_\mathcal{A})$ is in $n\textcat{-P}|_\mathcal{A}$ then, for any $X=\textnormal{Spec}(A)\in\mathcal{A}$, any $x:X\rightarrow{\mathcal{F}}$ and any $M\in\textcat{M}_{A,1}$, we have that 
    $\pi_0(\textnormal{Map}_{\textcat{Mod}_A}(\mathbb{L}_{\mathcal{F}/\mathcal{G},x},M))=0$.
\end{assumption}

\begin{prop}\phantomsection\label{sqzlift}Suppose that
\begin{enumerate}
    \item $\bm{\tau}$ and $\textcat{P}$ satisfy the obstruction conditions relative to $\mathcal{A}$ for a suitable class of morphisms $\textcat{S}$, 
    \item For any finite collection $\{U_i\}_{i\in I}$ of $\mathcal{A}$-admissible objects, the map $\coprod_{i\in I} h(U_i)\rightarrow{h(\coprod_{i\in I} U_i)}$ is an equivalence in $\textcat{Stk}(\mathcal{A},\bm\tau|_\mathcal{A})$,
    \item Assumption \ref{assumptiononP} is satisfied.
\end{enumerate}Then, every $n\textcat{-P}|_\mathcal{A}$-representable morphism of stacks in $\textcat{Stk}(\mathcal{A},\bm{\tau}|_\mathcal{A})$ satisfies the square-zero lifting property, and hence the Postnikov lifting property with respect to $\mathcal{A}$.\end{prop}

\begin{proof}
    Indeed, suppose that we have an $n\textcat{-P}|_\mathcal{A}$-representable map $f:\mathcal{F}\rightarrow{\mathcal{G}}$ in $\textcat{Stk}(\mathcal{A},\bm{\tau}|_\mathcal{A})$ and a lifting problem
\begin{equation*}\begin{tikzcd}
    X\arrow{d} \arrow{r} & \mathcal{F}\arrow{d}{f}\\
    X_d[\Omega M] \arrow{r}\arrow[ur,dotted]\arrow{r}  & \mathcal{G}
\end{tikzcd}\end{equation*}where $X=\textnormal{Spec}(A)$, $M\in\textcat{M}_{A,1}$, and $d\in\pi_0(\textcat{Der}(A,M))$. Consider the $n$-geometric stack $\mathcal{H}:=\mathcal{F}\times_\mathcal{G} X_d[\Omega M]$ and consider the map $x:X\rightarrow{\mathcal{F}\times_\mathcal{G} X_d[\Omega M]=\mathcal{H}}$ induced by the pullback. We see that it suffices to find a solution to the following lifting problem 
\begin{equation*}\begin{tikzcd}
X\arrow{d} \arrow{r}{x} & \mathcal{H}\arrow{d}\\
   X_d[\Omega M] \arrow[ur,dotted]\arrow{r}{id} & X_d[\Omega M]
\end{tikzcd}\end{equation*}
Now, since $\bm{\tau}$ and $\textcat{P}$ satisfy the obstruction conditions, then by Theorem \ref{obstructiongeometric}, $\mathcal{H}$ has an obstruction theory. Hence, by Proposition \ref{liftpropositionpath}, the obstruction to lifting the map $X\rightarrow{X_d[\Omega M]}$ to a map $X_d[\Omega M]\rightarrow{\mathcal{H}}$ lies in 
\begin{equation*}
    \pi_0(\textnormal{Map}_{\textcat{Mod}_A}(\mathbb{L}_{\mathcal{H}/X_d[\Omega M],x},M))
\end{equation*}which is zero by our assumption, since the map $\mathcal{H}\rightarrow{X_d[\Omega M]}$ is in $n\textcat{-P}|_\mathcal{A}$.

\end{proof}

\begin{remark}
    It follows that, in the above context, $n$-representable$|_\mathcal{A}$ morphisms of stacks in $\textcat{Stk}(\mathcal{A},\bm\tau|_\mathcal{A})$ are nilcomplete. 
\end{remark}

\section{Representability Contexts}\label{chapter5}

In this section, we want to put conditions on our derived geometry contexts such that, in some sense, the derived geometry on them can be determined by the classical geometry on their hearts. We will make this statement more precise when we prove our representability theorem. 

\subsection{Affines in the Heart}

Suppose that $(\mathcal{C},\mathcal{C}_{\geq 0},\mathcal{C}_{\leq 0},\mathcal{C}^0,\bm\tau,\textcat{P},\mathcal{A},\textcat{M})$ is a Postnikov compatible derived geometry context. We recall that this consists of a derived algebraic context $(\mathcal{C},\mathcal{C}_{\geq 0},\mathcal{C}_{\leq 0},\mathcal{C}^0)$, a collection of connective affines $\mathcal{A}\subseteq \textcat{DAff}^{cn}(\mathcal{C})$, and a certain collection $\textcat{M}$ of \textit{good modules}. Since we are assuming that our context is Postnikov compatible, there is an adjunction 
\begin{equation*}
    \pi_0=\tau_{\leq 0}|_{\textcat{DAlg}^{cn}(\mathcal{C})}:\textcat{DAlg}^{cn}(\mathcal{C})\leftrightarrows{\textcat{DAlg}^\heartsuit(\mathcal{C})}:\iota_{\leq 0}|_{\textcat{DAlg}^\heartsuit(\mathcal{C})}
\end{equation*}We will denote the induced adjunction on the opposite categories by 
\begin{equation*}
    \iota:\textcat{DAff}^{\heartsuit}(\mathcal{C})\leftrightarrows\textcat{DAff}^{cn}(\mathcal{C}):t_0
\end{equation*}

\begin{defn}
    Define the full subcategory $\mathcal{A}^\heartsuit\subseteq\mathcal{A}\subseteq\textcat{DAff}^\heartsuit(\mathcal{C})$ to consist of objects $X$ in $\textcat{DAff}^\heartsuit(\mathcal{C})$ such that $X=t_0(Y)$ for some $Y\in \mathcal{A}$. 
\end{defn}

\begin{defn}
    We define a class $\textcat{P}^\heartsuit$ of maps in $\mathcal{A}^\heartsuit$ and a collection $\bm\tau^\heartsuit$ of covering families in $\textnormal{Ho}(\mathcal{A}^{\heartsuit})$ as follows
    \begin{enumerate}
        \item The collection $\textcat{P}^\heartsuit$ is defined to be the collection of morphisms $f$ such that $f\in\textcat{P}\cap \mathcal{A}^\heartsuit$,
        \item The collection of $\bm\tau^\heartsuit$-covers is the collection $\{t_0(U_i)\rightarrow{t_0(X)}\}_{i\in I}$ such that $\{U_i\rightarrow{X}\}_{i\in I}$ is a $\bm\tau$-cover in $\mathcal{A}$.
    \end{enumerate}
\end{defn}

It is clear from our definition that the collection of $\bm\tau^\heartsuit$-covering families in $\textnormal{Ho}(\mathcal{A}^\heartsuit)$ defines a Grothendieck pre-topology on $\mathcal{A}^\heartsuit$. If $\iota|_{\mathcal{A}^\heartsuit}:(\mathcal{A}^\heartsuit,\bm\tau^\heartsuit)\rightarrow{(\mathcal{A},\bm\tau|_\mathcal{A})}$ is a continuous functor of $(\infty,1)$-sites, then by our remarks in Section \ref{continuousfunctorsection}, $\iota|_{\mathcal{A}^\heartsuit}$ induces the following functors on the corresponding categories of stacks.

\begin{defn}
\begin{enumerate}
    \item The \textit{truncation functor} $t_0$ is defined to be the functor \begin{equation*}
    t_0:=(\iota|_{\mathcal{A}^\heartsuit})^*:\textcat{Stk}(\mathcal{A},\bm{\tau}|_{\mathcal{A}})\rightarrow{\textcat{Stk}(\mathcal{A}^\heartsuit,\bm{\tau}^\heartsuit)}
\end{equation*}
\item The \textit{extension functor} $i$ is defined to be its left adjoint \begin{equation*}
    i:=(\iota|_{\mathcal{A}^\heartsuit})_{\#}:\textcat{Stk}(\mathcal{A}^\heartsuit,\bm{\tau}^\heartsuit)\rightarrow{\textcat{Stk}(\mathcal{A},\bm{\tau}|_{\mathcal{A}})}
    \end{equation*}
\end{enumerate}

\end{defn}

Recall from Section \ref{geometricstacksection} that we use the notation $\textcat{Stk}_n(\mathcal{A},\bm\tau|_\mathcal{A},\textcat{P}|_\mathcal{A})$ to denote the collection of $n$-geometric stacks defined relative to the strong $(\infty,1)$-geometry tuple $(\mathcal{A},\bm\tau|_\mathcal{A},\textcat{P}|_\mathcal{A},\mathcal{A})$. In particular, $(-1)$-geometric stacks are those represented by objects in $\mathcal{A}$. We also recall that there is a fully faithful functor
\begin{equation*}
    \textcat{Stk}_n(\mathcal{A},\bm\tau|_\mathcal{A},\textcat{P}|_\mathcal{A})\rightarrow{\textcat{Stk}_n(\textcat{DAff}^{cn}(\mathcal{C}),\bm\tau,\textcat{P},\mathcal{A})}
\end{equation*}

\subsection{Transverse, Flat, and Derived Strong Modules}

We recall that the $t$-structure on $\textcat{Mod}_A$ is induced from the $t$-structure on $\mathcal{C}$, and hence, by our Postnikov compatibility condition, lines up with the $t$-structure on $\textcat{DAlg}^{cn}(\mathcal{C})$.  In general the homotopy groups $\pi_i$ don't commute with base-change. Under some conditions, such as derived strongness or with some flatness assumptions, they do. We note that these conditions are explored more in \cite[Section 2.3.2]{ben-bassat_perspective_2024}.

\begin{lem}\phantomsection\label{piipushout} Suppose that $C\in \textcat{DAlg}^{cn}(\mathcal{C})$ and $A,B\in\textcat{DAlg}^{\heartsuit}(\mathcal{C})$ along with morphisms $\iota(A)\rightarrow{\iota(B)}$ and $\iota(A)\rightarrow{C}$. Then, $\pi_i(C\otimes^\mathbb{L}_{\iota(A)}\iota(B))\simeq \pi_i(C)\otimes^\mathbb{L}_AB$ in $\textcat{Mod}_B^{\heartsuit}(\mathcal{C})$ for each $i$.
\end{lem}

\begin{proof}
We first note that, since the tensor product of objects in the image of the forgetful functor $U:\textcat{DAlg}^{cn}_B(\mathcal{C})\rightarrow{\textcat{Mod}_B}$ is given by the image of the coproduct in $\textcat{DAlg}^{cn}_B(\mathcal{C})$, there is an equivalence
\begin{equation*}
    \pi_0(C\otimes^\mathbb{L}_{\iota(A)}\iota(B))\simeq \pi_0(C)\otimes^\mathbb{L}_AB
\end{equation*}in $\textcat{Mod}_B^\heartsuit$, since $\pi_0$ is a left adjoint. By shifting, the result holds in the situation when $C$ has non-vanishing homotopy group only in degree $k$ for some $k\geq 0$. Now, consider $C\simeq \varinjlim_k C_{\leq k}$ in $\textcat{DAlg}^{cn}(\mathcal{C})$. For $i\leq k$, we note that $\pi_i(C)\simeq \pi_i(C_{\leq k})$ and $\pi_i(C\otimes^\mathbb{L}_{\iota(A)}\iota(B))\simeq \pi_i(C_{\leq k}\otimes^\mathbb{L}_{\iota(A)}\iota(B))$. Hence, it suffices to show that we have an equivalence $\pi_i(C_{\leq k}\otimes^\mathbb{L}_{\iota(A)}\iota(B))\simeq \pi_i(C_{\leq k})\otimes^\mathbb{L}_AB$ for each $k$.

When $k=0$, the result holds since $C_{\leq 0}$ is concentrated in degree $0$. Suppose that $k>0$ and that the result holds for all $n<k$. We define $K_k$ to be the fibre in $\textcat{DAlg}_{\iota(A)}^{cn}(\mathcal{C})$ of the morphism $C_{\leq k}\rightarrow{C_{\leq k-1}}$. Base-changing by the morphism $\iota(A)\rightarrow{\iota(B)}$, we obtain the following fibre sequence
\begin{equation*}
    K_k\otimes^\mathbb{L}_{\iota(A)} \iota(B)\rightarrow{C_{\leq k}\otimes^\mathbb{L}_{\iota(A)} \iota(B)}\rightarrow{C_{\leq k-1}\otimes^\mathbb{L}_{\iota(A)} \iota(B)}
\end{equation*}and hence obtain the following long exact sequence in homotopy
\begin{equation*}
   \dots\rightarrow \pi_i(K_k\otimes^\mathbb{L}_{\iota(A)} \iota(B))\rightarrow{\pi_i(C_{\leq k}\otimes^\mathbb{L}_{\iota(A)} \iota(B))}\rightarrow{\pi_i(C_{\leq k-1}\otimes^\mathbb{L}_{\iota(A)} \iota(B))}\rightarrow{\dots}
\end{equation*}We also have a long exact sequence in $\textcat{Mod}_B^\heartsuit$
\begin{equation*}
\dots\rightarrow\pi_i(K_k)\otimes^\mathbb{L}_AB\rightarrow{\pi
    _i(C_{\leq k})\otimes^\mathbb{L}_AB}\rightarrow{\pi_i(C_{\leq k-1})\otimes^\mathbb{L}_AB}\rightarrow{\dots}
\end{equation*}
We note that $K_k$ is concentrated in degree $k$ for each $k\geq 1$. Therefore, we see that $\pi_i(K_k\otimes^\mathbb{L}_{\iota(A)}\iota(B))\simeq \pi_i(K_k)\otimes^\mathbb{L}_AB$. By induction, $\pi_i(C_{\leq k-1}\otimes^\mathbb{L}_{\iota(A)}\iota(B))\simeq \pi_i(C_{\leq k-1})\otimes^\mathbb{L}_AB$, and hence it follows that $\pi_i(C_{\leq k}\otimes^\mathbb{L}_{\iota(A)}\iota(B))\simeq \pi_i(C_{\leq k})\otimes^\mathbb{L}_AB$. 

\end{proof}

If we impose a flatness condition on our generators, then we can obtain spectral sequences in our derived algebraic context $(\mathcal{C},\mathcal{C}_{\geq 0},\mathcal{C}_{\leq 0},\mathcal{C}^0)$ in a similar way to \cite[Proposition 7.2.1.17]{lurie_higher_2017}. Suppose that $A\in \mathcal{C}$ and consider $\pi_*(A)=\bigoplus_{i}\pi_i(A)$ as an object of $\mathcal{C}^\heartsuit$.

\begin{defn}\cite[Definition 2.3.77]{ben-bassat_perspective_2024}
    An object $P$ of $\mathcal{C}_{\geq 0}$ is \textit{homotopy flat} if, for any object $A$ of $\mathcal{C}$, the natural map of graded objects
    \begin{equation*}
        \pi_*(A)\otimes \pi_*(P)\rightarrow{\pi_*(A\otimes^\mathbb{L}P)}
    \end{equation*}is an equivalence. A derived algebraic context $(\mathcal{C},\mathcal{C}_{\geq 0},\mathcal{C}_{\leq 0},\mathcal{C}^0)$ is \textit{flat} if every object of $\mathcal{C}^0$ is homotopy flat. A derived geometry context is \textit{flat} if the underlying derived algebraic context is flat. 
\end{defn}

\begin{exmp}
    In the setting of Theorem \ref{modelderivedtheorem}, the resulting derived algebraic context will be flat by \cite[Proposition 3.1.69]{ben-bassat_perspective_2024}. 
\end{exmp}

For the rest of this section we fix a flat derived algebraic context $(\mathcal{C},\mathcal{C}_{\geq 0},\mathcal{C}_{\leq 0},\mathcal{C}^0)$. 

\begin{prop}\phantomsection\label{torspectralsequence}
    \cite[Lemma 2.3.79]{ben-bassat_perspective_2024} There is a spectral sequence
    \begin{equation*}
        \textnormal{Tor}^p_{\pi_*(A)}(\pi_*(M),\pi_*(N))_q\Rightarrow{\pi_{p+q}(M\otimes_A^\mathbb{L}N)}
    \end{equation*}for $M,N\in\textcat{Mod}_A^{cn}$. 
\end{prop}

\begin{cor}\phantomsection\label{torzerolemma}
    Suppose that $A\in \textcat{DAlg}^{cn}(\mathcal{C})$ and that $M$ is a connective $A$-module such that $M\otimes^\mathbb{L}_A\pi_0(A)\simeq 0$. Then, $M\simeq 0$. 
\end{cor}

\begin{proof}
    Indeed, consider the spectral sequence
    \begin{equation*}
        \textnormal{Tor}^p_{\pi_*(A)}(\pi_*(M),\pi_0(A))_q\Rightarrow{\pi_{p+q}(M\otimes_A^\mathbb{L}\pi_0(A))\simeq 0}
    \end{equation*}and proceed by induction on $i\geq 0$ to show that $\pi_i(M)\simeq 0$. 
\end{proof}

\begin{defn}Suppose that $A\in\textcat{DAlg}^{cn}(\mathcal{C})$,  $N\in\textcat{Mod}_A^\heartsuit$, and $M\in\textcat{Mod}_A^{cn}$.  Then, \begin{itemize}
    \item $N$ is \textit{transverse to $M\in\textcat{Mod}_A^{cn}$} if $M\otimes_A^\mathbb{L}N$ is in $\textcat{Mod}_A^\heartsuit$,
    \item $M$ is \textit{flat} if every $N\in\textcat{Mod}_A^\heartsuit$ is transverse to it. 
\end{itemize}
\end{defn}

\begin{prop}\phantomsection\label{basechangepi0}\cite[c.f. Lemma 2.3.82]{ben-bassat_perspective_2024}
    Suppose that $A\in\textcat{DAlg}^{cn}(\mathcal{C})$. Let $M$ be a connective $A$-module and $N$ a connective $A$-module. Suppose that, for each $i\leq n$, $\pi_i(N)$ is transverse to $M$. Then, there is an equivalence
    \begin{equation*}
        \pi_i(M\otimes_A^\mathbb{L}N)\simeq\pi_0(M)\otimes^\mathbb{L}_{\pi_0(A)}\pi_i(N)
    \end{equation*}for all $i\leq n$.
\end{prop}
\begin{defn}\phantomsection\label{derivedstrong}
    Suppose that $A\in\textcat{DAlg}^{cn}(\mathcal{C})$ and $M\in\textcat{Mod}_A^{cn}$. Then, $M$ is \textit{derived strong} if the map\begin{equation*}
    \pi_*(A)\otimes^\mathbb{L}_{\pi_0(A)}\pi_0(M)\rightarrow{\pi_*(M)}
\end{equation*}is an equivalence. A morphism $f:A\rightarrow{B}$ is \textit{derived strong} if $B$ is derived strong as an $A$-module. 
\end{defn}
We note that equivalences are derived strong maps, and compositions of derived
strong maps are derived strong. Using the spectral sequence, we get more general results about homotopy groups of tensor products under some strongness assumptions. 

\begin{prop}\phantomsection\label{derviedstronghomotopygroupprop}
    Suppose that $A\in\textcat{DAlg}^{cn}(\mathcal{C})$. Suppose that $M$ is a derived strong $A$-module and that $N$ is an $A$-module such that $\pi_*(N)$ is transverse to $\pi_0(M)$ over $\pi_0(A)$. Then, there is an isomorphism
    \begin{equation*}
        \pi_*(M\otimes_A^\mathbb{L}N)\simeq \pi_0(M)\otimes^\mathbb{L}_{\pi_0(A)}\pi_*(N)
    \end{equation*}
\end{prop}

\begin{proof}
    Indeed, if we consider the spectral sequence
    \begin{equation*}
        \textnormal{Tor}^p_{\pi_*(A)}(\pi_*(M),\pi_*(N))_q\Rightarrow{\pi_{p+q}(M\otimes_A^\mathbb{L}N)}
    \end{equation*}Then, since $M$ is a derived strong $A$-module, we see that 
    \begin{equation*}
       \pi_*(M)\otimes^\mathbb{L}_{\pi_*(A)}\pi_*(N)\simeq \pi_0(M)\otimes^\mathbb{L}_{\pi_0(A)}\pi_*(N)
    \end{equation*}and hence, using our transversality assumption, the spectral sequence degenerates on the first page, and we obtain our result.
\end{proof}
\begin{lem}\phantomsection\label{derivedstrongequiv}
    Suppose that $f:A\rightarrow{B}$ is a derived strong map in $\textcat{DAlg}^{cn}(\mathcal{C})$, then $\pi_0(B)\simeq B\otimes^\mathbb{L}_A\pi_0(A)$.
\end{lem}

\begin{proof}
    Indeed, if we consider the spectral sequence
    \begin{equation*}
        \textnormal{Tor}^p_{\pi_*(A)}(\pi_*(B),\pi_0(A))_q\Rightarrow{\pi_{p+q}(B\otimes_A^\mathbb{L}\pi_0(A))}
    \end{equation*}Then, since $f$ is derived strong, we see that 
    \begin{equation*}
       \pi_*(B)\otimes^\mathbb{L}_{\pi_*(A)}\pi_0(A)\simeq (\pi_*(A)\otimes^\mathbb{L}_{\pi_0(A)}\pi_0(B))\otimes^\mathbb{L}_{\pi_*(A)}\pi_0(A)\simeq \pi_0(B)
    \end{equation*}and hence the spectral sequence degenerates on the first page, and we obtain our result.
\end{proof}

\begin{cor}\phantomsection\label{cotangentderivedstronglyflat}
    If $f:A\rightarrow{B}$ in $\textcat{DAlg}^{cn}(\mathcal{C})$ is derived strong, then 
    \begin{equation*}
        \mathbb{L}_{B/A}\otimes^\mathbb{L}_B\pi_0(B)\simeq \mathbb{L}_{\pi_0(B)/\pi_0(A)}
    \end{equation*}
\end{cor}

\begin{proof}
  By the previous lemma, there is a pushout square in $\textcat{DAlg}^{cn}(\mathcal{C})$
    \begin{equation*}\begin{tikzcd}
        A\arrow{r}\arrow{d} & B \arrow{d}\\
        \pi_0(A)\arrow{r} & \pi_0(B)
    \end{tikzcd}\end{equation*}Hence, the result follows by Corollary \ref{cofibcotresults}.
\end{proof}

Suppose that we have some derived algebraic context $(\mathcal{C},\mathcal{C}_{\geq 0},\mathcal{C}_{\leq 0},\mathcal{C}^0)$, $\mathcal{A}$ is a full subcategory of $\textcat{DAff}^{cn}(\mathcal{C})$, and that we have a class $\textcat{P}$ of maps in $\textcat{DAff}^{cn}(\mathcal{C})$. In order for morphisms in $\textcat{P}$ to be well behaved with respect to their truncations, we will impose the following strongness condition.

\begin{defn}\phantomsection\label{dderivedstrong}
    Suppose that $f:Y=\textnormal{Spec}(B)\rightarrow{X=\textnormal{Spec}(A)}$ is a morphism in $\mathcal{A}$. Then, $f$ is \textit{derived strong relative to $\textcat{P}|_\mathcal{A}$} if
    \begin{enumerate}
        \item The corresponding morphism $f:A\rightarrow{B}$ is derived strong, 
        \item The morphism $t_0(f):t_0(Y)\rightarrow{t_0(X)}$ is in $\textcat{P}^\heartsuit$.
    \end{enumerate}
\end{defn}

\begin{remark}
    In the derived algebraic context corresponding to simplicial commutative rings, maps are smooth if and only if they are strongly smooth, and \'etale if and only if they are strongly \'etale. 
\end{remark}

Recall from Section \ref{geometriessection}, that if $\bm\tau$ is a topology on $\textnormal{Ho}(\textcat{DAff}
^{cn}(\mathcal{C}))$, then we can extend a class of maps $\textcat{P}$ to be local for a topology $\bm\tau$ by defining the class $\textcat{P}^{\bm\tau}$. Suppose that, whenever a morphism is derived strong relative to $\textcat{P}|_\mathcal{A}$, then it is in $\textcat{P}|_\mathcal{A}$. Then, we note that if a morphism is derived strong relative to $\textcat{P}^{\bm\tau}|_\mathcal{A}$ and morphisms in $\bm\tau$ are derived strong, then it is in $\textcat{P}^{\bm\tau}|_\mathcal{A}$.

\subsection{Representability Contexts}

In our proof of the representability theorem, we will assume representability of the truncated stack and then lift certain constructions to the derived stack. We will need to put suitable conditions on our derived geometry contexts such that they are well-behaved with respect to their truncation and such that we have the necessary obstruction theories. Recall the notation $n$-representable$|_\mathcal{A}$ from Notation \ref{ngeometricnotation}.

\begin{defn}\phantomsection\label{repcontextdefn}
    $(\mathcal{C},\mathcal{C}_{\geq 0},\mathcal{C}_{\leq 0},\mathcal{C}^0,\bm\tau,\textcat{P},\mathcal{A},\textcat{M},\textcat{S})$ is a \textit{representability context} if
    \begin{itemize}
        \item $(\mathcal{C},\mathcal{C}_{\geq 0},\mathcal{C}_{\leq 0},\mathcal{C}^0,\bm\tau,\textcat{P},\mathcal{A},\textcat{M})$ is a flat Postnikov compatible derived geometry context, 
        \item $\textcat{S}$ is a class of morphisms in $\mathcal{A}$,
    \end{itemize}such that 
    \begin{enumerate}
    \item\label{repcontexttuple}$(\mathcal{A},\bm\tau|_\mathcal{A},\textcat{P}|_\mathcal{A},\mathcal{A}^\heartsuit)$ is a strong relative $(\infty,1)$-geometry tuple, \item\label{repcontextcontinuous}$\iota|_{\mathcal{A}^\heartsuit}:(\mathcal{A}^\heartsuit,\bm\tau^\heartsuit)\rightarrow{(\mathcal{A},\bm\tau|_\mathcal{A})}$ is a continuous functor of $(\infty,1)$-sites,      
        \item\label{repcontextrepresentables}For any finite collection $\{U_i\}_{i\in I}$ of $\mathcal{A}$-admissible objects, the map $\coprod_{i\in I} h(U_i)\rightarrow{h(\coprod_{i\in I} U_i)}$ is an equivalence in $\textcat{Stk}(\mathcal{A},\bm\tau|_\mathcal{A})$,
        \item\label{repcontextstrong}A morphism is in $\textcat{P}|_\mathcal{A}$ if it is derived strong relative to $\textcat{P}|_\mathcal{A}$,
        \item\label{repcontextP}An $n$-representable$|_\mathcal{A}$ morphism of stacks $f:\mathcal{F}\rightarrow{\mathcal{G}}$ in $\textcat{Stk}(\mathcal{A},\bm\tau|_\mathcal{A})$ is in $n\textcat{-P}|_\mathcal{A}$ if $f$ satisfies that, for any $x:X=\textnormal{Spec}(A)\rightarrow{\mathcal{F}}$ and any $M\in\textcat{M}_{A,1}$, 
    \begin{equation*}
        \pi_0(\textnormal{Map}_{\textcat{Mod}_A}(\mathbb{L}_{\mathcal{F}/\mathcal{G},x},M))=0
    \end{equation*}The converse holds if $t_0(\mathcal{F})\rightarrow{t_0(\mathcal{G})}$ is in $n\textcat{-P}^\heartsuit$,
    \item\label{repcontextobstruction}$\bm{\tau}$ and $\textcat{P}$ satisfy the obstruction conditions of Definition \ref{artinsconditions} relative to $\mathcal{A}$, for the class $\textcat{S}$ of morphisms.
    \end{enumerate}
\end{defn}

We will give several motivating examples of these representability contexts in Section \ref{topologiessection}. In particular, we will describe how derived algebraic geometry in the sense of To\"en and Vezzosi \cite{toen_homotopical_2008}, where we work with simplicial commutative rings, can be reframed as a representability context in the above sense. 

\subsection{Truncation and Extension}

Fix a representability context $(\mathcal{C},\mathcal{C}_{\geq 0},\mathcal{C}_{\leq 0},\mathcal{C}^0,\bm\tau,\textcat{P},\mathcal{A},\textcat{M},\textcat{S})$. By Assumption (\ref{repcontextcontinuous}), the functor $\iota|_{\mathcal{A}^\heartsuit}:(\mathcal{A}^\heartsuit,\bm\tau^\heartsuit)\rightarrow{(\mathcal{A},\bm\tau|_\mathcal{A})}$ induces an adjunction
\begin{equation*}
i:\textcat{Stk}(\mathcal{A}^\heartsuit,\bm\tau^\heartsuit)\leftrightarrows \textcat{Stk}(\mathcal{A},\bm\tau|_{\mathcal{A}}):t_0
\end{equation*}
\begin{lem}\phantomsection\label{iotacocontinuous}
    $\iota|_{\mathcal{A}^\heartsuit}$ is a cocontinuous functor.
\end{lem}

\begin{proof}Consider a $\bm{\tau}$-cover $\{U_j\rightarrow{\iota(X)}\}_{j\in J}$ for $U_j\in \mathcal{A}$ and $X\in\mathcal{A}^\heartsuit$. We consider the collection $\{t_0(U_j)\rightarrow{t_0\circ\iota(X)\simeq X}\}_{j\in J}$, which is a $\bm{\tau}^\heartsuit$-covering family in $\textnormal{Ho}(\mathcal{A}^\heartsuit)$ by definition. We note that the family $\{\iota\circ t_0(U_j)\rightarrow{\iota(X)}\}_{j\in J}$ refines the cover $\{U_j\rightarrow{\iota(X)}\}_{j\in J}$ since each morphism $\iota\circ t_0(U_j)\rightarrow{\iota(X)}$ factors through $U_j\rightarrow{\iota(X)}$ via the counit. 
\end{proof}
\begin{lem}
    The extension functor $i$ is fully faithful. 
\end{lem}

\begin{proof}This follows from Proposition \ref{fullyfaithfulstackfunctor} since every object of $\mathcal{A}^\heartsuit$ is admissible by Assumption (\ref{repcontexttuple}) and the functor $\iota$ is fully faithful and both continuous and cocontinuous.
\end{proof}

\begin{lem}\phantomsection\label{representabilityextension}
    The extension functor $i$ sends $(-1)\textcat{-P}^\heartsuit$-morphisms between representable stacks on $\mathcal{A}^\heartsuit$ to $(-1)\textcat{-P}|_\mathcal{A}$-morphisms between representable stacks on $\mathcal{A}$.
\end{lem}

\begin{proof}
    Suppose that $f:Y\rightarrow{X}$ is a $(-1)\textcat{-P}^\heartsuit$-morphism between representable stacks on $\mathcal{A}^\heartsuit$. Consider the induced morphism $i(f):i(Y)\rightarrow{i(X)}$ of stacks on $\mathcal{A}$ and suppose that there exists some representable stack $Z$ on $\mathcal{A}$ along with a morphism $Z\rightarrow{i(X)}$. We note that the pullback $i(Y)\times_{i(X)}Z$ is representable, and hence, by Assumption (\ref{repcontextstrong}), it suffices to show that the induced morphism $i(Y)\times_{i(X)}Z\rightarrow{Z}$ is derived strong relative to $\textcat{P}|_\mathcal{A}$. Since $t_0(i(Y)\times_{i(X)}Z)\simeq Y\times_Xt_0(Z)$ and $f$ is in $(-1)\textcat{-P}^\heartsuit$, we know that $t_0(i(Y)\times_{i(X)}Z)\rightarrow{t_0(Z)}$ is in $\textcat{P}^\heartsuit$. 
    
    Suppose that $X=\textnormal{Spec}(A)$ and $Y=\textnormal{Spec}(B)$ for $A,B\in \mathcal{A}^\heartsuit$, and that $Z=\textnormal{Spec}(C)$ for $C\in \mathcal{A}$. We conclude that the map is derived strong by noticing that, by Proposition \ref{piipushout} and our assumptions on our representability context, there is an equivalence
    \begin{equation*}\begin{aligned}
        \pi_i(C)\otimes^\mathbb{L}_{\pi_0(C)}\pi_0(i(B)\otimes^\mathbb{L}_{i(A)}C)&\simeq  B\otimes^\mathbb{L}_A\pi_i(C)\simeq{\pi_i(i(B)\otimes^\mathbb{L}_{i(A)}C)}
    \end{aligned}\end{equation*}for each $i\geq 0$.
\end{proof}

\begin{prop}\phantomsection\label{extensiongeometricstack}
    If $\mathcal{A}$ is closed under $\bm\tau$-descent relative to $\mathcal{A}$, then the extension functor $i$ satisfies the following properties:
    \begin{enumerate}
        \item\label{extension1} $i$ preserves pullbacks of stacks in $\textcat{Stk}_n(\mathcal{A}^\heartsuit,\bm\tau^\heartsuit,\textcat{P}^\heartsuit)$ along $m\textcat{-P}|_\mathcal{A}$-morphisms for all $m\geq -1$, 
        \item\label{extension2} $i$ sends stacks in the category $\textcat{Stk}_n(\mathcal{A}^\heartsuit,\bm{\tau}^\heartsuit,\textcat{P}^\heartsuit)$ to stacks in $\textcat{Stk}_n(\mathcal{A},\bm{\tau}|_{\mathcal{A}},\textcat{P}|_\mathcal{A})$ and also sends $n\textcat{-P}^\heartsuit$ morphisms between stacks in $\textcat{Stk}_n(\mathcal{A}^\heartsuit,\bm{\tau}^\heartsuit,\textcat{P}^\heartsuit)$ to $n\textcat{-P}$ morphisms between stacks in $\textcat{Stk}_n(\mathcal{A},\bm{\tau}|_{\mathcal{A}},\textcat{P}|_\mathcal{A})$. 
    \end{enumerate}
\end{prop}

\begin{proof}
Indeed, we note that $i$ preserves pullbacks of representable stacks because $i(\textnormal{Spec}(A))=\textnormal{Spec}(A)$ and every object of $\mathcal{A}$ is $\mathcal{A}$-admissible. Therefore, Statement (\ref{extension1}) follows by Proposition \ref{pullbacksegal}. To show Statement (\ref{extension2}), we use Proposition \ref{ngeometricity} along with Statement (\ref{extension1}), the assumption that $\mathcal{A}$ is closed under $\bm{\tau}$-descent relative to $\mathcal{A}$, and that $i$ sends $(-1)\textcat{-P}^\heartsuit$-morphisms between representable stacks to $(-1)\textcat{-P}|_\mathcal{A}$-morphisms between representable stacks by Lemma \ref{representabilityextension}.
\end{proof}

\begin{prop}
If $\mathcal{A}$ is closed under $\bm\tau$-descent relative to $\mathcal{A}$, then $\mathcal{A}^\heartsuit$ is closed under $\bm\tau^\heartsuit$-descent relative to $\mathcal{A}^\heartsuit$.
\end{prop}

\begin{proof}
    Indeed, suppose that $\mathcal{F}$ is a stack in $\textcat{Stk}(\mathcal{A}^\heartsuit,\bm\tau^\heartsuit)$ and that we have a morphism $\mathcal{F}\rightarrow{X}$ for some $X\in\mathcal{A}^\heartsuit$. Suppose that we have a $\bm\tau^\heartsuit$-covering family $\{U_i\rightarrow{X}\}_{i\in I}$ such that $\mathcal{F}\times_XU_i$ is a representable stack for every $i$. We note that, since $\iota$ is continuous, $\{\iota(U_i)\rightarrow{\iota(X)}\}_{i\in I}$ is a $\bm\tau$-cover, and, moreover, since each morphism $\iota(U_i)\rightarrow{\iota(X)}$ is in $\textcat{P}|_\mathcal{A}$, then we see that $i(\mathcal{F}\times_XU_i)\simeq i(\mathcal{F})\times_{i(X)} i(U_i)$ is representable for each $i$ by Proposition \ref{extensiongeometricstack}. Therefore, since $\mathcal{A}$ is closed under $\bm\tau$-descent relative to $\mathcal{A}$, it follows that $i(\mathcal{F})$ is representable, say $i(\mathcal{F})=Y$ for some $Y\in\mathcal{A}$. Since $i$ is fully faithful, $\mathcal{F}\simeq t_0(i(\mathcal{F}))=t_0(Y)$ is a representable stack on $\mathcal{A}^\heartsuit$.
\end{proof}

\begin{prop}\phantomsection\label{t0preserves}
The truncation functor $t_0$ satisfies the following properties:
\begin{enumerate}
    \item\label{t01} $t_0$ preserves epimorphisms of stacks,
    \item\label{t02} $t_0$ sends stacks in $\textcat{Stk}_n(\mathcal{A},\bm{\tau}|_{\mathcal{A}},\textcat{P}|_\mathcal{A})$ to stacks in $\textcat{Stk}_n(\mathcal{A}^\heartsuit,\bm{\tau}^\heartsuit,\textcat{P}^\heartsuit)$ and sends $n\textcat{-P}|_\mathcal{A}$ morphisms between stacks in $\textcat{Stk}(\mathcal{A},\bm{\tau}|_{\mathcal{A}})$ to $n\textcat{-P}^\heartsuit$ morphisms between stacks in $\textcat{Stk}(\mathcal{A}^\heartsuit,\bm{\tau}^\heartsuit)$.
\end{enumerate}
\end{prop}

\begin{proof}
Property number (\ref{t01}) follows by Proposition \ref{epimorphismcocontinuous} because $\iota$ is a cocontinuous functor by Lemma \ref{iotacocontinuous}. Now we use the setting of Example \ref{adjunctionsFG} to show Property number (\ref{t02}). We note that $t_0$ preserves finite limits because it is a right adjoint. It sends morphisms in $\textcat{P}|_\mathcal{A}$ between representable stacks to morphisms in $\textcat{P}^\heartsuit$ between representable stacks. Moreover, we note that every representable stack $X$ in $\textcat{Stk}(\mathcal{A}^\heartsuit,\bm{\tau}^\heartsuit)$ is equivalent to $t_0(\iota(X))$. Hence, we can conclude using Proposition \ref{therealtruncationprop}. 
\end{proof}

\subsection{The Relative Cotangent Complex of the Truncation}

In this section, we will explore connectivity of the relative cotangent complex of the natural morphism $t_0(\mathcal{F})\rightarrow{\mathcal{F}}$ of stacks in $\textcat{Stk}(\mathcal{A},\bm\tau|_\mathcal{A})$. The following result follows in a similar way to \cite[Lemma 7.15]{porta_representability_2020}. 

\begin{lem}\phantomsection\label{corollaryfullyfaithful}
    Suppose that $A\in \mathcal{A}^{\heartsuit,op}$. If $M_1,M_2\in\textcat{M}_{A,-n}$ for some $n$ and
    \begin{equation*}\textnormal{Map}_{\textcat{Mod}_{A}}(M_1,N[m])\simeq \textnormal{Map}_{\textcat{Mod}_{A}}(M_2,N[m])
\end{equation*}for every $N\in\textcat{M}_{A}^\heartsuit$, then in $\textcat{Mod}_A$
    \begin{equation*}
        \tau_{\leq m}(M_1)\simeq \tau_{\leq m}(M_2)
    \end{equation*}
\end{lem}

\begin{lem}\phantomsection\label{2connectivelemma}Suppose that $\mathcal{F}$ is a stack in $\textcat{Stk}(\mathcal{A},\bm\tau|_\mathcal{A})$ such that the morphism $t_0(\mathcal{F})\rightarrow{\mathcal{F}}$ has an obstruction theory. Suppose that $u_0:U_0=\textnormal{Spec}(A_0)\rightarrow{t_0(\mathcal{F})}$ is a morphism in $\textcat{P}^\heartsuit$ such that $U_0\in\mathcal{A}^\heartsuit$. Then, $\pi_i(\mathbb{L}_{t_0(\mathcal{F})/\mathcal{F},u_0})\simeq 0$ for all $i\leq 1$.
        
    \end{lem}
    \begin{proof}
We consider the fibre-cofibre sequence $\mathbb{L}_{\mathcal{F},u_0}\rightarrow{\mathbb{L}_{t_0(\mathcal{F}),u_0}}\rightarrow{\mathbb{L}_{t_0(\mathcal{F})/\mathcal{F},u_0}}$ from Lemma \ref{globalcotangentresults}. There is an induced long exact sequence in homotopy
\begin{equation}\phantomsection\label{longexacthomotopy}
    \dots\pi_1(\mathbb{L}_{t_0(\mathcal{F})/\mathcal{F},u_0})\rightarrow\pi_0(\mathbb{L}_{\mathcal{F},u_0})\rightarrow{\pi_0(\mathbb{L}_{t_0(\mathcal{F}),u_0})}\rightarrow{\pi_0(\mathbb{L}_{t_0(\mathcal{F})/\mathcal{F},u_0})}\rightarrow{\dots}
\end{equation}We note that, for any $M\in\textcat{M}_{A_0}^{\heartsuit}$, the vertical fibres in the following diagram 
\begin{equation*}
    \begin{tikzcd}
        t_0(\mathcal{F})(A_0\oplus M)\arrow{r} \arrow{d}& \mathcal{F}(A_0\oplus M) \arrow{d}\\
        t_0(\mathcal{F})(A_0)\arrow{r} & \mathcal{F}(A_0)
    \end{tikzcd}
\end{equation*}induce a map 
\begin{equation}\phantomsection\label{equationcotangent}
    \textnormal{Map}_{\textcat{Mod}_{A_0}}(\mathbb{L}_{t_0(\mathcal{F}),u_0},M)\rightarrow \textnormal{Map}_{\textcat{Mod}_{A_0}}(\mathbb{L}_{\mathcal{F},u_0},M)
\end{equation}We note that $A_0\oplus M\in\textcat{DAlg}^{\heartsuit}(\mathcal{C})$. Therefore, the horizontal morphisms in the diagram are equivalences and hence the map in Equation (\ref{equationcotangent}) is an equivalence. By Lemma \ref{corollaryfullyfaithful}, we see that $\tau_{\leq 0}(\mathbb{L}_{t_0(\mathcal{F}),u_0})\simeq \tau_{\leq 0}(\mathbb{L}_{\mathcal{F},u_0})$. Hence, by examining the long exact sequence in Equation (\ref{longexacthomotopy}), we see that $\pi_i(\mathbb{L}_{t_0(\mathcal{F})/\mathcal{F},u_0})=0$ for $i\leq 0$.

Now, define $M:=\pi_1(\mathbb{L}_{t_0(\mathcal{F})/\mathcal{F},u_0})$ and suppose that it is non-zero. We have a long exact sequence in homotopy
\begin{equation*}
    \dots\pi_2(\mathbb{L}_{t_0(\mathcal{F})/\mathcal{F},u_0})\rightarrow\pi_1(\mathbb{L}_{\mathcal{F},u_0})\rightarrow{\pi_1(\mathbb{L}_{t_0(\mathcal{F}),u_0})}\rightarrow{\pi_1(\mathbb{L}_{t_0(\mathcal{F})/\mathcal{F},u_0})}\rightarrow{0}
\end{equation*}and hence the morphism $\pi_1(\mathbb{L}_{t_0(\mathcal{F}),u_0})\rightarrow{\pi_1(\mathbb{L}_{t_0(\mathcal{F})/\mathcal{F},u_0})}$ must be non-zero. Since
\begin{equation*}\begin{aligned}
    \textnormal{Map}_{\textcat{Mod}_{A_0}}(\mathbb{L}_{t_0(\mathcal{F}),u_0},M[1])&\simeq \textnormal{Map}_{\textcat{Mod}_{A_0}^{\leq 1}}(\tau_{\leq 1}(\mathbb{L}_{t_0(\mathcal{F}),u_0}),M[1])\\
    &\simeq \textnormal{Map}_{\textcat{Mod}_{A_0}^{\leq 1}}(\pi_1(\mathbb{L}_{t_0(\mathcal{F}),u_0})[1],\pi_1(\mathbb{L}_{t_0(\mathcal{F})/\mathcal{F},u_0})[1])
\end{aligned}\end{equation*}we see that there exists a non-zero map $\mathbb{L}_{t_0(\mathcal{F}),u_0}\rightarrow{M[1]}$. Since the morphism $U_0\rightarrow{t_0(\mathcal{F})}$ is in $\textcat{P}^\heartsuit$, we see that 
\begin{equation*}
    \pi_0(\textnormal{Map}_{\textcat{Mod}_{A_0}}(\mathbb{L}_{U_0/t_0(\mathcal{F}),u_0},M[1]))=0
\end{equation*}by Assumption (\ref{repcontextP}). Hence, by considering the fibre sequence of cotangent complexes $\mathbb{L}_{t_0(\mathcal{F}),u_0}\rightarrow{\mathbb{L}_{U_0,u_0}}\rightarrow{\mathbb{L}_{U_0/t_0(\mathcal{F}),u_0}}$, we see that there must exist a non-zero map $\mathbb{L}_{U_0,u_0}\rightarrow{M[1]}$. This defines a non-zero derivation $d:A_0\rightarrow{A_0\oplus M[1]}$.

Since the morphism $t_0(\mathcal{F})\rightarrow{\mathcal{F}}$ has an obstruction theory, then by Proposition \ref{liftpropositionpath}, the fibre of the morphism 
    \begin{equation}\phantomsection\label{t0pullbackequation}
        t_0(\mathcal{F})(A_0\oplus_d \Omega M[1])\rightarrow{t_0(\mathcal{F})(A_0)}\times_{\mathcal{F}(A_0)}\mathcal{F}(A_0\oplus_d \Omega M[1])
    \end{equation}at the point $x\in \pi_0(t_0(\mathcal{F})(A_0)\times_{\mathcal{F}(A_0)}\mathcal{F}(A_0\oplus_d \Omega M[1]))$ defined by the following commutative diagram
    \begin{equation*}
        \begin{tikzcd}
            U_0\arrow{r}{u_0}\arrow{d} & t_0(\mathcal{F}) \arrow{d}{j}\\
            (U_0)_d[ \Omega M[1]]\arrow{r} & \mathcal{F}
        \end{tikzcd}
    \end{equation*}is isomorphic, for some $\alpha(x)\in\pi_0(\textnormal{Map}_{\textcat{Mod}_{A_0}}(\mathbb{L}_{t_0(\mathcal{F})/\mathcal{F},u_0},M[1]))$, to the path space
\begin{equation*}\Omega_{\alpha(x),0}\textnormal{Map}_{\textcat{Mod}_{A_0}}(\mathbb{L}_{t_0(\mathcal{F})/\mathcal{F},u_0},M[1])
    \end{equation*}The morphism in Equation (\ref{t0pullbackequation}) is an equivalence because $A_0\oplus_d \Omega M[1]$ and $A_0$ are in $\textcat{DAlg}^{\heartsuit}(\mathcal{C})$. Hence, the path space is zero. However, since $d$ is non-zero and provides a path between $\alpha(x)$ and $0$, we get a contradiction. Therefore, $\pi_1(\mathbb{L}_{t_0(\mathcal{F})/\mathcal{F},u_0})\simeq 0$.
    \end{proof}

\subsection{The Representability Theorem} 

The following representability theorem bears many similarities to the one stated in \cite[Theorem C.0.9]{toen_homotopical_2008} and will be sufficient for our purposes.
\begin{thm}\phantomsection\label{representabilitytheorem}Suppose that $(\mathcal{C},\mathcal{C}_{\geq 0},\mathcal{C}_{\leq 0},\mathcal{C}^0,\bm\tau,\textcat{P},\mathcal{A},\textcat{M},\textcat{S})$ is a representability context and that $\mathcal{F}$ is a stack in $\textcat{Stk}(\mathcal{A},\bm{\tau}|_{\mathcal{A}})$. The following conditions are equivalent. 
\begin{enumerate}
    \item $\mathcal{F}$ is an $n$-geometric stack in $\textcat{Stk}_n(\mathcal{A},\bm\tau|_\mathcal{A},\textcat{P}|_\mathcal{A})$, 
    \item $\mathcal{F}$ satisfies the following three conditions:
    \begin{enumerate}
\item\makeatletter\def\@currentlabel{a}\makeatother\label{representability1}The truncation $t_0(\mathcal{F})$ is an $n$-geometric stack in $\textcat{Stk}_n(\mathcal{A}^\heartsuit,\bm\tau^\heartsuit,\textcat{P}^\heartsuit)$, 
        \item \makeatletter\def\@currentlabel{b}\makeatother\label{representability2} $\mathcal{F}$ has an obstruction theory relative to $\mathcal{A}$,
        \item \makeatletter\def\@currentlabel{c}\makeatother\label{representability3} $\mathcal{F}$ is nilcomplete with respect to $\mathcal{A}$. 
    \end{enumerate}
\end{enumerate}
\end{thm}

\begin{proof}We first suppose that $\mathcal{F}$ is an $n$-geometric stack. Then, (\ref{representability1}) is satisfied by Proposition \ref{t0preserves}. Condition (\ref{representability2}) follows from Theorem \ref{obstructiongeometric} and Assumptions (\ref{repcontextrepresentables}) and (\ref{repcontextobstruction}). Finally, Condition (\ref{representability3}) follows from Proposition \ref{postnikovnilcomplete} along with Proposition \ref{sqzlift} and Assumptions (\ref{repcontextrepresentables}), (\ref{repcontextP}) and (\ref{repcontextobstruction}). 

To prove the converse, suppose that $\mathcal{F}$ is a stack in $\textcat{Stk}(\mathcal{A},\bm\tau|_\mathcal{A})$ satisfying conditions (\ref{representability1}) to (\ref{representability3}). The proof goes by induction on $n$. Suppose that $n=-1$. We prove the following lemma.

\begin{lem}\phantomsection\label{Usequence}
Suppose that $U_0\in\mathcal{A}^{\heartsuit}$ and that we have a morphism $U_0\rightarrow{t_0(\mathcal{F})}$ in $n\textcat{-P}^\heartsuit$ such that, for any base point $x:X\rightarrow{U_0}$, we have that $\mathbb{L}_{U_0/t_0(\mathcal{F}),x}\simeq 0$. Then, there exists a representable stack $U=\textnormal{Spec}(A)\in\mathcal{A}$ in $\textcat{Stk}(\mathcal{A},\bm\tau|_\mathcal{A})$ and a morphism $u:U\rightarrow{\mathcal{F}}$ such that $t_0(U)\simeq U_0$ and $\mathbb{L}_{U/\mathcal{F},u}\simeq 0$. 
\end{lem}

\begin{proof}
    We will construct by induction a sequence of representable stacks
    \begin{equation*}
        U_0\rightarrow{U_1}\rightarrow{\dots}\rightarrow{U_k}\rightarrow{U_{k+1}}\rightarrow{\dots}\rightarrow{\mathcal{F}}
    \end{equation*}satisfying the following properties:
    \begin{itemize}
        \item We have $U_k=\textnormal{Spec}(A_k)$ with $A_k\in\mathcal{A}^{op}$ $k$-truncated, 
        \item The corresponding morphism $A_{k+1}\rightarrow{A_k}$ induces an equivalence on the $k$-th truncation,
        \item The morphisms $u_k:U_k\rightarrow{\mathcal{F}}$ are such that $\pi_i(\mathbb{L}_{U_k/\mathcal{F},u_k})=0$ for all $i\leq k+1$.
    \end{itemize}

    Let $A=\varprojlim_k A_k$ and let $U=\textnormal{Spec}(A)$. By construction $A_k=\tau_{\leq k}(A)$, and we note that $A\in\mathcal{A}^{op}$ by Assumption (\ref{postnikovcomp4}) on our Postnikov compatible derived geometry context. Therefore, $U$ is a representable stack on $\mathcal{A}$. The points $u_k$ define a well defined point in $\pi_0(\varprojlim_k \mathcal{F}(A_k))$ which we know, by condition (\ref{representability3}), defines a point in $\pi_0(\mathcal{F}(A))$. Therefore, there is a well-defined morphism of stacks $u:U\rightarrow{\mathcal{F}}$ which, by Lemma \ref{relativeobstructionresult}, has a global cotangent complex. If $M\in\textcat{M}_{A,1}$, then
    \begin{equation*}
        \begin{aligned}
            \textnormal{Map}_{\textcat{Mod}_A}(\mathbb{L}_{U/\mathcal{F},u},M)&\simeq \textnormal{Map}_{{}^{U/}\textcat{Stk}(\mathcal{A},\bm{\tau}|_{\mathcal{A}})_{/\mathcal{F}}}(U[M],\mathcal{F})\\
            &\simeq \textnormal{Map}_{{}^{U/}\textcat{Stk}(\mathcal{A},\bm{\tau}|_{\mathcal{A}})_{/\mathcal{F}}}(\varinjlim_kU_k[M_{\leq k}],\mathcal{F})\\
            &\simeq \varprojlim_k\textnormal{Map}_{{}^{U_k/}\textcat{Stk}(\mathcal{A},\bm{\tau}|_{\mathcal{A}})_{/\mathcal{F}}}(U_k[M_{\leq k}],\mathcal{F})\\
            &\simeq \varprojlim_k\textnormal{Map}_{\textcat{Mod}_{A_k}}(\mathbb{L}_{U_k/\mathcal{F},u_k},M_{\leq k})\\
        \end{aligned}
    \end{equation*}By our last assumption on the $U_k$, $\textnormal{Map}_{\textcat{Mod}_A}(\mathbb{L}_{U/\mathcal{F},u},M)\simeq 0$. Hence, by Lemma \ref{Dfullyfaithfulconnective}, $\mathbb{L}_{U/\mathcal{F},u}\simeq 0$.

    It now remains to construct such a sequence. Indeed, when $k=0$, since $U_0$ is a representable stack on $\mathcal{A}^\heartsuit$, $U_0=\textnormal{Spec}(A_0)$ for some $A_0\in\mathcal{A}^{\heartsuit,op}$. We let $u_0:U_0\rightarrow{\mathcal{F}}$ denote the morphism induced by $u_0:U_0\rightarrow{t_0(\mathcal{F})}$ and the morphism $j:t_0(\mathcal{F})\rightarrow{\mathcal{F}}$.

    If we consider the fibre sequence of $A_0$-modules from Corollary \ref{cofibrecotangentstack},
    \begin{equation*}
\mathbb{L}_{t_0(\mathcal{F})/\mathcal{F},u_0}\rightarrow{\mathbb{L}_{U_0/\mathcal{F},u_0}\rightarrow{\mathbb{L}_{U_0/t_0(\mathcal{F}),u_0}}}
    \end{equation*}
then we obtain the following long exact sequence
\begin{equation*}
\dots\rightarrow\pi_1(\mathbb{L}_{U_0/t_0(\mathcal{F}),u_0})\rightarrow\pi_0(\mathbb{L}_{t_0(\mathcal{F})/\mathcal{F},u_0})\rightarrow{\pi_0(\mathbb{L}_{U_0/\mathcal{F},u_0})\rightarrow{\pi_0(\mathbb{L}_{U_0/t_0(\mathcal{F}),u_0}})}\rightarrow{\dots}
\end{equation*}By assumption, $\mathbb{L}_{U_0/t_0(\mathcal{F}),u_0}\simeq 0$. By Lemma \ref{2connectivelemma}, since $t_0(\mathcal{F})$ and $\mathcal{F}$ have obstruction theories, we see that $\pi_i(\mathbb{L}_{t_0(\mathcal{F})/\mathcal{F},u_0})=0$ for all $i\leq 1$. Therefore, $\pi_i(\mathbb{L}_{U_0/\mathcal{F},u_0})=0$ for all $i\leq 1$.  

 Now, suppose that all the $U_i$ for $i\leq k$ have been constructed for some $k\geq 0$. Consider $u_k:U_k\rightarrow{\mathcal{F}}$ and the natural morphism \begin{equation*}d_k:\mathbb{L}_{U_k}\rightarrow{\mathbb{L}_{U_k/\mathcal{F},u_k}}\xrightarrow{\tau_{\leq k+2}}{(\mathbb{L}_{U_k/\mathcal{F},u_k})_{\leq k+2}}\simeq \pi_{k+2}(\mathbb{L}_{U_k/\mathcal{F},u_k})[k+2]
    \end{equation*}We note that this defines an element of 
    \begin{equation*}
        \textcat{Der}_{U_k}(U_k,\pi_{k+2}(\mathbb{L}_{U_k/\mathcal{F},u_k})[k+2])\simeq\textnormal{Map}_{{}^{U_k/}\textcat{Stk}}(\textnormal{Spec}(A_k\oplus\pi_{k+2}(\mathbb{L}_{U_k/\mathcal{F},u_k})[k+2]),U_k)
    \end{equation*}We let $A_{k+1}=A_k\oplus_{d_k} \pi_{k+2}(\mathbb{L}_{U_k/\mathcal{F},u_k})[k+1]$ and $U_{k+1}=\textnormal{Spec}(A_{k+1})$, which, by our assumptions on $\textcat{M}$, lies in $\mathcal{A}$. Since the obstruction $\alpha(u_k):\mathbb{L}_{\mathcal{F},u_k}\rightarrow{\pi_{k+2}(\mathbb{L}_{U_k/\mathcal{F},u_k})[k+2]}$ from Proposition \ref{liftingproposition} is induced by $d_k$, we see that it is nullhomotopic. It follows that the morphism $u_{k}:U_k\rightarrow{\mathcal{F}}$ extends to a morphism $u_{k+1}:U_{k+1}\rightarrow{\mathcal{F}}$
\begin{equation*}
    \begin{tikzcd}
        U_k\arrow{rr} \arrow{rd}{u_k}& & U_{k+1} \arrow{ld}{u_{k+1}}\\
        &\mathcal{F}
    \end{tikzcd}
\end{equation*}in ${}^{U_k/}\textcat{Stk}(\mathcal{A},\bm{\tau}|_{\mathcal{A}})_{/\mathcal{F}}$. We note that $U_{k+1}$ satisfies the first two conditions by construction. It remains to prove the final condition. Consider the fibre sequence in $\textcat{Mod}_{A_{k}}$ associated to the above triangle
\begin{equation*}
\mathbb{L}_{U_{k+1}/\mathcal{F},u_{k+1}}\otimes^\mathbb{L}_{A_{k+1}}A_k\rightarrow{\mathbb{L}_{U_k/\mathcal{F},u_k}}\rightarrow{\mathbb{L}_{U_k/{U_{k+1}}}}
    \end{equation*}and consider the corresponding long exact sequence
\begin{equation*}
    .\,.\rightarrow\pi_{i+1}(\mathbb{L}_{U_k/{U_{k+1}}})\rightarrow{\pi_i(\mathbb{L}_{U_{k+1}/\mathcal{F},u_{k+1}}\otimes^\mathbb{L}_{A_{k+1}}A_k)}\rightarrow{\pi_i(\mathbb{L}_{U_k/\mathcal{F},u_k})}\rightarrow{\pi_i(\mathbb{L}_{U_k/{U_{k+1}}})}\rightarrow{.\,.}
\end{equation*}
We note that $\pi_i(\mathbb{L}_{U_k/\mathcal{F},u_k})=0$ for all $i\leq k+1$. Moreover, by  Lemma \ref{postnikovcotangentvanish}, we have that $\pi_i(\mathbb{L}_{U_k/U_{k+1}})=0$ for $i\leq k+1$ and $i=k+3$ and 
\begin{equation*}
    \pi_{k+2}(\mathbb{L}_{U_k/U_{k+1}})\simeq \pi_{k+1}(A_{k+1})\simeq\pi_{k+1}(A_k\oplus_{d_k}\pi_{k+2}(\mathbb{L}_{U_k/\mathcal{F},u_k})[k+1])\simeq\pi_{k+2}(\mathbb{L}_{U_k/\mathcal{F},u_k})
\end{equation*}Therefore, using the long exact sequence, we see that $\pi_{i}(\mathbb{L}_{U_{k+1}/\mathcal{F},u_{k+1}}\otimes^\mathbb{L}_{A_{k+1}}A_k)=0$ for $i\leq k+2$. 

It remains to show that $\pi_i(\mathbb{L}_{U_{k+1}/\mathcal{F},u_{k+1}})=0$ for all $i\leq k+2$. We denote by $K_{k+1}$ the fibre of the morphism $A_{k+1}\rightarrow{A_k}$. We note that this is concentrated in degree $k+1$. By considering the following long exact sequence of homotopy groups
\begin{equation*}
.\,.\rightarrow\pi_i(\mathbb{L}_{U_{k+1}/\mathcal{F},u_{k+1}}\otimes^\mathbb{L}_{A_{k+1}}K_{k+1})\rightarrow\pi_i(\mathbb{L}_{U_{k+1}/\mathcal{F},u_{k+1}})\rightarrow\pi_i(\mathbb{L}_{U_{k+1}/\mathcal{F},u_{k+1}}\otimes^\mathbb{L}_{A_{k+1}}A_k)\rightarrow\dots 
\end{equation*}we see that it suffices to show that $\pi_i(\mathbb{L}_{U_{k+1}/\mathcal{F},u_{k+1}}\otimes^\mathbb{L}_{A_{k+1}}K_{k+1})\simeq 0$ for all $i\leq k+2$. Indeed, using the Tor-spectral sequence
\begin{equation*}
    \textnormal{Tor}^p_{\pi_*(A_{k+1})}(\pi_*(\mathbb{L}_{U_{k+1}/\mathcal{F},u_{k+1}}),\pi_{*}(K_{k+1}))_q\Rightarrow{\pi_{p+q}(\mathbb{L}_{U_{k+1}/\mathcal{F},u_{k+1}}\otimes^\mathbb{L}_{A_{k+1}}K_{k+1})}
\end{equation*}we easily see that the bottom $k$-rows of the second page are zero. Therefore, we can see that $\pi_i(\mathbb{L}_{U_{k+1}/\mathcal{F},u_{k+1}}\otimes^\mathbb{L}_{A_{k+1}}K_{k+1})=0$ for $i\leq k$. Since $k\geq 1$, this implies, in particular, that $\pi_0(\mathbb{L}_{U_{k+1}/\mathcal{F},u_{k+1}})$ and $\pi_1(\mathbb{L}_{U_{k+1}/\mathcal{F},u_{k+1}})$ are zero. We then see that $\pi_i(\mathbb{L}_{U_{k+1}/\mathcal{F},u_{k+1}}\otimes^\mathbb{L}_{A_{k+1}}K_{k+1})=0$ for $i=k+1$ and $k+2$. 
\end{proof}

Now, coming back to the original theorem. Suppose that $n=-1$. Then, we note that $t_0(\mathcal{F})$ is a representable stack. Therefore, by Lemma \ref{Usequence} applied to the identity morphism $t_0(\mathcal{F})\rightarrow{t_0(\mathcal{F})}$, we see that there exists a representable stack $U$ in $\textcat{Stk}(\mathcal{A},\bm\tau|_\mathcal{A})$ and a morphism $u:U\rightarrow{\mathcal{F}}$ with $\mathbb{L}_{U/\mathcal{F},u}\simeq 0$ and $t_0(U)\simeq t_0(\mathcal{F})$.

We will prove by induction on $k\geq 0$ that, for every $A\in \mathcal{A}^{op}$, $U(\tau_{\leq k}A)\simeq \mathcal{F}(\tau_{\leq k}A)$. Then, using Condition (\ref{representability3}), we see that $U(A)\simeq\mathcal{F}(A)$ for every $A\in \mathcal{A}^{op}$, and hence $\mathcal{F}\simeq U$ is representable. When $k=0$, we note that $t_0(U)\simeq t_0(\mathcal{F})$, and hence we are done. Now, suppose that $k>0$ and that $U(\tau_{\leq m} A)\simeq\mathcal{F}(\tau_{\leq m}A)$ for all $m<k$ and all $X=\textnormal{Spec}(A)\in \mathcal{A}$. By Lemma \ref{postnikovsquarezero}, the map $A_{\leq k}\rightarrow{A_{\leq k-1}}$ is isomorphic in $\mathcal{A}^{op}_{/A_{\leq k-1}}$ to the morphism 
\begin{equation*}
    A_{\leq k-1}\oplus_{d_{k}}\pi_k(A)[k]\rightarrow{A_{\leq k-1}}
\end{equation*}for some derivation $d_{k}\in\pi_0(\textcat{Der}(A_{\leq k-1},\pi_k(A)[k+1]))$. Since $\mathbb{L}_{U/\mathcal{F},u}\simeq 0$, then $\mathbb{L}_{U}\simeq \mathbb{L}_{\mathcal{F},u}$. Hence, since $\mathcal{F}$ has a global cotangent complex, we see that
\begin{equation*}
    \textnormal{Map}_{X_{\leq k-1}/\textcat{Stk}}(X_{\leq k-1}[\pi_k(A)[k+1]],U)\simeq \textnormal{Map}_{X_{\leq k-1}/\textcat{Stk}}(X_{\leq k-1}[\pi_k(A)[k+1]],\mathcal{F})
\end{equation*}
Now, since $U$ and $\mathcal{F}$ are infinitesimally cartesian, then, using our inductive hypothesis,
\begin{equation*}\begin{aligned}
    U(A_{\leq k-1}\oplus_{d_{k}}\pi_k(A)[k])&\simeq U(A_{\leq k-1})\times_{U(A_{\leq k-1}\oplus \pi_k(A)[k+1])}U(A_{\leq k-1})\\
    &\simeq \mathcal{F}(A_{\leq k-1})\times_{\mathcal{F}(A_{\leq k-1}\oplus\pi_k(A)[k+1])}\mathcal{F}(A_{\leq k-1})\\
    &\simeq \mathcal{F}(A_{\leq k-1}\oplus_{d_{k}}\pi_k(A)[k])
\end{aligned}\end{equation*}from which it follows that $U(A_{\leq k})\simeq \mathcal{F}(A_{\leq k})$.\\

Now, suppose that $n> -1$ and that the theorem holds for $m< n$. Suppose that $\mathcal{F}$ is a stack in $\textcat{Stk}(\mathcal{A},\bm\tau|_\mathcal{A})$ satisfying conditions (\ref{representability1}) to (\ref{representability3}) of the theorem. We will first show that the diagonal morphism $\mathcal{F}\rightarrow{\mathcal{F}\times\mathcal{F}}$ is $(n-1)$-representable. Suppose that we have a map $X\rightarrow{\mathcal{F}\times\mathcal{F}}$ where $X$ is a $(-1)$-geometric stack. It suffices to show that $\mathcal{F}\times_{\mathcal{F}\times\mathcal{F}}X$ is $(n-1)$-geometric. We note that, since $t_0(\mathcal{F})$ is an $n$-geometric stack, then the diagonal morphism $t_0(\mathcal{F})\rightarrow{t_0(\mathcal{F})\times t_0(\mathcal{F})}$ is $(n-1)$-representable, and hence
\begin{equation*}
t_0(\mathcal{F}\times_{\mathcal{F}\times\mathcal{F}}X)\simeq t_0(\mathcal{F})\times_{t_0(\mathcal{F})\times t_0(\mathcal{F})}t_0(X)
\end{equation*}is an $(n-1)$-geometric stack. Since $\mathcal{F}$ has an obstruction theory, it follows that $\mathcal{F}\times_{\mathcal{F}\times\mathcal{F}}X$ has an obstruction theory by Lemma \ref{relativeobstructionresult}. Moreover, it is easy to see that 
\begin{equation*}
    (\mathcal{F}\times_{\mathcal{F}\times\mathcal{F}}X)(A)=\varprojlim_k(\mathcal{F}\times_{\mathcal{F}\times\mathcal{F}}X)(\tau_{\leq k}A)
\end{equation*}for any $A\in \mathcal{A}^{op}$. Therefore, since $\mathcal{F}\times_X(\mathcal{F}\times\mathcal{F})$ satisfies Conditions (\ref{representability1}) to (\ref{representability3}) of the theorem at level $n-1$, we can conclude by the induction hypothesis.\\

It remains to show that $\mathcal{F}$ has an $n$-atlas. Suppose that $t_0(\mathcal{F})$ has an $n$-atlas $\{V_j\rightarrow{t_0(\mathcal{F})}\}_{j\in J}$ in $\textcat{Stk}(\mathcal{A}^\heartsuit,\bm\tau^\heartsuit)$. We will lift this to an atlas of $\mathcal{F}$. Consider any morphism $V_j\rightarrow{t_0(\mathcal{F})}$, which we note is in $(n-1)\textcat{-P}^\heartsuit$, and let $U_0:=V_j$. Similarly to before, we will construct by induction a sequence of representable stacks
\begin{equation*}
U_0\rightarrow{U_1}\rightarrow{\dots}\rightarrow{U_k}\rightarrow{U_{k+1}}\rightarrow{\dots}\rightarrow{\mathcal{F}}
\end{equation*}in $\textcat{Stk}(\mathcal{A},\bm{\tau}|_{\mathcal{A}})$ satisfying the following properties
\begin{itemize}
    \item We have $U_k=\textnormal{Spec}(A_k)$ with $A_k\in\mathcal{A}^{op}$ $k$-truncated, 
    \item The corresponding morphism $A_{k+1}\rightarrow{A_k}$ induces an equivalence on the $k$-th truncation,
    \item The morphisms $u_k:U_k\rightarrow{\mathcal{F}}$ are such that, for any $M\in\textcat{M}_{A_k,1}$ with $\pi_i(M)=0$ for $i>k+1$, one has $[\mathbb{L}_{U_k/\mathcal{F},u_k},M]:=\pi_0(\textnormal{Map}_{\textcat{Mod}_{A_k}}(\mathbb{L}_{U_k/\mathcal{F},u_k},M))\simeq 0$. 
\end{itemize}

We construct the sequence inductively. Indeed, when $k=0$, we let $u_0:U_0\rightarrow{\mathcal{F}}$ denote the morphism induced by $u_0:U_0\rightarrow{t_0(\mathcal{F})}$ and the morphism $j:t_0(\mathcal{F})\rightarrow{\mathcal{F}}$. Suppose that $M\in\textcat{M}_{A_0}$ is concentrated in degree $1$ and consider the long exact sequence
    \begin{equation*}\dots\rightarrow[\mathbb{L}_{U_0/t_0(\mathcal{F}),u_0},M]\rightarrow{[\mathbb{L}_{U_0/\mathcal{F},u_0},M]}\rightarrow{[\mathbb{L}_{t_0(\mathcal{F})/\mathcal{F},u_0},M]}\rightarrow{0}
    \end{equation*}Since $\mathbb{L}_{t_0(\mathcal{F})/\mathcal{F},u_0}$ is $1$-connective by Lemma \ref{2connectivelemma} and the morphism $U_0\rightarrow{t_0(\mathcal{F})}$ is in $(n-1)\textcat{-P}$, then it follows using Assumption (\ref{repcontextP}), that 
    \begin{equation*}
        \pi_0(\textnormal{Map}_{\textcat{Mod}_{A_0}}(\mathbb{L}_{U_0/\mathcal{F},u_0},M))=0
    \end{equation*}

    Now, suppose that all the $U_i$ for $i\leq k$ have been constructed. We may  construct the next term in the sequence in a similar way to Lemma \ref{Usequence}. Indeed, consider $u_k:U_k\rightarrow{\mathcal{F}}$ and the natural morphism 
    \begin{equation*}
d_k:\mathbb{L}_{U_k}\rightarrow{\mathbb{L}_{U_k/\mathcal{F},u_k}}\xrightarrow{\tau_{\leq k+2}}{(\mathbb{L}_{U_k/\mathcal{F},u_k})_{\leq k+2}}\xrightarrow{\pi_{k+2}}{ \pi_{k+2}(\mathbb{L}_{U_k/\mathcal{F},u_k})}
    \end{equation*}Then, as before, we let $A_{k+1}=A_k\oplus_{d_k}\pi_{k+2}(\mathbb{L}_{U_k/\mathcal{F}})[k+1]$ and $U_{k+1}=\textnormal{Spec}(A_{k+1})$. As before, there is an induced morphism $u_{k+1}:U_{k+1}\rightarrow\mathcal{F}$. The only thing we need to check is the third condition. Suppose we have a $(k+2)$-truncated $M\in\textcat{M}_{A_{k+1},1}$. From the sequence $U_k\rightarrow{U_{k+1}}\rightarrow{\mathcal{F}}$, we obtain a long exact sequence
    \begin{equation*}
\dots\rightarrow{[\mathbb{L}_{U_k/U_{k+1}},M_{\leq k+1}]}\rightarrow{[\mathbb{L}_{U_k/\mathcal{F},u_k},M_{\leq k+1}]}\rightarrow{[\mathbb{L}_{U_{k+1}/\mathcal{F},u_{k+1}}\otimes^\mathbb{L}_{A_{k+1}}A_{k},M_{\leq k+1}]}\rightarrow{0}
    \end{equation*}Since $M_{\leq k+1}$ is a $(k+1)$-truncated module in $\textcat{M}_{A_k,1}$, then, by our assumptions on our sequence, $\pi_0(\textnormal{Map}_{\textcat{Mod}_{A_k}}(\mathbb{L}_{U_k/\mathcal{F},u_k},M_{\leq k+1}))\simeq 0$, and hence
\begin{equation*}\pi_0(\textnormal{Map}_{\textcat{Mod}_{A_{k}}}(\mathbb{L}_{U_{k+1}/\mathcal{F},u_{k+1}}\otimes^\mathbb{L}_{A_{k+1}}A_k,M_{\leq k+1}))\simeq 0
\end{equation*}By Lemma \ref{postnikovcotangentvanish} and by similar reasoning to Lemma \ref{Usequence}, $\pi_{k+2}(\mathbb{L}_{U_{k+1}/\mathcal{F},u_{k+1}}\otimes^\mathbb{L}_{A_{k+1}}A_k)=0$. Therefore, since $\pi_i(M)=0$ for $i>k+2$, 
\begin{equation*}
  \pi_0(\textnormal{Map}_{\textcat{Mod}_{A_{k+1}}}(\mathbb{L}_{U_{k+1}/\mathcal{F},u_{k+1}},M))\simeq \pi_0(\textnormal{Map}_{\textcat{Mod}_{A_{k}}}(\mathbb{L}_{U_{k+1}/\mathcal{F},u_{k+1}}\otimes^\mathbb{L}_{A_{k+1}}A_k,M))\simeq 0
\end{equation*}as required.

Suppose that the sequence in question has been constructed. We let $A=\varprojlim_k A_k$ and let $U=\textnormal{Spec}(A)$. Since the diagonal morphism $\mathcal{F}\rightarrow{\mathcal{F}\times\mathcal{F}}$ is $(n-1)$-representable, we can see that the morphism $u:U\rightarrow{\mathcal{F}}$ is $(n-1)$-representable using Proposition \ref{coolcorollary}.

Suppose that $M\in\textcat{M}_{A,1}$. Then, for each $k\geq 0$, there is a fibre sequence
\begin{equation*}
\mathbb{L}_{U/\mathcal{F},u}\otimes^\mathbb{L}_{A}A_k\rightarrow{\mathbb{L}_{U_k/\mathcal{F},u_k}}\rightarrow{\mathbb{L}_{U_k/U,u_k}}
\end{equation*}and hence we have a long exact sequence
\begin{equation*}
    \dots\rightarrow[\mathbb{L}_{U_k/U,u_k},M_{\leq k+1}]\rightarrow{[\mathbb{L}_{U_k/\mathcal{F},u_k},M_{\leq k+1}]}\rightarrow{[\mathbb{L}_{U/\mathcal{F},u}\otimes^\mathbb{L}_{A}A_k,M_{\leq k+1}]}\rightarrow{0}
\end{equation*}Since $[\mathbb{L}_{U_k/\mathcal{F},u_k},M_{\leq k+1}]\simeq 0$ by construction, then
\begin{equation*} \pi_0(\textnormal{Map}_{\textcat{Mod}_{A}}(\mathbb{L}_{U/\mathcal{F},u},M_{\leq k+1}))\simeq\pi_0(\textnormal{Map}_{\textcat{Mod}_{A_k}}(\mathbb{L}_{U/\mathcal{F},u}\otimes^\mathbb{L}_{A}A_k,M_{\leq k+1})) \simeq 0
\end{equation*}Therefore, since this holds for all $k\geq 0$, then 
\begin{equation*}
    \pi_0(\textnormal{Map}_{\textcat{Mod}_{A}}(\mathbb{L}_{U/\mathcal{F},u},M))=0
\end{equation*}Since the morphism $U_0=V_j\rightarrow{t_0(\mathcal{F})}$ is in $(n-1)\textcat{-P}^\heartsuit$, we see that $U\rightarrow{\mathcal{F}}$ is in $(n-1)\textcat{-P}$ by Assumption (\ref{repcontextP}). 

    Let $U_{(j)}$ denote the $U$ corresponding to taking $U_0=V_j$. Now, since the total morphism $\coprod_j V_j\rightarrow{t_0(\mathcal{F})}$ is an epimorphism of stacks, it follows that the total morphism $\coprod_{j\in J} U_{(j)}\rightarrow{\mathcal{F}}$ is an epimorphism of stacks since $\pi_0^{\bm{\tau}}(\mathcal{F})\simeq\pi_0^{\bm{\tau}}(t_0(\mathcal{F}))$. Hence, we see that we have constructed an $n$-atlas $\{U_{(j)}\rightarrow{\mathcal{F}}\}_{j\in J}$ for $\mathcal{F}$.

\end{proof}

\section{Derived Analytic Geometry Contexts}\phantomsection\label{topologiessection}

Suppose that $R$ is a Banach ring. Recall from Example \ref{indbanderivedalgebraic}, that we obtain a flat derived algebraic context 
\begin{equation*}
    (\textcat{Ch}(\textnormal{Ind}(\textnormal{Ban}_R)),\textcat{Ch}_{\geq 0}(\textnormal{Ind}(\textnormal{Ban}_R)),\textcat{Ch}_{\leq 0}(\textnormal{Ind}(\textnormal{Ban}_R)),\textcat{L}^H(P^0))
\end{equation*}where $P^0$ is the collection of compact projective generators. Moreover, we have an equivalence of categories 
\begin{equation*}
    \textcat{DAlg}^{cn}(\textcat{Ch}(\textnormal{Ind}(\textnormal{Ban}_R)))\simeq \textcat{L}^H(\textnormal{Comm}(\textnormal{sInd}(\textnormal{Ban}_R)))
\end{equation*}As explained in the Introduction, this is an appropriate setting to develop theories of derived analytic and derived smooth geometry. 

In this section, we will endow a general derived algebraic context $(\mathcal{C},\mathcal{C}_{\geq 0},\mathcal{C}_{\leq 0},\mathcal{C}^0)$ with some extra structures such that we obtain representability contexts of the form
\begin{equation*}
    (\mathcal{C},\mathcal{C}_{\geq 0},\mathcal{C}_{\leq 0},\mathcal{C}^0,\bm\tau,\textcat{P},\mathcal{A},\textcat{M},\textcat{S})
\end{equation*}We will then give a specific example of a derived complex analytic geometry context. A context for derived smooth geometry will be presented in future work.
\subsection{Formal Covering Families}

Suppose that $(\mathcal{C},\mathcal{C}_{\geq 0},\mathcal{C}_{\leq 0},\mathcal{C}^0)$ is a flat Postnikov compatible derived algebraic context. 

\begin{defn}
    A family of morphisms $f_i:U_i=\textnormal{Spec}(B_i)\rightarrow{\textnormal{Spec}(A)=X}$ in $\textnormal{Ho}(\textcat{DAff}^{cn}(\mathcal{C}))$ is a \textit{formal covering family} if the family of functors \begin{equation*}
        \{(f_i)_!:\textcat{Mod}_A^{cn}\rightarrow{\textcat{Mod}^{cn}_{B_i}}\}_{i\in I}
    \end{equation*}is conservative, i.e. for every morphism $u\in \textcat{Mod}_A^{cn}$, $u$ is an equivalence if and only if all the $(f_i)_!(u)$ are equivalences. 
\end{defn}
It is easy to see that equivalences define formal covering families, and that compositions and pullbacks of formal covering families also define formal covering families. Moreover, the truncation of any formal covering family is a formal covering family.
\begin{lem}\phantomsection\label{truncationconservative}
 If $\{f_i:U_i=\textnormal{Spec}(B_i)\rightarrow{\textnormal{Spec}(A)=X}\}_{i\in I}$ is a formal covering family, then so is $\{t_0(f_i):t_0(U_i)=\textnormal{Spec}(\pi_0(B_i))\rightarrow{\textnormal{Spec}(\pi_0(A))=t_0(X)}\}_{i\in I}$.
\end{lem}

\begin{proof}
    We just need to show that the family $\{t_0(f_i)_!:\textcat{Mod}^{cn}_{\pi_0(A)}\rightarrow{\textcat{Mod}^{cn}_{\pi_0(B_i)}}\}_{i\in I}$ is conservative. Suppose that we have a morphism $u:M\rightarrow{N}$ in $\textcat{Mod}^{cn}_{\pi_0(A)}$. Then, since
    \begin{equation*}
        M\otimes^\mathbb{L}_{\pi_0(A)}\pi_0(B_i)\simeq \pi_0(M\otimes^\mathbb{L}_A B_i)
    \end{equation*}it is clear that $u$ is an equivalence if and only if $t_0(f_i)_!(u)$ is an equivalence for each $i\in I$. 
\end{proof}

Recall from Section \ref{geometriessection}, that we can extend a class of maps $\textcat{P}$ to be local for a topology $\bm\tau$ by defining the category $\textcat{P}^{\bm\tau}$. 

\begin{lem}\phantomsection\label{formallyetalelocalmaps}
    Suppose that $\textcat{P}$ is a class of maps which are formally \'etale, and $\bm\tau$ is a Grothendieck topology whose covers are formal covering families. Then, $\textcat{P}^{\bm{\tau}}$ consists of formally \'etale morphisms. 
\end{lem}

\begin{proof}Suppose that we have a morphism $Y=\textnormal{Spec}(B)\rightarrow\textnormal{Spec}(A)=X$ in $\textcat{P}^{\bm{\tau}}$. Then, we note that there is a ${\bm{\tau}}$-cover $\{U_i=\textnormal{Spec}(C_i)\rightarrow{\textnormal{Spec}(B)=Y}\}_{i\in I}$ such that each induced morphism $A\rightarrow{C_i}$ is in $\textcat{P}$. There is a fibre-cofibre sequence
\begin{equation*}
    \mathbb{L}_{B/A}\otimes^\mathbb{L}_BC_i\rightarrow{\mathbb{L}_{C_i/A}}\rightarrow{\mathbb{L}_{C_i/B}}
\end{equation*}in $\textcat{Mod}_{C_i}$. Therefore, since the morphisms $A\rightarrow{C_i}$ and $B\rightarrow{C_i}$ are formally \'etale, then they are formally unramified and we see that, for each $i\in I$, $\mathbb{L}_{B/A}\otimes_B^\mathbb{L}C_i\simeq 0$. Since our covering family is formal, it follows that $\mathbb{L}_{B/A}\simeq 0$. Therefore, $Y\rightarrow{X}$ is formally \'etale. 
    
\end{proof}

\begin{lem}\phantomsection\label{localisderivedstrong}
    Suppose that $\textcat{P}$ is a class of maps which are derived strong, and that $\bm\tau$ is a Grothendieck topology whose covers are formal covering families and consist of derived strong morphisms. Then, $\textcat{P}^{\bm\tau}$ consists of derived strong morphisms. 
\end{lem}

\begin{proof}
    Indeed, suppose that we have a morphism $f:Y=\textnormal{Spec}(B)\rightarrow{\textnormal{Spec}(A)=X}$ in $\textcat{P}^{\bm\tau}$. Then, there exists a cover $\{U_i=\textnormal{Spec}(C_i)\rightarrow{Y}\}_{i\in I}$ in $\bm\tau$ such that the induced morphism $A\rightarrow{C_i}$ is in $\textcat{P}$. We can then see, using Lemma \ref{derivedstrongequiv}, that for each $i\in I$,
    \begin{align*}
       (\pi_*(A)\otimes^\mathbb{L}_{\pi_0(A)}\pi_0(B))\otimes^\mathbb{L}_{B}C_i&\simeq \pi_*(C_i)\simeq \pi_*(B)\otimes^\mathbb{L}_{\pi_0(B)}\pi_0(C_i)\simeq \pi_*(B)\otimes^\mathbb{L}_{B}C_i
\end{align*}and therefore, since $\{U_i\rightarrow{Y}\}_{i\in I}$ is a formal covering family, we see that the morphism $A\rightarrow{B}$ is derived strong. 
\end{proof}

\begin{lem}\phantomsection\label{formalcoveringsqz}
    Suppose that $X=\textnormal{Spec}(A)\in\textcat{DAff}^{cn}(\mathcal{C})$, $M\in\textcat{Mod}_A^{\geq 1}$, and that $d\in\pi_0(\textcat{Der}(A,M))$ is a derivation. Suppose that $\{V_i=\textnormal{Spec}(A_i)\rightarrow{X}\}_{i\in I}$ is a formal covering family with each morphism $V_i\rightarrow{X}$ formally smooth. Then, $\{W_i\rightarrow{X_d[\Omega M]}\}_{i\in I}$ is a formal covering family where $W_i$ is defined to be $\textnormal{Spec}(A_i\oplus_{d_i'} \Omega M_i')$ with $d_i'$ the derivation induced by Lemma \ref{liftingderivations} and $M_i'=M\otimes_A^\mathbb{L} A_i$.
\end{lem}
\begin{proof}
Indeed, suppose that $W_i=\textnormal{Spec}(B_i)$ with $B_i=A_i\oplus_{d_i'}\Omega M_i'$. Consider the family $\{g_i:W_i\rightarrow{X_d[\Omega M]}\}_{i\in I}$. Suppose that we have a morphism $u:N\rightarrow{N'}$ in $\textcat{Mod}^{cn}_{A\oplus_d\Omega M}$. It suffices to show that if, for each $i\in I$, we have an equivalence $(g_i)_!(u):N\otimes^\mathbb{L}_{A\oplus_d\Omega M} B_i\rightarrow{N'\otimes^\mathbb{L}_{A\oplus_d\Omega M} B_i}$, then $u$ is also an equivalence. We note that, by Lemma \ref{Ecoverequiv}, there is an equivalence \begin{equation*}
    N\otimes^\mathbb{L}_A A_i
    \simeq N\otimes^\mathbb{L}_{A\oplus_d\Omega M} B_i
    \simeq N'\otimes^\mathbb{L}_{A\oplus_d\Omega M} B_i
    \simeq N'\otimes^\mathbb{L}_A A_i
\end{equation*}for all $i\in I$. Hence, $u:N\rightarrow N'$ is an equivalence because $\{V_i\rightarrow{X}\}_{i\in I}$ is a conservative subfamily. 

\end{proof}

\subsection{Derived Algebraic Geometry}\phantomsection\label{derivedalgebraicgeosection}

We recall from Example \ref{simplicialcommringexample} that, for $k$ a commutative ring, there is an equivalence of $(\infty,1)$-categories
\begin{equation*}
    \textcat{DAlg}^{cn}(\textcat{Mod}_k)\simeq \textcat{L}^H(\textnormal{Comm}(\textnormal{sMod}_k))
\end{equation*}

\begin{defn}A morphism $f:A\rightarrow{B}$ in $\textcat{DAlg}^{cn}(\textcat{Mod}_k)$ is
\begin{itemize}
    \item \textit{finitely presented} if, for any filtered diagram of objects $\{A\rightarrow{C_i}\}_{i\in I}$, there is an equivalence of $\infty$-groupoids
    \begin{equation*}
        \varinjlim_{i\in I}\textnormal{Map}_{\textcat{DAlg}_A^{cn}(\textcat{Mod}_k)}(B,C_i)\rightarrow{\textnormal{Map}_{\textcat{DAlg}_A^{cn}(\textcat{Mod}_k)}(B,\varinjlim_{i\in I}C_i)}
    \end{equation*}
    \item \textit{\'etale} (resp. \textit{smooth}) if it is formally \'etale (resp. formally smooth) and finitely presented. \end{itemize} 
\end{defn}

\begin{defn}
    The \textit{\'etale topology} on $\textnormal{Ho}(\textcat{DAff}^{cn}(\textcat{Mod}_k))$ has covers of the form $\{U_j\rightarrow{X}\}_{j\in J}$ with $U_j=\textnormal{Spec}(B_j)$ and $X=\textnormal{Spec}(A)$, satisfying the following properties. 
    \begin{enumerate}
        \item Each morphism $A\rightarrow{B_j}$ is \'etale, 
        \item There is a finite subset $K\subseteq J$ such that the family $\{U_k\rightarrow{X}\}_{k\in K}$ is a formal covering family. 
    \end{enumerate}We denote this topology by $\bm{\acute{e}t}$. 
\end{defn}

Consider the category $\textcat{DAff}^{cn}(\textcat{Mod}_k)$ endowed with the \'etale topology, $\bm{\acute{e}t}$, and the class of smooth maps, $\textcat{Sm}$. Let $\textcat{M}$ denote the full categories of modules $\textcat{Mod}_A$ for each $A\in\textcat{DAlg}^{cn}(\textcat{Mod}_k)$. Let $\textcat{E}$ denote the class of \'etale morphisms. 
\begin{prop}
    The tuple \begin{equation*}
        (\textcat{Mod}_k,\textcat{Mod}_{k,\geq 0},\textcat{Mod}_{k,\leq 0},\textnormal{Mod}_k^{fgf},\bm{\acute{e}t},\textcat{Sm},\textcat{DAff}^{cn}(\textcat{Mod}_k),\textcat{M},\textcat{E})
    \end{equation*}is a representability context and is moreover closed under $\bm{\acute{e}t}$-descent. 
\end{prop}

\begin{proof}
    We see that $(\textcat{Mod}_k,\textcat{Mod}_{k,\geq 0},\textcat{Mod}_{k,\leq 0},\textnormal{Mod}_k^{fgf},\bm{\acute{e}t},\textcat{Sm},\textcat{DAff}^{cn}(\textcat{Mod}_k),\textcat{M})$ is a flat Postnikov compatible derived geometry context using results from \cite{toen_homotopical_2008}. We note that Conditions (\ref{repcontexttuple}) and (\ref{repcontextcontinuous}) for being a representability context are clearly satisfied. Furthermore, $\textcat{DAff}^{cn}(\textcat{Mod}_k)$ satisfies the descent condition in Definition \ref{descentcondition} for hypercovers by \cite[Lemma 2.2.2.13]{toen_homotopical_2008}, and hence can easily be shown to be closed under $\bm{\acute{e}t}$-descent by Proposition \ref{taudescentundercondn}. Condition (\ref{repcontextrepresentables}) follows by \cite[Lemma 1.3.2.3]{toen_homotopical_2008}. Conditions (\ref{repcontextstrong}) and (\ref{repcontextP}) follow from \cite[Theorem 2.2.2.6]{toen_homotopical_2008} and \cite[Corollary 2.2.5.3]{toen_homotopical_2008} respectively. Finally, Condition (\ref{repcontextobstruction}) follows from \cite[Proposition 2.2.3.2]{toen_homotopical_2008}. 
\end{proof}

This is the correct representability context to work in to do derived algebraic geometry in the sense of To\"en and Vezzosi \cite{toen_homotopical_2008}. In \cite[Section 2.2.6]{toen_homotopical_2008}, they apply their version of the representability theorem to prove conditions under which certain mapping stacks are $n$-geometric.

\subsection{Derived Strongly Formally Perfect Morphisms}

Suppose that $(\mathcal{C},\mathcal{C}_{\geq 0},\mathcal{C}_{\leq 0},\mathcal{C}^0)$ is a flat Postnikov compatible derived algebraic context. Recall from Proposition \ref{formallyperfectformallysmooth}, that formally perfect morphisms are formally smooth. In the derived algebraic geometry context from the previous section, a finitely presented morphism is formally perfect if and only if it is formally smooth. In the rest of this section, we will work with formally perfect morphisms rather than formally smooth morphisms because we obtain stronger results. 

\begin{defn}A morphism $f:A\rightarrow{B}$ in $\textcat{DAlg}^{cn}(\mathcal{C})$ is
\begin{enumerate}
    \item \textit{flat} if, whenever $M$ is an $A$-module in $\textcat{DAlg}^{\heartsuit}(\mathcal{C})$, then $M\otimes_A^\mathbb{L}B$ is in $\textcat{DAlg}^{\heartsuit}(\mathcal{C})$, 
    \item \textit{derived strongly flat} if it is derived strong and $\pi_0(f)$ is flat, 
    \item \textit{flat formally perfect} if it is flat and formally perfect. 
\end{enumerate}
\end{defn}

\begin{lem}\phantomsection\label{derivedstronglyflatflat}
    \cite[c.f. Corollary 2.3.86]{ben-bassat_perspective_2024} Given a morphism $f:A\rightarrow{B}$ in $\textcat{DAlg}^{cn}(\mathcal{C})$, the following are equivalent,
\begin{enumerate}
    \item $f$ is a flat morphism, 
    \item $f$ is a derived strongly flat morphism, 
    \item $B\otimes^\mathbb{L}_A -:\textcat{Mod}_A^{cn}\rightarrow{\textcat{Mod}_B^{cn}}$ commutes with finite limits.
\end{enumerate}
\end{lem}

In the definition of a representability context, we specified that our class of maps should consist of maps which are derived strong. 

\begin{defn}
    Suppose that we have a morphism $f:A\rightarrow{B}$ in $\textcat{DAlg}^{cn}(\mathcal{C})$. Then, $f$ is \textit{derived strongly (flat) formally perfect} if it is derived strong and $\pi_0(f)$ is (flat) formally perfect.
\end{defn}

\begin{cor}\phantomsection\label{derivedstronglyperfectcorollary}Suppose that $f:A\rightarrow{B}$ is a morphism in $\textcat{DAlg}^{cn}(\mathcal{C})$. Then, 
\begin{enumerate}
    \item\label{strongitem1} If $f$ is derived strongly (flat) formally perfect, then it is (flat) formally perfect, 
    \item\label{strongitem2} If $f$ is formally perfect and $\pi_0(f)$ is formally perfect with $\pi_n(A)$ transverse to $\pi_0(B)$ as $\pi_0(A)$-modules for all $n$, then $f$ is derived strong.
\end{enumerate} 
\end{cor}

\begin{proof}
For (\ref{strongitem1}), suppose that $f$ is derived strongly formally perfect. By Corollary \ref{cotangentderivedstronglyflat}, there is an equivalence 
\begin{equation*}
\mathbb{L}_{B/A}\otimes^\mathbb{L}_B\pi_0(B)\simeq \mathbb{L}_{\pi_0(A)/\pi_0(B)}
\end{equation*}
Since $\mathbb{L}_{\pi_0(B)/\pi_0(A)}$ is a perfect $\pi_0(B)$-module, it is a retract of some finite colimit of objects of the form $\coprod_{E_i} \pi_0(B)$ where $E_i$ is a finite set, i.e. there exist maps $r:\varinjlim_i \coprod_{E_i}\pi_0(B)\rightarrow{\mathbb{L}_{\pi_0(B)/\pi_0(A)}}$ and $i:\mathbb{L}_{\pi_0(B)/\pi_0(A)}\rightarrow{\varinjlim_i \coprod_{E_i}\pi_0(B)}$ such that $r\circ i$ is the identity. We note that, by Lemma \ref{projectivevanishinghomotopy},
    \begin{equation*}
        \pi_0(\textnormal{Map}_{\textcat{Mod}_B}(\varinjlim_i \coprod_{E_i}B,\varinjlim_i \coprod_{E_i}B))\simeq \textnormal{Hom}_{\textcat{Mod}_{\pi_0(B)}}(\varinjlim_i \coprod_{E_i}\pi_0(B),\varinjlim_i \coprod_{E_i}\pi_0(B))    \end{equation*}and hence there is a map $i':\varinjlim_i \coprod_{E_i}B\rightarrow{\varinjlim_i \coprod_{E_i}B}$ such that $i'\circ i'=i'$ and $\pi_0(i')=i\circ r$. Similarly, there is a map $r':\varinjlim_i \coprod_{E_i}B\rightarrow{\mathbb{L}_{B/A}}$ such that $\pi_0(r')=r$. Since $\textcat{Mod}_B$ is idempotent complete by \cite[Corollary 5.4.3.6]{lurie_higher_2009}, $i'$ induces a split fibre sequence of $B$-modules
\begin{equation*}K\rightarrow{\varinjlim_i \coprod_{E_i}B}\rightarrow{C}
    \end{equation*}where $K$ is the fibre of $i'$ and $C$ is the retract of $i'$. We note that $r'$ induces a map $r'':C\rightarrow{\mathbb{L}_{B/A}}$ which is an equivalence on $\pi_0$. If we can show that $C\simeq \mathbb{L}_{B/A}$, then since $C$ is a retract of $\varinjlim_i \coprod_{E_i}B$, it follows that $\mathbb{L}_{B/A}$ is perfect, as desired.
    
    We note that $C$ is derived strong as an $A$-module since it is a retract of a free derived strong $A$-module. Therefore, by Lemma \ref{derivedstrongequiv}, we have an equivalence $C\otimes^\mathbb{L}_ B\pi_0(B)\simeq C\otimes^\mathbb{L}_{A}\pi_0(A)\simeq \pi_0(C)$. If we let $K'$ be the fibre of the morphism $C\rightarrow{\mathbb{L}_{B/A}}$, then we note that $K'\otimes^\mathbb{L}_B\pi_0(B)\simeq 0$. Therefore, by Lemma \ref{torzerolemma}, $K'\simeq 0$, and hence $\mathbb{L}_{B/A}\simeq C$ is perfect. The result about flatness follows from Lemma \ref{derivedstronglyflatflat}.

    Now, for (\ref{strongitem2}), suppose that $f$ is a formally perfect morphism and that $\pi_0(f)$ is formally perfect with $\pi_n(A)$ transverse to $\pi_0(B)$ as $\pi_0(A)$-modules for all $n$. Then, we note that $f$ and $\pi_0(f)$ are formally smooth by Lemma \ref{formallyperfectformallysmooth}, and hence we can conclude by \cite[Proposition 2.6.160]{ben-bassat_perspective_2024}. 
\end{proof}

\subsection{Quasicoherent Sheaves}

Suppose that we have an $(\infty,1)$-site $(\textcat{DAff}^{cn}(\mathcal{C}),\bm\tau)$. For $X=\textnormal{Spec}(A)\in\textcat{DAff}^{cn}(\mathcal{C})$, we define $\textcat{QCoh}(X):=\textcat{Mod}_A$ and for any morphism $f:Y\rightarrow{X}$, $\textcat{QCoh}(f)$ is defined to be the colimit-preserving functor $B\otimes_A^\mathbb{L}-:\textcat{Mod}_A\rightarrow{\textcat{Mod}_B}$. 
This defines a $\textcat{Cat}$-valued presheaf
\begin{equation*}
    \textcat{QCoh}:\textcat{DAff}^{cn}(\mathcal{C})^{op}\rightarrow{\textcat{Pr}^{\mathbb{L},\otimes}}
\end{equation*}where $\textcat{Pr}^{\mathbb{L},\otimes}$ is the $(\infty,1)$-category whose objects are locally presentable monoidal $(\infty,1)$-categories and morphisms are left adjoint functors. If we restrict to perfect modules, we obtain a $\textcat{Pr}^{\mathbb{L},\otimes}$-valued presheaf which we will denote by $\textcat{Perf}$. If we restrict to strongly dualisable modules, as in Appendix \ref{perfectappendix}, we obtain a presheaf which we will denote by $\textcat{Dls}$. We note that we can define similar notions of (\v{C}ech-)descent and hyperdescent for $\bm\tau$-covers for $\textcat{Pr}^{\mathbb{L},\otimes}$-valued presheaves as in Section \ref{descentsubsection}. See also \cite[Section 7.1.3]{ben-bassat_perspective_2024}. 

\begin{lem}\phantomsection\label{perfectqcohdescent}
    \cite[c.f. Lemma 7.2.35]{ben-bassat_perspective_2024} Suppose that $\mathcal{N}\subseteq\textcat{QCoh}$ is such that 
    \begin{enumerate}
        \item $\mathcal{N}$ satisfies descent for $\bm\tau$-covers, 
        \item $\mathcal{N}(\textnormal{Spec}(A))$ contains $A$, 
        \item $\mathcal{N}(\textnormal{Spec}(A))\subseteq\textcat{Mod}_A$ is closed under the tensor product and internal hom. 
    \end{enumerate}Then, $\textcat{Dls}$ is a local sub-presheaf of $\mathcal{N}$. In particular it satisfies (\v{C}ech-)descent for $\bm\tau$-covers, and also satisfies hyperdescent for $\bm\tau$-covers whenever $\mathcal{N}$ does.
\end{lem}

 \begin{lem}\phantomsection\label{cechcocartesiandescent} Suppose that
\begin{enumerate}
    \item $\mathcal{A}$ is closed under geometric realisations,
    \item Any $\bm\tau$-cover has a finite subcover,
    \item For any finite collection $\{U_i\}_{i\in I}$ of $U_i\in\mathcal{A}$, the map $\coprod_{i\in I} h(U_i)\rightarrow{h(\coprod_{i\in I} U_i)}$ is an equivalence in $\textcat{Stk}(\mathcal{A},\bm\tau|_\mathcal{A})$,
    \item $\textcat{QCoh}$ satisfies (\v{C}ech)-descent for $\bm\tau$-covers.
\end{enumerate}Then, 
\begin{enumerate}
    \item\label{descent1} Any \v{C}ech nerve of a $\bm\tau$-cover of objects in $\mathcal{A}$ satisfies the co-cartesian descent condition of Definition \ref{descentcondition}, 
    \item\label{descent2} $\mathcal{A}$ is closed under $\bm\tau$-descent relative to $\mathcal{A}$ in the sense of Definition \ref{taudescent}.
\end{enumerate}
\end{lem}

\begin{proof}We note that (\ref{descent2}) follows from (\ref{descent1}) and Proposition \ref{taudescentundercondn}.

To prove (\ref{descent1}), suppose that $\{U_i=\textnormal{Spec}(A_i)\rightarrow{\textnormal{Spec}(A)=X}\}_{i\in I}$ is a $\bm\tau$-cover of objects in $\mathcal{A}$. By the second assumption, we can assume it is finite. Consider the \v{C}ech nerve $\mathcal{U}_*$ of the epimorphism of stacks $\coprod_{i\in I}h(U_i)\rightarrow{h(X)}$. By our assumptions, $\mathcal{U}_m=h((U)_m)$ where $(U)_*$ is the \v{C}ech nerve of the morphism $\coprod_{i\in I} U_i\rightarrow{X}$ in $\mathcal{A}$. Further, $(U)_m=\textnormal{Spec}(B_m)$ for some $B_m\in\mathcal{A}^{op}$. Let $B_*$ be the associated simplicial object.  

By an extension of \cite[Theorem B.0.7]{toen_homotopical_2008}, we can identify $\varprojlim_{[n]\in\Delta}\textcat{Mod}_{B_n}$ with the category $\textcat{csMod}_{B_*,cart}$ of cosimplicial cartesian $B_*$-modules, i.e. cosimplicial $B_*$-modules $M_*$ such that, for each $[n]\rightarrow{[m]}$ in $\Delta$, $M_n\otimes^\mathbb{L}_{B_n}B_m\rightarrow{M_m}$ is an equivalence.  By descent for quasi-coherent sheaves with respect to $\bm\tau$-covers, there is therefore an equivalence of categories $B_*\otimes_A^\mathbb{L}-:\textcat{Mod}_A\rightarrow{\textcat{csMod}_{B_*,cart}}$. Hence, if we have a co-cartesian morphism $Y_*=\textnormal{Spec}(C_*)\rightarrow{(U)_*}$ in the sense of Definition \ref{descentcondition}, then $C_*$ is a cosimplicial cartesian $B_*$-module. Hence, $C_*\simeq B_*\otimes_A^\mathbb{L}C$, where $C:=\varprojlim_{[n]\in\Delta}C_n$. Therefore, $U_*\times_X Y\simeq Y_*$ as required. 
\end{proof}

\begin{remark}
    In practice, many of our representability contexts will satisfy that $\mathcal{A}$ is closed under $\bm\tau$-descent relative to $\textcat{DAff}^{cn}(\mathcal{C})$, by a similar proof to the above, but not necessarily with respect to $\mathcal{A}$. For example, the category $\textcat{DSt}^{op}$ of derived Steins is not closed under geometric realisations and so the conditions of the above proof will not hold.
\end{remark}

\subsection{Infinitesimal Criteria for Formally Perfect Morphisms}

Suppose that $(\mathcal{C},\mathcal{C}_{\geq 0},\mathcal{C}_{\leq 0},\mathcal{C}^0)$ is a flat Postnikov compatible derived algebraic context. We will denote by $\textcat{fP}$ the class of formally perfect morphisms in $\textcat{DAff}^{cn}(\mathcal{C})$. Suppose that $\textcat{P}\subseteq \textcat{fP}$ is a subclass of formally perfect morphisms satisfying the following conditions:
\begin{enumerate}
    \item The class $\textcat{P}$ is stable under equivalences, compositions, and pullbacks, 
    \item If a morphism $f:Y=\textnormal{Spec}(B)\rightarrow{\textnormal{Spec}(A)=X}$ in $\mathcal{A}$ is formally perfect and $t_0(f)\in \textcat{P}^\heartsuit$, then $f$ is in $\textcat{P}$.
\end{enumerate}

Suppose that there exists a topology $\bm\tau$ such that $(\mathcal{C},\mathcal{C}_{\geq 0},\mathcal{C}_{\leq 0},\mathcal{C}^0,\bm\tau,\textcat{P}^{\bm\tau},\mathcal{A},\textcat{M})$ is a flat Postnikov compatible derived geometry context. Suppose that $\bm\tau$ and $\textcat{P}^{\bm\tau}$ satisfy the obstruction conditions for an appropriate class $\textcat{S}$ of morphisms and that $\iota|_{\mathcal{A}^\heartsuit}:(\mathcal{A}^\heartsuit,\bm\tau^\heartsuit)\rightarrow{(\mathcal{A},\bm\tau|_\mathcal{A})}$ is a continuous functor of $(\infty,1)$-sites.

\begin{lem}\phantomsection\label{perfectconnduallemma}Suppose that $\textcat{Dls}|_{\mathcal{A}^{op}}$ satisfies descent for $\bm\tau$-covers. Suppose that $f:\mathcal{F}\rightarrow{Y}$ is an $n\textcat{-P}^{\bm\tau}|_\mathcal{A}$-morphism from an $n$-geometric stack in $\textcat{Stk}_n(\mathcal{A},\bm\tau|_\mathcal{A},\textcat{P}^{\bm\tau}|_\mathcal{A})$ to a representable stack $Y=\textnormal{Spec}(B)\in\mathcal{A}$. Then, for any $x:X=\textnormal{Spec}(A)\rightarrow{\mathcal{F}}$ with $A\in\mathcal{A}^{op}$, we have that $\mathbb{L}_{\mathcal{F}/Y,x}$ is strongly dualisable and its dual $\mathbb{L}_{\mathcal{F}/Y,x}^\vee$ is $0$-connective.
\end{lem}

\begin{proof}
    We proceed by induction on $n$. Indeed, when $n=-1$ and $\mathcal{F}$ is representable, say by $Z=\textnormal{Spec}(D)\in\mathcal{A}$, then $f$ is in $\textcat{P}^{\bm\tau}$. Hence, there exists a $\bm{\tau}$-cover of $Z$, $\{U_i=\textnormal{Spec}(C_i)\rightarrow{Z}\}_{i\in I}$, such that the induced morphism $U_i\rightarrow{Y}$ is in $\textcat{P}$. By considering the fibre sequence
    \begin{equation*}
        \mathbb{L}_{Z/Y}\otimes^\mathbb{L}_DC_i\rightarrow{\mathbb{L}_{U_i/Y}}\rightarrow{\mathbb{L}_{U_i/Z}}
    \end{equation*}we see that, since $\mathbb{L}_{U_i/Z}$ is perfect and $\mathbb{L}_{U_i/Y}$ is perfect, then $\mathbb{L}_{Z/Y}\otimes^\mathbb{L}_DC_i$ is strongly dualisable by Lemma \ref{fibreperfect}. Since $\textcat{Dls}|_{\mathcal{A}^{op}}$ satisfies descent for $\bm\tau$-covers, we see that $\mathbb{L}_{Z/Y}$ is a strongly dualisable $0$-connective $D$-module. Moreover, $\mathbb{L}_{Z/Y}^\vee$ is $0$-connective. 

    Now, when $n>-1$, since $\textcat{Dls}|_{\mathcal{A}^{op}}$ satisfies descent for $\bm\tau$-covers and $\mathcal{F}$ has a global cotangent complex by Theorem \ref{obstructiongeometric}, we see that the properties we need to show are local for the $\bm\tau$-topology. We can therefore assume, since $\mathcal{F}$ is $n$-geometric, that the point $x$ lifts to a point of an $n$-atlas for $\mathcal{F}$ by Proposition \ref{localepimorphismprop}. So, suppose that there is a representable stack $U\in\textcat{Stk}(\mathcal{A},\bm\tau|_\mathcal{A})$, and an $(n-1)\textcat{-P}^{\bm\tau}|_{\mathcal{A}}$-morphism $U\rightarrow{\mathcal{F}}$ such that $x\in\pi_0(\mathcal{F}(A))$ is the image of a point $u\in\pi_0(U(A))$. We consider the induced fibre sequence of $A$-modules
    \begin{equation*}
\mathbb{L}_{\mathcal{F}/Y,x}\rightarrow{\mathbb{L}_{U/Y,u}}\rightarrow{\mathbb{L}_{U/\mathcal{F},u}}
    \end{equation*}Up to replacing $U$ by an appropriate cover, we can assume that the induced morphism $U\rightarrow{Y}$ is in $\textcat{P}|_\mathcal{A}$. Hence, $\mathbb{L}_{U/Y,u}$ is perfect and its dual is $0$-connective. Since the morphism $U\rightarrow{\mathcal{F}}$ is in $(n-1)\textcat{-P}^{\bm\tau}|_\mathcal{A}$, by our inductive hypothesis $\mathbb{L}_{U/\mathcal{F},u}$ is strongly dualisable and its dual is $0$-connective. We can then easily see that $\mathbb{L}_{\mathcal{F}/Y,x}$ is strongly dualisable and its dual is $0$-connective using Lemma \ref{fibreperfect} and the induced long exact sequence of homotopy groups. 
\end{proof}

\begin{prop}\phantomsection\label{perfectrepresentabilitycriteria}
    Suppose that $\textcat{Dls}|_{\mathcal{A}^{op}}$ satisfies descent for $\bm\tau$-covers. Suppose that $f:\mathcal{F}\rightarrow{\mathcal{G}}$ is an $n$-representable$|_\mathcal{A}$ morphism of stacks in $\textcat{Stk}(\mathcal{A},\bm\tau|_\mathcal{A})$ with $t_0(f):t_0(\mathcal{F})\rightarrow{t_0(\mathcal{G})}$ in $n\textcat{-P}^{\bm\tau,\heartsuit}$. Then the following statements are equivalent
    \begin{enumerate}
        \item\label{smitem1} $f$ is in $n\textcat{-P}^{\bm\tau}|_\mathcal{A}$,
        \item\label{smitem2}For any $x:X=\textnormal{Spec}(A)\rightarrow{\mathcal{F}}$, $\mathbb{L}_{\mathcal{F}/\mathcal{G},x}$ is strongly dualisable and $\mathbb{L}_{\mathcal{F},\mathcal{G},x}^\vee$ is $0$-connective, 
        \item\label{smitem3}For any $x:X=\textnormal{Spec}(A)\rightarrow{\mathcal{F}}$ and $M\in \textcat{M}_{A,1}$, $\pi_0(\textnormal{Map}_{\textcat{Mod}_A}(\mathbb{L}_{\mathcal{F}/\mathcal{G},x},M))=0$.
    \end{enumerate}
\end{prop}

\begin{proof}
    To show that (\ref{smitem1}) implies (\ref{smitem2}), suppose that $f$ is in $n\textcat{-P}^{\bm\tau}|_\mathcal{A}$ and that there is a morphism $x:X=\textnormal{Spec}(A)\rightarrow{\mathcal{F}}$ with $X\in\mathcal{A}$. Then, we note that $\mathcal{F}\times_\mathcal{G}X$ is $n$-geometric and $\mathcal{F}\times_\mathcal{G}X\rightarrow{X}$ is in $n\textcat{-P}^{\bm\tau}|_\mathcal{A}$. Moreover, by Lemma \ref{globalcotangentresults}, we have that $\mathbb{L}_{\mathcal{F}/\mathcal{G},x}\simeq\mathbb{L}_{\mathcal{F}\times_\mathcal{G}X/X,x^\cdot}$, and hence we just need to show that (\ref{smitem2}) holds for any $n\textcat{-P}^{\bm\tau}|_\mathcal{A}$-morphism from an $n$-geometric stack to a representable stack on $\mathcal{A}$. This is the content of Lemma \ref{perfectconnduallemma}.

    To show that (\ref{smitem2}) implies (\ref{smitem3}), since $\mathbb{L}_{\mathcal{F}/\mathcal{G},x}$ is strongly dualisable, we have that
\begin{equation*}\pi_0(\textnormal{Map}_{\textcat{Mod}_A}(\mathbb{L}_{\mathcal{F}/\mathcal{G},x},M))\simeq\pi_0(\textnormal{Map}_{\textcat{Mod}_A}(A,\mathbb{L}_{\mathcal{F}/\mathcal{G},x}^\vee\otimes_A^\mathbb{L}M))
    \end{equation*}Since $M\in\textcat{M}_{A,1}$, it follows that $\mathbb{L}_{\mathcal{F}/\mathcal{G},x}^\vee\otimes_A^\mathbb{L}M$ is $1$-connective. Hence, the left hand side must be zero by Lemma \ref{projectivevanishinghomotopy}.

    Finally, to prove that (\ref{smitem3}) implies (\ref{smitem1}), we suppose that $f:\mathcal{F}\rightarrow{\mathcal{G}}$ is an $n$-representable$|_\mathcal{A}$ morphism of stacks with $t_0(f):t_0(\mathcal{F})\rightarrow{t_0(\mathcal{G})}$ in $n\textcat{-P}^{\bm\tau,\heartsuit}$. Since $f$ is $n$-representable we can just suppose, using familiar reasoning and Lemma \ref{globalcotangentresults}, that $\mathcal{F}$ is $n$-geometric and that $\mathcal{G}$ is a representable stack $X$ on $\mathcal{A}$. It also suffices to show that, whenever we have a representable stack $Y$ along with an $(n-1)\textcat{-P}^{\bm\tau}|_\mathcal{A}$ morphism $Y\rightarrow{\mathcal{F}}$, the composite morphism $Y\rightarrow{X}$ is in $\textcat{P}^{\bm\tau}$.
    
    We note that $t_0(Y)\rightarrow{t_0(\mathcal{F})}$ is in $(n-1)\textcat{-P}^{\bm\tau,\heartsuit}$ by Proposition \ref{t0preserves}, and hence the composite $t_0(Y)\rightarrow{t_0(X)}$ is in $\textcat{P}^{\bm\tau,\heartsuit}$. Therefore, we see that there exists a $\bm\tau^\heartsuit$-cover $\{U_i
    \rightarrow{t_0(Y)}\}_{i\in I}$ such that the induced morphisms $U_i\rightarrow{t_0(X)}$ are in $\textcat{P}^{\heartsuit}$. 

    By definition, the $\bm\tau^\heartsuit$-cover is the truncation of a $\bm\tau$-cover $\{U_i'=\textnormal{Spec}(C_i)\rightarrow{Y}\}_{i\in I}$. We need to show that the induced morphisms $U_i'\rightarrow{X}$ are in $\textcat{P}$. Since the morphism $U_i\rightarrow{t_0(X)}$ is in $\textcat{P}^\heartsuit$, we see that $\mathbb{L}_{\pi_0(C_i)/\pi_0(A)}$ is a perfect $\pi_0(C_i)$-module and there is a retract
        \begin{equation*}
\varinjlim_{j\in J}\coprod_{E_j}\pi_0(C_i)\rightarrow{\mathbb{L}_{\pi_0(C_i)/\pi_0(A)}}
        \end{equation*}for finite sets $J$ and $E_j$. By Corollary \ref{cotangentcomplexpi0} and Lemma \ref{projectivevanishinghomotopy}, this lifts to a map $\varinjlim_{j\in J}\coprod_{E_j}C_i\rightarrow{\mathbb{L}_{C_i/A}}$. Let $K$ be the fibre of this morphism. If we consider the fibre sequence
    \begin{equation*}
        \textnormal{Map}_{\textcat{Mod}_{C_i}}\bigg(\mathbb{L}_{C_i/A},\varinjlim_{j\in J}\coprod_{E_j}C_i\bigg)\rightarrow{\textnormal{Map}_{\textcat{Mod}_{C_i}}(\mathbb{L}_{C_i/A},\mathbb{L}_{C_i/A})\rightarrow{\textnormal{Map}_{\textcat{Mod}_{C_i}}(\mathbb{L}_{C_i/A},K[1])}}
    \end{equation*}and the induced long exact sequence in homotopy then, by our assumption, since $K[1]\in \textcat{M}_{C_i,1}$, there is a surjection
    \begin{equation*}
       \pi_0\bigg(\textnormal{Map}_{\textcat{Mod}_{C_i}}\bigg(\mathbb{L}_{{C_i}/A},\varinjlim_{j\in J}\coprod_{E_j}{C_i}\bigg)\bigg)\rightarrow{\pi_0(\textnormal{Map}_{\textcat{Mod}_{C_i}}(\mathbb{L}_{{C_i}/A},\mathbb{L}_{{C_i}/A}))}
    \end{equation*}and hence we see that $\mathbb{L}_{{C_i}/A}$ is perfect, being a retract of $\varinjlim_{j\in J}\coprod_{E_j}{C_i}$. Therefore, the morphism $U_i'\rightarrow{X}$ is formally perfect and its truncation lies in $\textcat{P}^\heartsuit$. By our assumptions on $\textcat{P}$, it must therefore lie in $\textcat{P}$.

\end{proof}

\subsection{The Finite Homotopy Monomorphism Topology}

Suppose that $(\mathcal{C},\mathcal{C}_{\geq 0},\mathcal{C}_{\leq 0},\mathcal{C}^0)$ is a flat Postnikov compatible derived algebraic context. 

\begin{defn}\phantomsection\label{homotopyepi}
    Suppose that $f:A\rightarrow{B}$ is a morphism in $\textcat{DAlg}^{cn}(\mathcal{C})$. Then, $f$ is a \textit{homotopy epimorphism} if the map $B\otimes_A^\mathbb{L}B\rightarrow{B}$ is an equivalence. The corresponding morphism in $\textcat{DAff}^{cn}(\mathcal{C})$ is called a \textit{homotopy monomorphism}. 
\end{defn}

\begin{lem}\phantomsection\label{formallyetalehomotopymono}\cite[c.f. Lemma 2.1.42]{ben-bassat_perspective_2024}
    If $f:A\rightarrow{B}$ in $\textcat{DAlg}^{cn}(\mathcal{C})$ is a homotopy epimorphism, then $f$ is formally \'etale, and hence formally perfect. 
\end{lem}

\begin{defn}
    The \textit{finite homotopy monomorphism topology} on $\textnormal{Ho}(\textcat{DAff}^{cn}(\mathcal{C}))$ has finite covers $\{U_j\rightarrow{X}\}_{j\in J}$ such that
    \begin{enumerate}
        \item Each morphism $U_j\rightarrow{X}$ is a homotopy monomorphism, 
        \item The family $\{U_j\rightarrow{X}\}_{j\in J}$ is a formal covering family. 
    \end{enumerate}It is easy to see that this defines a Grothendieck topology. We will denote the topology by $\bm{hm}^{fin}$.
\end{defn}

\begin{lem}\phantomsection\label{qcohdescenthmfin}
    $\textcat{QCoh}$ satisfies (\v{C}ech-)descent for $\bm{hm}^{fin}$-covers.
\end{lem}

\begin{proof}
    Suppose that we have a $\bm{hm}^{fin}$-cover $\{U_j=\textnormal{Spec}(B_j)\rightarrow{\textnormal{Spec}(A)=X}\}_{j\in J}$. Let $B=\prod_{j\in J}B_j$ and consider the cobar resolution $\textcat{CB}^*(B)$ defined in each degree by $\textcat{CB}^n(B)=B^{\otimes^\mathbb{L}_A n+1}$. We note that, since each morphism $A\rightarrow{B_j}$ is a homotopy epimorphism, the limit of this resolution is equivalent to $B$ as a $B$-module. Hence, by conservativity of our cover, we see that $A\rightarrow{\textcat{CB}^*(B)}$ is a limit diagram. Therefore, by \cite[Proposition 3.20]{mathew_galois_2016}, $A\rightarrow{B}$ is descendable in the sense of \cite[Definition 3.18]{mathew_galois_2016}. 
    
    Consider the \v{C}ech nerve $\mathcal{U}_*$ of the morphism $\coprod_{j\in J}h(U_j)\rightarrow{h(X)}$. We note that $h(\textnormal{Spec}(\textcat{CB}^n(B)))$ is equivalent to $\mathcal{U}_n$. Therefore, by \cite[Proposition 3.22]{mathew_galois_2016}, there is an equivalence of categories
    \begin{equation*}
        \textcat{QCoh}(X)=\textcat{Mod}_A\rightarrow{\varprojlim_{n}\textcat{Mod}_{B^{\otimes_A^\mathbb{L} n+1}}}\simeq \varprojlim_{n\in\Delta}\textcat{Mod}_{\textcat{CB}^n(B)}=\varprojlim_{n\in \Delta}\textcat{QCoh}(\mathcal{U}_n)
    \end{equation*}and therefore we have descent in the sense of Definition \ref{cechdescentdef}. 
\end{proof}

\begin{lem}\cite[Proposition 2.6.165]{ben-bassat_perspective_2024}\phantomsection\label{stronghomotopyzar} Suppose that $f:A\rightarrow{B}$ is a morphism in $\textcat{DAlg}^{cn}(\mathcal{C})$. 
\begin{enumerate}
    \item If $f$ is a homotopy epimorphism, then $\pi_0(f):\pi_0(A)\rightarrow{\pi_0(B)}$ is an epimorphism, 
    \item If $f$ is a derived strong morphism such that $\pi_0(f):\pi_0(A)\rightarrow{\pi_0(B)}$ is an epimorphism, then $f$ is a homotopy epimorphism,
    \item If $f$ is a homotopy epimorphism such that $\pi_0(f)$ is a homotopy epimorphism, and each $\pi_n(A)$ is transverse to $\pi_0(B)$ over $\pi_0(A)$, then $f$ is derived strong. 
\end{enumerate} 
\end{lem}

 Suppose that there exists a category $\mathcal{A}$ such that $(\textcat{DAff}^{cn}(\mathcal{C}),\bm{hm}^{fin},\textcat{fP}^{\bm{hm}^{fin}},\mathcal{A})$ is a relative $(\infty,1)$-geometry tuple. 

 \begin{lem}\phantomsection\label{coproductrephmfin}
    For any finite collection $\{U_i\}_{i\in I}$ of $\mathcal{A}$-admissible $U_i\in\textcat{DAff}^{cn}(\mathcal{C})$,  $\coprod_{i\in I} h(U_i)\rightarrow{h(\coprod_{i\in I} U_i)}$ is an equivalence in $\textcat{Stk}(\mathcal{A},\bm{hm}^{fin}|_\mathcal{A})$.
 \end{lem}
 \begin{proof}
     We note that the morphism $\prod_{j\in I} A_j\rightarrow{A_i}$ is a homotopy epimorphism by \cite[Proposition 2.5.134]{ben-bassat_perspective_2024} and Lemma \ref{stronghomotopyzar}. Moreover, it is clear that the family $\{U_i\rightarrow{\coprod_{j\in I} U_j\}_{i\in I}}$ is a formal covering family. Hence, we can conclude by Corollary \ref{coproductsdisjoint}. 
 \end{proof}

In this situation, we obtain the following representability context. 

\begin{cor}\phantomsection\label{hzrepresentabilitycontext}
     The tuple $(\mathcal{C},\mathcal{C}_{\geq 0},\mathcal{C}_{\leq 0},\mathcal{C}^0,\bm{hm}^{fin},\textcat{fP}^{\bm{hm}^{fin}},\mathcal{A},\textcat{M},\textcat{S})$ is a representability context in the sense of Definition \ref{repcontextdefn} if
      \begin{equation*}
          (\mathcal{C},\mathcal{C}_{\geq 0},\mathcal{C}_{\leq 0},\mathcal{C}^0,\bm{hm}^{fin},\textcat{fP}^{\bm{hm}^{fin}},\mathcal{A},\textcat{M})
      \end{equation*}is a flat Postnikov compatible derived geometry context, such that
    \begin{enumerate}\item$(\mathcal{A},\bm{hm}^{fin}|_\mathcal{A},\textcat{fP}^{\bm{hm}^{fin}}|_\mathcal{A},\mathcal{A}^\heartsuit)$ is a strong relative $(\infty,1)$-geometry tuple, 
    \item $\iota|_{\mathcal{A}^\heartsuit}:(\mathcal{A}^\heartsuit,\bm\tau^\heartsuit)\rightarrow{(\mathcal{A},\bm\tau|_\mathcal{A})}$ is a continuous functor of $(\infty,1)$-sites,
    \item\label{hmitem3} A morphism is in $\textcat{fP}^{\bm{hm}^{fin}}|_\mathcal{A}$ if it is derived strong relative to $\textcat{fP}^{\bm{hm}^{fin}}|_\mathcal{A}$,
    \item $\bm{hm}^{fin}$ and $\textcat{fP}^{{\bm{hm}^{fin}}}$ satisfy the obstruction conditions of Definition \ref{artinsconditions} relative to $\mathcal{A}$ for $\textcat{S}$. 
    \end{enumerate}
\end{cor}

\begin{proof}We note that Conditions (\ref{repcontexttuple}), (\ref{repcontextcontinuous}), (\ref{repcontextstrong}), and (\ref{repcontextobstruction}) are satisfied by assumption. Condition (\ref{repcontextrepresentables}) is satisfied by Lemma \ref{coproductrephmfin}. Condition (\ref{repcontextP}) follows from the proof of Proposition \ref{perfectrepresentabilitycriteria} and the statement that $\textcat{Dls}$ satisfies descent for $\bm{hm}^{fin}$-covers by Lemma \ref{perfectqcohdescent} and Lemma \ref{qcohdescenthmfin}. 

\end{proof}

\begin{remark}
    We note that if covers in $\bm{hm}^{fin}|_\mathcal{A}$ consist of derived strong morphisms, then Condition (\ref{hmitem3}) is satisfied using Lemma \ref{stronghomotopyzar}.
\end{remark}

\subsection{Derived Complex Analytic Geometry}

We will not give a full introduction to complex analytic geometry, a good reference is \cite{grauert_theory_2004}. The basic algebraic building blocks of complex analytic geometry are Stein algebras. These are objects in the image of the functor
\begin{equation*}
    \mathcal{O}:\textnormal{Stein Spaces}\rightarrow{\textnormal{Comm}(\textnormal{Fr}_\mathbb{C})^{op}}
\end{equation*}taking a Stein space $X$, for example $X=\mathbb{C}^n$, to the algebra of holomorphic functions on $X$. 

We will say that a Stein algebra $A$ is \textit{finitely presented} if it is of the form $\mathcal{O}(\mathbb{C}^n)/I$ for some $n\geq 0$ and some finitely generated ideal $I$. We note that we can write $A$ as a \textit{Fr\'echet-Stein algebra} in the sense of \cite[Definition 5.2.52]{ben-bassat_perspective_2024} in the following way. Write $\mathcal{O}(\mathbb{C}^n)$ as a limit $\varprojlim_k\mathcal{O}(D_k)$ where $\mathcal{O}(D_k)$ is the Banach space of holomorphic functions on the disk $D_k$ of radius $k$ in $\mathbb{C}^n$. Then, we can define $I_k$ to be the closure of the ideal generated by the image of $I$ in $\mathcal{O}(D_k)$. We then easily see that $A=\varprojlim_kA_k$ where $A_k=\mathcal{O}(D_k)/I_k$. 

\begin{defn}\cite[c.f. Definition 5.2.56]{ben-bassat_perspective_2024}
    A module $M$ over a finitely presented Stein algebra $A$ is \textit{coadmissible} if it can be written as a limit $\varprojlim_k M_k$ such that
    \begin{enumerate}
        \item Each $M_k$ is finitely generated over $A_k$, 
        \item The natural morphism $A_k\hat{\otimes}_{A_{k+1}}M_{k+1}\rightarrow{M_k}$ is an equivalence. 
    \end{enumerate}
\end{defn}

\begin{remark}
    We think of coadmissible modules as being the analogues for complex analytic geometry of coherent modules in algebraic geometry.
\end{remark}

Consider the derived algebraic context
\begin{equation*}
    (\textcat{Ch}(\textnormal{Ind}(\textnormal{Ban}_\mathbb{C})),\textcat{Ch}_{\geq 0}(\textnormal{Ind}(\textnormal{Ban}_\mathbb{C})),\textcat{Ch}_{\leq 0}(\textnormal{Ind}(\textnormal{Ban}_\mathbb{C})),\textcat{L}^H(P^0))
\end{equation*}
We remark that the following definition of a derived Stein is due to Kelly and we refer to \cite{ben-bassat_perspective_2024} for more details. 

\begin{defn}
 $A\in\textcat{DAlg}^{cn}(\textcat{Ch}(\textnormal{Ind}(\textnormal{Ban}_\mathbb{C})))$ is a \textit{derived Stein algebra over $\mathbb{C}$} if
    \begin{enumerate}
        \item $\pi_0(A)$ is a finitely presented Stein algebra,
        \item $\pi_n(A)$ is a coadmissible $\pi_0(A)$-module.
    \end{enumerate} 
\end{defn}

\begin{remark}
    It is shown in \cite[Section 9.2.1.1]{ben-bassat_perspective_2024} that the complex analytic theory of Porta and Yue Yu, as described in \cite{porta_derived_2018-1}, built up from their notion of derived Stein algebras, embeds fully faithfully into our theory.
\end{remark}

Let $\textcat{DSt}$ denote the full subcategory of $\textcat{DAlg}^{cn}(\textcat{Ch}(\textnormal{Ind}(\textnormal{Ban}_\mathbb{C})))$ consisting of derived Stein algebras, and $\textnormal{St}$ the ordinary category of Stein algebras. We note that $\textcat{DSt}$ is closed under finite products by \cite[Proposition 5.4.115]{ben-bassat_perspective_2024} and \cite[Theorem V.1.1]{grauert_theory_2004}. We recall that $\textcat{fP}^{{\bm{hm}^{fin}}}$ denotes the expansion of our class of formally perfect maps to those which are local for the ${{\bm{hm}^{fin}}}$-topology. We note that covers of Steins are often not flat, but we can impose a transversality condition, as described in the following definition.

\begin{defn}\phantomsection\label{maptopologysteindef}
We make the following definitions.
\begin{itemize}
    \item The class $\textcat{fP}^{\bm{hm}}_{\textcat{DSt}}$ is the subclass of maps $f:Y=\textnormal{Spec}(B)\rightarrow{\textnormal{Spec}(A)=X}$ in ${\textcat{fP}}^{\bm{hm}^{fin}}$ such that
    \begin{enumerate}
        \item\label{fpsteinitem1} Whenever $A\rightarrow{C}$ is a map with $C\in\textnormal{St}$, then $B\otimes_A^\mathbb{L}C$ is in $\textnormal{St}$,
        \item\label{fpsteinitem2} Any coadmissible $\pi_0(A)$-module $M$ is transverse to $\pi_0(B)$.
    \end{enumerate} 
    \item The \textit{finite Stein homotopy monomorphism topology}, which we will denote by $\bm{hm}^{fin}_{\textcat{DSt}}$, consists of covers $\{U_j=\textnormal{Spec}(B_j)\rightarrow{\textnormal{Spec}(A)=X}\}_{j\in J}$ in ${\bm{hm}^{fin}}$ such that 
    \begin{enumerate}
        \item Whenever there is some map $A\rightarrow{C}$, then $C\in\textnormal{St}$ if and only if $B_j\otimes_A^\mathbb{L}C$ is in $\textnormal{St}$ for each $j\in J$,
        \item Any coadmissible $\pi_0(A)$-module $M$ is transverse to $\pi_0(B_j)$.
    \end{enumerate}
\end{itemize}
\end{defn}

\begin{remark}
    We note that an epimorphism $A\rightarrow{B}$ of Stein algebras such that $\pi_n(A)$ is transverse to $\pi_0(B)$ as a $\pi_0(A)$-module is equivalent to a homotopy epimorphism of Stein algebras. This follows easily using the Tor spectral sequence from Proposition \ref{torspectralsequence}. 
\end{remark}

It is easy to see that $\textcat{fP}^{\bm{hm}}_{\textcat{DSt}}$ is closed under equivalences, compositions, and pullbacks. We can also see that the finite Stein homotopy monomorphism topology defines a Grothendieck pre-topology on $\textnormal{Ho}(\textcat{DAff}^{cn}(\textcat{Ch}(\textnormal{Ind}(\textnormal{Ban}_\mathbb{C}))))$. Using Lemma \ref{qcohdescenthmfin}, we see that $\textcat{QCoh}$ satisfies descent for $\bm{hm}^{fin}_{\textcat{DSt}}$-covers. It follows by Lemma \ref{perfectqcohdescent}, that $\textcat{Dls}$ satisfies descent for $\bm{hm}^{fin}_{\textcat{DSt}}$-covers.

\begin{lem}\phantomsection\label{steinfpstrong}
    Morphisms in $\textcat{fP}^{\bm{hm}}_{\textcat{DSt}}$ are derived strong.
\end{lem}

\begin{proof}
    This is a consequence of Lemma \ref{localisderivedstrong}. Indeed, if $\textnormal{Spec}(B)\rightarrow{\textnormal{Spec}(A)}$ is a formally perfect morphism or a homotopy monomorphism with each $\pi_n(A)$ transverse to $\pi_0(B)$ as a $\pi_0(A)$-module, then the morphism $A\rightarrow{B}$ is derived strong by Corollary \ref{derivedstronglyperfectcorollary} and Lemma \ref{stronghomotopyzar}. 
\end{proof}

  \begin{lem}\phantomsection\label{relativestein} Suppose that we have a morphism $Y=\textnormal{Spec}(B)\rightarrow{\textnormal{Spec}(A)=X}$ in  $\textcat{fP}^{\bm{hm}}_{\textcat{DSt}}$ and a map $A\rightarrow{C}$ with $C\in\textcat{DSt}$. Then, $B\otimes_A^\mathbb{L}C$ is in $\textcat{DSt}$. 
    
\end{lem}

\begin{proof}
    We note that, since $C\in\textcat{DSt}$, $B\otimes_A^\mathbb{L}\pi_0(C)\simeq \pi_0(B\otimes_A^\mathbb{L}C)$ is a Stein algebra. The map $C\rightarrow{B\otimes_A^\mathbb{L}C}$ is derived strong by Lemma \ref{steinfpstrong} since morphisms in $\textcat{fP}^{\bm{hm}}_{\textcat{DSt}}$ are closed under pullback. Therefore, 
    \begin{equation*}
\pi_i(C)\otimes^\mathbb{L}_{\pi_0(C)}\pi_0(B\otimes_A^\mathbb{L}C)\simeq \pi_i(B\otimes_A^\mathbb{L}C)
    \end{equation*}and, since $\pi_i(C)$ is coadmissible as a $\pi_0(C)$-module, we see that $\pi_i(B\otimes_A^\mathbb{L}C)$ is coadmissible as a $\pi_0(B\otimes_A^\mathbb{L}C)$-module. 
\end{proof}

\begin{lem}\phantomsection\label{theextrarelativecondition}Suppose that $\{U_j=\textnormal{Spec}(B_j)\rightarrow{\textnormal{Spec}(A)=X}\}_{j\in J}$ is a $\bm{hm}^{fin}_{\textcat{DSt}}$-cover and we have some map $A\rightarrow{C}$. Then, if $B_j\otimes_A^\mathbb{L}C\in\textcat{DSt}$ for each $j$, then $C\in\textcat{DSt}$.
\end{lem}

\begin{proof}
    It is clear that $\pi_0(C)\in\textnormal{St}$ by our conditions on our topology. Since we have descent for coadmissible modules on Steins by \cite[Theorem 9.2.14]{ben-bassat_perspective_2024}, then we can conclude that $\pi_n(C)$ is a coadmissible $\pi_0(C)$-module. 
\end{proof}

\begin{lem}\phantomsection\label{strongstein}
    A morphism $f:Y=\textnormal{Spec}(B)\rightarrow{\textnormal{Spec}(A)=X}$ is in $\textcat{fP}^{\bm{hm}}_{\textcat{DSt}}|_{\textcat{DSt}^{op}}$ if it is derived strong and $t_0(f)$ is in $\textcat{fP}^{\bm{hm},\heartsuit}_{\textnormal{St}}$. 
\end{lem}

\begin{proof}Suppose that there is a morphism $Z=\textnormal{Spec}(C)\rightarrow{X}$ with $Z\in\textnormal{St}^{op}$. Then, we note that, since $t_0(f)\in \textcat{fP}^{\bm{hm},\heartsuit}_{\textnormal{St}}$,  $\pi_0(B\otimes_A^\mathbb{L}C)\simeq \pi_0(B)\otimes^\mathbb{L}_{\pi_0(A)}C\in\textnormal{St}$. Since $f$ is derived strong, $\pi_n(B\otimes_A^\mathbb{L}C)\simeq \pi_0(B)\otimes^\mathbb{L}_{\pi_0(A)}\pi_n(C)\simeq 0$ for $n>0$, by Proposition \ref{derviedstronghomotopygroupprop}. Now, we note that, since covers in ${\bm{hm}^{fin}}$ are derived strong by Lemma \ref{stronghomotopyzar}, we can conclude that the morphism $f:Y\rightarrow{X}$ is in $\textcat{fP}^{{\bm{hm}^{fin}}}$ using Corollary \ref{derivedstronglyperfectcorollary}. The condition on transversality of coadmissible modules follows from the statement that $t_0(f)$ is in $\textcat{fP}^{\bm{hm},\heartsuit}_{\textnormal{St}}$ and that $X,Y\in\textcat{DSt}^{op}$.
\end{proof}

\begin{lem}\phantomsection\label{steinadmissiblelemma}
    Every object of $\textcat{DAff}^{cn}(\textcat{Ch}(\textnormal{Ind}(\textnormal{Ban}_\mathbb{C})))$ is $\textcat{DSt}^{op}$-admissible with respect to the finite Stein homotopy monomorphism topology.
\end{lem}

\begin{proof}We adapt the proof of \cite[Lemma 2.2.2.13]{toen_homotopical_2008} to our setting. Suppose that we have a $\bm{hm}^{fin}_{\textcat{DSt}}$-hypercover $\mathcal{U}_*\rightarrow{h(X)}$ in $\textcat{PSh}(\textcat{DSt}^{op})$. We note that each $\mathcal{U}_n$ is equivalent to a coproduct $\coprod_{i_n\in I_n} h(U_{i_n})$ of representable stacks such that we have a corresponding $\bm{hm}^{fin}_{\textcat{DSt}}$-cover $\{U_{i_n}\rightarrow{(\textnormal{cosk}_{n-1}\mathcal{U_*})_n}\}_{i_n\in I_n}$. Let $U_{i_n}=\textnormal{Spec}(B_{i_n})$ and $X=\textnormal{Spec}(A)$ for some $A,B_{i_n}\in\textcat{DSt}$. Let $U_n=\coprod_{i_n\in I_n} U_{i_n}$ and consider the augmented simplicial object $\textnormal{Spec}(B_*)=U_*\rightarrow{X}$.

We will show that $A\simeq \varprojlim_ n B_n$ as then $X\simeq |U_*|$. It then follows easily that any $Y\in\textcat{DAff}^{cn}(\textcat{Ch}(\textnormal{Ind}(\textnormal{Ban}_\mathbb{C})))$ satisfies descent for $\bm{hm}^{fin}_{\textcat{DSt}}$-hypercovers as required. Indeed, consider the spectral sequence 
\begin{equation*}
    \pi_p(\textnormal{Tot}(\pi_q(B_*)))\Rightarrow{\pi_{q-p}\bigg(\varprojlim_n \prod_{i_n\in I_n} B_{i_n}}\bigg)
\end{equation*}We note that, by definition of our topology, each $\pi_q(B_{i_n})$ is transverse to $\pi_0(B_{i_n})$ over $\pi_0(A)$. Using the definition of derived Stein algebras, we therefore see that the $\pi_q(B_{i_n})$ glue together to define an object of $\textnormal{Coad}(t_0(U_*))$. Since we have hyperdescent for coadmissible sheaves on $\textnormal{St}^{op}$ with respect to $\bm{hm}^{fin,\heartsuit}$ by \cite[Theorem 9.2.14]{ben-bassat_perspective_2024}, we see that $\pi_p(\textnormal{Tot}(\pi_q(B_*)))\simeq 0$ for $p\neq 0$. Therefore, the spectral sequence degenerates and we have that
\begin{equation*}
    \pi_p\bigg(\varprojlim_n \prod_{i_n\in I_n}B_{i_n}\bigg)\simeq \textnormal{Ker}\bigg(\pi_p\bigg(\prod_{i_0\in I_0}B_{i_0}\bigg)\rightarrow{\pi_p\bigg(\prod_{i_1\in I_1}B_{i_1}\bigg)}\bigg)
\end{equation*}This implies that $ \pi_p\bigg(\varprojlim_n \prod_{i_n\in I_n}B_{i_n}\bigg)$ is the $\pi_0(A)$-module obtained by descent from $\pi_p(B_{i_*})$. Therefore, it follows that the natural morphism
\begin{equation*}
    \pi_p\bigg(\varprojlim_n\prod_{i_n\in I_n}B_{i_n}\bigg)\otimes^\mathbb{L}_{\pi_0(A)}\pi_0(B_{i_0})\rightarrow{\pi_p(B_{i_0})}
\end{equation*}is an isomorphism for all $p>0$ and $i_0\in I_0$. We note that the morphism $A\rightarrow{B_{i_0}}$ is derived strong by Lemma \ref{steinfpstrong}. Hence, we can use Lemma \ref{derivedstrongequiv} to show that there is an equivalence
\begin{equation*}
    B_{i_0}\rightarrow{\bigg(\varinjlim_n\prod_{i_n\in I_n} B_{i_n}}\bigg)\otimes^\mathbb{L}_AB_{i_0}
\end{equation*}Since $B_{i_0}$ was arbitrary, and the family $\{A\rightarrow{B_{i_0}}\}_{i_0\in I_0}$ corresponds to a formal covering family, we see that $|B_*|\simeq \varprojlim_{n} \prod_{i\in I_n} B_{i_n}\simeq A$ as required. 
\end{proof}

\begin{cor}\phantomsection\label{infinitesimalcriteriastein}
    Suppose that $f:\mathcal{F}\rightarrow{\mathcal{G}}$ is an $n$-representable$|_{\textcat{DSt}^{op}}$ morphism of stacks in $\textcat{Stk}(\textcat{DSt}^{op},\bm{hm}^{fin}_{\textcat{DSt}})$. Then, $f$ is in $n\textcat{-fP}^{\bm{hm}}_{\textcat{DSt}}|_{\textcat{DSt}^{op}}$ if $f$ satisfies that, for any $x:X=\textnormal{Spec}(A)\rightarrow{\mathcal{F}}$ and any $M\in\textcat{M}_{A,1}$, 
    \begin{equation*}
        \pi_0(\textnormal{Map}_{\textcat{Mod}_A}(\mathbb{L}_{\mathcal{F}/\mathcal{G},x},M))=0
    \end{equation*}The converse holds if $t_0(\mathcal{F})\rightarrow{t_0(\mathcal{G})}$ is in $n\textcat{-fP}^{\bm{hm},\heartsuit}_{\textnormal{St}}$.
\end{cor}

\begin{proof}
   Indeed, we note that, by Lemma \ref{steinadmissiblelemma} and Lemma \ref{perfectqcohdescent}, $\textcat{Dls}$ satisfies descent for $\bm{hm}^{fin}_{\textcat{DSt}}$-covers. Therefore, the forwards direction follows from Proposition \ref{perfectrepresentabilitycriteria} since $\textcat{fP}^{hm}_{\textcat{DSt}}\subseteq\textcat{fP}^{\bm{hm}^{fin}_{\textcat{DSt}}}$. To prove the converse, we note that, by the proof of Proposition \ref{perfectrepresentabilitycriteria}, it suffices to show that if we have a morphism $f:Y\rightarrow{X}$ in $\textcat{DSt}^{op}$ which is in $\textcat{fP}^{\bm{hm}^{fin}_{\textcat{DSt}}}$ and whose truncation lies in $\textcat{fP}^{\bm{hm},\heartsuit}_{\textnormal{St}}$, then $f$ lies in $\textcat{fP}^{\bm{hm}}_{\textcat{DSt}}$. Condition (\ref{fpsteinitem1}) of Definition \ref{maptopologysteindef} follows by definition of $\textcat{hm}^{fin}_{\textcat{DSt}}$ and Condition (\ref{fpsteinitem2}) follows since $t_0(f)$ is in $\textcat{fP}^{\bm{hm},\heartsuit}_{\textnormal{St}}$.
\end{proof}

\begin{cor}\phantomsection\label{steinstronggeometry}
The tuple $(\textcat{DAff}^{cn}(\textcat{Ch}(\textnormal{Ind}(\textnormal{Ban}_\mathbb{C}))),\bm{hm}^{fin}_{\textcat{DSt}},{\textcat{fP}^{hm}_{\textcat{DSt}}},\textcat{DSt})$ is a strong relative $(\infty,1)$-geometry tuple. 
\end{cor}

\begin{proof}We note that, by Lemmas \ref{relativestein} and \ref{theextrarelativecondition}, together with descent for transversality, we can show that it is a relative $(\infty,1)$-pre-geometry tuple. Further, by Lemma \ref{makingthingsstrong} it is strong. By Lemma \ref{steinadmissiblelemma}, every object of $\textcat{DAff}^{cn}(\textcat{Ch}(\textnormal{Ind}(\textnormal{Ban}_\mathbb{C})))$ is $\textcat{DSt}^{op}$-admissible. 
\end{proof}

\subsection{A Representability Context for Derived Complex Analytic Geometry}

We now want to check that the geometry constructed in the previous section defines a representability context. 

\begin{lem}\phantomsection\label{steinstronggeometrytuple}
$(\textcat{DSt}^{op},\bm{hm}^{fin}_{\textcat{DSt}}|_{\textcat{DSt}^{op}},{\textcat{fP}^{\bm{hm}}_{\textcat{DSt}}}|_{\textcat{DSt}^{op}},\textnormal{St}^{op})$ is a strong relative $(\infty,1)$-geometry tuple. 
\end{lem}

\begin{proof}We easily see, using Corollary \ref{steinstronggeometry} and our definition of $\textcat{fP}^{\bm{hm}}_{\textcat{DSt}}$, that the tuple is a relative $(\infty,1)$-geometry tuple. Furthermore, every object of $\textcat{DSt}^{op}$ is $\textnormal{St}^{op}$-admissible since Stein algebras satisfy descent for the $\bm{hm}^{fin,\heartsuit}_{\textnormal{St}}$-topology. It is strong by our conditions on morphisms in $\textcat{fP}^{\bm{hm}}_{\textcat{DSt}}$.
\end{proof}

\begin{lem}\phantomsection\label{continuousstein}
    $\iota|_{\textnormal{St}^{op}}:(\textnormal{St}^{op},\bm{hm}^{fin,\heartsuit}_{\textnormal{St}})\rightarrow{(\textcat{DSt}^{op},\bm{hm}^{fin}_{\textcat{DSt}})}$ is a continuous functor of $(\infty,1)$-sites. 
\end{lem}

\begin{proof}Suppose that $\{t_0(U_j)=\textnormal{Spec}(\pi_0(B_j))\rightarrow{\textnormal{Spec}(\pi_0(A))=t_0(X)}\}_{j\in J}$ is a cover in $\bm{hm}^{fin,\heartsuit}_{\textnormal{St}}$ corresponding to the truncation of a cover $\{U_j\rightarrow{X}\}_{j\in J}$ in $\bm{hm}^{fin}_{\textcat{DSt}}$. Then, we note that  $\pi_0(A)\rightarrow{\pi_0(B_j)}$ is an epimorphism of Stein algebras such that $\pi_n(A)$ is transverse to $\pi_0(B_j)$ as a $\pi_0(A)$-module. Hence, it is a homotopy epimorphism. Moreover, we know that $U_j\rightarrow{X}$ is strong, and hence we can show the remaining condition on $\bm{hm}^{fin}_{\textcat{DSt}}$ using Lemma \ref{derivedstrongequiv}. We know that the truncation of a formal covering family is a formal covering family by Lemma \ref{truncationconservative}.
\end{proof}

 \begin{lem}\phantomsection\label{coproductsteinrepresentable}
     For any finite collection $\{U_i\}_{i\in I}$ of $U_i\in\textcat{DAff}^{cn}(\textcat{Ch}(\textnormal{Ind}(\textnormal{Ban}_\mathbb{C})))$, $\coprod_{i\in I} h(U_i)\rightarrow{h(\coprod_{i\in I} U_i)}$ is an equivalence in $\textcat{Stk}(\textcat{DSt}^{op},\bm{hm}^{fin}_{\textcat{DSt}})$,
 \end{lem}
 \begin{proof}
     By Lemma \ref{coproductrephmfin}, we note that the family $\{U_i\rightarrow{\coprod_{j\in I} U_j\}_{i\in I}}$ is a $\bm{hm}^{fin}$-covering family. Moreover, we can easily check that each morphism $U_i\rightarrow{\coprod_{j\in I} U_j}$ lies in ${\textcat{fP}^{\bm{hm}}_{\textcat{DSt}}}$. Hence, we can conclude by Corollary \ref{coproductsdisjoint}. 
 \end{proof}

Suppose that $A\in\textcat{DSt}$. Let $\textcat{Mod}_A^{coad,cn}$ denote the subcategory of connective $A$-modules $M$ such that $\pi_n(M)$ is coadmissible over $\pi_0(A)$ for all $n\geq 0$. Consider the system $\textcat{Mod}^{coad,cn}$ defined by these categories. 

\begin{lem}\phantomsection\label{coadsteinmodules}
    $\textcat{Mod}^{coad,cn}$ is a good system of categories of modules on $\textcat{DSt}^{op}$ in the sense of Definition \ref{goodsystemmodules}. 
\end{lem}

\begin{proof}
    Indeed, we note that Conditions (\ref{good1}), (\ref{good3}), and (\ref{good6}) follow easily from our definition of derived Stein algebras. We note that Condition (\ref{good2}) follows using \cite[Lemma 5.4.121]{ben-bassat_perspective_2024}, (\ref{good4}) follows by \cite[Theorem 5.4.118]{ben-bassat_perspective_2024} and (\ref{good7}) follows by \cite[Proposition 5.4.116]{ben-bassat_perspective_2024}. Condition (\ref{good5}) follows by Lemma \ref{steinadmissiblelemma}. To show that Condition (\ref{good5half}) holds, suppose that $A\in\textcat{DSt}^{\leq k}$ for some $k$, $M\in\textcat{Mod}_A^{coad,cn,\heartsuit}$, and $d\in\pi_0(\textcat{Der}(A,M))$, and consider the following fibre sequence 
    \begin{equation*}
        A\oplus_d\Omega M[k+1]\rightarrow{A}\rightarrow{M[k+1]}
    \end{equation*}from Lemma \ref{pushoutAsqzM}. Then, for $n<k$ we note that, by considering the long exact sequence in homotopy, $\pi_n(A\oplus_d\Omega M[k+1])\simeq \pi_n(A)$, which is a finitely presented Stein algebra for $n=0$ and is coadmissible as a $\pi_0(A)$-module for $0<n<k$. When $n=k$, we see that we have an exact sequence
    \begin{equation*}
        0\rightarrow M\rightarrow{\pi_{k}(A\oplus_d\Omega M[k+1])}\rightarrow{\pi_k(A)}\rightarrow{0}
    \end{equation*}and it follows that $\pi_k(A\oplus_d\Omega M[k+1])$ is a coadmissible $\pi_0(A)$-module since the category of coadmissible $\pi_0(A)$-modules is abelian by \cite[Proposition 5.2.58]{ben-bassat_perspective_2024}. For $n>k$, $\pi_n(A\oplus_d\Omega M[k+1])\simeq 0$. 
\end{proof}

\begin{lem}\phantomsection\label{steinobstruction}    $\bm{hm}^{fin}_{\textcat{DSt}}$ and ${\textcat{fP}^{\bm{hm}}_{\textcat{DSt}}}$ satisfy the obstruction conditions relative to $\textcat{DSt}^{op}$ for the class $\textcat{hm}$ of homotopy monomorphisms.
\end{lem}
\begin{proof}

 Indeed, we note that morphisms in ${\textcat{fP}^{\bm{hm}}_{\textcat{DSt}}}$ are formally \'etale by Lemma \ref{formallyetalelocalmaps}, and hence are formally $i$-smooth. Therefore, Condition (\ref{artin0}) is satisfied. Condition (\ref{artin1}) is also clearly satisfied. To show that Condition (\ref{artin5}) is satisfied, we suppose that we have a $\bm{hm}^{fin}_{\textcat{DSt}}$-covering family $\{V_j=\textnormal{Spec}(A_j)\rightarrow{\textnormal{Spec}(A)=X}\}_{j\in J}$ and some $M\in\textcat{Mod}^{coad,cn}_{A,1}$. Now, define $W_j=\textnormal{Spec}(B_j)=\textnormal{Spec}(A_j\oplus_{d_j'} \Omega M_j')\in\textcat{DSt}^{op}$ with $d_j'$ the derivation induced by Lemma \ref{liftingderivations} and $M_j'=M\otimes^\mathbb{L}_A A_j$. This can trivially be refined to a $\bm{hm}^{fin}$-covering family. The collection $\{W_j\rightarrow{X_d[\Omega M]}\}_{j\in J}$ is a formal covering family by Lemma \ref{formalcoveringsqz}. By Lemma \ref{Ecoverequiv}, and using that the morphism $V_j\rightarrow{X}$ is a homotopy monomorphism, we have an equivalence
\begin{equation*}
B_j\otimes_{A\oplus_d\Omega M}^\mathbb{L}A\simeq A_j\simeq A_j\otimes^\mathbb{L}_AA_j\simeq (B_j\otimes_{A\oplus_d\Omega M}^\mathbb{L}B_j)\otimes_{A\oplus_d\Omega M}^\mathbb{L}A
\end{equation*}Hence, by Corollary \ref{Amodulepullbackpushout}, we have an equivalence $B_j\simeq B_j\otimes^\mathbb{L}_{A\oplus_d\Omega M}B_j$. Therefore, $\{W_j\rightarrow{X_d[\Omega M]}\}_{j\in J}$ is also a $\bm{hm}^{fin}$-covering family. 

By Lemma \ref{qcohdescenthmfin}, since we have $\textcat{QCoh}$-descent for $\bm{hm}^{fin}$-covers and since $W_j$ and $X_d[\Omega M]$ are $\textcat{DSt}^{op}$-admissible, we see that this defines an epimorphism of stacks $\coprod_{j\in J}W_j\rightarrow{X_d[\Omega M]}$ in $\textcat{Stk}(\textcat{DSt}^{op},\bm{hm}^{fin}_{\textcat{DSt}})$.
\end{proof}

Denote by $\underline{\textcat{Ch}}(\textnormal{Ind}(\textnormal{Ban}_\mathbb{C}))$ the derived algebraic context
\begin{equation*}
    (\textcat{Ch}(\textnormal{Ind}(\textnormal{Ban}_\mathbb{C})),\textcat{Ch}_{\geq 0}(\textnormal{Ind}(\textnormal{Ban}_\mathbb{C})),\textcat{Ch}_{\leq 0}(\textnormal{Ind}(\textnormal{Ban}_\mathbb{C})),\textcat{L}^H(P^0))
\end{equation*}

\begin{cor}\phantomsection\label{steinflatderived}The tuple $(\underline{\textcat{Ch}}(\textnormal{Ind}(\textnormal{Ban}_\mathbb{C})),\bm{hm}^{fin}_{\textcat{DSt}},{\textcat{fP}^{\bm{hm}}_{\textcat{DSt}}},\textcat{DSt}^{op},\textcat{Mod}^{coad,cn})$ is a flat Postnikov compatible derived geometry context.
    
\end{cor}

\begin{proof}
    Indeed, we know that it is a derived geometry context by Corollary \ref{steinstronggeometry} and Lemma \ref{coadsteinmodules}. To see that it is Postnikov compatible, we note that the first two conditions follow easily from standard results. We note that Assumption (\ref{postnikovcomp4}) follows from the definition of derived Stein algebras and, since we are working over $\mathbb{C}$, Assumption (\ref{postnikovcomp5}) is satisfied using a similar proof to \cite[Proposition 25.2.4.1]{lurie_spectral_2018}. 
\end{proof}

\begin{cor}\phantomsection\label{complexanalyticrepcontext} The tuple \begin{equation*}
    (\underline{\textcat{Ch}}(\textnormal{Ind}(\textnormal{Ban}_\mathbb{C})),\bm{hm}^{fin}_{\textcat{DSt}},{\textcat{fP}^{\bm{hm}}_{\textcat{DSt}}},\textcat{DSt}^{op},\textcat{Mod}^{coad,cn},\textcat{hm})
\end{equation*}is a representability context in the sense of Definition \ref{repcontextdefn}
\end{cor}

\begin{proof}By Corollary \ref{steinflatderived}, our context is a flat Postnikov compatible derived geometry context. Condition (\ref{repcontexttuple}) is satisfied by Lemma \ref{steinstronggeometrytuple}. Condition (\ref{repcontextcontinuous}) follows by Lemma \ref{continuousstein}. Condition (\ref{repcontextrepresentables}) follows by Lemma \ref{coproductsteinrepresentable}. Condition (\ref{repcontextstrong}) follows by Lemma \ref{strongstein} and Condition (\ref{repcontextP}) by Corollary \ref{infinitesimalcriteriastein}. Finally, Condition (\ref{repcontextobstruction}) follows by Lemma \ref{steinobstruction}. 

\end{proof}

\appendix

\section{The Categories $\textnormal{CBorn}_R$ and $\textnormal{Ind}(\textnormal{Ban}_R)$}\label{appendixbornological}

Suppose we have a Banach ring $R$, i.e. a commutative unital ring with a complete norm defined on it. Denote by $\textnormal{Ban}_R$ the category of Banach $R$-modules, i.e. complete normed $R$-modules. 

\begin{defn}
    The category $\textnormal{Ind}(\textnormal{Ban}_R)$ of \textit{inductive systems of Banach $R$-modules} has as its objects functors  $A:I\rightarrow{\textnormal{Ind}(\textnormal{Ban}_R)}$, denoted by $A=(A_i)_{i\in I}$. The set of morphisms is given by  \begin{equation*}
        \textnormal{Hom}_{\textnormal{Ind}(\textnormal{Ban}_R)}(A,B)\simeq \varprojlim_{i\in I}\varinjlim_{j\in J}\textnormal{Hom}_{\textnormal{Ban}_R}(A_i,B_j)
    \end{equation*}
\end{defn}

\begin{defn}
An object $A\in\textnormal{Ind}(\textnormal{Ban}_R)$ is \textit{essentially monomorphic} if it is isomorphic to a system where all the morphisms are monomorphic. Denote by $\textnormal{Ind}^m(\textnormal{Ban}_R)$ the subcategory of essentially monomorphic objects in $\textnormal{Ind}(\textnormal{Ban}_R)$. 
\end{defn}

We refer to \cite[Definition 2.13]{kelly_analytic_2022} for the definition of the category $\textnormal{CBorn}_R$ of complete bornological $R$-modules. By \cite[Proposition 2.14]{kelly_analytic_2022}, there is a faithful functor 
\begin{equation*}
    \textnormal{Ind}^m(\textnormal{Ban}_R)\rightarrow{\textnormal{CBorn}_R}
\end{equation*}which defines an equivalence when $k$ is a non-trivial valued field \cite[Proposition 3.60]{bambozzi_dagger_2016}.

\section{Perfect, Compact, Dualisable Objects}\phantomsection\label{perfectappendix}

Let $\mathcal{C}$ be an additive closed symmetric monoidal $(\infty,1)$-category with monoidal product $\otimes^\mathbb{L}$ and unit $I$. Denote the internal mapping space by $\texthom{Map}_\mathcal{C}:\mathcal{C}\rightarrow{\mathcal{C}}$. We fix the following definitions.

\begin{defn}\phantomsection\label{differentduals} Suppose that $A\in\mathcal{C}$. Then its \textit{dual} object is $A^\vee:=\texthom{Map}_\mathcal{C}(A,I)$. An object $A\in\mathcal{C}$ is 
    \begin{enumerate}
            \item \textit{compact} if $\textnormal{Map}_\mathcal{C}(A,-)$ commutes with filtered colimits,
        \item \textit{projective} if $\textnormal{Map}_\mathcal{C}(A,-)$ commutes with geometric realisations,
        \item \textit{perfect} if it is a retract of a finite colimit of objects of the form $\coprod_E I$ for some finite set $E$, 
        \item \textit{dualisable} if the map $A^\vee\otimes^{\mathbb{L}}A\rightarrow{\texthom{Map}_\mathcal{C}(A,A)}$ is an equivalence,
        \item \textit{strongly dualisable} if the map $A^\vee\otimes^{\mathbb{L}}B\rightarrow{\texthom{Map}_\mathcal{C}(A,B)}$ is an equivalence for any $B\in\mathcal{C}$,
        \item \textit{reflexive} if $(A^\vee)^\vee\simeq A$.
    \end{enumerate}
\end{defn}

We note that perfect objects are strongly dualisable and projective. We will frequently use the following result. 

\begin{lem}\phantomsection\label{fibreperfect}
    If $\mathcal{C}$ is additionally a stable $(\infty,1)$-category and $A\rightarrow{B}\rightarrow{C}$ is a fibre sequence in $\mathcal{C}$ with two of the objects strongly dualisable, then the third object is strongly dualisable.
\end{lem}

\begin{proof}
Suppose that $D\in\mathcal{C}$. The result follows from considering the morphism of fibre-cofibre sequences induced by $(-)^\vee\otimes^\mathbb{L}D$ and $\texthom{Map}_\mathcal{C}(-,D)$ and then applying the five lemma. 
\end{proof}

\bibliographystyle{plain}
\bibliography{references} 

\end{document}